\documentclass[11pt]{amsart}
\usepackage[leqno]{amsmath}
\usepackage{amssymb}
\usepackage{enumerate}
\usepackage{subfigure}
\usepackage{graphicx}
\usepackage{hyperref}
\usepackage{mathtools}
\usepackage{mathrsfs}
\usepackage{mathabx}
\usepackage[all]{xy}
\usepackage[usenames,dvipsnames]{xcolor}

\usepackage{microtype}
\usepackage{stmaryrd}

\usepackage{comment}

\usepackage{bm}

\numberwithin{equation}{section}

\newtheoremstyle{boldremark} 
    {3pt} 
    {3pt} 
    {} 
    {} 
    {\bfseries} 
    {.} 
    {.5em} 
    {} 
    
\newtheorem{theorem}{Theorem}[section]
\newtheorem{lemma}[theorem]{Lemma}
\newtheorem{proposition}[theorem]{Proposition}

\theoremstyle{definition}
\newtheorem{defn}[theorem]{Definition}
\theoremstyle{boldremark}
\newtheorem{rem}[theorem]{Remark}

\newcommand{\C}{\mathbf{C}}
\newcommand{\D}{\mathbf{D}}
\newcommand{\E}{\mathbf{E}}

\newcommand{\N}{\mathbf{N}}
\newcommand{\Z}{\mathbf{Z}}
\newcommand{\p}{\mathbf{P}}
\newcommand{\Q}{\mathbf{Q}}
\newcommand{\R}{\mathbf{R}}

\newcommand{\Fh}{\mathfrak {h}}

\newcommand{\SLE}{{\rm SLE}}

\newcommand{\wt}{\widetilde}
\newcommand{\wh}{\widehat}

\newcommand{\giv}{\,|\,}

\newcommand{\len}{{\mathrm {len}}}

\begin{document}

\title[]{Continuity in $\xi$ and $\xi \to \infty$ tightness of the LQG metric}
\author[]{Konstantinos Kavvadias}
\address{Konstantinos Kavvadias, Department of Mathematics, Massachusetts Institute of Technology, Cambridge, MA, USA}
\email{kavva941@mit.edu}

\begin{abstract}
We show that the law of the $\gamma$-LQG metric (appropriately renormalized) is continuous in $\gamma \in (0,2)$ with respect to the local uniform topology of metrics on $\C \times \C$ whenever $\gamma$ lies on compact subsets of $(0,2)$.   Moreover we show that as $\gamma \to 0$,  the $\gamma$-LQG metric (appropriately renormalized) converges to the Euclidean metric with respect to the local uniform topology of metrics on $\C \times \C$.  More generally,  we show that the law of the LQG metric with parameter $\xi>0$ (appropriately renormalized) is tight with respect to the topology on lower semicontinuous functions on $\C \times \C$ whenever $\xi$ lies on compact subsets of $(0,\infty)$,  and any subsequential limit in law is non-trivial almost surely.  If in addition we assume that the limit satisfies the triangle inequality almost surely,  then it has the law of an LQG metric with an appropriate parameter $\xi$.  Finally we examine the limit as $\xi \to \infty$,  which is a regime that has not been studied before.  More precisely we show that if $D_h^{\xi}$ denotes the LQG metric with parameter $\xi>0$ (appropriately renormalized) associated with the whole-plane GFF $h$,  the family of metrics $(D_h^{\xi})^{1 / \xi}$ is tight as $\xi \to \infty$ and any subsequential limit is non-trivial almost surely.  If in addition we assume that the subsequential limit satisfies the triangle inequality almost surely,  then the limit is a metric almost surely.
\end{abstract}

\date{\today}
\maketitle

\setcounter{tocdepth}{1}
\tableofcontents

\parindent 0 pt
\setlength{\parskip}{0.2cm plus1mm minus1mm}

\section{Introduction}
\label{sec:intro}

\subsection{Overview}
\label{subsec:overview}
Liouville quantum gravity (LQG) surfaces were first introduced in the physics literature by Polyakov \cite{polyakov1981quantum,polyakov1996quantum}.  More precisely,  we fix $\gamma \in (0,2)$,  an open and connected set $U \subseteq \C$,  and let $h$ be the Gaussian free field (GFF) on $U$.  Then the \emph{$\gamma$-Liouville quantum gravity (LQG)} surface parameterized by $(U,h)$ is formally the random two-dimensional Riemannian manifold with metric tensor
\begin{align}\label{eqn:lqg_metric_tensor}
e^{\gamma h} (dx^2 + dy^2)
\end{align}
where $dx^2 + dy^2$ denotes the Euclidean metric tensor.

Apart from the physics literature,  the study of LQG surfaces is motivated by the description of scaling limits of certain random planar maps.  Different choices of the parameter $\gamma$ correspond to different families of measures on planar maps.  For instance,  the case $\gamma = \sqrt{\frac{8}{3}}$ corresponds to uniformly random planar maps such as uniform triangulations,  quadrangulations,  etc,  while the case $\gamma \neq \sqrt{\frac{8}{3}}$ corresponds to random planar maps weighted by the partition functions of certain statistical mechanics models on the map such as the uniform spanning tree $(\gamma = \sqrt{2})$ or the Ising model $(\gamma = \sqrt{3})$.

Note that the definition \eqref{eqn:lqg_metric_tensor} does not make literal sense since $h$ is not well-defined as a function taking pointwise values but rather as a distribution (generalized function).  However it has been shown that \eqref{eqn:lqg_metric_tensor} can be associated with a random volume form $e^{\gamma h(z)} dz$ where $dz$ denotes the Lebesgue measure via regularization procedures (see \cite{kahane1985chaos,ds2011lqgkpz,rhodes2014gaussian}).  

One can also construct a metric $D_h$ on $U$ associated with \eqref{eqn:lqg_metric_tensor}.  This metric should be considered as the candidate for the scaling limits of the graph distance on random planar maps with respect to the Gromov-Hausdorff topology.   The $\gamma$-LQG metric associated with \eqref{eqn:lqg_metric_tensor} was first constructed in the case when $\gamma = \sqrt{\frac{8}{3}}$ in \cite{ms2020qle1,ms2021qle2,ms2021qle3},  by Miller and Sheffield.  The construction of the $\sqrt{\frac{8}{3}}$-LQG metric metric is not based on a direct regularization procedure but instead it uses a metric growth process called quantum Loewer evolution (QLE) which is constructed using $\SLE_6$.  The resulting metric space for $\gamma = \sqrt{\frac{8}{3}}$ has the same law with the Brownian map constructed in \cite{lg2013bm,m2013bm},  or the Brownian disk constructed in \cite{bettinelli2017compact}.  For general $\gamma \in (0,2)$,  the LQG metric was constructed in \cite{gwynne2021existence} (see also \cite{ding2020tightness2,dubedat2020weak}) using a regularization procedure.

Let us now describe the construction of the $\gamma$-LQG metric for general $\gamma \in (0,2)$.  For $s>0$ and $z,w \in \C$,  we let 
\begin{align*}
p_s(z,w) = \frac{1}{2\pi s} \exp\left(-\frac{|z-w|^2}{2s}\right)
\end{align*}
be the transition kernel of the two-dimensional Brownian motion.  For $\epsilon > 0$,  we consider a mollified version of the GFF $h$ by 
\begin{align}\label{eqn:gff_molification}
h_{\epsilon}^*(z):=(h \ast p_{\epsilon^2 / 2})(z) = \int_U h(w) p_{\epsilon^2 / 2}(z,w)dw = (h,p_{\epsilon^2 / 2}(z,\cdot)) \quad \text{for all} \quad z \in U,
\end{align}
where we interpret the integral in the distributional sense.  Also for all $\xi>0$,  we define
\begin{align}\label{eqn:lffp_definition}
D_h^{\epsilon}(z,w):=\inf_P \left\{\int_0^1 e^{\xi h_{\epsilon}^*(P(t))} |P'(t)| dt : P \,\,\text{is a left-right crossing of} \,\,[0,1]^2\right\} \quad \text{for all} \quad z, w \in \C
\end{align}
where the infimum is over all piecewise continuously differentiable paths $P : [0,1] \to \C$ from $z$ to $w$.  The metrics $D_h^{\epsilon}$ are called \emph{$\epsilon$-Liouville first passage percolation} (LFPP).                      

We would like to renormalize $D_h^{\epsilon}$ and then take a limit as $\epsilon \to 0$ in the case that $h$ has the law of a whole-plane GFF plus a bounded continuous function for an appropriate choice of $\xi$.  Let us briefly explain heuristically what the choice of $\xi$ should be.  Recall that scaling areas by $C>0$ with respect to the volume form $e^{\gamma h(z)} dz$ associated with \eqref{eqn:lqg_metric_tensor} corresponds to replacing the field $h$ by $h+\frac{1}{\gamma} \log(C)$.  Moreover by \eqref{eqn:lffp_definition},  we obtain that replacing $h$ by $h+\frac{1}{\gamma} \log(C)$ scales distances by a factor of $C^{\frac{\xi}{\gamma}}$ and hence $\frac{\xi}{\gamma}$ should be the scaling exponent relating areas and distances.  Therefore the exponent $\frac{\gamma}{\xi}$ should be interpreted as the Hausdorff dimension of a $\gamma$-LQG surface.

It was shown in \cite{ding2020fractal,ding2019heat} that for all $\gamma \in (0,2)$,  there is an exponent $d_{\gamma} > 2$ which describes distances in various approximations of $\gamma$-LQG.  Furthermore,  it was shown in \cite{ding2020fractal} that $d_{\gamma}$ is the ball volume exponent for certain random planar maps.  Therefore the above intuitive reasoning implies that we should have that
\begin{align}\label{eqn:xi_subcritical}
\xi = \frac{\gamma}{d_{\gamma}}.
\end{align}

For $\epsilon>0$,  we let $a_{\epsilon}$ be the median of 
\begin{align}\label{eqn:median_of_lffp_definition}
\inf\left\{\int_0^1 e^{\xi h_{\epsilon}^*(P(t))} |P'(t)| dt : P \,\,\text{is a left-right crossing of} \,\,[0,1]^2\right\}
\end{align}
in the case that $h$ is a whole-plane GFF normalized such that its average on the unit circle $\partial \D$ is equal to zero,  where a left-right crossing of $[0,1]^2$ is a piecewise continuously differentiable path in $[0,1]^2$ joining the left and right boundaries of $[0,1]^2$.

One of the main results of \cite{gwynne2021existence} is the following.

\begin{theorem}(\cite[Theorem~1.1]{gwynne2021existence})
\label{thm:convergence_of_lffp}
The random metrics $a_{\epsilon}^{-1} D_h^{\epsilon}$ converge in probability as $\epsilon \to 0$ with respect to the local uniform topology in $\C \times \C$ to a random metric in $\C$ which is a.s.  determined by $h$.
\end{theorem}

Moreover it was shown in \cite{gwynne2021existence} that the limit in Theorem~\ref{thm:convergence_of_lffp} satisfies certain list of axioms which uniquely characterize it modulo a multiplicative deterministic constant (see Section~\ref{sec:preliminaries}).

Once the limit in Theorem~\ref{thm:convergence_of_lffp} is constructed,  it is shown in \cite{gwynne2022kpz} that $d_{\gamma}$ is its Hausdorff dimension.  Note that the explicit value of $d_{\gamma}$ for general $\gamma \in (0,2)$ is not known except for the case that $\gamma = \sqrt{\frac{8}{3}}$.  In the latter case,  we have that $d_{\sqrt{\frac{8}{3}}} = 4$ and this corresponds to the almost sure Hausdorff dimension of the Brownian map.

It is shown in \cite[Proposition~1.7]{ding2020fractal} that $\gamma \to \frac{\gamma}{d_{\gamma}}$ is an increasing function in $(0,2]$.   In particular for $\xi > \frac{2}{d_2}$,  the scaling limit of LFPP metrics in \eqref{eqn:lffp_definition} does not correspond to $\gamma$-LQG for $\gamma \in (0,2)$.  Nevertheless,  the LFPP metrics make sense for all $\xi>0$.

\begin{defn}\label{def:lfpp}
We refer to LFPP with $\xi < \xi_{\text{crit}},  \xi = \xi_{\text{crit}}$ and $\xi > \xi_{\text{crit}}$ as the \emph{subcritical,}  \emph{critical,} and \emph{supercritical} phases,  respectively,  where we set $\xi_{\text{crit}}:=\frac{2}{d_2}$,
\end{defn}

LFPP for $\xi > \xi_{\text{crit}}$ corresponds to LQG with parameter $\gamma \in \C$ with $|\gamma| = 2$,  or equivalently with central charge in $(1,25)$.  We refer the reader to \cite{ding2022introduction,ding2020tightness,ding2023uniqueness} for further discussion.

It turns out that the limit in probability of $a_{\epsilon}^{-1} D_h^{\epsilon}$ as $\epsilon \to 0$ exists for all $\xi > 0$ whenever $h$ is a whole-plane GFF plus a bounded and continuous function.  However the limit is not a continuous metric for $\xi > \xi_{\text{crit}}$ but it is lower semicontinuous and takes infinite values.  This motivates the following definition of convergence of lower semicontinuous functions on $U \times U$,  where $U \subseteq \C$ is an open set.

\begin{defn}\label{def:lower_semicontinuous_topology}
We say that a sequence of lower semicontinuous functions $(f_n)$ converges to a function $f$ with respect to the lower semicontinuous topology if and only if the following hold.
\begin{enumerate}[(i)]
\item For every sequence of points $(z_n,w_n) \in U\times U$ converging to a point $(z,w) \in U \times U$,  we have that
\begin{align*}
f(z,w) \leq \liminf_{n \to \infty} f_n(z_n,w_n).
\end{align*}
\item For every point $(z,w) \in U \times U$,  there exists a sequence of points $(z_n,w_n) \in U \times U$ converging to $(z,w)$ for which 
\begin{align*}
f(z,w) = \lim_{n \to \infty} f_n(z_n,w_n).
\end{align*}
\end{enumerate}
\end{defn}

It is easy to see that if $f_n \to f$ as $n \to \infty$ in the sense of Definition~\ref{def:lower_semicontinuous_topology},  then $f$ is lower semicontinuous.  It is not hard to see (see \cite[Section~1.2]{ding2020tightness}) that the topology of the set of lower semicontinuous functions $U \times U \to \R \cup \{\pm \infty\}$ induced by the convergence in Definition~\ref{def:lower_semicontinuous_topology} is metrizable and the resulting metric space is separable and complete.

Then we have the following convergence result.

\begin{theorem}(\cite[Theorem~1.3]{ding2023uniqueness})
\label{thm:convergence_of_supercritical_lfpp}
Let $h$ be a whole-plane GFF plus a bounded continuous function. For each $\xi>0$,  the re-scaled LFPP metrics $a_{\epsilon}^{-1} D_h^{\epsilon}$ converge in probability as $\epsilon \to 0$ with respect to the topology of Definition~\ref{def:lower_semicontinuous_topology} to a limit $D_h$.  The limit $D_h$ is a random metric on $\C$,  except that it is allowed to take on infinite values.
\end{theorem}

As in \cite{gwynne2021existence},  it is shown in \cite{ding2023uniqueness} that the limiting metric in Theorem~\ref{thm:convergence_of_supercritical_lfpp} satisfies a certain list of axioms which uniquely characterize it modulo a multiplicative deterministic constant (see Section~\ref{sec:preliminaries}).

\subsection{Main results}
\label{subsec:main_results}

Let us now describe the main results of the paper.  Recall that from Theorems~\ref{thm:convergence_of_lffp} and ~\ref{thm:convergence_of_supercritical_lfpp},  for all $\xi>0$,  we have that the LFPP metrics $D_h^{\epsilon}$ with parameter $\xi$ converge in probability as $\epsilon \to 0$ to some random metric $D_h^{\xi}$ whenever $h$ is a whole-plane GFF plus a bounded continuous function.  It is then natural to ask whether the family of metrics $(D_h^{\xi_n})_{n \in \N}$ (properly re-normalized) is tight or it converges in law or in probability to some non-trivial limit whenever $(\xi_n)_{n \in \N}$ is a sequence in $(0,\infty)$ converging to some $\xi \in (0,\infty)$.  We will show that the above question has an affirmative answer by treating the subcritical and critical / supercritical regimes separately.

We note that it is a much easier problem to show that the volume form associated with \eqref{eqn:lqg_metric_tensor} is continuous in $\gamma \in (0,2)$.  Indeed we have by \cite[Proposition~3.37]{berestycki2024gaussian} that
\begin{align}\label{eqn:measure_convergence}
\int_S \epsilon^{\frac{\gamma^2}{2}} e^{\gamma h_{\epsilon}(z)} dz \to \int_S e^{\gamma h(z)} dz \quad \text{in} \quad \text{L}^q \quad \text{as} \quad \epsilon \to 0
\end{align}
for all $0 < q < \min\{\frac{4}{\gamma^2} , 2\}$ and all bounded Borel sets $S \subseteq \C$.  Note also that
\begin{align}\label{eqn:gff_covariance}
\text{Cov}(h_{\epsilon}(x) ,  h_{\epsilon}(y)) = \log\left(\frac{1}{|x-y| \vee \epsilon}\right) + O(1)
\end{align}
for all $x , y \in S$ and all $\epsilon \in (0,1)$,  where the term $O(1)$ is uniform in $x,y \in S,  \epsilon \in (0,1)$.  Hence \eqref{eqn:gff_covariance} implies that that there exists a constant $C \in (0,\infty)$ depending only on $S$ such that for all $x,y \in S,  \epsilon \in (0,1)$,  and all $0<\gamma_1 < \gamma_2 < 2$ with $\gamma_2 < 2\gamma_1$,  we have that
\begin{align}\label{eqn:gff_covariance_bounds}
\text{Cov}(\gamma_2 h_{\epsilon}(x) ,  \gamma_2 h_{\epsilon}(y)) &= \text{Cov}(\gamma_1 h_{\epsilon}(x) ,  \gamma_1 h_{\epsilon}(y)) + 2\gamma_1 (\gamma_2 - \gamma_1) \text{Cov}(h_{\epsilon}(x) , h_{\epsilon}(y))\notag \\
&+ (\gamma_2 - \gamma_1)^2 \text{Cov}(h_{\epsilon}(x) ,  h_{\epsilon}(y))\notag \\
&\leq \text{Cov}(\gamma_1 h_{\epsilon}(x) ,  \gamma_1 h_{\epsilon}(y)) + C + C(\gamma_2 - \gamma_1)^2.
\end{align}
Hence combining \eqref{eqn:gff_covariance_bounds} with Kahane's convexity inequality (see \cite[Theorem~3.19]{berestycki2024gaussian}),  and using the convergence in \eqref{eqn:measure_convergence} together with Kolmogorov's continuity criterion,  we obtain that the law of $\int_S e^{\gamma h(z)} dz$ is continuous in $\gamma \in (0,2)$.  However a similar reasoning does not apply for the metrics $D_h^{\epsilon}$ defined in \eqref{eqn:lffp_definition} since they are not defined as a direct regularization of the field $h$ due to the presence of the infimum over all paths $P : [0,1] \to \C$ from $z$ to $w$ which are piecewise continuously differentiable.

We have the following result in the subcritical case.

\begin{theorem}\label{thm:main_theorem_subcritical_intro}
Let $h$ be a whole-plane GFF normalized so that $h_1(0) = 0$.  For all $\gamma \in (0,2)$,  we denote the metric $D_h^{\frac{\gamma}{d_{\gamma}}}$ by $D_h^{\gamma}$,  and let $\beta(\gamma) \in (0,\infty)$ be such that
\begin{align*}
\p\left[D_h^{\gamma}(0,1) \leq \beta(\gamma) \right] = \frac{1}{2}
\end{align*}
and set $\wh{D}_h^{\gamma}:=\beta(\gamma)^{-1} D_h^{\gamma}$. Then the following is true.  Fix $\gamma \in (0,2)$ and let $(\gamma_n)_{n \in \N}$ be a sequence in $(0,2)$ such that $\gamma_n \to \gamma$ as $n \to \infty$.  Then we have that
\begin{align*}
\wh{D}_h^{\gamma_n} \to \wh{D}_h^{\gamma} \quad \text{as} \quad n \to \infty
\end{align*}
in probability with respect to the local uniform topology in $\C \times \C$.
\end{theorem}

We note that the re-normalization in Theorem~\ref{thm:main_theorem_subcritical_intro} is similar to the re-normalization of the LFPP metrics except that we normalize with respect to the median of the LQG distance between $0$ and $1$ instead of normalizing with respect to the median of the minimum length of the left-right crossing of $[0,1]^2$.

Let us now focus on the critical / supercritical regime.   We have the following result for general $\xi>0$.

\begin{theorem}\label{thm:main_theorem_general_case_intro}
Let $h$ be a whole-plane GFF normalized so that $h_1(0) = 0$.  For all $\xi>0$,  we let $\beta(\xi) \in (0,\infty)$ be such that
\begin{align*}
\p\left[D_h^{\xi}(0,1) \leq \beta(\xi)\right] = \frac{1}{2}
\end{align*}
and set $\wh{D}_h^{\xi}:=\beta(\xi)^{-1} D_h^{\xi}$.  Then the following is true. Fix $\xi>0$ and let $(\xi_n)_{n \in \N}$ be a sequence in $(0,\infty)$ such that $\xi_n \to \xi$ as $n \to \infty$.  Then the sequence of metrics $(\wh{D}_h^{\xi_n})_{n \in \N}$ is tight with respect to the topology of Definition~\ref{def:lower_semicontinuous_topology}.  Moreover if $\wh{D}$ is any subsequential limit in law,  we have that $\wh{D}$ is non-trivial in the sense that $\wh{D}(z,w) \in (0,\infty)$ a.s.  for any fixed distinct and  deterministic points $z,w \in \C$.  If in addition we assume that $\wh{D}$ satisfies the triangle inequality a.s.,  then the following is true.  If $(\xi_{k_n})_{n \in \N}$ is a subsequence of $(\xi_n)_{n \in \N}$ such that 
\begin{align*}
\wh{D}_h^{\xi_{k_n}} \to \wh{D} \quad \text{in law} \quad \text{as} \quad n \to \infty,
\end{align*}
then we have that
\begin{align*}
\wh{D}_h^{\xi_{k_n}} \to \wh{D}_h^{\xi} \quad \text{in probability} \quad \text{as} \quad n \to \infty.
\end{align*}
\end{theorem}

It remains to treat the cases where $\gamma \to 0$ and $\xi \to \infty$.  Note that in the subcritical regime,  it is shown in \cite[Theorem~1.2]{ding2020fractal} that $d_{\gamma} = 2 + o_{\gamma}(1)$ as $\gamma \to 0$.  This suggests that the metric $D_h^{\gamma}$ should converge (properly re-scaled) to the Euclidean metric as $\gamma \to 0$.  Indeed,  this is the content of our next main result.

\begin{theorem}\label{thm:main_theorem_gamma_to_zero_intro}
Let $h$ be a whole-plane GFF normalized so that $h_1(0) = 0$,  and for all $\gamma \in (0,2)$,  we let $\beta(\gamma) \in (0,\infty)$ be as in Theorem~\ref{thm:main_theorem_subcritical_intro}
and set $\wh{D}_h^{\gamma}:= \beta(\gamma)^{-1} D_h^{\gamma}$.  Then if $(\gamma_n)_{n \in \N}$ is any sequence in $(0,2)$ such that $\gamma_n \to 0$ as $n \to \infty$,  we have that $(\wh{D}_h^{\gamma_n})_{n \in \N}$ converges in probability as $n \to \infty$ with respect to the local uniform topology on $\C \times \C$ to the Euclidean metric.
\end{theorem}

Finally we focus on the limit as $\xi \to \infty$.  Recall that it follows from  \cite[Proposition~1.8]{dubedat2020weak} that for all $r>0$ and fixed disjoint,  compact and connected sets $K_1, K_2 \subseteq \C$ which are not singletons,  we have that 
\begin{align*}
A^{-1} r^{\xi Q(\xi)} e^{\xi h_r(0)} \leq D_h^{\xi}(r K_1 ,  rK_2) \leq A r^{\xi Q(\xi)} e^{\xi h_r(0)}
\end{align*}
with superpolynomially high probability at $A \to \infty$ (at a rate which is uniform in $r$) for some continuous function $Q : (0,\infty) \to (0,\infty)$.  This suggests that in order to obtain a non-trivial limit as $\xi \to \infty$,  we need to consider the metric $D_h^{\xi}$ (properly re-normalized) and take its $\frac{1}{\xi}$th power.  This is the content of our last main result.

We introduce some notation first before stating our last main result.

For $z \in \C$ and $r>0$,  we write $B_r(z)$ for the open Euclidean ball of radius $r$ centered at $z$.  Moreover we define the open annulus
\begin{align}\label{eqn:around_annulus_definition}
\mathbb{A}_{r_1,r_2}(z):=B_{r_2}(z) \setminus \overline{B_{r_1}(z)} \quad \text{for all} \quad 0 < r_1 < r_2 < \infty.
\end{align}

Next we recall the definitions of distances across and around annuli with respect to a metric.

\begin{defn}\label{def:distances_across_and_around}
Let $d$ be a metric on $\C$ and let $A \subseteq \C$ be a Euclidean annulus.  We define $d(\text{across} \,\,  A)$ to be the $d$-distance between the inner and outer boundaries of $A$,  and we define $d(\text{around} \,\,  A)$ to be the infimum of the $d$-lengths of paths in $A$ which disconnect the inner from the outer boundaries of $A$.
\end{defn}

\begin{theorem}\label{thm:main_thm_xi_to_infty_intro}
There exists a deterministic constant $\mathfrak{p} \in (0,1)$ such that the following is true.  Let $h$ be a whole-plane GFF normalized such that $h_1(0) = 0$,  and for all $\xi \in [1,\infty)$,  we let $\alpha(\xi) \in (0,\infty)$ be such that
\begin{align*}
\p\left[D_h^{\xi}(\text{around} \,\,  \mathbb{A}_{1,2}(0)) \leq \alpha(\xi) \right] = \mathfrak{p}
\end{align*}
and set $\wt{D}_h^{\xi}:=(\alpha(\xi)^{-1} D_h^{\xi})^{1/\xi}$.  Let also $(\xi_n)_{n \in \N}$ be any sequence in $[1,\infty)$ such that $\xi_n \to \infty$ as $n \to \infty$.  Then the sequence of metrics $\{\wt{D}_h^{\xi_n}\}_{n \in \N}$ is tight with respect to the topology of Definition~\ref{def:lower_semicontinuous_topology}.  Moreover if $\wt{D}$ is any subsequential limit in law,  we have that $\wt{D}$ is non-trivial in the sense that $\wt{D}(z,w) > 0$ a.s.  for any fixed distinct and deterministic points $z,w \in \C$,  and $\wt{D} \not \equiv \infty$ a.s.  If in addition we assume that $\wt{D}$ satisfies the triangle inequality a.s. ,  then we have that $\wt{D}$ is a metric a.s.
\end{theorem}

We note that it was one of the open problems in \cite{ding2023uniqueness} whether the metric $D_h^{\xi}$,  appropriately re-scaled,  converges in some topology as $\xi \to \infty$ (see \cite[Problem~1.10]{ding2023uniqueness}).  Theorem~\ref{thm:main_thm_xi_to_infty_intro} gives a partial solution to the above question.

\subsection{Outline}
\label{subsec:outline}

Here we will give a rough outline of the content and the main ideas of the rest of the paper.

In Section~\ref{sec:preliminaries},  we will recall some preliminary facts about GFFs and LQG metrics.

In Section~\ref{sec:tightness},  we will show that the family of metrics $\wh{D}_h^{\gamma}$ in Theorem~\ref{thm:main_theorem_subcritical_intro} is tight whenever $\gamma$ lies in a compact subset of $(0,2)$.  More precisely,  for fixed $0 < \gamma_* < \gamma^* < 2$,  we will show that the family of random metrics $(\wh{D}_h^{\gamma})_{\gamma \in [\gamma_* ,  \gamma^*]}$ is tight with respect to the local uniform topology on $\C \times \C$.  This will be done in two main steps.  The first step is to prove an upper bound on the $\gamma$-LQG distance between sets which is uniform in $\gamma \in [\gamma_* ,  \gamma^*]$.  More precisely,  we will show that for any two fixed disjoint compact and connected sets $K_1 ,  K_2 \subseteq \C$ which are not singletons,  we have that
\begin{align}\label{eqn:main_bound_subcritical}
\wh{D}_h^{\gamma}(r K_1 ,  r K_2) \leq A r^{\xi Q} e^{\xi h_r(0)}
\end{align}
with polynomially high probability as $A \to \infty$,  at a rate which is uniform in $r$ and $\gamma \in [\gamma_* ,  \gamma^*]$,  where we set 
\begin{align*}
\xi = \frac{\gamma}{d_{\gamma}} \quad \text{and} \quad Q = Q\left(\frac{\gamma}{d_{\gamma}}\right) = \frac{2}{\gamma} + \frac{\gamma}{2}.
\end{align*}

In the second step,  we will use \eqref{eqn:main_bound_subcritical} to obtain a H\"older continuity estimate which is uniform in $\gamma \in [\gamma_* ,  \gamma^*]$.  In particular,  we will show that for all $r>0$ and every compact set $K \subseteq \C$,  we have that
\begin{align}\label{eqn:main_holder_cont_estimate_subcritical}
\wh{D}_h^{\gamma}(u,v) \leq \left|\frac{u-v}{r}\right|^{\chi} r^{\xi Q} e^{\xi h_r(0)} \quad \text{for all} \quad u,v \in r K \quad \text{with} \quad |u-v| \leq \epsilon r
\end{align}
with polynomially high probability as $\epsilon \to 0$ at a rate which is uniform in $r$ and $\gamma \in [\gamma_* ,  \gamma^*]$,  and the exponent $\chi$ depends only on $\gamma_*$ and $\gamma^*$.  Thus the tightness will follow from \eqref{eqn:main_holder_cont_estimate_subcritical}    combined with the Arzela-Ascoli theorem.

In Section~\ref{sec:identifying_limit_subcritical},  we will identify any subsequential limit $\wh{D}$ from Section~\ref{sec:tightness} as the limiting metric in Theorem~\ref{thm:convergence_of_lffp} (modulo a multiplicative deterministic constant) and hence complete the proof of Theorem~\ref{thm:main_theorem_subcritical_intro}.

As a first step,  we will show that $\wh{D}$ is a metric almost surely.  The argument goes roughly as follows and it is similar to the argument in \cite[Section~6.2]{ding2023tightness}.  First we note that by scaling and translating,  it suffices to show that
\begin{align*}
\p\left[\wh{D}(\text{across} \,\,  \mathbb{A}_{1,2}(0))>0 \right] = 1.
\end{align*}
We set $\mathcal{Z} = \{\wh{D}(\text{across} \,\,  \mathbb{A}_{1,2}(0)) > 0\}$ and 
\begin{align*}
E_1 = \{x \in \C : \wh{D}(x,\partial B_1(0)) = 0\} \quad \text{and} \quad E_2 = \{x \in \C : \wh{D}(x,\partial B_2(0)) = 0\}.
\end{align*}
We will show that for any fixed disjoint open sets $U_1 ,  U_2 \subseteq \C$,  we have that the events $\{E_1 \subseteq U_1\}$ and $\{E_2 \subseteq U_2\}$ are independent.  Thus combining with the fact that
\begin{align}\label{eqn:main_equality_limit_metric}
\mathcal{Z}= \{E_1 \cap E_2 = \emptyset\} = \{E_1 \cap \partial B_2(0) = \emptyset\} = \{E_2 \cap \partial B_1(0) = \emptyset\}.
\end{align}
(by the triangle inequality for $\wh{D}$) and summing over a suitable collection of possible choices of $U_1$ and $ U_2$,  we obtain that
\begin{align*}
\p[\mathcal{Z}] = \p[E_1 \cap E_2 = \emptyset] \leq \p[E_1 \cap \partial B_2(0) = \emptyset] \p[E_2 \cap \partial B_1(0) = \emptyset] = \p[\mathcal{Z}]^2
\end{align*}
and hence $\p[\mathcal{Z}] \in \{0,1\}$.  Moreover we will show that $\p[\mathcal{Z}] > 0$ and hence $\p[\mathcal{Z}] = 1$.  This will prove that $\wh{D}$ is a metric almost surely.

As a second step,  in order to identify the limit,  we will show that $\wh{D}$ satisfies the axioms of $\gamma$-LQG metrics (see Section~\ref{sec:preliminaries}).  Therefore since these axioms uniquely characterize the law of the metric (modulo a multiplicative deterministic constant) and the limiting metric in Theorem~\ref{thm:convergence_of_lffp} satisfies these axioms (see \cite[Theorem~1.2]{gwynne2021existence}),  this will complete the proof of Theorem~\ref{thm:main_theorem_subcritical_intro}.  

In Section~\ref{sec:gamma_to_zero_case},  we will prove that the metrics $\wh{D}_h^{\gamma}$ in the statement of Theorem~\ref{thm:main_theorem_gamma_to_zero_intro} are tight as $\gamma \to 0$.  In order to show this,  we will use the same argument as in Section~\ref{sec:tightness}.  However,  since the way that we choose to re-normalize the metrics $D_h^{\gamma}$ will be different,  we will obtain slightly different estimates.  We will emphasize on these differences.

In Section~\ref{sec:limit_euclidean_metric},  we will identify the limit as $\gamma \to 0$ as the Euclidean metric (modulo a global multiplicative constant) and hence complete the proof of Theorem~\ref{thm:main_theorem_gamma_to_zero_intro}. This will be achieved via a $0-1$ law type argument.  In particular,  we will write
\begin{align*}
h = \lim_{n \to \infty}\left( \sum_{m=1}^n a_m f_m \right)
\end{align*}
where the limit is considered in the space of distributions (generalized functions),  $(a_n)_{n \in \N}$ is an i.i.d.  sequence of normal Gaussians,  and $(f_n)_{n \in \N}$ is an orthonormal basis of the Sobolev space $H_0(\C)$ with respect to the Dirichlet inner product.   Then we will show that for all $m \in \N$,  the metrics $\wh{D}_{h-h_m}^{\gamma}$ converge to the same limit in law as $\gamma \to 0$.  In particular we can find a coupling of $h$ with $\wh{D}$ in the same probability space so that $\wh{D}$ is $\bigcap_{m \geq 1} \sigma\left(\{a_n : n \geq m\}\right)$-measurable.  But the latter $\sigma$-algebra is trivial by Kolmogorov's $0-1$ law and hence $\wh{D}$ has to be deterministic almost surely from which Theorem~\ref{thm:main_theorem_gamma_to_zero_intro} follows easily.

In Section~\ref{sec:tightness_general_case},  we will prove tightness of the metrics $(\wh{D}_h^{\xi})_{\xi > 0}$ in Theorem~\ref{thm:main_theorem_general_case_intro} whenever $\xi$ lies in compact subsets of $(0,\infty)$.  The argument will be similar as that of Section~\ref{sec:tightness}.  In particular we will first prove tightness of the metrics $D_h^{\xi}$ when they are re-normalized in a different way from the metrics $\wh{D}_h^{\xi}$ appearing in the statement of Theorem~\ref{thm:main_theorem_general_case_intro} and then deduce tightness of the metrics $\wh{D}_h^{\xi}$.  More precisely,   we fix a function $\mathfrak{p}_0 : (0,\infty) \to (0,1)$ and for all $\xi > 0$,  we let $\alpha(\xi) \in (0,\infty)$ be such that
\begin{align*}
\p\left[D_h^{\xi}(\text{around} \,\,\mathbb{A}_{1,2}(0)) \leq \alpha(\xi) \right] = \mathfrak{p}_0(\xi).
\end{align*}

The main idea behind our normalization is the following.  For all $\xi>0$,  if we choose $\mathfrak{p}_0(\xi) \in (0,1)$ sufficiently close to $1$ and set 
\begin{align}\label{eqn:normalization_for_general_xi}
\wt{D}_h^{\xi}:=\alpha(\xi)^{-1} D_h^{\xi},
\end{align}
then the collection of random metrics $(\wt{D}_h^{\xi})_{\xi > 0}$ is tight when indexed by compact subsets of $(0,\infty)$.  The intuitive reason for this is the following.   First we note that combining the facts that the metric $\wt{D}_h^{\xi}$ is locally determined by $h$ with the independence across scales property of $h$ (\cite[Lemma~3.1]{gwynne2020local}),  we obtain that if we choose $\mathfrak{p}_0(\xi) \in (0,1)$ sufficiently close to $1$,  then we can cover the space with very high probability by annuli $\mathbb{A}_{r,2r}(z)$ with the property that 
\begin{align}\label{eqn:upper_bound_on_distance_around}
\wt{D}_h^{\xi}(\text{around} \,\,  \mathbb{A}_{r,2r}(z)) \leq r^{\xi Q(\xi)} e^{\xi h_r(0)}.
\end{align}
Thus by concatenating paths around $\mathbb{A}_{r,2r}(z)$ satisfying \eqref{eqn:upper_bound_on_distance_around},  we obtain an upper bound with high probability of the $\wt{D}_h^{\xi}$-distance between any two fixed and deterministic compact and connected sets $K_1 ,  K_2 \subseteq \C$.

We will show that for all $r>0,  \xi^* > 0$,  and compact sets $K_1 ,  K_2 \subseteq \C$ as in \eqref{eqn:main_bound_subcritical},  we have that
\begin{align*}
\wt{D}_h^{\xi}(r K_1 ,  r K_2) \leq A r^{\xi Q(\xi)} e^{\xi h_r(0)}
\end{align*}
with polynomially high probability as $A \to \infty$,  at a rate which is uniform in $\xi \in (0,\xi^*]$ and $r$.  However since $\wt{D}_h^{\xi}$ can take infinite values,  we have that \eqref{eqn:main_holder_cont_estimate_subcritical} does not hold.  In order to obtain a subsequential limit,  we will use the same construction as in \cite[Section~5.2]{ding2020tightness}.  More precisely if $(\xi_n)_{n \in \N}$ is a sequence in $(0,\infty)$ which converges to some $\xi>0$ as $n \to \infty$,  we will choose a subsequence $(\xi_{k_n})_{n \in \N}$ along which certain countable collection of functionals of the metrics $\wt{D}_h^{\xi_{k_n}}$ converge.  This will include the $\wt{D}_h^{\xi_{k_n}}$-distances across and around annuli whose boundaries are circles with rational radii centered at rational points.  We will then use these estimates to construct a subsequential limit in law $\wt{D}$ of $(\wt{D}_h^{\xi_{k_n}})_{n \in \N}$.

In Section~\ref{lem:identifying_limit_general_case},  we will complete the proof of Theorem~\ref{thm:main_theorem_general_case_intro}.  In order to prove that the subsequential limits for general $\xi>0$ are metrics almost surely,  we will follow the same argument as in the subcritical case (Section~\ref{sec:identifying_limit_subcritical}).  However there is one major difference and that is the reason why we make the assumption that the limit satisfies the triangle inequality almost surely.  More precisely,  suppose that we define the sets $E_1$ and $E_2$ as in \eqref{eqn:main_equality_limit_metric}.  Recall that in the subcritical case,  we have that
\begin{align*}
\{E_1 \cap \partial B_2(0) = \emptyset\} = \{E_2 \cap \partial B_1(0) = \emptyset\} = \{E_1 \cap E_2 = \emptyset\}
\end{align*}
since the triangle inequality is preserved under limits with respect to the local uniform topology on $\C \times \C$.  However since  this is not always true when we consider limits with respect to the lower semicontinuous topology,  we have to assume that the limit satisfies the triangle inequality almost surely so that  \eqref{eqn:main_equality_limit_metric} holds.  The rest of the argument is then similar to the subcritical case.

One could try to use the argument in \cite[Section~5.4]{ding2020tightness} in order to prove that any subsequential limit in law $\wt{D}$ of $\wt{D}_h^{\xi_n}$ satisfies the triangle inequality a.s. 	However the aforementioned argument wouldn't work in our case for the following reason.  Recall that in order to be able to apply \cite[Lemma~5.13]{ding2020tightness},  we need to know that there exists a deterministic constant $C>0$ such that for all $z \in \C,  r>0,  n \in \N$ fixed,  we have with sufficiently high probability (uniformly in $z,r$ and $n$) that
\begin{align*}
\wt{D}_h^{\xi_n}(\text{around} \,\,  \mathbb{A}_{r,2r}(z)) \leq C \wt{D}_h^{\xi_n}(\text{across} \,\,  \mathbb{A}_{r,2r}(z)).
\end{align*}
However by the way that we have chosen to re-normalize (so that $\wt{D}_h^{\xi_n}(\text{around} \,\,  \mathbb{A}_{1,2}(0)) \leq 1$ with high probability),  it is not clear whether $\wt{D}_h^{\xi_n}(\text{across} \,\,  \mathbb{A}_{1,2}(0))$ is not too small with sufficiently high probability so that we can justify the existence of the constant $C$. If we could prove that such $C$ exists,  then the argument in \cite[Section~5.4]{ding2020tightness} would work in our setup.

Next we will show that $\wt{D}$ satisfies the axioms of LQG metrics with parameter $\xi$ (see Section~\ref{sec:preliminaries}).  As in Section~\ref{sec:identifying_limit_subcritical},  since these axioms uniquely characterize the law of the metric (modulo a multiplicative deterministic constant) and the limiting metric in Theorem~\ref{thm:convergence_of_supercritical_lfpp} satisfies these axioms (see \cite[Theorem~1.4,  Lemma~1.13 and Lemma~1.14]{ding2023uniqueness}),  this will complete the proof of Theorem~\ref{thm:main_theorem_general_case_intro}.

In Section~\ref{sec:xi_to_infty},  we will complete the proof of Theorem~\ref{thm:main_thm_xi_to_infty_intro}.  The argument will be similar to the arguments in Section~\ref{sec:tightness_general_case}.  However since we re-normalize the metrics in a different way as $\xi \to \infty$,  there will be some differences compared to the arguments in Section~\ref{sec:tightness_general_case} that we are going to emphasize on and make clear.

Finally in Section~\ref{sec:apendix},  we will prove some facts that we are going to use about LQG metrics and convergence of random variables taking values on separable and complete metric spaces.

\subsection{Notation}
\label{subsec:notation}

We write $\N = \{1,2,3,\cdots\},  \N_0 = \N \cup \{0\}$,  and denote by $\C$ and $\Q$ the set of complex and rational numbers respectively.  We also denote by $\Z$ the set of integers and by $\R$ the set of real numbers. Moreover we will denote the probability of an event $E$ by $\p[E]$ and expectation with respect to $\p$ by $\E$.  We also abbreviate \emph{almost surely }by \emph{a.s.}

If $f : (0,\infty) \to \R$ and $g : (0,\infty) \to (0,\infty)$,  we say that $f(\epsilon) = O_{\epsilon}(g(\epsilon))$ (resp.  $f(\epsilon) = o_{\epsilon}(g(\epsilon))$) as $\epsilon \to 0$ if $\frac{f(\epsilon)}{g(\epsilon)}$ remains bounded (resp.  tends to zero) as $\epsilon \to 0$.  We similarly define $O(\cdot)$ and $o(\cdot)$ errors as parameters converge to infinity.

Let $\{E^{\epsilon}\}_{\epsilon > 0}$ be a one-parameter family of events.  We say that $E^{\epsilon}$ occurs with 
\begin{itemize}
\item \emph{polynomially high probability} as $\epsilon \to 0$ if there exists $p>0$ independent from $\epsilon$ such that $\p[E^{\epsilon}] \geq 1 - O_{\epsilon}(\epsilon^p)$.
\item \emph{superpolynomially high probability} as $\epsilon \to 0$ if $\p[E^{\epsilon}] \geq 1 - O_{\epsilon}(\epsilon^p)$ for every $p>0$.
\end{itemize}
We similarly define events which occur with polynomially and superpolynomially high probability as a parameter tends to $\infty$.

We write $\mathbb{S} = (0,1)^2$ for the open Euclidean unit square.  Also for a positive integer $N \in \N$,  we set $[1,N]_{\Z}:=[1,N] \cap \Z$ and we denote by $|\cdot|$ the Euclidean metric on $\C$.

For a metric $D : \C \times \C \to \R \cup \{\pm \infty\}$ and sets $A,B \subseteq \C$,  we define the distance between $A$ and $B$ with respect to $D$ by 
\begin{align*}
D(A,B):=\inf\{D(x,y) : x \in A,  y \in B\}.
\end{align*}

Finally for an open set $D \subseteq \C$,  we will denote by $C_0^{\infty}(D)$ the space of smooth and compactly supported functions $f : \C \to \R$ whose support is a subset of $D$.

\subsection*{Acknowledgements}
We would like to thank Ewain Gwynne for helpful comments on an earlier version of the draft.

\section{Preliminaries}
\label{sec:preliminaries}

In this section,  we will first introduce the zero-boundary and the whole-plane GFF and describe some of its main properties that we are going to use.  Next we will mention some results from \cite{gwynne2021existence,ding2023uniqueness},  which uniquely characterize (modulo a multiplicative deterministic constant) the limiting metrics in Theorems~\ref{thm:convergence_of_lffp} and ~\ref{thm:convergence_of_supercritical_lfpp} via certain lists of axioms.

\subsection{Gaussian free fields}
\label{subsec:gff}

Let $D \subseteq \C$ be a simply connected domain with harmonically non-trivial boundary and let $H_0(D)$ denote the Hilbert space closure of $C_0^{\infty}(D)$ with respect to the Dirichlet inner product
\begin{align*}
(f,g)_{\nabla} = \frac{1}{2\pi} \int_D \nabla f(z) \cdot \nabla g(z) dz.
\end{align*}
The \emph{zero-boundary Gaussian free field (GFF)} $h$ is the random distribution defined by 
\begin{align}\label{eqn:gff_definition}
h = \lim_{n \to \infty} \left( \sum_{m=1}^n a_m \phi_m \right)
\end{align}
where $(\phi_n)_{n \in \N}$ is a $(\cdot ,  \cdot)_{\nabla}$-orthonormal basis of $H_0(D)$,  $(a_n)_{n \in \N}$ is a sequence of independent $\mathcal{N}(0,1)$-distributed random variables,  and the limit in \eqref{eqn:gff_definition} is considered in the space of distributions $H^{-1}(D)$,  which denotes the dual space of $H_0(D)$.  

An important property of the GFF that we are going to use is the \emph{Markov property}.  It states that if $h$ is a zero-boundary GFF on $D$ and $U \subseteq D$ is a deterministic open set,  then the law of $h$ restricted to $U$,  conditional on the values of $h$ outside of $U$ is that of a zero-boundary GFF $h_U$ on $U$ plus the harmonic extension of its values on $\partial U$ to $U$.  Also the zero-boundary part and the harmonic part are independent.

Now we discuss about the whole-plane GFF (we will mostly work with this random field throughout the paper).  Let $H_0(\C)$ denote the Hilbert space closure with respect to $(\cdot ,  \cdot)_{\nabla}$ of the set of $f \in C_0^{\infty}(\C)$ such that $\int_{\C} f(z) dz = 0$.  We define the whole-plane GFF in the same way as we defined the zero-boundary GFF but with the orthonormal basis of $H_0(\C)$ instead in the sum in \eqref{eqn:gff_definition}.  The convergence in \eqref{eqn:gff_definition} is in the space $H_{\text{loc}}^{-1}(\C)$,  where we denote by $H_{\text{loc}}^{-1}(\C)$ the set of all generalized functions $\wt{h}$ on $\C$ such that $\wt{h}|_{B_R(0)} \in H^{-1}(B_R(0))$ for all $R>0$.  Moreover we can define a topology on $H_{\text{loc}}^{-1}(\C)$ by requiring that a sequence $(\wt{h}_n)_{n \in \N}$ in $H_{\text{loc}}^{-1}(\C)$ converges to some $\wt{h} \in H_{\text{loc}}^{-1}(\C)$ as $n \to \infty$ if and only if 
\begin{align*}
\wt{h}_n|_{B_R(0)} \to \wt{h}|_{B_R(0)} \quad \text{as} \quad n \to \infty \quad \text{in} \quad H^{-1}(B_R(0)) \quad \text{for all} \quad R>0.
\end{align*}

Note that the whole-plane GFF is defined modulo additive constant.  We usually fix the normalization by requiring that $h_1(0) = 0$.  We also note that the whole-plane GFF satisfies the Markov property as well.  Namely if $h$ is a whole-plane GFF and $U \subseteq \C$ is an open set,  we have that conditional on $h|_{\C \setminus U}$,  the field $h$ can be decomposed as the sum of a zero-boundary GFF on $U$ and the harmonic extension to $U$ of the values of $h$ on $\partial U$.

Finally we note that if $h$ is a whole-plane GFF,  then it is possible to define a version of $h$ such that for all $z \in \C$ and $r>0$,  the average of $h$ on $\partial B_r(z)$ is well-defined,  and we denote this quantity by $h_r(z)$ and call it the \emph{circle average} of $h$ on $\partial B_r(z)$.  Also there exists a modification of $h$ such that $(z,r) \to h_r(z)$ is a.s.  H\"older continuous on all the compact subsets of $\C \times (0,\infty)$ (see e.g.,  \cite[Proposition~1.58]{berestycki2024gaussian}).  Moreover we have that $h$ is scale and translation invariant in the sense that for every fixed $z \in \C,r>0$,  we have that the fields $h,h(\cdot + z) - h_1(z)$ and $h(r \cdot) - h_r(0)$ all have the same law.

\subsection{LQG metrics}
\label{subsec:lqg_metrics}

In order to prove in \cite{gwynne2021existence} that the limit in Theorem~\ref{thm:convergence_of_lffp} is unique,  the authors prove that for all $\gamma \in (0,2)$,  there is a unique (up to a multiplication by a deterministic positive constant) metric satisfying certain axioms.  We begin by giving some preliminary definitions.

\begin{defn}\label{def:metric_space}
Let $(X,d)$ be a metric space with $d$ allowed to take on infinite values.
\begin{itemize}
\item A \emph{path} in $(X,d)$ is a continuous function $P : [a,b] \to X$ for some interval $[a,b]$.
\item For a curve $P : [a,b] \to X$,  the \emph{$d$-length} of $P$ is defined by 
\begin{align*}
\len(P ; d) := \sup_T \sum_{i=1}^{|T|}  d(P(t_i) ,  P(t_{i-1}))
\end{align*}
where the supremum is over all partitions $T : a = t_0 < \cdots <t_{|T|} = b$ of $[a,b]$.  Note that the $d$-length of a curve can be infinite.
\item We say that $(X,d)$ is a \emph{length space} if for all $x,y \in X ,  \epsilon>0$,  there exist a curve with $d$-length at most $d(x,y) + \epsilon$ from $x$ to $y$.  If $d(x,y) < \infty$,  a curve in $X$ from $x$ to $y$ with $d$-length exactly $d(x,y)$ is called a \emph{geodesic}.  We say that $(X,d)$ is a \emph{geodesic space} if for all $x,y \in X$,  we have that $d(x,y) < \infty$ and there exists a geodesic in $X$ from $x$ to $y$.
\item For $Y \subseteq X$,  the \emph{internal metric} of $d$ on $Y$ is defined by 
\begin{align*}
d(x,y ; Y) := \inf_{P \subseteq Y} \len(P ; d), \quad \text{for all} \quad x,y \in Y
\end{align*}
where the infimum is over all curves $P$ in $Y$ from $x$ to $y$.  We note that $d(\cdot ,  \cdot ; Y)$ is a metric on $Y$.
\item If $X \subseteq \C$,  we say that $d$ is a \emph{lower semicontinuous metric} if the function $(x,y) \to d(x,y)$ is lower semicontinuous with respect to the Euclidean topology.  We equip the set of lower semicontinuous metrics on $X$ with the topology of lower semicontinuous functions on $X \times X$.
\end{itemize}
\end{defn}

We are now ready to state the axioms of a strong $\gamma$-LQG metric.

\begin{defn}\label{def:strong_lqg_metric_subcritical}
Let $\mathcal{D}'(\C)$ be the space of distributions (generalized functions) on $\C$,  equipped with the standard weak topology.  For $\gamma \in (0,2)$,  we say that a  \emph{strong $\gamma$-Liouville quantum gravity (LQG) metric} is a measurable function $h \to D_h$ from $\mathcal{D}'(\C)$ to the space of continuous metrics on $\C$ such that the following is true whenever $h$ is a whole-plane GFF plus a continuous function.

\begin{enumerate}[(I)]
\item \label{it:length_space}
\textbf{Length space}: A.s.,  $(\C,D_h)$ is a length space, i.e.,  the $D_h$-distance between any two points of $\C$ is the infimum of the $D_h$-lengths of $D_h$-continuous paths between the two points.
\item \label{it:locality}
\textbf{Locality}: Let $U \subseteq \C$ be a deterministic open set. The $D_h$-internal metric $D_h(\cdot , \cdot ; U)$ is a.s.  determined by $h|_U$.
\item \label{it:weyl_scaling}
\textbf{Weyl scaling}: For each continuous function $f : \C \to \R$,  we define
\begin{align*}
(e^{\xi f} \cdot D_h)(z,w):=\inf_{P : z \to w} \left \{ \int_0^{\text{len}(P;D_h)} e^{\xi f(P(t))} dt \right\} \quad \text{for all} \quad z,w \in \C,
\end{align*}
where the infimum is over all continuous paths from $z$ to $w$ parameterized by $D_h$-legth.  Then we have that $e^{\xi f} \cdot D_h = D_{h+f}$ for all continuous functions $f : \C \to \R$ a.s.
\item \label{it:coordinate_change}
\textbf{Coordinate change for translation and scaling}: For each $r>0,z \in \C$ deterministic,  a.s.
\begin{align*}
D_h(ru + z ,  rv + z) = D_{h(r \cdot + z) + Q \log(r)}(u,v) \quad \text{for all} \quad u,v \in \C
\end{align*}
where $Q = \frac{2}{\gamma} + \frac{\gamma}{2}$.
\end{enumerate}
\end{defn}

\begin{theorem}(\cite[Theorem~1.2]{gwynne2021existence})
\label{thm:uniqueness_subcritical}
Fix $\gamma \in (0,2)$.  There is a $\gamma$-LQG metric $D$ such that the limiting metric in Theorem~\ref{thm:convergence_of_lffp} is a.s.  equal to $D_h$ whenever $h$ is a whole-plane GFF plus a bounded and continuous function.  Furthermore the $\gamma$-LQG metric is unique in the following sense.  If $D$ and $\wt{D}$ are two $\gamma$-LQG metrics,  then there exists a deterministic constant $C \in (0,\infty)$ such that if $h$ is a whole-plane GFF plus a continuous function,  then a.s.  $D_h = C \wt{D}_h$. 
\end{theorem}

We are going to use Theorems~\ref{thm:uniqueness_subcritical} in Section~\ref{sec:identifying_limit_subcritical} in order to identify the limiting metric in Theorem~\ref{thm:main_theorem_subcritical_intro}.

Next we review the extension of the above results for general $\xi>0$. 

Let us first briefly explain why the limiting metric in Theorem~\ref{thm:convergence_of_supercritical_lfpp} takes infinite values almost surely (see \cite[Proposition~5.22]{ding2020tightness} for a rigorous proof). First we note that it is shown in \cite[Proposition~1.1]{ding2020tightness} that for all $\xi>0$,  there exists $Q = Q(\xi) > 0$ such that
\begin{align}\label{eqn:supercritical_lfpp_exponent}
a_{\epsilon} = \epsilon^{1-\xi Q + o_{\epsilon}(1)} \quad \text{as} \quad \epsilon \to 0.
\end{align}
It is also shown that $\xi \to Q(\xi)$ is a continuous function and $Q(\xi) \to \infty$ as $\xi \to 0$ and $Q(\xi) \to 0$ as $\xi \to \infty$.  Moreover we have that $\xi_{\text{crit}}$ is the unique value of $\xi$ for which $Q(\xi) = 2$ and $Q>2$ for $\xi < \xi_{\text{crit}}$ and $Q \in (0,2)$ for $\xi > \xi_{\text{crit}}$.

Recall that for $\xi > \xi_{\text{crit}}$ we have that there exist points $z \in \C$  such that
\begin{align*}
\limsup_{\epsilon \to 0} \frac{h_{\epsilon}(z)}{\log(\epsilon^{-1})} > Q
\end{align*}
since $Q \in (0,2)$.  Thus combining with the definition of $D_h^{\epsilon}$ (see \eqref{eqn:lffp_definition}) and \eqref{eqn:supercritical_lfpp_exponent},  we obtain (at least heuristically) that if $z \in \C$ is as above,  we have that
\begin{align}\label{eqn:infinite_distance}
D_h(z,w) = \infty \quad \text{for all} \quad w \in \C \setminus \{z\}.
\end{align}
\begin{defn}\label{def:singular_points}
The points $z \in \C$ for which \eqref{eqn:infinite_distance} holds are called \emph{singular points} for the metric $D_h$.
\end{defn}

Furthermore it is shown in \cite{hu2010thick} that almost surely
\begin{align*}
\limsup_{\epsilon \to 0} \frac{h_{\epsilon}(z)}{\log(\epsilon^{-1})} \in [-2,2]
\end{align*}
which explains the lack of singular points when $\xi \in (0,\xi_{\text{crit}}]$.  We also mention that in the subcritical and critical cases,  we have that
\begin{align}\label{eqn:parameter_Q}
Q\left(\frac{\gamma}{d_{\gamma}}\right) = \frac{2}{\gamma}+\frac{\gamma}{2} \quad \text{for all} \quad \gamma \in (0,2].
\end{align}

Now we give the definition of strong LQG metrics with parameter $\xi>0$. We will present results from \cite{ding2023uniqueness}.

\begin{defn}\label{def:strong_lqg_metric_general}
Fix $\xi>0$.  A \emph{(strong) LQG metric with parameter $\xi$} is a measurable function $h \to D_h$ from $\mathcal{D}'(\C)$ to the space of lower semicontinuous metrics on $\C$ such that the following holds whenever $h$ is a whole-plane GFF plus a continuous function on $\C$.  Axioms~\ref{it:length_space},  ~\ref{it:locality} and ~\ref{it:weyl_scaling} in Definition~\ref{def:strong_lqg_metric_subcritical} are satisfied and in addition we have that
\begin{enumerate}[(I)]
\setcounter{enumi}{3}
\item 
\textbf{Scale and translation invariance} \label{it:translation_and_scale_invariance}
Let $Q$ be as in \eqref{eqn:supercritical_lfpp_exponent}.  For each fixed and deterministic $r>0$ and $z \in \C$,  a.s.
\begin{align*}
D_h(r u + z ,  rv +z) = D_{h(r \cdot + z) + Q \log(r)}(u,v) \quad \text{for all} \quad u,v \in \C.
\end{align*}
\item
\textbf{Finiteness} \label{it:finiteness}
Let $U \subseteq \C$ be a deterministic,  open,  connected set and let $K_1,K_2 \subseteq U$ be disjoint,  deterministic,  compact,  connected sets which are not singletons.  Almost surely,  $D_h(K_1,K_2 ; U) < \infty$.
\end{enumerate}
\end{defn}

Then we have the following analogue of Theorem~\ref{thm:uniqueness_subcritical} for general $\xi>0$.

\begin{theorem}(\cite[Theorem~1.8]{ding2023uniqueness})
\label{thm:uniqueness_supercritical}
Fix $\xi>0$.  Then there is an LQG metric $D$ with parameter $\xi$ such that the limiting metric of Theorem~\ref{thm:convergence_of_supercritical_lfpp} is a.s.  equal to $D_h$ whenever $h$ is a whole-plane GFF plus a bounded continuous function.  Furthermore this LQG metric is unique in the following sense.  If $D$ and $\wt{D}$ are two LQG metrics with parameter $\xi$,  then there exists a deterministic constant $C \in (0,\infty)$ such that if $h$ is a whole-plane GFF plus a continuous function,  then a.s.  $D_h = C \wt{D}_h$. 
\end{theorem}

\begin{rem}
(Metrics associated with other fields).  Theorem~\ref{thm:uniqueness_supercritical} gives us a way to associate a canonical LQG metric with parameter $\xi$ with a whole-plane GFF plus a continuous function.

We can also associate metrics with GFFs defined on proper subdomains of $\C$.  In particular,  let $U \subseteq \C$ be a fixed open set and let $h$ be a whole-plane GFF.  Recall that the Markov property for $h$ (see Section~\ref{subsec:gff}) implies that we can write $h|_U = h_U^0 + \Fh_U$,  where $h_U^0$ is a zero-boundary GFF on $U$ and $\Fh_U$ is a harmonic function on $U$.  Then for each fixed $\xi>0$,  we can associate a metric $D_{h_U^0}^{\xi}$ with $h_U^0$ by defining 
\begin{align*}
D_{h_U^0}^{\xi}:=e^{-\xi \Fh_U} \cdot D_{h|_U}^{\xi},
\end{align*}
where $D_{h|_U}^{\xi}:=D_h^{\xi}(\cdot ,  \cdot ; U)$ is the internal metric on $U$ associated with $D_h^{\xi}$.  Then it is easy to see by combining Axioms~\ref{it:locality} and ~\ref{it:weyl_scaling} that $D_{h_U^0}^{\xi}$ is a measurable function of $h_U^0$.  We will call the metric $D_{h_U^0}^{\xi}$ to be the LQG metric with parameter $\xi$ for the zero-boundary GFF on $U$.
\end{rem}

\section{Tightness: Subcritical case}
\label{sec:tightness}

In this section,  we will prove that the family of metrics $\wh{D}_h^{\gamma}$ in Theorem~\ref{thm:main_theorem_subcritical_intro} is tight whenever $\gamma$ lies on fixed compact subintervals of $(0,2)$.  As explained in Subsection~\ref{subsec:outline},  this will be achieved in two main steps.  As a first step,  we will give probability estimates on the $\gamma$-LQG distances between two fixed and disjoint compact sets which are uniform in $\gamma$.  This will be carried out in Subsection~\ref{subsec:estimates_on_lqg_distances_subcritical}.  As a second step,  we will obtain a moment diameter bound which is uniform in $\gamma$. This will be done in Subsection~\ref{subsec:moment_bound_for_diameters}.  Finally in Subsection~\ref{subsec:tightness_subcritical},  we will use the estimates from Subsections~\ref{subsec:estimates_on_lqg_distances_subcritical} and ~\ref{subsec:moment_bound_for_diameters} to obtain a H\"older continuity estimate which is uniform in $\gamma$ and hence prove the main tightness result of the section.

Let us now describe the setup.  For the rest of the section,  we fix $h$ to be whole-plane GFF normalized so that $h_1(0) = 0$.  Fix $\gamma \in (0,2)$ and let $D_h^{\gamma}$ denote the metrics as in the statement of Theorem~\ref{thm:main_theorem_subcritical_intro} which is the limit in probability as $\epsilon \to 0$ of the metrics $a_{\epsilon}^{-1} D_h^{\epsilon}$ as in Theorem~\ref{thm:convergence_of_lffp}.  Fix also $\mathfrak{p} \in (0,1)$ and let $\alpha = \alpha(\gamma ,  \mathfrak{p}) \in (0,\infty)$ be such that
\begin{align}\label{eqn:main_normalization}
\p[D_h^{\gamma}(\text{around} \,\,  \mathbb{A}_{1,2}(0)) \leq \alpha] = \mathfrak{p}.
\end{align}
Note that Lemma~\ref{lem:normalization_well_defined} implies that $\alpha(\gamma,\mathfrak{p})$ is well-defined for all $\mathfrak{p} \in (0,1)$.
Eventually we will define $\mathfrak{p} : (0,2) \to (0,1)$ as a decreasing function of $\gamma$.  Moreover we define $\xi$ and $Q$ as in \eqref{eqn:xi_subcritical} and \eqref{eqn:parameter_Q} respectively.  

In order to prove our main estimates in Subsections~\ref{subsec:estimates_on_lqg_distances_subcritical} and ~\ref{subsec:moment_bound_for_diameters},  it will be more convenient to work with a different re-normalization of the metrics $D_h^{\gamma}$ which is analogous to the re-normalization that we use in Theorem~\ref{thm:main_thm_xi_to_infty_intro}.  More precisely we set
\begin{align}\label{eqn:normalization_subcritical}
\wt{D}_h^{\gamma}:=\alpha^{-1} D_h^{\gamma}
\end{align}
and for the rest of the section we will prove estimates for $\wt{D}_h^{\gamma}$.

The main result that we are going to prove is the following.

\begin{theorem}\label{thm:tightness_subcritical}
There exists a choice of the function $\mathfrak{p} : (0,2) \to (0,1)$ such that the following is true.  Fix $0<\gamma_*<\gamma^*<2$ and let $U \subseteq \C$ be an open and connected set (including $\C$).  Then the collection of laws of the random internal metrics $\wt{D}_h^{\gamma}(\cdot ,  \cdot ; U)$ on $U \times U$ for $\gamma \in [\gamma_*,\gamma^*]$ is tight with respect to the local uniform topology on $U \times U$.
\end{theorem}

\subsection{Estimates for the distance between sets: Subcritical case}
\label{subsec:estimates_on_lqg_distances_subcritical}

The main goal of this subsection is to prove the following precise estimate which is uniform in $\gamma$.

\begin{proposition}\label{prop:bound_on_lqg_distances_of_sets}
Fix $M>1$.  Then there exists a universal constant $b \in (0,1)$ and a constant $\mathfrak{p}_0 \in (0,1)$ depending only on $M$ such that the following holds for all $\mathfrak{p} \in (\mathfrak{p}_0 , 1)$.  Let $U \subseteq \C$ be an open and connected set (including $\C$) and let $K_1,K_2 \subseteq U$ be two connected,  disjoint compact sets which are not singletons.  Then for all $\mathfrak{r} > 0$,  it holds with probability at least $1-O_A(A^{-b \sqrt{M}})$ as $A \to \infty$ at a rate that depends only on $K_1,K_2$ and $U$ and it is uniform in the choice of $\mathfrak{r}$,  that
\begin{align*}
\wt{D}_h^{\gamma}(\mathfrak{r} K_1 ,  \mathfrak{r} K_2 ; \mathfrak{r} U) \leq A \mathfrak{r}^{\xi Q} e^{\xi h_{\mathfrak{r}}(0)}.
\end{align*}
\end{proposition}

Let us explain the main idea of the proof of Proposition~\ref{prop:bound_on_lqg_distances_of_sets}.  We will use a similar argument as in the proof of \cite[Proposition~3.1]{dubedat2020weak} but the events that we are going to consider are different.  More precisely the main step will be to prove that with high probability (which is uniform in $\mathfrak{r}$ and $\gamma$),  we can cover the set $\mathfrak{r} U$ by Euclidean balls $B_r(w)$ for which there exists a path in $\mathbb{A}_{r,2r}(w)$ whose $\wt{D}_h^{\gamma}$-length is at most $r^{\xi Q} e^{\xi h_r(w)}$.  Then we will concatenate a collection of such paths to obtain a path from $\mathfrak{r} K_1$ to $\mathfrak{r} K_2$ whose $\wt{D}_h^{\gamma}$-length is bounded from above.  The assumption that both $K_1$ and $K_2$ are connected guarantees that some of the above paths in the annuli $\mathbb{A}_{r,2r}(w)$ intersect $\mathfrak{r} K_1$ and $\mathfrak{r} K_2$.

In order to cover $\mathfrak{r} U$ by Euclidean balls with the desired properties,  for all $z \in \C$ and $r>0$,  we let $E_r^{\gamma}(z)$ be the event that
\begin{align*}
\wt{D}_h^{\gamma}(\text{around} \,\,  \mathbb{A}_{r,2r}(z)) \leq r^{\xi Q} e^{\xi h_r(z)}.
\end{align*}
Note that our choice of re-normalization combined with the scaling properties of LQG metrics (Axiom~\ref{it:coordinate_change}) and Weyl scaling (Axiom~\ref{it:weyl_scaling})
imply that $\p[E_r^{\gamma}(z)] = \mathfrak{p}$.

We would like to have that the event $E_r^{\gamma}(z)$ occurs for sufficiently dense sets of annuli with high probability (uniformly in $\gamma$) provided we choose $\mathfrak{p} \in (0,1)$ sufficiently close to $1$.  This is the content of the following lemma.

\begin{lemma}\label{lem:good_event_occurs_almost_everywhere}
For each $M>0$,  there exists $\mathfrak{p}_0 \in (0,1)$ depending only on $M$ such that for all $\mathfrak{p} \in (\mathfrak{p}_0 ,  1)$, and all $\mathfrak{r}>0$,  it holds with probability at least $1 - \epsilon^M$ as $\epsilon \to 0$,  at a rate which is universal,   that the following is true.  For each $z \in B_{\mathfrak{r} \epsilon^{-M}}(0)$,  there exist 
\begin{align*}
r \in [\epsilon^2 \mathfrak{r} ,  \epsilon \mathfrak{r}] \cap \{2^{-k} \mathfrak{r}\}_{k \in \N},  w \in B_{\mathfrak{r} \epsilon^{-M}}(0) \cap \left(\frac{\epsilon^2 \mathfrak{r}}{4} \Z^2 \right)
\end{align*}
such that $E_r^{\gamma}(w)$ occurs and $z \in B_{\frac{\mathfrak{r} \epsilon^2}{2}}(w)$.
\end{lemma}

We will prove Lemma~\ref{lem:good_event_occurs_almost_everywhere} using the near -independence across disjoint annuli property of the GFF.  In particular we will use the following more precise version of \cite[Lemma~3.1]{gwynne2020local}.

\begin{lemma}\label{lem:good_event_at_many_scales}
Fix $0<s_1<s_2<1$.  Let $\{r_k\}_{k \in \N}$ be a decreasing sequence of positive numbers such that $\frac{r_{k+1}}{r_k} \leq s_1$ for all $k \in \N$,  and let $\{E_{r_k}\}_{k \in \N}$ be events such that each $E_{r_k}$ is measurable with respect to the $\sigma$-algebra generated by 
\begin{align*}
((h-h_{r_k}(0))|_{\mathbb{A}_{s_1 r_k ,  s_2 r_k}(0)}).
\end{align*}
For $K \in \N$,  we let $\mathcal{N}(K)$ be the number of $k \in [1,K]_{\Z}$ for which $E_{r_k}$ occurs.  Then for each $a>0,b \in (0,1)$,  there exists $p \in (0,1)$ depending only on $a$ and $b$ such that if 
\begin{align*}
\p[E_{r_k}] \geq p \quad \text{for all} \quad k \in \N,
\end{align*}
then
\begin{align*}
\p[\mathcal{N}(K) < bK] \leq e^{-a K} \quad \text{for all} \quad K \in \N.
\end{align*}
\end{lemma}

Lemma~\ref{lem:good_event_at_many_scales} is proved by comparing the Radon-Nikodym derivatives of the laws of a whole-plane GFF and a zero-boundary GFF when restricted to concentric annuli.  Recall that for $r>0$,  the Markov property of the GFF implies that we can write 
\begin{align*}
(h-h_r(0))|_{\C \setminus B_r(0)} = h^{0,r} + \Fh^r,
\end{align*}
where $\Fh^r$ is a harmonic function on $B_r(0)$ which is determined by $(h-h_r(0))|_{\C \setminus B_r(0)}$ and $h^{0,r}$ is a zero-boundary GFF on $B_r(0)$ which is independent from  $(h-h_r(0))|_{\C \setminus B_r(0)}$ . 

For $0<r<R<\infty$,  we let 
\begin{align*}
\mathcal{M}_r^R:=\sup_{z \in B_r(0)} |\Fh^R(z) - \Fh^R(0)|.
\end{align*}
Then we have the following.

\begin{lemma}\label{lem:harmonic_function_bounded_at_many_scales}
Fix $s \in (0,1)$ and let $\{r_k\}_{k \in \N_0}$ be a decreasing sequence of positive numbers such that $\frac{r_{k+1}}{r_k} \leq s$ for all $k \in \N_0$.  For $K \in \N$ and $M>0$,  we let $\mathcal{N}_M(K)$ be the number of $k \in [1,K]_{\Z}$ for which $\mathcal{M}_{s r_k}^{r_k} \leq M$.  Then for each $a>0,b \in (0,1)$,  there exists $M>0$ depending only on $a,b$ and $s$ such that
\begin{align*}
\p[\mathcal{N}_M(K) \geq b K] \geq 1 - e^{-a K} \quad \text{for all} \quad K \in \N.
\end{align*}
\end{lemma}

\begin{proof}
It follows from the exact same argument used to prove \cite[Lemma~3.4]{gwynne2020local}.
\end{proof}

\begin{proof}[Proof of Lemma~\ref{lem:good_event_at_many_scales}]
It follows from the same argument used to prove \cite[Lemma~3.1]{gwynne2020local} with \cite[Lemma~3.4]{gwynne2020local} replaced by Lemma~\ref{lem:harmonic_function_bounded_at_many_scales}.
\end{proof}

\begin{proof}[Proof of Lemma~\ref{lem:good_event_occurs_almost_everywhere}]
Fix $z \in \C ,  r>0$ and recall that Axioms~\ref{it:weyl_scaling} and ~\ref{it:coordinate_change} imply that 
\begin{align*}
\p[E_r^{\gamma}(z)] = \p[E_1^{\gamma}(0)] = \mathfrak \quad \text{for all} \quad \gamma \in (0,2).
\end{align*}
Also by the locality property of $D_h^{\gamma}$ (Axiom~\ref{it:locality}),  we have that the event $E_r^{\gamma}(z)$ is determined by $(h-h_{3r}(z))|_{\mathbb{A}_{r,2r}(z)}$.

We are going to apply Lemma~\ref{lem:good_event_at_many_scales} with $r_k=32^{-2k}\mathfrak{r}, s_1 = \frac{1}{3},s_2 = \frac{2}{3}, b= \frac{1}{2}$ and $\wt{E}_{r_k} = E_{r_k / 3}^{\gamma}(z)$.  Fix $a > 2 \log(2)(3M + 4)$ (depending only on $M$).  Then Lemma~\ref{lem:good_event_at_many_scales} implies that there exists $\mathfrak{p}_0 \in (0,1)$ such that for all $\mathfrak{p} \in (\mathfrak{p}_0,1)$,  we have that
\begin{align*}
\p[\mathcal{N}(K) < K/2] \leq e^{-aK} \quad \text{for all} \quad K \in \N,
\end{align*}
where $\mathcal{N}(K)$ is as in the statement of Lemma~\ref{lem:good_event_at_many_scales}.  Set $K:=\left \lfloor \frac{\log(\epsilon^{-1})}{\log(2)} \right \rfloor$ and let $\epsilon_0 \in (0,1)$ be sufficiently small (chosen in a universal way) such that $K > \frac{\log(\epsilon^{-1})}{2\log(2)}$ for all $\epsilon \in (0,\epsilon_0)$.  Then we have that
\begin{align*}
\p[\mathcal{N}(K) < K/2] \leq \exp\left(-\frac{a \log(\epsilon^{-1})}{2\log(2)} \right) = \epsilon^{\frac{a}{2\log(2)}} < \epsilon^M
\end{align*}
for all $\epsilon \in (0,\epsilon_0)$ by the choice of $a$.

Note that if $\mathcal{N}(K) > K/2$,  the event $\wt{E}_{r_k}$ occurs for at least one value of $r_k \in [3\epsilon^2 \mathfrak{r},3\epsilon \mathfrak{r}] \cap \{32^{-k} \mathfrak{r}\}_{k \in \N}$,  which implies that $E_r^{\gamma}(z)$ occurs for at least one value of $r \in [\epsilon^2 \mathfrak{r} ,  \epsilon \mathfrak{r}] \cap \{2^{-k} \mathfrak{r}] \cap \{2^{-k} \mathfrak{r}\}_{k \in \N}$.  Note also that
\begin{align*}
\left | B_{\mathfrak{r} \epsilon^{-M}}(0) \cap \left(\frac{\epsilon^2 \mathfrak{r}}{4} \Z^2 \right) \right | \asymp \epsilon^{-2M-4} \quad \text{as} \quad \epsilon \to 0
\end{align*}
at a universal rate.  Therefore the choice of $a$ implies that possibly by taking $\epsilon_0 \in (0,1)$ to be smaller (in a universal way),  we have by taking a union bound that the following is true.  For all $\epsilon \in (0,\epsilon_0)$,  off an event with probability at most $\epsilon^M$,  it holds that for all $w \in B_{\mathfrak{r} \epsilon^{-M}}(0) \cap \left(\frac{\mathfrak{r} \epsilon^2}{4} \Z^2 \right)$,  there exists $r \in [\epsilon^2 \mathfrak{r} ,  \epsilon \mathfrak{r}] \cap \{2^{-k} \mathfrak{r}\}_{k \in \N}$ such that $E_r^{\gamma}(w)$ occurs.  This completes the proof of the lemma.
\end{proof}

Now we would like to compare the circle averages of the field around any point in $U$ with the circle averages around $0$.  Then we have the following more precise version of \cite[Lemma~3.4]{dubedat2020weak}.

\begin{lemma}\label{lem:upper_bound_on_circle_averages}
Fix $\nu \geq 0$ and $R>0$.  Then there exists a constant $C \in (0,\infty)$ depending only on $\nu$ and $R$ such that for all $q>0,\mathfrak{r}>0$,  we have that
\begin{align*}
&\p\left[\sup\left\{|h_r(w)-h_{\mathfrak{r}}(0)| : w \in B_{R \mathfrak{r}}(0) \cap \left(\frac{\epsilon^{1+\nu} \mathfrak{r}}{4} \Z^2 \right) ,  r \in [\epsilon^{1+\nu} \mathfrak{r} , \epsilon \mathfrak{r}] \right\} \geq q \log(\epsilon^{-1}) \right]\\
& \leq C \epsilon^{\frac{q^2}{2(1+\sqrt{\nu})^2(1+C (\log(\epsilon^{-1}))^{-1})} - 2 -2\nu} \quad \text{for all} \quad \epsilon \in (0,1).
\end{align*}
\end{lemma}

\begin{proof}
Fix $s \in [0,q)$ (to be chosen later in the proof) and note that \cite[Theorem~1.59]{berestycki2024gaussian} implies that for all $w \in B_{R \mathfrak{r}}(0)$,  the random variable $t \to h_{e^{-t} \epsilon \mathfrak{r}}(w) - h_{\epsilon \mathfrak{r}}(w)$ has the law of a standard Brownian motion.   Hence the standard Gaussian tail bound implies that
\begin{align*}
&\p\left[ \sup_{\epsilon^{1+\nu} \leq r \leq  \epsilon \mathfrak{r}} |h_r(w) - h_{\epsilon \mathfrak{r}}(w)| \geq s \log(\epsilon^{-1}) \right] = \p\left[ \sup_{0 \leq t \leq \nu \log(\epsilon^{-1})} |B_t| \geq s \log(\epsilon^{-1})\right]\\
&=\p\left[ \sup_{0 \leq t \leq 1} |B_t| \geq s \sqrt{\log(\epsilon^{-1})} / \sqrt{\nu} \right] = 2 \p\left[ B_1 \geq s \sqrt{\log(\epsilon^{-1})} / \sqrt{\nu} \right]\\
&\leq 2 \exp\left(-\frac{s^2 \log(\epsilon^{-1})}{2\nu} \right) = 2 \epsilon^{\frac{s^2}{2\nu}} \quad \text{for all} \quad \epsilon \in (0,1),
\end{align*}
where $B$ is a standard Brownian motion.  Moreover the random variables $h_{\epsilon \mathfrak{r}}(w) - h_{\mathfrak{r}}(0)$ for $w \in B_{R \mathfrak{r}}(0)$ are centered Gaussians and there exists a constant $C>0$ depending only on $R$ such that
\begin{align*}
\text{Var}(h_{\epsilon \mathfrak{r}}(w) - h_{\mathfrak{r}}(0)) \leq \log(\epsilon^{-1}) + C \quad \text{for all} \quad w \in B_{R \mathfrak{r}}(0) ,  \mathfrak{r}>0.
\end{align*}
Thus by applying the Gaussian tail bound again,  we obtain that 
\begin{align*}
&\p\left[ |h_{\epsilon \mathfrak{r}}(w) - h_{\mathfrak{r}}(0)| \geq (q-s) \log(\epsilon^{-1})\right] = \p\left[ \frac{|h_{\epsilon \mathfrak{r}}(w)-h_{\mathfrak{r}}(0)|}{\sqrt{\text{Var}(h_{\epsilon \mathfrak{r}}(w)-h_{\mathfrak{r}}(0))}} \geq \frac{(q-s) \log(\epsilon^{-1})}{\sqrt{\text{Var}(h_{\epsilon \mathfrak{r}}(w)-h_{\mathfrak{r}}(0))}} \right]\\
&\leq 2 \exp\left(-\frac{(q-s)^2 \log(\epsilon^{-1})^2}{2\text{Var}(h_{\epsilon \mathfrak{r}}(w)-h_{\mathfrak{r}}(0))} \right) \leq 2 \exp\left(-\frac{(q-s)^2 \log(\epsilon^{-1})^2}{2(\log(\epsilon^{-1}) + C)}\right)\\
&=2 \epsilon^{\frac{(q-s)^2}{2(1+C (\log(\epsilon^{-1}))^{-1})}} \quad \text{for all} \quad \epsilon \in (0,1).
\end{align*}

Combining the above estimates and setting $s = \frac{q\sqrt{\nu}}{1+\sqrt{\nu}}$,  we obtain that for all $w \in B_{R \mathfrak{r}}(0)$ and all $\epsilon \in (0,1)$,  we have that
\begin{align*}
\p\left[\sup_{\epsilon^{1+\nu}\mathfrak{r} \leq r \leq \epsilon \mathfrak{r}} |h_r(w) - h_{\mathfrak{r}}(0)| \geq q \log(\epsilon^{-1})\right] &\leq \p\left[ |h_{\epsilon \mathfrak{r}}(w) - h_{\mathfrak{r}}(0)| \geq (q-s) \log(\epsilon^{-1})\right]\\
&+\p\left[ \sup_{\epsilon^{1+\nu} \mathfrak{r} \leq r \leq \epsilon \mathfrak{r}} |h_r(w) - h_{\epsilon \mathfrak{r}}(w)| \geq s \log(\epsilon^{-1})\right]\\
&\leq 2 \epsilon^{\frac{s^2}{2\nu}} + 2 \epsilon^{\frac{(q-s)^2}{2(1+C (\log(\epsilon^{-1}))^{-1})}} \leq 4 \epsilon^{\frac{q^2}{2(1+\sqrt{\nu})^2(1+C (\log(\epsilon^{-1}))^{-1})}}.
\end{align*}

Finally we have that
\begin{align*}
\left| B_{R \mathfrak{r}}(0) \cap \left(\frac{\epsilon^{1+\nu} \mathfrak{r}}{4} \Z^2 \right) \right| \lesssim \epsilon^{-2-2\nu} \quad \text{for all} \quad \epsilon \in (0,1),
\end{align*}
where the implicit constant depends only on $\nu$ and $R$.  Therefore the proof of the lemma is complete by taking a union bound over all $w \in B_{R \mathfrak{r}}(0) \cap \left(\frac{\epsilon^{1+\nu} \mathfrak{r}}{4} \Z^2 \right)$.
\end{proof}

Recall that our goal is to construct paths connecting $\mathfrak{r} K_1$ to $\mathfrak{r} K_2$ by concatenating the paths in the definition of the events $E_r^{\gamma}(z)$.  The next lemma ensures that we can construct such paths.

\begin{lemma}\label{lem:connecting_compact_sets_with_good_annuli}
Let $U \subseteq \C$ be a bounded and open set and let $K_1,K_2 \subseteq U$ be connected,  disjoint compact sets which are not singletons.  Fix $M>1$ and $\mathfrak{r}>0$ and for $\epsilon \in (0,1)$,  we let $F_{\mathfrak{r}}^{\epsilon}$ be the event in the statement of Lemma~\ref{lem:good_event_occurs_almost_everywhere}.  For $w \in B_{\epsilon \mathfrak{r}}(\mathfrak{r} U) \cap \left(\frac{\epsilon^2 \mathfrak{r}}{4} \Z^2 \right)$ and $r \in [\epsilon^2 \mathfrak{r} , \epsilon \mathfrak{r}] \cap \{2^{-k} \mathfrak{r}\}_{k \in \N}$ such that $E_r^{\gamma}(w)$ occurs,  we let $P_{w,r}$ be a path disconnecting $\partial B_r(w)$ from $\partial B_{2r}(w)$ such that
\begin{align*}
 \len(P_{w,r} ; \wt{D}_h^{\gamma}) \leq 1
\end{align*}
where we choose $P_{w,r}$ in some arbitrary but fixed way which is measurable with respect to $\wt{D}_h^{\gamma}$.  Let also $\mathcal{B}_{w,r}$ denote the region bounded by $P_{w,r}$.  Then on the event $F_{\mathfrak{r}}^{\epsilon}$,  the union of the paths $P_{w,r}$ for $w \in B_{\epsilon \mathfrak{r}}(\mathfrak{r} U) \cap \left(\frac{\epsilon^2 \mathfrak{r}}{4} \Z^2 \right)$ and $r \in [\epsilon^2 \mathfrak{r} ,  \epsilon \mathfrak{r}] \cap \{2^{-k} \mathfrak{r}\}_{k \in \N}$ such that $E_r^{\gamma}(w)$ occurs contains a path from $\mathfrak{r} K_1$ to $\mathfrak{r} K_2$ which is contained in $\mathfrak{r} U$.
\end{lemma}

\begin{proof}
It follows from the same argument used to prove \cite[Lemma~3.5]{dubedat2020weak}.
\end{proof}

Finally we are ready to prove Proposition~\ref{prop:bound_on_lqg_distances_of_sets}.

\begin{proof}[Proof of Proposition~\ref{prop:bound_on_lqg_distances_of_sets}]
Let $\mathfrak{p}_0 \in (0,1)$ be the constant in the statement of Lemma~\ref{lem:good_event_occurs_almost_everywhere} and fix $\mathfrak{p} \in (\mathfrak{p}_0 , 1)$.  Suppose first that $U$ is bounded.  Then Lemma~\ref{lem:upper_bound_on_circle_averages} applied with $q = 2 \sqrt{2} \sqrt{4+M},\nu =1$ and $R>0$ such that $U \subseteq B_R(0)$ implies that there exists a constant $C>1$ depending only on $R$ such that for all $\epsilon \in (0,1)$,  we have off an event with probability at most $C \epsilon^{\frac{4+M}{1+C (\log(\epsilon^{-1}))^{-1}}}$ that
\begin{align}\label{eqn:bound_on_circle_averages}
\sup\left\{|h_r(w)-h_{\mathfrak{r}}(0)| : w \in B_{R \mathfrak{r}}(0) \cap \left(\frac{\epsilon^{2} \mathfrak{r}}{4} \Z^2 \right) ,  r \in [\epsilon^{1+\nu} \mathfrak{r} , \epsilon \mathfrak{r}] \right\} \leq 2\sqrt{2} \sqrt{M+4}\log(\epsilon^{-1}).
\end{align}

Let $F_{\mathfrak{r}}^{\epsilon}$ be the event of Lemmas~\ref{lem:good_event_occurs_almost_everywhere} and ~\ref{lem:connecting_compact_sets_with_good_annuli}. Then combining \eqref{eqn:bound_on_circle_averages} with Lemmas~\ref{lem:good_event_occurs_almost_everywhere} and ~\ref{lem:connecting_compact_sets_with_good_annuli},  we obtain that off an event with probability at most $\epsilon^M$ as $\epsilon \to 0$ at a rate depending only on $K_1,K_2$ and $U$,  we have that \eqref{eqn:bound_on_circle_averages} occurs and the union of the paths $P_{w,r}$ for $w \in B_{\epsilon \mathfrak{r}}(\mathfrak{r} U) \cap \left(\frac{\epsilon^2 \mathfrak{r}}{4} \Z^2\right)$ and $r \in [\epsilon^2 \mathfrak{r} ,  \epsilon \mathfrak{r}] \cap \{2^{-k} \mathfrak{r}\}_{k \in \N}$ such that $E_r^{\gamma}(w)$ occurs contains a path from $\mathfrak{r} K_1$ to $\mathfrak{r} K_2$ which is contained in $\mathfrak{r} U$.  From now on,  we assume that we are working on this event.  Note that the number of such paths is at most $\epsilon^{-4-o_{\epsilon}(1)}$ as $\epsilon \to 0$ (at a rate which depends only on $K_1,K_2$ and $U$) and so combining with the triangle inequality we obtain that
\begin{align*}
\wt{D}_h^{\gamma}(\mathfrak{r} K_1 ,  \mathfrak{r} K_2 ; \mathfrak{r} U)&\leq \epsilon^{-4-o_{\epsilon}(1)} \sup\left\{\len(P_{w,r} ; \wt{D}_h^{\gamma}) : w \in B_{\epsilon \mathfrak{r}}(\mathfrak{r} U) \cap \left(\frac{\epsilon^2 \mathfrak{r}}{4} \Z^2\right) ,  r \in [\epsilon^2 \mathfrak{r} ,  \epsilon \mathfrak{r}] \cap \{2^{-k} \mathfrak{r}\}_{k \in \N} \right\}\\
&\leq \epsilon^{-4 - o_{\epsilon}(1)} \sup\left\{r^{\xi Q} e^{\xi h_r(w)} : w \in B_{\epsilon \mathfrak{r}}(\mathfrak{r} U) \cap \left(\frac{\epsilon^2 \mathfrak{r}}{4} \Z^2\right) ,  r \in [\epsilon^2 \mathfrak{r} ,  \epsilon \mathfrak{r}] \cap \{2^{-k} \mathfrak{r}\}_{k \in \N} \right\}\\
&\leq \epsilon^{-4+\xi Q -\xi 2 \sqrt{2}\sqrt{4+M} -o_{\epsilon}(1)} \mathfrak{r}^{\xi Q} e^{\xi h_{\mathfrak{r}}(0)}.
\end{align*}

Given $A>1$,  we choose $\epsilon = A^{-\frac{b}{\sqrt{M}}}$ and note that
\begin{align*}
\epsilon^{-4+\xi Q -\xi 2 \sqrt{2}\sqrt{4+M} -o_{\epsilon}(1)} \leq A
\end{align*}
if and only if
\begin{align*}
\frac{4-\xi Q + \xi 2 \sqrt{2}\sqrt{4+M}}{\sqrt{M}} < b^{-1}.
\end{align*}
Therefore since $\xi \in (0,\xi_{\text{crit}})$ for all $\gamma \in (0,2)$ and 
\begin{align*}
\frac{4+\xi_{\text{crit}} 2 \sqrt{2}\sqrt{4+M}}{\sqrt{M}} \to 2 \sqrt{2} \xi_{\text{crit}} \quad \text{as} \quad M \to \infty,
\end{align*}
we obtain combining with \eqref{eqn:bound_on_circle_averages} and since $\epsilon^M = A^{-b \sqrt{M}}$ that by choosing a universal constant $b \in (0,1)$ such that
\begin{align}\label{eqn:choice_of_constant_b}
\frac{4+\xi_{\text{crit}} 2 \sqrt{2}\sqrt{4+M}}{\sqrt{M}} < b^{-1} \quad \text{for all} \quad M>1,
\end{align}
we have that
\begin{align*}
\p[\wt{D}_h^{\gamma}(\mathfrak{r} K_1,\mathfrak{r} K_2 ; \mathfrak{r} U) \leq A \mathfrak{r} ^{\xi Q} e^{\xi h_{\mathfrak{r}}(0)}] \geq 1 - O_A(A^{-b \sqrt{M}}) \quad \text{as} \quad M \to \infty,
\end{align*}
at a rate which depends only on $K_1,K_2$ and $U$.  

Finally if $V$ is a possibly unbounded subset of $\C$ with $U \subseteq V$,  then we have that
\begin{align}\label{eqn:general_case}
\wt{D}_h^{\gamma}(\mathfrak{r} K_1 ,  \mathfrak{r} K_2) \leq \wt{D}_h^{\gamma}(\mathfrak{r} K_1 ,  \mathfrak{r} K_2 ; \mathfrak{r} V) \leq \wt{D}_h^{\gamma}(\mathfrak{r} K_1 ,  \mathfrak{r} K_2 ; \mathfrak{r} U).
\end{align}
Therefore using \eqref{eqn:general_case},  we obtain the claim in the statement of the lemma with $U$ possibly unbounded.  This completes the proof of the proposition.
\end{proof}

\subsection{Moment bound for diameters: Subcritical case}
\label{subsec:moment_bound_for_diameters}

The main goal of this subsection is to prove the following moment bound.

\begin{proposition}\label{prop:moment_lqg_diameter_bound_for_compact_sets}
There exists a choice of the function $\mathfrak{p} : (0,2) \to (0,1)$ in \eqref{eqn:main_normalization} such that the following is true.  Fix $0 < \gamma_* < \gamma^* < 2$.  Let $U \subseteq \C$ be an open and connected set and let $K \subseteq U$ be a compact and connected set with more than a point. Then there exists a constant $a  > 2$ depending only on $\gamma_*$ and $\gamma^*$ such that the following is true.  Fix $\beta \in (2,a)$ and set $p = \frac{\beta}{\xi}$.  Then we have that
\begin{align*}
\E\left[\left(\mathfrak{r}^{-\xi(\gamma) Q(\gamma)} e^{-\xi(\gamma) h_{\mathfrak{r}}(0)} \sup_{z,w \in \mathfrak{r} K} \wt{D}_h^{\gamma}(z,w ;  \mathfrak{r} U)\right)^p \right] \lesssim 1
\end{align*}
for all $\gamma \in [\gamma_*,\gamma^*]$ and all $\mathfrak{r}>0$,  where the implicit constant depends only on $K,U,\beta,\gamma_*$ and $\gamma^*$.
\end{proposition}

We will deduce Proposition~\ref{prop:moment_lqg_diameter_bound_for_compact_sets} from the following variant.

\begin{proposition}\label{prop:moment_lqg_diameter_bound_for_square}
There exists a choice of the function $\mathfrak{p} : (0,2) \to (0,1)$ in \eqref{eqn:main_normalization} such that the following is true. 
Fix $0 < \gamma_* < \gamma^* < 2$.  Then there exists a constant $a = a(\gamma_* ,  \gamma^*) > 2$ depending only on $\gamma_*$ and $\gamma^*$ such that the following is true.  Fix $\beta \in (2,a)$ and set $p = \frac{\beta}{\xi}$.  Then we have that
\begin{align*}
\E\left[\left(\mathfrak{r}^{-\xi(\gamma) Q(\gamma)} e^{-\xi(\gamma) h_{\mathfrak{r}}(0)} \sup_{z,w \in \mathfrak{r} \mathbb{S}} \wt{D}_h^{\gamma}(z,w ;  \mathfrak{r} \mathbb{S})\right)^p \right] \lesssim 1
\end{align*}
for all $\gamma \in [\gamma_*,\gamma^*]$ and all $\mathfrak{r}>0$,  where the implicit constant depends only on $\beta,\gamma_*$ and $\gamma^*$.
\end{proposition}

Next we fix $q \in (2,Q),  \zeta \in (0,Q-q) \cap (0,1)$ and $\wt{\zeta} \in (0,\zeta)$.  We will choose these parameters later as functions of $\gamma \in (0,2)$.  In order to prove Proposition~\ref{prop:moment_lqg_diameter_bound_for_square},  we will use Proposition~\ref{prop:bound_on_lqg_distances_of_sets} combined with a union bound to construct paths between the two shorter sides of rectangles of the form either $2^{-n} \mathfrak{r} \times 2^{-n-1} \mathfrak{r}$ or $2^{-n-1} \mathfrak{r} \times 2^{-n} \mathfrak{r}$ contained in $\mathfrak{r} \mathbb{S}$.  The next lemma ensures that the circle averages of $h$ on the above rectangles are not too large.

\begin{lemma}\label{lem:circle_averages_uniform_bound}
Fix $\mathfrak{r}>0,C>1$ and set 
\begin{align*}
E_{\mathfrak{r}}^n:=\left\{\sup\left\{ |h_{2^{-n} \mathfrak{r}}(w) - h_{\mathfrak{r}}(0)| : w \in B_{\mathfrak{r}}(0) \cap (2^{-n-1} \mathfrak{r} \Z^2) \right\} \leq \log(C 2^{qn}) \right\} \quad \text{for all} \quad n \in \N.
\end{align*}
Then there exists a universal constant $\wt{C}>1$ such that
\begin{align*}
&\p[E_{\mathfrak{r}}^n,  \text{for all} \,\,n \in \N \,\,\text{with} \,\,2^n \geq C^{\zeta}]\\
&\geq 1 - \frac{\wt{C}\wt{\zeta}^{-1} \log(C)}{\log(2)} \exp\left(\left(-\frac{q+\sqrt{q^2-4}}{1+\wt{C} \zeta^{-1} \log(C)^{-1}} + \frac{\wt{\zeta} (q+\zeta^{-1})^2}{2} + 2\wt{\zeta} \wt{C} + 2\wt{C} \log(C)^{-1}\right)\log(C)\right)\\
&-\frac{\wt{C}}{1-2^{2-\frac{q^2}{2(1+\wt{C} \wt{\zeta})}}} \exp\left(\left(\left(\frac{2}{1+\wt{C}\wt{\zeta}}\right) \left(\frac{4-q^2}{2\wt{\zeta}}\right) + \frac{2 \wt{C}}{1+\wt{C} \wt{\zeta}} \right)\log(C)\right).
\end{align*}
\end{lemma}

\begin{proof}
We start by fixing a partition of $[\zeta ,  \wt{\zeta}^{-1}]$, 
$\zeta = a_0 < \cdots <a_N = \wt{\zeta}^{-1}$
with 
\begin{align*}
\max_{k=1,\cdots,N} (a_k - a_{k-1}) \leq \wt{\zeta}.
\end{align*}
Fix $n \in \N$ such that $2^n \geq C^{\zeta}$.  Then we have that either $2^n \leq C^{1 / \wt{\zeta}}$ or $2^n > C^{1/ \wt{\zeta}}$.  We will bound from above the probability of the event $(E_{\mathfrak{r}}^n)^c$ by dealing with the two cases separately.

\emph{Case 1.  $2^n \leq C^{1 / \wt{\zeta}}$.}
In that case,  there exists $k \in [1,N]_{\Z}$ such that $2^n \in [C^{a_{k-1}},C^{a_k}]$.  Then Lemma~\ref{lem:upper_bound_on_circle_averages} applied with $\epsilon = 2^{-n},  \nu = 0,R=1$ and $q$ replaced by $q+\frac{1}{a_k}$ implies that there exists a universal constant $\wt{C}>1$ such that
\begin{align*}
&\p[(E_{\mathfrak{r}}^n)^c] \leq \p\left[\sup\left\{|h_{2^{-n}\mathfrak{r}}(w)-h_{\mathfrak{r}}(0)| : w \in B_{\mathfrak{r}}(0) \cap \left(2^{-n-1} \mathfrak{r} \Z^2 \right)  \right\} > \left(q+\frac{1}{a_k}\right) \log(2^n)\right]\\
&\leq \wt{C} 2^{-n\left(\frac{(q+1/a_k)^2}{2(1+\wt{C} (n\log(2))^{-1})}-2\right)} \leq \wt{C} C^{-a_{k-1} \left(\frac{(q+1/a_k)^2}{2(1+\wt{C} (n\log(2))^{-1})} - 2\right)} \leq \wt{C} C^{-a_{k-1} \left(\frac{(q+1/a_k)^2}{2(1+\wt{C} (a_{k-1} \log(C))^{-1})}-2\right)}.
\end{align*}

Note that
\begin{align*}
\frac{|a_{k-1}-a_k| (q+1/a_k)^2}{2(1+\wt{C} (a_{k-1} \log(C))^{-1})} \leq \frac{\wt{\zeta} (q+\zeta^{-1})^2}{2}.
\end{align*}
It follows that
\begin{align*}
\p[(E_{\mathfrak{r}}^n)^c] \leq \wt{C} C^{2a_k - \frac{(qa_k+1)^2}{2a_k (1+\wt{C} (a_{k-1} \log(C))^{-1})} + \frac{\wt{\zeta} (q+\zeta^{-1})^2}{2}} \quad \text{for all} \quad C>1.
\end{align*}
Since the number of values of $n \in \N$ such that $C^{a_{k-1}} \leq 2^n \leq C^{a_k}$ is at most $\frac{\wt{\zeta} \log(C)}{\log(2)}$,  we obtain by taking a union bound that
\begin{align*}
\p[E_{\mathfrak{r}}^n \,\,\text{for all} \,\,n \in \N \,\,\text{with} \,\,C^{a_{k-1}} \leq 2^n \leq C^{a_k}] \geq 1 - \frac{\wt{C} \wt{\zeta} \log(C)}{\log(2)} C^{2a_k - \frac{(qa_k + 1)^2}{2a_k (1+\wt{C} (a_{k-1} \log(C))^{-1})} + \frac{\wt{\zeta} (q+\zeta^{-1})^2}{2}}.
\end{align*}
Note that
\begin{align*}
\left |2 a_k - \frac{2a_k}{1+\wt{C} (a_{k-1} \log(C))^{-1}} \right | &= \frac{2a_k \wt{C} (a_{k-1} \log(C))^{-1}}{1+\wt{C} (a_{k-1} \log(C))^{-1}} \\
&= \frac{2(a_k - a_{k-1}) \wt{C} (a_{k-1} \log(C))^{-1}}{1+ \wt{C} (a_{k-1} \log(C))^{-1}} + \frac{2 a_{k-1} \wt{C} (a_{k-1} \log(C))^{-1}}{1+\wt{C} (a_{k-1} \log(C))^{-1}}\\
&\leq 2 \wt{\zeta} \wt{C} + \frac{2\wt{C}}{\log(C)}.
\end{align*}
Thus we have that
\begin{align*}
&\p[E_{\mathfrak{r}}^n \,\,\text{for all} \,\,n \in \N \,\,\text{with} \,\,C^{a_{k-1}} \leq 2^n \leq C^{a_k}] \\
&\geq 1 - \frac{\wt{C} \wt{\zeta} \log(C)}{\log(2)} \exp\left(\left(\frac{2a_k - \frac{(qa_k + 1)^2}{2a_k}}{1+ \wt{C} (\zeta \log(C))^{-1}} + \frac{\wt{\zeta} (q+\zeta^{-1})^2}{2} + 2\wt{\zeta} \wt{C} + \frac{2\wt{C}}{\log(C)} \right) \log(C) \right).
\end{align*}

Since the quantity $2a - \frac{(qa+1)^2}{2a}$ is maximized over all $a>0$ when $a = (q^2 - 4)^{-1/2}$,  in which case it is equal to $-(q+\sqrt{q^2-4})$,  we obtain that
\begin{align*}
&\p[E_{\mathfrak{r}}^n \,\,\text{for all} \,\,n \in \N \,\,\text{with} \,\,C^{a_{k-1}} \leq 2^n \leq C^{a_k}] \\
&\geq 1 - \frac{\wt{C} \wt{\zeta} \log(C)}{\log(2)} \exp\left(\left(\frac{-(q+\sqrt{q^2-4})}{1+ \wt{C} (\zeta \log(C))^{-1}} + \frac{\wt{\zeta} (q+\zeta^{-1})^2}{2} + 2\wt{\zeta} \wt{C} + \frac{2\wt{C}}{\log(C)} \right) \log(C) \right).
\end{align*}
Moreover it holds that $N \leq \wt{\zeta}^{-2}$ and so a union bound gives that
\begin{align}\label{eqn:lower_bound_case_1}
&\p[E_{\mathfrak{r}}^n \,\,\text{for all} \,\,n \in \N \,\,\text{with} \,\,C^{\zeta} \leq 2^n \leq C^{1/\wt{\zeta}}]\notag \\
&\geq 1 - \frac{\wt{C} \wt{\zeta}^{-1} \log(C)}{\log(2)} \exp\left(\left(\frac{-(q+\sqrt{q^2-4})}{1+ \wt{C} (\zeta \log(C))^{-1}} + \frac{\wt{\zeta} (q+\zeta^{-1})^2}{2} + 2\wt{\zeta} \wt{C} + \frac{2\wt{C}}{\log(C)} \right) \log(C) \right).
\end{align}

\emph{Case 2.  $2^n \geq C^{1 / \wt{\zeta}}$.}
Lemma~~\ref{lem:upper_bound_on_circle_averages} applied with $\epsilon = 2^{-n}$ and $\nu = 0$ implies that 
\begin{align*}
&\p[(E_{\mathfrak{r}}^n)^c] \leq \p\left[\sup\left\{|h_{2^{-n}\mathfrak{r}}(w)-h_{\mathfrak{r}}(0)| : w \in B_{\mathfrak{r}}(0) \cap \left(2^{-n-1} \mathfrak{r} \Z^2 \right)  \right\} > q \log(2^n)\right]\\
&\leq \wt{C} 2^{-n\left(\frac{q^2}{2(1+\wt{C} (n \log(2))^{-1})} -2\right)} \leq \wt{C} 2^{-n\left(\frac{q^2}{2(1+\wt{C} \wt{\zeta} \log(C)^{-1})}-2\right)}.
\end{align*}
Summing over all $n \geq C^{1 / \wt{\zeta}}$ gives that
\begin{align*}
\p[E_{\mathfrak{r}}^n \,\,\text{for all} \,\,  n \in \N \,\,\text{with} \,\,2^n \geq C^{1 / \wt{\zeta}}] &\geq 1 - \frac{\wt{C}}{1 - 2^{2-\frac{q^2}{2(1+\wt{C} \wt{\zeta})}}} \exp\left(\left(-\frac{q^2}{2\wt{\zeta} (1 + \wt{C} \wt{\zeta} \log(C)^{-1})} + \frac{2}{\wt{\zeta}}\right) \log(C) \right)\\
&\geq 1 - \frac{\wt{C}}{1 - 2^{2-\frac{q^2}{2(1+\wt{C} \wt{\zeta})}}} \exp\left(\left(-\frac{q^2}{2\wt{\zeta} (1 + \wt{C} \wt{\zeta})} + \frac{2}{\wt{\zeta}}\right) \log(C) \right).
\end{align*}

Note that
\begin{align*}
-\frac{q^2}{2\wt{\zeta} (1+\wt{C} \wt{\zeta})} + \frac{2}{\wt{\zeta}} = \frac{1}{1+\wt{C} \wt{\zeta}} \left(\frac{4-q^2}{2\wt{\zeta}}\right) + \frac{2\wt{C}}{1+\wt{C} \wt{\zeta}}.
\end{align*}
Therefore we obtain that
\begin{align}\label{eqn:lower_bound_case_2}
\p[E_{\mathfrak{r}}^n \,\,\text{for all} \,\,  n \in \N \,\,\text{with} \,\,2^n \geq C^{1 / \wt{\zeta}}] \geq 1-\frac{\wt{C}}{1-2^{2-\frac{q^2}{2(1+\wt{C} \wt{\zeta})}}} \exp\left(\left(\left(\frac{2}{1+\wt{C}\wt{\zeta}}\right) \left(\frac{4-q^2}{2\wt{\zeta}}\right) + \frac{2 \wt{C}}{1+\wt{C} \wt{\zeta}} \right)\log(C)\right).
\end{align}
Hence the proof of the lemma is complete by combining ~\eqref{eqn:lower_bound_case_1} with ~\eqref{eqn:lower_bound_case_2}.
\end{proof}

Next we will use Proposition~\ref{prop:bound_on_lqg_distances_of_sets} and Lemma~\ref{lem:circle_averages_uniform_bound} to construct short paths across rectangles with high probability.  We will follow the same notation as in the proof of \cite[Proposition~3.10]{dubedat2020weak}.  For $n \in \N_0 ,  \mathfrak{r}>0$,  we let $\mathcal{R}_{\mathfrak{r}}^n$ denote the set of open $2^{-n} \mathfrak{r} \times 2^{-n-1} \mathfrak{r}$ or $2^{-n-1} \mathfrak{r} \times 2^{-n} \mathfrak{r}$ rectangles $R \subseteq \mathfrak{r} \mathbb{S}$,  where recall that $\mathbb{S} = (0,1)^2$.  For $R \in \mathcal{R}_{\mathfrak{r}}^n$,  we let $w_R$ be the bottom-left corner of $R$ and set $N_C:=\lfloor \log_2(C^{\zeta}) \rfloor +1$ for $C>1$.  Then applying Proposition~\ref{prop:bound_on_lqg_distances_of_sets} with $2^{-n} \mathfrak{r}$ in place of $\mathfrak{r}$ and with $A = 2^{\zeta \xi n}$ and then taking a union bound,  we obtain that for some universal constants $b$ and $\wt{C}$,  we have that off an event with probability at most
\begin{align*}
\wt{C} \sum_{n \geq N_C} 2^{-\zeta \xi b \sqrt{M} n} \leq \left(\frac{\wt{C}}{1-2^{-\zeta \xi b \sqrt{M}}}\right) C^{-\zeta^2 \xi b \sqrt{M}}
\end{align*}
the following holds.  For all $n \geq N_C$ and all $R \in \mathcal{R}_{\mathfrak{r}}^n$,  the distance between the two shorter sides of $R$ with respect to the metric $\wt{D}_h^{\gamma}(\cdot,\cdot ;  R)$ is at most
\begin{align*}
2^{\zeta \xi n} (2^{-n} \mathfrak{r})^{\xi Q} e^{\xi h_{2^{-n} \mathfrak{r}}(w_R)}.
\end{align*}
Hence combining with Lemma~\ref{lem:circle_averages_uniform_bound},  we obtain that off an event with probability at most
\begin{align*}
&\left(\frac{\wt{C}\wt{\zeta}^{-1} \log(C)}{\log(2)}\right) \exp\left(\left(-\frac{q+\sqrt{q^2-4}}{1+\wt{C} \zeta^{-1} \log(C)^{-1}} + \frac{\wt{\zeta} (q+\zeta^{-1})^2}{2} + 2\wt{\zeta} \wt{C} + 2\wt{C} \log(C)^{-1}\right)\log(C)\right)\\
&+\left(\frac{\wt{C}}{1-2^{2-\frac{q^2}{2(1+\wt{C} \wt{\zeta})}}}\right) \exp\left(\left(\left(\frac{2}{1+\wt{C}\wt{\zeta}}\right) \left(\frac{4-q^2}{2\wt{\zeta}}\right) + \frac{2 \wt{C}}{1+\wt{C} \wt{\zeta}} \right)\log(C)\right) + \left(\frac{\wt{C}}{1-2^{-\zeta \xi b \sqrt{M}}}\right) C^{-\zeta^2 \xi b \sqrt{M}}
\end{align*}
the following holds.  For all $n \geq N_C$ and all $R \in \mathcal{R}_{\mathfrak{r}}^n$,  there is a path $P_R$ in $R$ between the two shorter sides of $R$ with $\wt{D}_h^{\gamma}$-length at most 
\begin{align*}
C^{\xi} 2^{(q+\zeta) \xi n} (2^{-n} \mathfrak{r})^{\xi Q} e^{\xi h_{\mathfrak{r}}(0)} = C^{\xi} 2^{-(Q-q-\zeta) \xi n} \mathfrak{r}^{\xi Q} e^{\xi h_{\mathfrak{r}}(0)}.
\end{align*}
Henceforth we assume that such paths $P_R$ exist.

Next we have the following lemma which gives a precise estimate on the tail probability of $\sup_{z ,  w \in \mathfrak{r} \mathbb{S}} \wt{D}_h^{\gamma}(z,w ; \mathfrak{r} \mathbb{S})$.

\begin{lemma}\label{lem:upper_bound_on_internal_diamters_of_spuares}
Let $\wt{C},b$ be the universal constants as in the statements of Proposition~\ref{prop:bound_on_lqg_distances_of_sets} and Lemma~\ref{lem:circle_averages_uniform_bound}.  Then we have that
\begin{align*}
&\p\left[(1-2^{-(Q-q-\zeta) \xi}) \mathfrak{r}^{\xi Q} e^{-\xi h_{\mathfrak{r}}(0)} \sup_{z ,w  \in \mathfrak{r} \mathbb{S}} \wt{D}_h^{\gamma}(z,w ; \mathfrak{r} \mathbb{S}) > C\right]\\
&\leq \left(\frac{\wt{C} \log(C)}{\log(2)\wt{\zeta}(\zeta + \xi)}\right) \exp\left(\left(-\frac{q+\sqrt{q^2-4}}{1+\wt{C}(\zeta + \xi) \zeta^{-1} \log(C)^{-1}} + \frac{\wt{\zeta} (q+\zeta^{-1})^2}{2} + 2\wt{\zeta} \wt{C} + 2\wt{C}(\zeta + \xi) \log(C)^{-1}\right)\frac{\log(C)}{\zeta + \xi}\right)\\
&+\left(\frac{\wt{C}}{1-2^{2-\frac{q^2}{2(1+\wt{C} \wt{\zeta})}}}\right) \exp\left(\left(\left(\frac{2}{1+\wt{C}\wt{\zeta}}\right) \left(\frac{4-q^2}{2\wt{\zeta}}\right) + \frac{2 \wt{C}}{1+\wt{C} \wt{\zeta}} \right)\frac{\log(C)}{\zeta + \xi}\right) + \left(\frac{\wt{C}}{1-2^{-\zeta \xi b \sqrt{M}}}\right) C^{-\frac{\zeta^2 \xi b \sqrt{M}}{\zeta+\xi}}
\end{align*}
for all $C>1$.
\end{lemma}

\begin{proof}
Suppose that we have the same setup as in the paragraph just before the statement of the lemma.  We also assume that we are working on the event that the paths $P_R$ constructed above exist.  The main idea of the proof is to string together the paths $P_R$ in order to bound from above the distance between any two fixed points in $\mathfrak{r} \mathbb{S}$.  This is very similar to the main idea of the proof of \cite[Proposition~3.10]{dubedat2020weak}.  Let us now proceed with the details of the proof.

As in \cite[Proposition~3.10]{dubedat2020weak},  for each square $S \subseteq \mathfrak{r} \mathbb{S}$ with side length $2^{-n} \mathfrak{r}$ and corners in $2^{-n} \mathfrak{r} \mathbb{S}$,  there are exactly four rectangles in $\mathcal{R}_{\mathfrak{r}}^n$ which are contained in $S$.  If $n \geq N_C$,  we let $X_S$ be the union of the paths $P_R$ for these four rectangles.  If $S'$ is one of the four dyadic children of $S$,  then we have that $X_S \cap X_{S'} \neq \emptyset$.  Hence every point of $X_S$ can be joined to $X_{S'}$ by a path in $S$ with $\wt{D}_h^{\gamma}$-length at most
\begin{align*}
4C^{\xi} 2^{-(Q-q-\zeta)\xi n} \mathfrak{r}^{\xi Q} e^{\xi h_{\mathfrak{r}}(0)}.
\end{align*}

Since $\wt{D}_h^{\gamma}$ is a continuous function on $\C \times \C$,  we obtain that if $S_n(z)$ for $n \in \N_0$ denotes the square of side length $2^{-n} \mathfrak{r}$ with corners in $2^{-n} \mathfrak{r} \Z^2$ which contains $z$,  then the $\wt{D}_h^{\gamma}$-diameter of $S_n(z)$ tends to zero as $n \to \infty$.  Therefore we have that
\begin{align*}
\sup_{w \in S_{N_C}(z)} \wt{D}_h^{\gamma}(z,w ; \mathfrak{r} \mathbb{S})&\leq  4C^{\xi} \mathfrak{r}^{\xi Q} e^{\xi h_{\mathfrak{r}}(0)} \sum_{n \geq N_C} 2^{-(Q-q-\zeta)\xi n} \leq 4\left(\frac{C^{-\zeta \xi (Q-q-\zeta)}}{1-2^{-(Q-q-\zeta)\xi}}\right) C^{\xi} \mathfrak{r}^{\xi Q} e^{\xi h_{\mathfrak{r}}(0)}\\
&\leq 4\left(\frac{C^{\xi}}{1-2^{-(Q-q-\zeta)\xi}}\right) \mathfrak{r}^{\xi Q} e^{\xi h_{\mathfrak{r}}(0)}.
\end{align*}
Since the above holds for all $z \in \mathfrak{r} \mathbb{S}$,  we obtain that for all $n \geq N_C$ and each $2^{-n} \mathfrak{r} \times 2^{-n} \mathfrak{r}$ square $S \subseteq \mathfrak{r} \mathbb{S}$ with corners in $2^{-n} \mathfrak{r} \Z^2$ has $\wt{D}_h^{\gamma}$-diameter at most 
\begin{align*}
4\left(\frac{C^{\xi}}{1-2^{-(Q-q-\zeta)\xi}}\right) \mathfrak{r}^{\xi Q} e^{\xi h_{\mathfrak{r}}(0)}.
\end{align*}
Note that $2^{N_C} \leq 2 C^{\zeta}$ and so combining with the triangle inequality we obtain that
\begin{align*}
\sup_{z,w \in \mathfrak{r} \mathbb{S}} \wt{D}_h^{\gamma}(z,w ; \mathfrak{r} \mathbb{S}) \leq \left(\frac{8C^{\xi + \zeta}}{1-2^{-(Q-q-\zeta)\xi}}\right) \mathfrak{r}^{\xi Q} e^{\xi h_{\mathfrak{r}}(0)}.
\end{align*}
The proof of the lemma is then complete by replacing $C$ with $C^{\frac{1}{\xi + \zeta}}$.
\end{proof}

Henceforth we assume that the constants $\wt{C},b$ are as in Lemma~\ref{lem:upper_bound_on_internal_diamters_of_spuares}.  Now we choose continuous functions $q,\zeta,\wt{\zeta},M : (0,2) \to \R_+$ satisfying the following properties.  First we note that since $Q(\gamma) + \sqrt{Q(\gamma)^2 - 4} = \frac{4}{\gamma}$ for all $\gamma \in (0,2)$,  and $Q$ is a decreasing function in $\gamma$,  we can choose a continuous and decreasing function $q : (0,2) \to \R_+$ such that
\begin{align}\label{eqn:conditions_for_the_function_q}
q(\gamma) < Q(\gamma) \quad \text{and} \quad q(\gamma) + \sqrt{q(\gamma)^2 - 4} > 2 \quad \text{for all} \quad \gamma \in (0,2).
\end{align}
Also since $\xi : (0,2) \to \R_+$ is a continuous and increasing function (see \cite[Proposition~1.7]{ding2020fractal}),  we can choose a continuous and increasing  function $\zeta : (0,2) \to \R_+$ such that
\begin{align}\label{eqn:conditions_for_the_function_zeta}
0<\zeta(\gamma) < Q(\gamma) - q(\gamma) \quad,   \frac{q(\gamma) + \sqrt{q(\gamma)^2 - 4}}{\xi(\gamma) + \zeta(\gamma)} > \frac{2}{\xi(\gamma)} \quad \text{and} \quad \frac{q(\gamma)^2}{2(1+\wt{C} \zeta(\gamma))} > 2 \quad \text{for all} \quad \gamma \in (0,2).
\end{align}
Moreover we choose a continuous and increasing function $\wt{\zeta} : (0,2) \to \R_+$ such that
\begin{align}\label{eqn:condition_1_for_the_function_wt{zeta}}
\left(q(\gamma) + \sqrt{q(\gamma)^2 - 4} - \frac{\wt{\zeta}(\gamma) (q(\gamma) + 1)^2}{2} - 2\wt{C} \wt{\zeta}(\gamma)\right) > 2\left(1+\frac{\zeta(\gamma)}{\xi(\gamma)}\right) \quad \text{for all} \quad \gamma \in (0,2)
\end{align}
and
\begin{align}\label{eqn:condition_2_for_the_function_wt{zeta}}
\left(\frac{1}{1+\wt{C} \wt{\zeta}(\gamma)}\right) \left(\frac{q(\gamma)^2 - 4}{2 \wt{\zeta}(\gamma)} - \frac{2\wt{C}}{1+\wt{C} \wt{\zeta}(\gamma)}\right) >2\left(1+\frac{\zeta(\gamma)}{\xi(\gamma)}\right) \quad \text{for all} \quad \gamma \in (0,2). 
\end{align}
Finally we choose a continuous and decreasing function $M : (0,2) \to \R_+$ such that 
\begin{align}\label{eqn:condition_for_the_function_M}
\frac{\zeta(\gamma)^2 \xi(\gamma) b \sqrt{M(\gamma)}}{\xi(\gamma) + \zeta(\gamma)} > \frac{2}{\xi(\gamma)} \quad \text{for all} \quad \gamma \in (0,2).
\end{align}

Now we are ready to choose our normalization constant $\mathfrak{p}$ in \eqref{eqn:main_normalization} as a decreasing function in $\gamma \in (0,2)$.  More precisely,   for all $\gamma \in (0,2)$,  we let $\mathfrak{p}(\gamma)$ be the minimum value of $\mathfrak{p} \in (0,1)$ so that the statement of Lemma~\ref{lem:good_event_occurs_almost_everywhere} holds with $M = M(\gamma)$.  Note that the proof of Lemma~\ref{lem:good_event_occurs_almost_everywhere} (see also Lemma~\ref{lem:good_event_at_many_scales}) implies that $\mathfrak{p}(\gamma)$ can be chosen in a way which is continuous in $\gamma$ and since $M(\gamma)$ is decreasing in $\gamma$,  we obtain that the function $\mathfrak{p}(\gamma)$ is decreasing in $\gamma$.

From now on we will assume that $D_h^{\gamma}$ is normalized as in \eqref{eqn:main_normalization} with $\mathfrak{p} = \mathfrak{p}(\gamma)$.

\begin{proof}[Proof of Proposition~\ref{prop:moment_lqg_diameter_bound_for_square}]
First we note that it follows from \eqref{eqn:conditions_for_the_function_q} - \eqref{eqn:condition_for_the_function_M} combined with Lemma~\ref{lem:upper_bound_on_internal_diamters_of_spuares} and the continuity in $\gamma$ of the functions $q,\zeta,\wt{\zeta},M$ and $\mathfrak{p}$ that there exist constants $a>2,C_0>1$ depending only on $\gamma_*$ and $\gamma^*$ such that
\begin{align*}
\p\left[\mathfrak{r}^{-\xi(\gamma) Q(\gamma)} e^{-\xi(\gamma) h_{\mathfrak{r}}(0)} \sup_{z,w \in \mathfrak{r} \mathbb{S}} \wt{D}_h^{\gamma}(z,w ; \mathfrak{r} \mathbb{S} > C \right] \leq \wt{C} C^{-\frac{a}{\xi(\gamma)}} \quad \text{for all} \quad \gamma \in [\gamma_* ,  \gamma^*],\,\,C \geq C_0.
\end{align*}
Thus it follows that for fixed $\beta \in (2,a)$ and $p = \frac{\beta}{\xi}$,
\begin{align*}
&\E\left[\left(\mathfrak{r}^{-\xi(\gamma) Q(\gamma)} e^{-\xi(\gamma) h_{\mathfrak{r}}(0)} \sup_{z,w \in \mathfrak{r} \mathbb{S}} \wt{D}_h^{\gamma}(z,w ; \mathfrak{r} \mathbb{S} \right)^p \right] \\
&\leq C_0 + \int_{C_0}^{\infty} C^{p-1} \p\left[\mathfrak{r}^{-\xi(\gamma) Q(\gamma)} e^{-\xi(\gamma) h_{\mathfrak{r}}(0)} \sup_{z,w \in \mathfrak{r} \mathbb{S}} \wt{D}_h^{\gamma}(z,w ; \mathfrak{r} \mathbb{S}) > C \right]dC\\
&\leq C_0 + \wt{C} \int_{C_0}^{\infty} C^{p-\frac{a}{\xi(\gamma)}-1} dC \lesssim 1,
\end{align*}
where the implicit constant depends only on $\beta,\gamma_*$ and $\gamma^*$.  This completes the proof of the lemma.
\end{proof}

\begin{proof}[Proof of Proposition~\ref{prop:moment_lqg_diameter_bound_for_compact_sets}]
We will follow the argument in the proof of \cite[Proposition~3.9]{dubedat2020weak}.  More precisely,  since $K$ is compact and connected,  we cover $K$ by finitely many Euclidean squares $S_1,\cdots,S_N$ which are contained in $U$,  chosen in a way which depends only on $K$ and $U$,  and such that $S_j \cap S_{j+1} \neq \emptyset$ for all $1 \leq j \leq N-1$.  For $n=1,\cdots,N$,  we let $u_n$ be the bottom left corner of $S_n$ and let $\rho_n$ be the side length of $S_n$.  Let $a>2$ be the constant of Proposition~\ref{prop:moment_lqg_diameter_bound_for_square} and fix $\beta \in (2,a)$. Set  $p = \frac{\beta}{\xi(\gamma)}$ and let $q>1$ be such that $\beta q < a$ and $q$ depends only on $\beta$ and $a$.  Let also $q'>1$ be such that $\frac{1}{q} + \frac{1}{q'} = 1$.

Then the choice of the $S_j$s implies that
\begin{align*}
\sup_{z,w \in \mathfrak{r} K} \wt{D}_h^{\gamma}(z,w ; \mathfrak{r} U) \leq \sum_{n=1}^N \sup_{z,w \in \mathfrak{r} S_n} \wt{D}_h^{\gamma}(z,w ; \mathfrak{r} S_n)
\end{align*}
which implies that
\begin{align}\label{eqn:basic_inequality_1}
\E\left[\left(\mathfrak{r}^{-\xi(\gamma) Q(\gamma)} e^{-\xi(\gamma) h_{\mathfrak{r}}(0)} \sup_{z,w \in \mathfrak{r} K} \wt{D}_h^{\gamma}(z,w ; \mathfrak{r} U) \right)^p\right] \leq N^p \sum_{n=1}^N \E\left[\left(\mathfrak{r}^{-\xi(\gamma) Q(\gamma)} e^{-\xi(\gamma) h_{\mathfrak{r}}(0)} \sup_{z,w \in \mathfrak{r} S_n} \wt{D}_h^{\gamma}(z,w ; \mathfrak{r} S_n) \right)^p\right].
\end{align}
Note that H\"older's inequality implies that for each $1\leq n \leq N$,  we have that
\begin{align}\label{eqn:basic_inequality_2}
&\E\left[\left(\mathfrak{r}^{-\xi(\gamma) Q(\gamma)} e^{-\xi(\gamma) h_{\mathfrak{r}}(0)} \sup_{z,w \in \mathfrak{r} S_n} \wt{D}_h^{\gamma}(z,w ; \mathfrak{r} S_n) \right)^p\right]\notag\\
&\leq \rho_n^{p \xi(\gamma) Q(\gamma)} \E\left[\exp\left(q' p \xi(\gamma)(h_{\rho_n \mathfrak{r}}(\mathfrak{r} u_n) - h_{\mathfrak{r}}(0))\right)\right]^{1/q'} \E\left[\left((\mathfrak{r} \rho_n)^{-\xi(\gamma) Q(\gamma)} e^{-\xi(\gamma) h_{\rho_n \mathfrak{r}}(\mathfrak{r} u_n)} \sup_{z,w \in \mathfrak{r} S_n} \wt{D}_h^{\gamma}(z,w ; \mathfrak{r} S_n) \right)^{p q} \right]^{1/q}.
\end{align}
Also the scale invariance of the whole-plane GFF implies that
\begin{align*}
\E\left[\exp\left(q' p \xi(\gamma) (h_{\mathfrak{r} \rho_n}(\mathfrak{r} u_n) - h_{\mathfrak{r}}(0))\right) \right] = \E\left[\exp\left(q' p \xi(\gamma) h_{\rho_n}(u_n)\right)\right] = \exp\left(\frac{(q')^2 p^2 \xi(\gamma)^2}{2} \text{Var}(h_{\rho_n}(u_n))\right).
\end{align*}

Note that $p^2 \xi(\gamma)^2 = \beta^2$ and $\text{Var}(h_{\rho_n}(u_n)) \lesssim 1$,  where the implicit constant depends only on $K$ and $U$.  Hence the proof of the proposition is complete by combining Proposition~\ref{prop:moment_lqg_diameter_bound_for_square} with \eqref{eqn:basic_inequality_1} and summing over all $n = 1,\cdots,N$ in \eqref{eqn:basic_inequality_2}.
\end{proof}

\subsection{Proving tightness: Subcritical case}
\label{subsec:tightness_subcritical}

In this subsection,  we are going to complete the proof of Theorem~\ref{thm:tightness_subcritical}.  For the rest of the section,  we will assume that the metrics $\wt{D}_h^{\gamma}$ are as in \eqref{eqn:normalization_subcritical},  where the function $\mathfrak{p} : (0,2) \to (0,1)$ is chosen according to Proposition~\ref{prop:moment_lqg_diameter_bound_for_compact_sets}.  

Theorem~\ref{thm:tightness_subcritical} will be a consequence of the following proposition.

\begin{proposition}\label{prop:tightness_proposition_subcritical}
Fix $0<\gamma_* < \gamma^* < 2$ and $\chi \in (0,(Q(\gamma^*)-2) \xi(\gamma_*))$.  Fix also an open and connected set $U \subseteq \C$ and a   compact and connected set $K \subseteq U$ with more than one point.  Then we have that
\begin{align*}
\sup_{\gamma \in [\gamma_*,\gamma^*]} \p\left[\sup_{z,w \in K} \frac{\wt{D}_h^{\gamma}(z,w ; U)}{|z-w|^{\chi}} \geq R \right] \to 0 \quad \text{as} \quad R \to \infty.
\end{align*}
\end{proposition}

Proposition~\ref{prop:tightness_proposition_subcritical} will follow from the following lemma which gives a H\"older continuity estimate for $\wt{D}_h^{\gamma}$-distances.

\begin{lemma}\label{lem:lqg_diameter_upper_bound_internal_metric_on_points}
Fix $0<\gamma_* < \gamma^* < 2$ and a compact set $K \subseteq \C$.  Let $a >2$ be the constant in the statement of Proposition~\ref{prop:moment_lqg_diameter_bound_for_compact_sets}.  Fix $\beta \in (2,a \wedge Q(\gamma^*))$ and set $\chi(\gamma):=(Q(\gamma)-\beta)\xi(\gamma)$.  Then for all $\mathfrak{r}>0$ and all $\gamma \in [\gamma_*,\gamma^*]$,   it holds with polynomially high probability as $\epsilon \to 0$,  at a rate depending only on $a,\beta,\gamma_*,\gamma^*$ and $K$,  that
\begin{align*}
\mathfrak{r}^{-\xi(\gamma) Q(\gamma)} e^{\xi(\gamma) h_{\mathfrak{r}}(0)} \wt{D}_h^{\gamma}(u,v ; B_{2|u-v|}(u)) \leq \left |\frac{u-v}{\mathfrak{r}}\right |^{\chi(\gamma)}
\end{align*}
for all $u,v \in \mathfrak{r} K$ with $|u-v| \leq \epsilon \mathfrak{r}$.
\end{lemma}

Before we give the proof of Lemma~\ref{lem:lqg_diameter_upper_bound_internal_metric_on_points},  we state and prove the following lemma about tail probabilities of $\wt{D}_h^{\gamma}$-diameters of small Euclidean balls.

\begin{lemma}\label{lem:lqg_diameter_upper_bound_internal_metric_on_balls}
Fix $0<\gamma_* < \gamma^*<2$ and let $a >2$ be the constant in the statement of Proposition~\ref{prop:moment_lqg_diameter_bound_for_compact_sets}.  Let $K \subseteq \C$ be a compact set and fix $\beta \in (2,a \wedge Q(\gamma^*))$ and $\delta \in (0,\frac{\beta^2}{2}-2)$.  Then there exists a constant $C>0$ depending only on $\gamma_*,\gamma^*,\beta,\delta$ and $K$ such that
\begin{align*}
\p\left[\sup_{u,v \in B_{\epsilon \mathfrak{r}}(z)} \wt{D}_h^{\gamma}(u,v ; B_{2\epsilon \mathfrak{r}}(z)) \leq \epsilon^{(Q(\gamma)-\beta)\xi(\gamma)} \mathfrak{r}^{\xi(\gamma) Q(\gamma)} e^{\xi(\gamma) h_{\mathfrak{r}}(0)}\right] \geq 1 - C \epsilon^{\frac{\beta^2}{2}-\delta}
\end{align*}
for all $\epsilon \in (0,1),\mathfrak{r}>0,z \in \mathfrak{r} K$ and all $\gamma \in [\gamma_*,\gamma^*]$.
\end{lemma}

\begin{proof}
We will follow the same argument as in the proof of \cite[Lemma~3.19]{dubedat2020weak}.  Set $p = \frac{\beta}{\xi(\gamma)}$ and $s = \xi(\gamma) Q(\gamma) - p \xi(\gamma)^2 = (Q(\gamma)-\beta) \xi(\gamma)$.  Fix also $\mathfrak{r} > 0$.

First we note that $h_{2\epsilon \mathfrak{r}}(z)-h_{\mathfrak{r}}(z)$ is a centered Gaussian random variable with variance $\log(\epsilon^{-1})-\log(2)$ which is independent from $(h-h_{2\epsilon \mathfrak{r}}(z))|_{B_{2\epsilon \mathfrak{r}}(z)}$.  Hence Axioms~\ref{it:locality} and ~\ref{it:weyl_scaling} imply that the random variable $h_{2\epsilon \mathfrak{r}}(z)-h_{\mathfrak{r}}(z)$ is also independent from the internal metric
\begin{align*}
D_{h-h_{2\epsilon \mathfrak{r}}(z)}(u,v ; B_{2\epsilon \mathfrak{r}}(z)) = e^{-\xi(\gamma) h_{2\epsilon \mathfrak{r}}(z)} D_h(u,v ; B_{2\epsilon \mathfrak{r}}(z)).
\end{align*}
Therefore combining with Proposition~\ref{prop:moment_lqg_diameter_bound_for_compact_sets},  we obtain that
\begin{align*}
\E\left[\left(\mathfrak{r}^{-\xi(\gamma)Q(\gamma)} e^{-\xi(\gamma) h_{\mathfrak{r}}(z)} \sup_{u,v \in B_{\epsilon \mathfrak{r}}(z)} \wt{D}_h^{\gamma}(u,v ; B_{2\epsilon \mathfrak{r}}(z))\right)^p\right] \leq C \epsilon^{p\xi(\gamma) Q(\gamma) - \frac{\xi(\gamma)^2 p^2}{2}},
\end{align*}
for all $\epsilon \in (0,1),\gamma \in [\gamma_*,\gamma^*]$ and all $z \in \mathfrak{r} K$ where $C<\infty$ depends only on $\beta,\gamma_*$ and $\gamma^*$.  Thus combining with Markov's inequality and since $p\xi(\gamma) Q(\gamma) -\frac{p^2 \xi(\gamma)^2}{2} - ps = \frac{\beta^2}{2}$,  we obtain that
\begin{align}\label{eqn:tail_bound_without_recentering}
\p\left[\sup_{u,v \in B_{\epsilon \mathfrak{r}}(z)} \wt{D}_h^{\gamma}(u,v ; B_{2\epsilon \mathfrak{r}}(z)) \geq \epsilon^{s} \mathfrak{r}^{\xi(\gamma) Q(\gamma)} e^{\xi(\gamma) h_{\mathfrak{r}}(z)}\right] \leq C \epsilon^{\frac{\beta^2}{2}} \quad \text{for all} \quad \epsilon \in (0,1).
\end{align}

Fix $\theta \in (0,1)$ sufficiently small (to be chosen and depending only on $\beta,\delta,\gamma_*$ and $\gamma^*$).  Note that the random variables $h_{\mathfrak{r}}(z)-h_{\mathfrak{r}}(0)$ for $z \in \mathfrak{r} K$ are Gaussians and there exists a constant $A<\infty$ depending only on $K$ such that $\text{Var}(h_{\mathfrak{r}}(z)-h_{\mathfrak{r}}(0)) \leq A$ for all $z \in \mathfrak{r} K$.  Hence the standard Gaussian tail bound implies that
\begin{align}\label{eqn:gaussian_tail_bound}
\p[|h_{\mathfrak{r}}(z)-h_{\mathfrak{r}}(0)\ \geq \theta \log(\epsilon^{-1})] \leq 2 \exp\left(-\frac{\theta^2 \log(\epsilon^{-1})^2}{2A}\right) \quad \text{for all} \quad \epsilon \in (0,1),  z \in \mathfrak{r} K.
\end{align}

Hence combining \eqref{eqn:tail_bound_without_recentering} with \eqref{eqn:gaussian_tail_bound},  we obtain that
\begin{align*}
&\p\left[\sup_{u,v \in B_{\epsilon \mathfrak{r}}(z)} \wt{D}_h^{\gamma}(u,v ; B_{2\epsilon \mathfrak{r}}(z)) \geq \epsilon^{s} \mathfrak{r}^{\xi(\gamma) Q(\gamma)} e^{\xi(\gamma) h_{\mathfrak{r}}(0)}\right] \\
&\leq 2 \exp\left(-\frac{\theta^2 \log(\epsilon^{-1})^2}{2A}\right) +\p\left[\sup_{u,v \in B_{\epsilon \mathfrak{r}}(z)} \wt{D}_h^{\gamma}(u,v ; B_{2\epsilon \mathfrak{r}}(z)) \geq \epsilon^{s+\theta \xi(\gamma)} \mathfrak{r}^{\xi(\gamma) Q(\gamma)} e^{\xi(\gamma) h_{\mathfrak{r}}(z)}\right]\\
&\leq 2 \exp\left(-\frac{\theta^2 \log(\epsilon^{-1})^2}{2A}\right) + C \epsilon^{\frac{\beta^2}{2(1+\theta \xi(\gamma) s^{-1})}} \quad \text{for all} \quad \epsilon \in (0,1) ,  z \in \mathfrak{r} K.
\end{align*}

Note that $\xi(\gamma) s^{-1} \leq \frac{\xi(\gamma^*)}{\xi(\gamma^*) (Q(\gamma^*) - \beta)}$ for all $\gamma \in [\gamma_*,\gamma^*]$,  and so the proof of the lemma is complete by choosing $\theta \in (0,1)$ sufficiently small (depending only only on $\beta,\delta,\gamma_*$ and $\gamma^*$) such that
\begin{align*}
\frac{\beta^2}{2(1+\theta \xi(\gamma)s^{-1})} > \frac{\beta^2}{2}-\delta.
\end{align*}
\end{proof}

\begin{proof}[Proof of Lemma~\ref{lem:lqg_diameter_upper_bound_internal_metric_on_points}]
It follows from the same argument used to prove \cite[Lemma~3.20]{dubedat2020weak} using Lemma~\ref{lem:lqg_diameter_upper_bound_internal_metric_on_balls} in place of \cite[Lemma~3.19]{dubedat2020weak}.
\end{proof}

\begin{proof}[Proof of Proposition~\ref{prop:tightness_proposition_subcritical}]
Let $a>2$ be the constant in Lemma~\ref{lem:lqg_diameter_upper_bound_internal_metric_on_points} and let $\beta \in (2,a \wedge Q(\gamma^*))$ be such that 
\begin{align*}
\chi < (Q(\gamma^*) - \beta) \xi(\gamma_*) \leq (Q(\gamma)-\beta) \xi(\gamma) \quad \text{for all} \quad \gamma \in [\gamma_*,\gamma^*].
\end{align*}

Fix $\delta \in (0,1)$.  Then Lemma~\ref{lem:lqg_diameter_upper_bound_internal_metric_on_points} implies that there exists $\epsilon \in (0,1)$ depending only on $\delta,\beta,K,U,\gamma_*$ and $\gamma^*$ such that
\begin{align}\label{eqn:holder_constant_upper_bound}
\inf_{\gamma \in [\gamma_*,\gamma^*]} \p\left[\wt{D}_h^{\gamma}(z,w ; U) \leq |z-w|^{\chi} \,\,\text{for all}\,\,z,w \in K \,\,\text{with}\,\,|z-w| \leq \epsilon \right] \geq 1 - \frac{\delta}{2}.
\end{align}
Moreover Proposition~\ref{prop:moment_lqg_diameter_bound_for_compact_sets} combined with Markov's inequality imply that that there exists $R>0$ depending only on $\delta,\beta,K,U,\gamma_*$ and $\gamma^*$ such that
\begin{align}\label{eqn:diameter_upper_bound}
\sup_{\gamma \in [\gamma_*,\gamma^*]} \p\left[\sup_{z,w \in K} \wt{D}_h^{\gamma}(z,w ; U) \geq R \right] \leq \frac{\delta}{2}.
\end{align}

Therefore combining \eqref{eqn:holder_constant_upper_bound} with \eqref{eqn:diameter_upper_bound},  we obtain that for all $\gamma \in [\gamma_*,\gamma^*]$,  we have that the following holds off an event with probability at most $\delta$:
\begin{align*}
&\wt{D}_h^{\gamma}(z,w ; U) \leq |z-w|^{\chi} \quad \text{for all}\quad z,w \in K \quad \text{with} \quad |z-w| \leq \epsilon,\\
&\wt{D}_h^{\gamma}(z,w ; U) \leq R \quad \text{for all} \quad z,w \in K.
\end{align*}
In particular,  for all $\gamma \in [\gamma_*,\gamma^*]$,  off an event with probability at most $\delta$ we have that
\begin{align*}
\wt{D}_h^{\gamma}(z,w ; U) \leq R \epsilon^{-\chi} |z-w|^{\chi} \quad \text{for all} \quad z,w \in K.
\end{align*}
This completes the proof of the proposition.
\end{proof}

\begin{proof}[Proof of Theorem~\ref{thm:tightness_subcritical}]
Fix $\epsilon \in (0,1)$ and let $(K_n)_{n \in \N}$ be a sequence of compact and connected subsets of $U$ such that $K_n \subseteq K_{n+1}$ for all $n \in \N$ and $U = \bigcup_{n \in \N} K_n$.  Then Proposition~\ref{prop:tightness_proposition_subcritical} implies that for fixed $\chi \in (0,(Q(\gamma^*)-2) \xi(\gamma_*))$,  we have that for all $n \in \N$ there exists $R_n >0$ depending only on $\epsilon,\chi,n,U,\gamma_*$ and $\gamma^*$ such that
\begin{align}\label{eqn:tail_bound_locally_uniformly}
\sup_{\gamma \in [\gamma_*,\gamma^*]} \p\left[\sup_{z,w \in K_n} \frac{\wt{D}_h^{\gamma}(z,w ; U)}{|z-w|^{\chi}} \geq R_n \right] \leq \frac{\epsilon}{2^n}.
\end{align}

Let $\mathcal{K}$ denote the set of all continuous functions $f$ on $U \times U \to \R$ such that
\begin{align*}
\sup_{z,w \in K_n} \frac{|f(z)-f(w)|}{|z-w|^{\chi}} \leq R_n \quad \text{for all} \quad n \in \N.
\end{align*}
Then clearly the set $\mathcal{K}$ is compact with respect to the local uniform topology on $U \times U$ by the Arzela-Ascoli theorem.  Also \eqref{eqn:tail_bound_locally_uniformly} implies that
\begin{align*}
\p[\wt{D}_h^{\gamma}(\cdot ,  \cdot ; U) \in \mathcal{K}] \geq 1-\epsilon \quad \text{for all} \quad \gamma \in [\gamma_*,\gamma^*].
\end{align*}
This completes the proof of the theorem.
\end{proof}

\section{Identifying the limit: Subcritical case}
\label{sec:identifying_limit_subcritical}

In this section,  we are going to identify the law of any possible subsequential limit in Theorem~\ref{thm:tightness_subcritical} as the limiting metric in Theorem~\ref{thm:convergence_of_lffp} (modulo a multiplicative deterministic constant) for an appropriate parameter $\gamma \in (0,2)$ and hence complete the proof of Theorem~\ref{thm:main_theorem_subcritical_intro}.  As explained in Subsection~\ref{subsec:outline},  this will be done in two main steps.  First we will prove that any subsequential limit is a metric a.s.  (Subsection~\ref{subsec:limit_is_a_metric_subcritical}).  Next we will show in Subsections~\ref{subsec:limit_geodesic_and_complete_metric}-\ref{subsec:locality_subcritical} that any subsequential limit satisfies the axioms of LQG metrics (Definition~\ref{def:strong_lqg_metric_subcritical}).  Finally we will complete the proof of Theorem~\ref{thm:main_theorem_subcritical_intro} in Subsection~\ref{subsec:limit_is_the_gamma_lqg_metric} by combining with Theorem~\ref{thm:uniqueness_subcritical}.

For the rest of the section,  we fix $\gamma \in (0,2)$ and let $(\gamma_n)_{n \in \N}$ be a sequence in $(0,2)$ such that $\gamma_n \to \gamma$ as $n \to \infty$.  Note that Theorem~\ref{thm:tightness_subcritical} combined with Skorokhod's representation theorem imply that we can find a subsequence $(\gamma_{k_n})_{n \in \N}$ and a coupling of random fields $(h^n)_{n \in \N},h$ on the same probability space such that $h^n \to h$ as $n \to \infty$ in $H^{-1}_{\text{loc}}(\C)$ a.s.,  and there exists a random function $\wt{D} : \C \times \C \to \R$ on the same probability space such that a.s.
\begin{align*}
\wt{D}_{h^n}^{\gamma_{k_n}} \to \wt{D} \quad \text{as} \quad n \to \infty
\end{align*}
with respect to the local uniform topology on $\C \times \C$.  Henceforth we assume that we are working with that coupling.  Also in order to make the notation easier,  we will assume without loss of generality that $\gamma_n = \gamma_{k_n}$ and we will denote the pair $(h^n ,  \wt{D}_{h^n}^{\gamma_n})$ by $(h,\wt{D}_h^n)$.

\subsection{The limit is a metric a.s.}
\label{subsec:limit_is_a_metric_subcritical}

In this subsection,  we will prove that $\wt{D}$ is a metric a.s.  In particular we will show the following.

\begin{proposition}\label{prop:limit_is_a_metric}
$\wt{D}$ is a.s.  a metric which induces the Euclidean topology.
\end{proposition}

Clearly the function $\wt{D}$ satisfies the triangle inequality a.s.  Hence it remains to show that 
\begin{align*}
\wt{D}(z,w) >0 \quad \text{for all} \quad z,w \in \C,  z \neq w \quad \text{a.s.}
\end{align*}

In order to prove Proposition~\ref{prop:limit_is_a_metric},  we will follow the argument in \cite[Section~6.2]{ding2023tightness}.  As in \cite{ding2023tightness},  the main ingredient of the proof of Proposition~\ref{prop:limit_is_a_metric} is the following proposition which is the analogue of \cite[Proposition~6.4]{ding2023tightness}.

\begin{proposition}\label{prop:annuli_crossing_always_positive}
It holds that
\begin{align*}
\lim_{\epsilon \to 0} \liminf_{n \to \infty} \left(\p\left[\wt{D}_h^{\gamma_n}(\text{across} \,\,  \mathbb{A}_{r_1,r_2}(0)) > \epsilon \right] \right) = 1 \quad \text{for all} \quad 0<r_1 < r_2 < \infty.
\end{align*}
\end{proposition}

Now we focus on proving Proposition~\ref{prop:annuli_crossing_always_positive}.  The proof of Proposition~\ref{prop:annuli_crossing_always_positive} will follow from combining Lemmas~\ref{lem:uniform_lower_bound_across_annuli} and ~\ref{lem:crossing_positive}.  Lemma~\ref{lem:uniform_lower_bound_across_annuli} shows that the crossing distance of a fixed annulus with respect to $\wt{D}_h^{\gamma_n}$ remains bounded away from zero with positive probability which is uniform in $n \in \N$.  In Lemma~\ref{lem:crossing_positive},  we will carry out the $0-1$ law argument described in Subsection~\ref{subsec:outline} by proving that $\p[\wt{D}(\text{across} \,\,  A) > 0] \in \{0,1\}$ for a fixed Euclidean annulus $A$.

Before stating Lemma~\ref{lem:uniform_lower_bound_across_annuli},  we state the following lemma which shows that the $\wt{D}_h^{\gamma_n}$-distance between any two fixed points on the inner and outer boundary respectively of a fixed annulus is bounded away from zero with positive probability (uniformly in $n \in \N$).

\begin{lemma}\label{lem:uniform_lower_bound_between_points}
There exists some universal constant $r_0 \in (0,1)$ such that for all $r \in (0,r_0)$,  there exists a constant $c>0$ depending only on $r$ and the sequence $(\gamma_n)$ such that
\begin{align*}
\liminf_{n \to \infty} \inf_{x \in \partial B_r(0),  y \in \partial B_{2r}(0)}\left( \p\left[\wt{D}_h^{\gamma_n}(x,y ; B_{3r}(0)) > c\right]\right)>0.
\end{align*}
\end{lemma}

\begin{proof}
\emph{Step 1.  Outline and setup.}
Let us first describe the strategy of the proof of the lemma. First in Step 2,  we will show that there exist points $\wt{x} ,  \wt{y} \in B_{\frac{3r}{2}}(0)$ and constants $c>0,p \in (0,1)$ such that
\begin{align}\label{eqn:lower_tail_bound_for_fixed_points}
\p\left[\wt{D}_h^{\gamma_n}(\wt{x},\wt{y}; B_{3r}(\wt{x})) > c \right] \geq p \quad \text{for all} \quad n \in \N.
\end{align}
Next in Step 3,  we will argue as in the proof of \cite[Lemma~6.5]{ding2023tightness}.  More precisely,  we will argue that there exists a sequence of points $\wt{x} = z_1,\cdots,z_J = \wt{y}$ and $y_1,\cdots,y_J$ such that for all $1 \leq i \leq J$,  the triple $(z_i,z_{i+1},B_{3r}(y_i))$ can be obtained from the triple $(x,y,B_{3r}(0))$ through translation and rotation.  Since 
\begin{align*}
\wt{D}_h^{\gamma_n}(\wt{x},\wt{y} ; B_{3r}(\wt{x})) \leq \sum_{i=1}^{J-1} \wt{D}_h^{\gamma_n}(z_i,z_{i+1} ; B_{3r}(y_i)),
\end{align*}
the proof will be complete by combining with \eqref{eqn:lower_tail_bound_for_fixed_points} and the translation and scaling properties of $\wt{D}_h^{\gamma_n}$ (Axiom~\ref{it:coordinate_change}).
 
\emph{Step 2.  $\wt{D}_h^{\gamma_n}(\wt{x},\wt{y} ; B_{3r}(\wt{x})) > c$ for some points $\wt{x},\wt{y}$ with positive probability uniformly in $n$.}
Fix $\delta \in (0,1)$ sufficiently small (to be chosen later in the proof in a universal way) and let $x_1,\cdots,x_N$ be distinct points on $\partial B_{\frac{3}{2}}(0)$ ordered in the clockwise way such that $\frac{\delta}{3} < |x_j-x_{j+1}| < \frac{\delta}{2}$ for all $1 \leq j \leq N$,  where $x_{N+1} = x_1$. Note that
\begin{align*}
\p\left[\wt{D}_h^{\gamma_n}(\text{around} \,\,\mathbb{A}_{1,2}(0)) > 1 \right] = 1-\mathfrak{p}(\gamma_n) \quad \text{for all} \quad n \in \N,
\end{align*}
where recall that the function $\mathfrak{p} : (0,2) \to (0,1)$ has been defined in Subsection~\ref{subsec:moment_bound_for_diameters}.  Since $\mathfrak{p}_0$ is a decreasing function,  we can choose $q \in (0,1)$ such that
\begin{align*}
\p\left[\wt{D}_h^{\gamma_n}(\text{around} \,\,\mathbb{A}_{1,2}(0)) > 1 \right] \geq q \quad \text{for all} \quad n \in \N.
\end{align*}

Fix $n \in \N$ and for all $1 \leq j \leq N$,  we let $P_{j,n}$ be a path in $\mathbb{A}_{1,2}(0)$ from $x_j$ to $x_{j+1}$ with minimal $\wt{D}_h^{\gamma_n}(\cdot,\cdot ; B_{2|x_j - x_{j+1}|}(x_j))$-length.  Let $P_n$ denote the concatenation of the paths $P_{1,n},\cdots,P_{N,n}$.  Then we can choose $\delta \in (0,1)$ sufficiently small (in a universal way) such that $P_n$ is a path disconnecting $\partial B_1(0)$ from $\partial B_2(0)$.  Suppose that the event $\{\wt{D}_h^{\gamma_n}(\text{around} \,\, \mathbb{A}_{1,2}(0)) > 1\}$ occurs.  Then we have that
\begin{align*}
1 < \len(P_n ; \wt{D}_h^{\gamma_n}(\cdot,\cdot ; \mathbb{A}_{1,2}(0))) \leq \sum_{j=1}^N \len(P_{j,N} ; \wt{D}_h^{\gamma_n}(\cdot,\cdot ; \mathbb{A}_{1,2}(0))
\end{align*}
and so there exists $1 \leq j_n \leq N$ such that
\begin{align*}
\len(P_{j_n,n} ; \wt{D}_h^{\gamma_n}(\cdot , \cdot ; \mathbb{A}_{1,2}(0))) > \frac{1}{N}.
\end{align*}
In particular we have that
\begin{align*}
\wt{D}_h^{\gamma_n}(x_{j_n} ,  x_{j_n + 1} ; \mathbb{A}_{1,2}(0)) > \frac{1}{N}.
\end{align*}
Combining with the previous paragraph,  we obtain that for all $n \in \N$,  there exist deterministic points $x_{j_n},x_{j_n + 1} \in \partial B_{\frac{3}{2}}(0)$ and $1 \leq j_n \leq N$ such that
\begin{align*}
\p\left[\wt{D}_h^{\gamma_n}(x_{j_n},x_{j_n+1} ; \mathbb{A}_{1,2}(0)) > \frac{1}{N}\right] \geq \frac{q}{N}.
\end{align*}
Note that possibly by taking $\delta \in (0,1)$ to be smaller (in a universal way) we can assume that $B_{2|x_{j_n} - x_{j_n+1}|}(x_{j_n}) \subseteq \mathbb{A}_{1,2}(0)$.  It follows that
\begin{align*}
\p\left[\wt{D}_h^{\gamma_n}(x_{j_n},x_{j_n+1} ; B_{2|x_{j_n}-x_{j_n+1}|}(x_{j_n})) > \frac{1}{N} \right] \geq q \quad \text{for all} \quad n \in \N.
\end{align*}

\emph{Step 3.  Conclusion of the proof.}
Note that Axiom~\ref{it:coordinate_change} combined with Step 2 imply that
\begin{align*}
\p\left[\wt{D}_h^{\gamma_n}(0,|x_{j_n}-x_{j_n+1}| ; B_{2|x_{j_n}-x_{j_n+1}|}(0)) > \frac{1}{N} \right] \geq \frac{q}{N} \quad \text{for all} \quad n \in \N.
\end{align*}
Set $d_n:=|x_{j_n}-x_{j_n+1}|>0$ and let $r_0 \in (0,1)$ be sufficiently small (to be chosen in a universal way).  Fix $r \in (0,r_0)$.  Then arguing as in the proof of \cite[Lemma~6.5]{ding2023tightness} and for $r_0 \in (0,1)$ sufficiently small (in a universal way),  we obtain that for fixed $x \in \partial B_r(0)$ and $y \in \partial B_{2r}(0)$,  there exists a sequence of points $0=\wt{z}_1,\cdots,\wt{z}_J = 1$ and a constant $C<\infty$,  depending only on $r$,  such $|\wt{z}_i - \wt{z}_{i+1}| = \frac{|x-y|}{d_n}$ for all $1 \leq i \leq J-1$,  and $J \leq C$.  Moreover if for all $i$ we let $\wt{x}_i$ be the point such that the triple $(\wt{z}_i,\wt{z}_{i+1},B_{3r / d_n}(\wt{x}_i))$ can be obtained from the triple $(\frac{x}{d_n},\frac{y}{d_n} , B_{3r / d_n}(0))$ through translation and rotation,  then we can have that $B_{3r / d_n}(\wt{x}_i) \subseteq B_2(0)$. 

By translating by $x_{j_n}$,  rotating and then scaling by $d_n$,  we obtain that we can find points $x_{j_n} = z_1,\cdots,z_J = x_{j_n+1}$ such that $|z_i - z_{i+1}| = |x-y|$ for all $1 \leq i \leq J-1$.  We can also define for all $i$ the point $y_i$ such that the triple $(z_i,z_{i+1},B_{3r}(y_i))$ can be obtained from the triple $(x,y,B_{3r}(0))$ through translation and rotation.  Moreover possibly by taking $r_0 \in (0,1)$ to be smaller (in a universal way),  we can assume that $B_{3r}(y_i) \subseteq B_{2|x_{j_n}-x_{j_n+1}|}(x_{j_n})$.  Therefore we have that
\begin{align*}
\wt{D}_h^{\gamma_n}(x_{j_n},x_{j_n+1} ; B_{2|x_{j_n}-x_{j_n+1}|}(x_{j_n})) \leq \sum_{i=1}^{J-1} \wt{D}_h^{\gamma_n}(z_i,z_{i+1} ; B_{3r}(y_i)).
\end{align*}
Hence combining with Step 2 and the translation and scaling properties of $\wt{D}_h^{\gamma_n}$ (Axiom~\ref{it:coordinate_change}) and since $J \leq C$,  we obtain that
\begin{align*}
\p\left[\wt{D}_h^{\gamma_n}(x,y ; B_{3r}(0)) > \frac{1}{CN} \right] \geq \frac{q}{CN} \quad \text{for all} \quad n \in \N.
\end{align*}
This completes the proof of the lemma.
\end{proof}

\begin{lemma}\label{lem:uniform_lower_bound_across_annuli}
Let $r_0 \in (0,1)$ be the constant of Lemma~\ref{lem:uniform_lower_bound_between_points}.  Then for all $r \in (0,r_0)$,  there exists a constant $c>0$ depending only on $r$ and the sequence $(\gamma_n)$ such that
\begin{align*}
\liminf_{n \to \infty} \p\left[\wt{D}_h^{\gamma_n}(\partial B_r(0) ,  \partial B_{2r}(0)) \geq c \right] > 0.
\end{align*}
\end{lemma}

Before proceeding with the proof of Lemma~\ref{lem:uniform_lower_bound_across_annuli},  let us first describe the setup and the main idea of the proof strategy.

Fix $r \in (0,r_0)$ and set $X:=\partial B_r(0) \cap (r 2^{-m}) \Z^2$ and $Y:=\partial B_{2r}(0) \cap (r 2^{-m}) \Z^2$ for $m \in \N$.  Then we can choose $m$ sufficiently large (depending only on $r$) such that the families of balls $\{B_{r 2^{-m}}(x)\}_{x \in X}$ and $\{B_{r 2^{-m}}(y)\}_{y \in Y}$ cover $\partial B_r(0)$ and $\partial B_{2r}(0)$ respectively and they are contained in $B_{3r}(0)$.  It follows that for all $c>0$,  we have that
\begin{align}\label{eqn:lower_bound_by_covering_the_boundary}
\p\left[\wt{D}_h^{\gamma_n}(\partial B_r(0) ,  \partial B_{2r}(0)) \geq c\right] \geq \p\left[\bigcap_{x \in X ,  y \in Y} \left\{\wt{D}_h^{\gamma_n}(B_{r2^{-m}}(x) ,  B_{r 2^{-m}}(y) ; B_{3r}(0)) \geq c\right\} \right].
\end{align}
Hence it suffices to give a lower bound on the probability in the right hand side of \eqref{eqn:lower_bound_by_covering_the_boundary}.  

Fix $n \in \N$ and recall the definitions of ~\eqref{eqn:gff_molification},  ~\eqref{eqn:lffp_definition} and ~\eqref{eqn:median_of_lffp_definition}.
Moreover we let $W$ be a space-time white noise on $\C \times [0,1]$ and set
\begin{align*}
\Phi_{0,k}(z):=\int_{2^{-2K}}^1 \int_{\C} p_{t/2}(z-u) W(du,dt) \quad \text{for all} \quad z \in \C,  k \in \N.
\end{align*}
Then we define $D_{0,k}$ in the same way that we defined $D_h^{\epsilon_k}$ but with $\epsilon_k = 2^{-k}$ and the field $h_{\epsilon_k}^*$ replaced by $\Phi_{0,k}$.

\begin{proof}[Proof of Lemma~
\ref{lem:uniform_lower_bound_across_annuli}.]
\emph{Step 1.  Outline.}
Suppose that we have the same setup as just after the statement of the lemma.  Fix $r \in (0,r_0)$ and let $c>0$ be the constant in the statement of Lemma~\ref{lem:uniform_lower_bound_between_points}.  As explained above,  it suffices to give a lower bound for the probability on the right hand side of \eqref{eqn:lower_bound_by_covering_the_boundary}.  We will complete the proof in two steps.  First in Step 2,  we will give a lower bound for the probability
\begin{align*}
\p\left[\wt{D}_h^{\gamma_n}(B_{r 2^{-m}}(x) ,  B_{r 2^{-m}}(y) ; B_{3r}(0)) > c /2 \right]
\end{align*}
which is uniform in $(x,y) \in X \times Y$ and $n \in \N$.  Next in Step 3,  we will complete the proof of the lemma by combining the convergence of $D_h^{\epsilon_k}$ and $D_{0,k}$ to $\wt{D}_h^{\gamma_n}$ (when properly rescaled) with the positive associativity property of $\Phi_{0,k}$ (see e.g.,  \cite[Section~2]{ding2023tightness}).

\emph{Step 2.  Lower bound on the probability that $\wt{D}_h^{\gamma_n}(B_{r 2^{-m}}(x) ,  B_{r 2^{-m}}(y) ; B_{3r}(0)) > c /2$.}
Define the constant
\begin{align*}
A:=\inf_{n \in \N} \left\{ \inf_{x \in \partial B_r(0),  y \in \partial B_{2r}(0)}\left( \p\left[\wt{D}_h^{\gamma_n}(x,y ; B_{3r}(0)) > c\right]\right)\right\}
\end{align*}
and note that Lemma~\ref{lem:uniform_lower_bound_between_points} implies that $A>0$.  Note also that Theorem~\ref{thm:tightness_subcritical} implies that the sequence of metrics $\{\wt{D}_h^{\gamma_n}(\cdot,\cdot ; U)\}_{n \in \N}$ is tight for every open set $U \subseteq \C$ with respect to the local uniform topology on $U \times U$. Therefore combining with Axiom~\ref{it:coordinate_change} we obtain that there exists $m \in \N$ such that
\begin{align}\label{eqn:diameter_upper_prob_bound}
\p\left[\text{Diam}_{\wt{D}_h^{\gamma_n}(\cdot,\cdot ; B_{3r}(0))}(B_{r 2^{-m}}(z)) > \frac{c}{4}\right] < \frac{A}{4} \quad \text{for all} \quad z \in \partial B_r(0) \cup \partial B_{3r}(0) ,  n \in \N.
\end{align}

Fix $x \in \partial B_r(0),  y \in \partial B_{2r}(0)$.  Note that the triangle inequality implies that
\begin{align*}
\wt{D}_h^{\gamma_n}(B_{r 2^{-m}}(x) ,  B_{r 2^{-m}}(y) ; B_{3r}(0))&\geq \wt{D}_h^{\gamma_n}(x,y ; B_{3r}(0)) - \text{Diam}_{\wt{D}_h^{\gamma_n}(\cdot,\cdot ; B_{3r}(0))}(B_{r 2^{-m}}(x))\\
&-\text{Diam}_{\wt{D}_h^{\gamma_n}(\cdot,\cdot ; B_{3r}(0))}(B_{r 2^{-m}}(y)).
\end{align*}
Hence combining with \eqref{eqn:diameter_upper_prob_bound} we obtain that
\begin{align}\label{eqn:point_distance_lower_bound_prob}
\p \left[\wt{D}_h^{\gamma_n}(B_{r2^{-m}}(x) ,  B_{r 2^{-m}}(y) ; B_{3r}(0)) > \frac{c}{2}\right] > \frac{A}{2} \quad \text{for all} \quad n \in \N,  x \in \partial B_r(0),  y \in \partial B_{2r}(0).
\end{align}

\emph{Step 3.  Conclusion of the proof.}
Fix $M>1$ sufficiently large (to be chosen independently of $n$) and recall the normalizing constant $\alpha(\gamma_n)$ introduced in \eqref{eqn:main_normalization}.  Then Theorem~\ref{thm:convergence_of_lffp} implies that
\begin{align}\label{eqn:limiting_prob}
&\p\left[\bigcap_{x \in X ,  y \in Y} \left\{\wt{D}_h^{\gamma_n}(B_{r2^{-m}}(x) ,  B_{r 2^{-m}}(y) ; B_{3r}(0)) \geq \frac{c e^{-2\xi(\gamma_n) M}}{2} \right\}\right]\notag\\
&\leq \limsup_{k \to \infty} \p\left[\bigcap_{x \in X ,  y \in Y} \left\{D_h^{\epsilon_k}(B_{r2^{-m}}(x) ,  B_{r 2^{-m}}(y) ; B_{3r}(0)) \geq \frac{c e^{-2\xi(\gamma_n) M} a_{\epsilon_k} \alpha(\gamma_n)}{2} \right\} \right].
\end{align}
Also \cite[Lemma~4.10]{ding2020tightness} implies that there exist universal constants $c_0,c_1 \in (0,\infty)$ and a coupling between $h$ and $W$ such that
\begin{align}\label{eqn:comparison_between_regularizations}
\p\left[\sup_{z \in B_2(0)} |\Phi_{0,k}(z) - h_{\epsilon_k}^*(z)| \geq x \right] \leq c_0 e^{-c_1 x^2} + o_{\epsilon_k}(1) \quad \text{for all} \quad x>0,k \in \N,
\end{align}
where the $o_{\epsilon_k}(1)$ is deterministic and tends to zero as $k \to \infty$.  Hence under the above coupling,  we have that
\begin{align}\label{eqn:intersection_ineq_between_regularizations}
&\p\left[\bigcap_{x \in X ,  y \in Y} \left\{D_h^{\epsilon_k}(B_{r2^{-m}}(x) ,  B_{r 2^{-m}}(y) ; B_{3r}(0)) \geq \frac{c e^{-2\xi(\gamma_n) M} a_{\epsilon_k} \alpha(\gamma_n)}{2} \right\} \right]\notag\\
&\geq \p\left[\bigcap_{x \in X ,  y \in Y} \left\{D_{0,k}(B_{r2^{-m}}(x) ,  B_{r 2^{-m}}(y) ; B_{3r}(0)) \geq \frac{c e^{-\xi(\gamma_n) M} a_{\epsilon_k} \alpha(\gamma_n)}{2} \right\} \right]\notag\\
&-\p\left[ \sup_{z \in B_2(0)} |\Phi_{0,k}(z) - h_{\epsilon_k}^*(z)| \geq M\right].
\end{align}

The positive associativity property of $\Phi_{0,k}$ combined with \cite[Theorem~2.3]{dubedat2020liouville} imply that
\begin{align}\label{eqn:positive_associativity}
&\p\left[\bigcap_{x \in X ,  y \in Y} \left\{D_{0,k}(B_{r2^{-m}}(x) ,  B_{r 2^{-m}}(y) ; B_{3r}(0)) \geq \frac{c e^{-\xi(\gamma_n) M} a_{\epsilon_k} \alpha(\gamma_n)}{2} \right\} \right]\notag\\
&\geq \prod_{x \in X ,  y \in Y} \p\left[ D_{0,k}(B_{r2^{-m}}(x) ,  B_{r 2^{-m}}(y) ; B_{3r}(0)) \geq \frac{c e^{-\xi(\gamma_n) M} a_{\epsilon_k} \alpha(\gamma_n)}{2} \right].
\end{align}
Note that for fixed $(x,y) \in X \times Y$,  we have that
\begin{align*}
&\p\left[ D_{0,k}(B_{r2^{-m}}(x) ,  B_{r 2^{-m}}(y) ; B_{3r}(0)) \geq \frac{c e^{-\xi(\gamma_n) M} a_{\epsilon_k} \alpha(\gamma_n)}{2} \right] \\
&\geq \p\left[ D_h^{\epsilon_k}(B_{r2^{-m}}(x) ,  B_{r 2^{-m}}(y) ; B_{3r}(0)) \geq \frac{c a_{\epsilon_k} \alpha(\gamma_n)}{2} \right]-\p\left[ \sup_{z \in B_2(0)} |\Phi_{0,k}(z) - h_{\epsilon_k}^*(z)| \geq M\right].
\end{align*}
Therefore combining with \eqref{eqn:intersection_ineq_between_regularizations} and \eqref{eqn:positive_associativity},  we obtain that

\begin{align}\label{eqn:lower_bound_of_intersection_prob}
&\p\left[\bigcap_{x \in X ,  y \in Y} \left\{D_h^{\epsilon_k}(B_{r2^{-m}}(x) ,  B_{r 2^{-m}}(y) ; B_{3r}(0)) \geq \frac{c e^{-2\xi(\gamma_n) M} a_{\epsilon_k} \alpha(\gamma_n)}{2} \right\} \right] \notag\\
&\geq \prod_{x \in X ,  y \in Y} \left(\p\left[ D_h^{\epsilon_k}(B_{r2^{-m}}(x) ,  B_{r 2^{-m}}(y) ; B_{3r}(0)) \geq \frac{c a_{\epsilon_k} \alpha(\gamma_n)}{2} \right]-\p\left[ \sup_{z \in B_2(0)} |\Phi_{0,k}(z) - h_{\epsilon_k}^*(z)| \geq M\right]\right)\notag\\
&-\p\left[ \sup_{z \in B_2(0)} |\Phi_{0,k}(z) - h_{\epsilon_k}^*(z)| \geq M\right].
\end{align}

Note that
\begin{align*}
\liminf_{k \to \infty} \p\left[D_h^{\epsilon_k}(B_{r 2^{-m}}(x) ,  B_{r 2^{-m}}(y) ; B_{3r}(0)) > \frac{c a_{\epsilon_k} \alpha(\gamma_n)}{2} \right] \geq \p\left[\wt{D}_h^{\gamma_n}(B_{r 2^{-m}}(x) , B_{r 2^{-m}}(y) ; B_{3r}(0)) > \frac{c}{2}\right]
\end{align*}
as $k \to \infty$ for all $(x,y) \in (X,Y)$ and so we can choose $M>1$ sufficiently large (independent of $n$ and $(x,y)$) and $k_0 \in \N$ such that
\begin{align*}
&\p\left[D_h^{\epsilon_k}(B_{r 2^{-m}}(x) ,  B_{r 2^{-m}}(y) ; B_{3r}(0)) > \frac{c a_{\epsilon_k} \alpha(\gamma_n)}{2} \right]  - \p\left[ \sup_{z \in B_2(0)} |\Phi_{0,k}(z) - h_{\epsilon_k}^*(z)| \geq M\right] \geq \frac{A}{4}\\
&\p\left[ \sup_{z \in B_2(0)} |\Phi_{0,k}(z) - h_{\epsilon_k}^*(z)| \geq M\right] \leq \frac{1}{2} \left(\frac{A}{4}\right)^{|X \times Y|}
\end{align*}
for all $k \geq k_0$. 

Therefore combining with \eqref{eqn:limiting_prob} and \eqref{eqn:lower_bound_of_intersection_prob},  we obtain that
\begin{align*}
\p\left[\bigcap_{x \in X ,  y \in Y} \left\{\wt{D}_h^{\gamma_n}(B_{r2^{-m}}(x) ,  B_{r 2^{-m}}(y) ; B_{3r}(0)) \geq \frac{c e^{-2\xi(\gamma_n) M}}{2} \right\}\right] \geq \frac{1}{2} \left(\frac{A}{4}\right)^{|X| \times |Y|} \quad \text{for all} \quad n \in \N.
\end{align*}
Hence the proof of the lemma is complete by combining with \eqref{eqn:lower_bound_by_covering_the_boundary}.
\end{proof}

Next we show that the $\wt{D}$-distance between any compact set (with respect to the Euclidean distance) and $\infty$ is infinite a.s.

\begin{lemma}\label{lem:metric_unbounded}
For all $r,T>0$,  it holds that
\begin{align*}
\lim_{R \to \infty} \liminf_{n \to \infty} \left(\p\left[ \wt{D}_h^{\gamma_n}(\text{across} \,\,  \mathbb{A}_{r,R}(0)) > T \right] \right) = 1.
\end{align*}
\end{lemma}

\begin{proof}
First we will show that there exists a constant $c>0$ such that
\begin{align}\label{eqn:lower_scaling_bound}
\p\left[\wt{D}_h^{\gamma_n}(\text{across} \,\,  \mathbb{A}_{2^k ,  2^{k+1}}(0)) \geq c 2^{\xi(\gamma_n) Q(\gamma_n) k } \exp(\xi(\gamma_n) h_{2^{k-1}}(0))\right] \geq c \quad \text{for all} \quad n\in \N, k \in \N_0.
\end{align}
Note that Axioms~\ref{it:weyl_scaling} and \ref{it:coordinate_change} imply that
\begin{align*}
\p\left[\wt{D}_h^{\gamma_n}(\text{across} \,\,   \mathbb{A}_{2^k ,  2^{k+1}}(0)) \geq c 2^{\xi(\gamma_n) Q(\gamma_n) k } \exp(\xi(\gamma_n) h_{2^k}(0))\right] =  \p\left[\wt{D}_h^{\gamma_n}(\text{across} \,\,  \mathbb{A}_{1,2}(0)) > c 2^{\xi(\gamma_n) Q(\gamma_n)}\right].
\end{align*}

Let $r_0 \in (0,1)$ be the constant in Lemma~\ref{lem:uniform_lower_bound_across_annuli} and fix $s \in (0,r_0)$.  Then Lemma~\ref{lem:uniform_lower_bound_across_annuli} implies that there exists a constant $\wt{c}>0$ independent of $n$ such that
\begin{align*}
\p\left[\wt{D}_h^{\gamma_n}(\text{across}\,\,  \mathbb{A}_{s,2s}(0)) > \wt{c}\right] \geq \wt{c} \quad \text{for all} \quad n \in \N.
\end{align*}
Combining with Axioms~\ref{it:weyl_scaling} and \ref{it:coordinate_change},  we obtain that
\begin{align}\label{eqn:crossing_prob_lower_bound}
\p\left[\wt{D}_h^{\gamma_n}(\partial B_s(0),\partial B_{2s}(0)) > \wt{c}\right] = \p\left[\wt{D}_h^{\gamma_n}(\partial B_1(0),\partial B_2(0)) > \wt{c} s^{\xi(\gamma_n) Q(\gamma_n)} e^{\xi(\gamma_n) h_s(0)}\right] \geq \wt{c} 
 \quad \text{for all} \quad n \in \N.
\end{align}

Note that $h_s(0) = B_{-\log(s)}$ where $B$ is a standard Brownian motion.  Since 
\begin{align*}
&0<\inf_{n \in \N} \xi(\gamma_n) \leq \sup_{n \in \N} \xi(\gamma_n) < \infty,\\
&0 < \inf_{n \in \N} \xi(\gamma_n) Q(\gamma_n) \leq \sup_{n \in \N} \xi(\gamma_n) Q(\gamma_n) < \infty,
\end{align*}
\eqref{eqn:lower_scaling_bound} follows from combining with \eqref{eqn:crossing_prob_lower_bound}.

Next we fix $0<r<R<\infty$ and note that
\begin{align}\label{eqn:lower_bound_by_summing}
\wt{D}_h^{\gamma_n}(\partial B_r(0),\partial B_R(0)) \geq \sum_{k=\lfloor \log_2(r)\rfloor + 1}^{\lfloor \log_2(R) \rfloor - 1} \wt{D}_h^{\gamma_n}(\partial B_{2^k}(0),\partial B_{2^{k+1}}(0)).
\end{align}
We set $\wt{h}:=h \circ \phi$ where $\phi(z) = -\frac{1}{z}$ for all $z \in \C \cup \{\infty\}$ and note that $\wt{h}$ also has the law of a whole-plane GFF normalized so that $\wt{h}_1(0) = 0$.  Moreover for all $k \in \N$,  we consider the event 
\begin{align*}
E_k = \left\{\wt{D}_h^{\gamma_n}(\partial B_{2^k}(0) ,  \partial B_{2^{k+1}}(0)) \geq c 2^{\xi(\gamma_n) Q(\gamma_n)k} \exp(\xi(\gamma_n) h_{2^{k-1}}(0))\right \},
\end{align*}
where $c>0$ is an in \eqref{eqn:lower_scaling_bound}.  Then the locality property of $\wt{D}_h^{\gamma_n}$ (Axiom~\ref{it:locality}) combined with Weyl scaling (Axiom~\ref{it:weyl_scaling}) imply that $E_k$ is determined by 
\begin{align*}
(\wt{h} - \wt{h}_{2^{-k+1}}(0))|_{\mathbb{A}_{2^{-k-1},2^{-k}}(0)} \quad \text{for all} \quad k \in \N.
\end{align*}
Thus since by \eqref{eqn:lower_scaling_bound} we have that
\begin{align*}
\p[E_k] = \p[E_0] \geq c \quad \text{for all} \quad k \in \N_0,
\end{align*}
we have by \cite[Lemma~3.1]{gwynne2020local} that there exist constants $a_1,a_3>0, a_2 \in (0,1)$ depending only on $r$ such that
\begin{align}\label{eqn:event_happens_at_many_scales}
\p[\mathcal{N}(R) < a_2 \log_2(R)] \leq a_3 \exp(-a_1 \log_2(R)) \quad \text{for all} \quad R>r,
\end{align}
where $\mathcal{N}(R)$ denotes the number of $\lfloor \log_2(r) \rfloor +1 \leq k \leq \lfloor \log_2(R) \rfloor - 1$ for which $E_k$ occurs.  Moreover set $B_t:=h_{e^t}(0)$ for all $t \geq 0$ and note that $B$ has the law of a standard Brownian motion.

Fix $p \in (0,1)$ and $T>0$.  Then there exists $M>0$ such that
\begin{align*}
\p\left[\exp\left(\xi(\gamma_n) \inf_{t \geq 0} (B_t + Q(\gamma_n) t)\right) \geq e^{-M}\right] \geq 1 - \frac{1-p}{2} \quad \text{for all}\quad  n \in \N.
\end{align*}
Also \eqref{eqn:event_happens_at_many_scales} implies that there exists $R_0 > 1$ large enough (independent of $n$) such that $R \geq R_0$ implies that
\begin{align*}
\p\left[\mathcal{N}(R) \geq a_2 \log_2(R)\right] \geq 1 - \frac{1-p}{2}.
\end{align*}

Therefore for fixed $R \geq R_0$,  we have off an event with probability at most $1-p$ that
\begin{align*}
\mathcal{N}(R) \geq a_2 \log_2(R) \quad \text{and} \quad \exp\left(\xi(\gamma_n) \inf_{t \geq 0} (B_t + Q(\gamma_n) t) \right) \geq e^{-M}.
\end{align*}
Hence possibly by taking $R_0$ to be larger,  we obtain combining with \eqref{eqn:event_happens_at_many_scales} that off an event with probability at most $1-p$,  we have that
\begin{align*}
\wt{D}_h^{\gamma_n}(\text{across} \,\,   \mathbb{A}_{r,R}(0)) \geq c a_2 \log_2(R) e^{-M}> T.
\end{align*}
This completes the proof of the lemma.
\end{proof}

Before stating and proving Lemma~\ref{lem:crossing_positive},  we mention the following lemma that we are going to use in the proof of Lemma~\ref{lem:crossing_positive}.

\begin{lemma}\label{lem:connectivity_lemma}
Fix $r>0$.  Then the set $\{x \in \C : \wt{D}(x,\partial B_r(0)) = 0\}$ is a.s.  a closed,  bounded and connected set.
\end{lemma}

\begin{proof}
Note that $\wt{D}$ is a.s.  continuous and $\wt{D}_h^{\gamma_n}$ is a length metric for all $n \in \N$ by Axiom~\ref{it:length_space}.  Therefore the proof of the lemma is complete by following the same argument as in the proof of \cite[Lemma~6.8]{ding2023tightness} and using Lemma~\ref{lem:metric_unbounded} in place of \cite[Lemma~6.7]{ding2023tightness}.
\end{proof}

\begin{lemma}\label{lem:crossing_positive}
Fix $r>0$ and set $\mathcal{Z} := \{\wt{D}(\text{across} \,\,\mathbb{A}_{r,2r}(0)) > 0\}$.  Then we have that 
\begin{align*}
\p[\mathcal{Z}] \in \{0,1\}.
\end{align*}
\end{lemma}

Let us first describe the setup of the proof of the lemma.  We will use the same notation as in the proof of \cite[Lemma~6.9]{ding2023tightness}. 

We consider the sets
\begin{align*}
&E_1 = \{z \in \C : \wt{D}(z,\partial B_r(0)) = 0\},\\
&E_2 = \{z \in \C : \wt{D}(z,\partial B_{2r}(0)) = 0\}
\end{align*}
and note that the triangle inequality implies that
\begin{align*}
\mathcal{Z} = \{E_1 \cap E_2 = \emptyset\} = \{E_1 \cap \partial B_{2r}(0) = \emptyset\} = \{E_2 \cap \partial B_r(0) = \emptyset\}.
\end{align*}

As in the proof of \cite[Lemma~6.9]{ding2023tightness},  the main idea is to show that the events $\{E_1 \subseteq U_1\}$ and $\{E_2 \subseteq U_2\}$ are independent for any fixed bounded and open sets $U_1 ,  U_2 \subseteq \C$ such that 
\begin{align*}
\text{dist}(U_1,U_2) > 0 \quad \text{and} \quad \partial B_r(0) \subseteq U_1,  \partial B_{2r}(0) \subseteq U_2.
\end{align*}

In the proof of Lemma~\ref{lem:crossing_positive},  we are going to consider $U_1,U_2$ to be $(r,j)$-dyadic domains,  i.e.,  they are open,  bounded and connected sets which can be expressed as the union of Euclidean balls with radius $2^{-j} r$ and centered at points on $2^{-j} r \Z^2$.  Then summing over all such sets $U_1,U_2$,  it follows that 
\begin{align*}
\p[\mathcal{Z}] = \p[E_1 \cap E_2 = \emptyset] \leq \p[E_1 \cap \partial B_{2r}(0) = \emptyset] \p[E_2 \cap \partial B_r(0) = \emptyset] = \p[\mathcal{Z}]^2
\end{align*}
and hence $\p[\mathcal{Z}] \in \{0,1\}$.

\begin{proof}[Proof of Lemma~\ref{lem:crossing_positive}.]
The argument is essentially the same as that of the proof of \cite[Lemma~6.9]{ding2023tightness} but we choose to repeat its main parts so that it is clear to the reader which properties of the metrics $\wt{D}_h^{\gamma_n}$ are used.

First we note that Theorem~\ref{thm:tightness_subcritical} implies that for any $(r,j)$-dyadic domain $U$,  we have that the sequence of metrics $\{\wt{D}_h^{\gamma_n}(\cdot,\cdot ; U)\}_{n \in \N}$ is tight with respect to the local uniform topology on $U$.  Thus since the collection of such domains is countable,  we obtain by combining with Skorokhod's representation theorem that we can find a subsequence $(\gamma_{k_n})$ of $(\gamma_n)$ and a coupling of the metrics $\{\wt{D}_h^{\gamma_{k_n}}(\cdot,\cdot ; U)\}_{n \in \N}$ for all $(r,j)$-dyadic domains $U$,  for all $j \in \N$,  such that 
\begin{align*}
\wt{D}_h^{\gamma_{k_n}}(\cdot ,  \cdot ; U) \to \wt{D}_U \quad \text{as} \quad n \to \infty
\end{align*}
with respect to the local uniform topology on $U$ for some pseudo-metric $\wt{D}_U$ on $U$,  for any such domain $U$.  Without loss if generality,  we can assume that $\gamma_{k_n} = \gamma_n$ for all $n \in \N$ and note that $\wt{D} = \wt{D}_{\C}$.

For all $j \in \N$,  we let $E_1^j$ be the open set which contains all the Euclidean balls of the form $B_{r 2^{-j}}(x)$ for $x \in (r 2^{-j}) \Z^2$ with $\text{dist}(x,E_1) < r 2^{-j}$.  Note that Lemma~\ref{lem:connectivity_lemma} implies that $E_1$ is bounded and connected and so $E_1^j$ is $(r,j)$-dyadic.  
We define the sets $E_2^j$ similarly with $E_2$ in place of $E_1$. As in the proof of \cite[Lemma~6.9]{ding2023tightness},  the main idea is to show that
\begin{align}\label{eqn:equality_of_dyadic_sets}
\p[E_1^j = U_1 ,  E_2^j = U_2] = \p[E_1^j = U_1] \p[E_2^j = U_2]
\end{align}
for any two disjoint $(r,j)$-dyadic domains $U_1$ and $U_2$ with 
\begin{align*}
\text{dist}(U_1,U_2) > r 2^{-j},\quad \partial B_r(0) \subseteq U_1 ,  \partial B_{2r}(0) \subseteq U_2,
\end{align*}
then sum \eqref{eqn:equality_of_dyadic_sets} over all such domains $U_1$ and $U_2$ and then take $j \to \infty$.  

In order to show \eqref{eqn:equality_of_dyadic_sets},  we will follow an argument which is similar to the argument in Step 3 of the proof of  \cite[Lemma~6.9]{ding2023tightness}.  Let $U_1,U_2$ be sets as in \eqref{eqn:equality_of_dyadic_sets} and let $\wh{E}_1$ (resp.  $\wh{E}_2$) be the closure of the connected component of $\{x \in U_1 : \wt{D}_{U_1}(x,\partial B_r(0)) = 0\}$ (resp.  $\{x \in U_2 : \wt{D}_{U_2}(x,\partial B_{2r}(0)) = 0\}$) containing $\partial B_r(0)$ (resp.  $\partial B_{2r}(0)$).  Then both $\wh{E}_1$ and $\wh{E}_2$ are closed,  bounded and connected sets.  Similarly we define sets $\wh{E}_1^j$ and $\wh{E}_2^j$ in the same way that we defined $E_1^j$ and $E_2^j$ but with $\wh{E}_1$ and $\wh{E}_2$ in place of $E_1$ and $E_2$ respectively.  Also for any $(r,j)$-dyadic domain $U$,  we fix a deterministic smooth and compactly supported function $\rho_U$ such that the support of $\rho_U$ is contained in $\C \setminus \overline{U}$.  Then the Markov property of the whole-plane GFF implies that $h-(h,\rho_U)$ can be decomposed as
\begin{align*}
h-(h,\rho_U) = h_U^0  + \Fh_U,
\end{align*}
where $h_U^0$ is a zero-boundary GFF on $U$ and $\Fh_U$ is a harminic function on $U$.  Moreover we have that $h_U^0$ and $\Fh_U$ are independent and $(h-(h,\rho_U))|_{\C \setminus \overline{U}}$ is a measurable function of $\Fh_U$.  Arguing as in the proof of Theorem~\ref{thm:tightness_subcritical},  we obtain that the family of random metrics $\{\wt{D}_{h_U^0}^{\gamma_n}(\cdot,\cdot)\}_{n \in \N}$ is tight for any $(r,j)$-dyadic domain $U$,  and so we can further assume using Skorokhod's representation theorem that we can a coupling such that 
\begin{align*}
\wt{D}_{h_U^0}^{\gamma_n}(\cdot,\cdot) \to \wt{D}_{0,U} \quad \text{as} \quad n \to \infty
\end{align*}
a.s.  to some random pseudo-metric $\wt{D}_{0,U}$ on $U$ with respect to the local uniform topology on $U$.

Note that the metrics $\wt{D}_{h_{U_1}^0}^{\gamma_n}$ and $\wt{D}_{h_{U_2}^0}^{\gamma_n}$ are independent for all $n$,  since $U_1 \cap U_2 = \emptyset$.  Thus we obtain that the metrics $\wt{D}_{0,U_1}$ and $\wt{D}_{0,U_2}$ are independent as well.  We define the sets $\wh{E}_1^0$ and $\wh{E}_2^0$ in the same way that we defined the sets $\wh{E}_1$ and $\wh{E}_2$ but with the metrics $\wt{D}_{0,U_1}$ and $\wt{D}_{0,U_2}$ in place of $\wt{D}_{U_1}$ and $\wt{D}_{U_2}$ respectively.  We also define the sets $\wh{E}_1^{0,j}$ and $\wh{E}_2^{0,j}$ as the $(r,j)$-dyadic domains associated with $\wh{E}_1^0$ and $\wh{E}_2^0$ respectively.  Note that Axiom~\ref{it:weyl_scaling} implies that $\wt{D}_h^{\gamma_n}(\cdot,\cdot ; U_1)$ can be bounded both above and below by $\wt{D}_{h_{U_1}^0}^{\gamma_n}$ up to random constants.  By the joint a.s.  convergence of $\{\wt{D}_h^{\gamma_n}(\cdot,\cdot ; U_1)\}_{n \in \N}$ to $\wt{D}_{U_1}$ and $\{\wt{D}_{h_{U_1}^0}^{\gamma_n}\}_{n \in \N}$ to $\wt{D}_{0,U_1}$,  we obtain that $\wt{D}_{U_1}$ can also be bounded from above and below by $\wt{D}_{0,U_1}$ up to random constants.  This implies that $\wh{E}_1 = \wh{E}_1^0$ and similarly we have that $\wh{E}_2 = \wh{E}_2^0$.  Therefore we obtain that $\wh{E}_1^j = \wh{E}_1^{0,j}$ and $\wh{E}_2^j = \wh{E}_2^{0,j}$ which implies that $\wh{E}_1^j$ and $\wh{E}_2^j$ are independent.  It follows that
\begin{align}\label{eqn:equality_of_dyadic_sets_2}
\p[\wh{E}_1^j = U_1 ,  \wh{E}_2^j = U_2] = \p[\wh{E}_1^j = U_1] \p[\wh{E}_2^j = U_2].
\end{align}
Therefore by arguing as in Step 3 of the proof of \cite[Lemma~6.9]{ding2023tightness} and using \eqref{eqn:equality_of_dyadic_sets_2},  we obtain that \eqref{eqn:equality_of_dyadic_sets} holds.

Finally the proof of the lemma is complete by summing over all domains $U_1,U_2$  in \eqref{eqn:equality_of_dyadic_sets} and taking $j \to \infty$,  and then using the $0-1$ law argument as in Step 4 of the proof of \cite[Lemma~6.9]{ding2023tightness}.
\end{proof}

\begin{proof}[Proof of Proposition~\ref{prop:annuli_crossing_always_positive}]
It follows from the same argument used to prove \cite[Proposition~6.4]{ding2023tightness} where we use Lemmas~\ref{lem:uniform_lower_bound_across_annuli} and ~\ref{lem:crossing_positive} in place of \cite[Lemma~6.6]{ding2023tightness} and \cite[Lemma~6.9]{ding2023tightness} respectively.
\end{proof}

\begin{proof}[Proof of Proposition~\ref{prop:limit_is_a_metric}]
The fact that $\wt{D}$ is a metric a.s.  follows from the analogous argument used in the proof of \cite[Theorem~1.2]{ding2023tightness} where we use Proposition~\ref{prop:annuli_crossing_always_positive} in place of \cite[Proposition~6.4]{ding2023tightness}.  Moreover we note that the proof of Theorem~\ref{thm:tightness_subcritical} implies that $\wt{D}$ is locally H\"older continuous a.s.  (see also Proposition~\ref{prop:tightness_proposition_subcritical}).  Hence following exactly the same argument as in the proof of \cite[Theorem~1.2]{ding2023tightness} and using again Proposition~\ref{prop:annuli_crossing_always_positive} in place of \cite[Proposition~6.4]{ding2023tightness},  we obtain that $\wt{D}$ induces the Euclidean topology a.s.  This completes the proof of the proposition.
\end{proof}

\subsection{The limit is a geodesic and complete metric.}
\label{subsec:limit_geodesic_and_complete_metric}

In this subsection,  we will show that it is a.s.  the case that $\wt{D}$ is a complete and geodesic metric.  We will start by proving the former.  This is the content of the following lemma.

\begin{lemma}\label{lem:complete_metric}
$\wt{D}$ is a complete metric a.s.
\end{lemma}

\begin{proof}
First we will show that it is a.s.  the case that any $\wt{D}$-bounded set is also Euclidean bounded.  In particular,  we will show that for any $r>0$ we have that
\begin{align}\label{eqn:unbounded_metric}
\inf_{z \in \C \setminus B_R(0),  w \in B_r(0)} \wt{D}(z,w) \to \infty \quad \text{as} \quad R \to \infty \quad \text{a.s.}
\end{align}
Indeed fix $\epsilon \in (0,1)$ and $T>0$. Then Lemma~\ref{lem:metric_unbounded} implies that there exists $R>r$ such that
\begin{align*}
\liminf_{n \to \infty} \p\left[\wt{D}_h^{\gamma_n}(\text{across} \,\,\mathbb{A}_{r,R}(0)) > T \right] \geq 1-\epsilon
\end{align*}
which implies that
\begin{align*}
\p\left[\inf_{z \in \C \setminus B_R(0) ,  w \in B_r(0)} \wt{D}(z,w) \geq T\right] \geq \limsup_{n \to \infty} \p\left[\wt{D}_h^{\gamma_n}(\text{across} \,\,\mathbb{A}_{r,R}(0)) > T \right] \geq 1-\epsilon.
\end{align*}
Therefore~ \eqref{eqn:unbounded_metric} follows since $\epsilon,T$ were arbitrary.

Next let $(z_n)_{n \in \N}$ be a $\wt{D}$-Cauchy sequence.  Then there exists  $R>0$ such that $\wt{D}(z_1,z_n) \leq R$ for all $n \in \N$.  Also there exists $r>0$ such that $z_1 \in B_r(0)$.  Thus combining with \eqref{eqn:unbounded_metric},  we obtain that it is a.s.  the case that there exists (random) $A \in (0,\infty)$ such that $\{z_n\}_{n \in \N} \subseteq B_A(0)$.  Fix any subsequence $\{z_{k_n}\}_{n \in \N}$.  Then since $(z_n)$ is bounded with respect to the Euclidean topology,  there exists a further subsequence $\{z_{k_n'}\}_{n \in \N}$ and $z \in \overline{B_A(0)}$ such that $z_{k_n'} \to z$ as $n \to \infty$.  Since $\wt{D}$ induces the Euclidean topology a.s.  by Proposition~\ref{prop:limit_is_a_metric},  we obtain that
\begin{align*}
\wt{D}(z_{k_n'},z) \to 0 \quad \text{as} \quad n \to \infty \quad \text{a.s.}
\end{align*}
Therefore since the subsequence $(z_{k_n})$ was arbitrary,  we obtain that $(z_n)$ is $\wt{D}$-convergent.  This completes the proof of the lemma combining with Proposition~\ref{prop:limit_is_a_metric}.
\end{proof}

Next we prove that $\wt{D}$ is a geodesic metric a.s.

\begin{lemma}\label{lem:geodesic_metric}
$\wt{D}$ is a geodesic metric a.s.
\end{lemma}

\begin{proof}
First we note that since $\wt{D}$ is a complete metric a.s.  (Lemma~\ref{lem:complete_metric}),  we obtain by \cite[Theorem~2.4.16]{burago2001course} that it suffices to show that it is a.s.  the case that the following is true for all $z,w \in \C$ distinct points.  There exists $x \in \C$ such that
\begin{align}\label{eqn:mid_point_property}
\wt{D}(z,x) = \frac{\wt{D}(z,w)}{2} = \wt{D}(x,w).
\end{align}

From now on,  we focus on proving \eqref{eqn:mid_point_property}.  Fix distinct points $z,w \in \C$ and let $r>0$ be such that $\{z,w\} \subseteq B_r(0)$.  Then it is a.s.  the case that there exists $T>0$ large such that
\begin{align*}
\wt{D}(u,v) \leq T \quad \text{for all} \quad u,v \in B_r(0).
\end{align*}
Moreover \eqref{eqn:unbounded_metric} in the proof of Lemma~\ref{lem:complete_metric} implies that it is a.s.  the case that there exists $R>r$ such that
\begin{align}\label{eqn:liminf_ineq}
\wt{D}_h^{\gamma_{k_n}}(\text{across} \,\,\mathbb{A}_{r,R}(0)) > 2T \quad \text{for all} \quad n \in \N,
\end{align}
for some subsequence $(\gamma_{k_n})$ of $(\gamma_n)$.  

Since $\wt{D}_h^{\gamma_{k_n}}$ is a length metric,  we obtain that there exists $x_n \in \C$ such that
\begin{align}\label{eqn:mid_point_for_sequence}
\wt{D}_h^{\gamma_{k_n}}(z,x_n) = \frac{\wt{D}_h^{\gamma_{k_n}}(z,w)}{2} = \wt{D}_h^{\gamma_{k_n}}(x_n,w).
\end{align}
Also there a.s.  exists $n_0 \in \N$ such that
\begin{align*}
|\wt{D}_h^{\gamma_{k_n}}(z,w) - \wt{D}(z,w)| < T \quad \text{for all} \quad n \geq n_0.
\end{align*}
In particular we have that $\wt{D}_h^{\gamma_{k_n}}(z,x_n) < T$ for all $n \geq n_0$ and so \eqref{eqn:liminf_ineq} implies that $x_n \in \overline{B_R(0)}$ for all $n \geq n_0$.  Hence possibly by passing into a subsequence,  we can assume that there exists $x \in \overline{B_R(0)}$ such that $x_n \to x$ as $n \to \infty$.  Therefore we obtain that
\begin{align*}
\wt{D}_h^{\gamma_{k_n}}(z,x_n) \to \wt{D}(z,x) \quad \text{and} \quad \wt{D}_h^{\gamma_{k_n}}(w,x_n) \to \wt{D}(w,x) \quad \text{as} \quad n \to \infty.
\end{align*}
Therefore combining with \eqref{eqn:mid_point_for_sequence},  we obtain that \eqref{eqn:mid_point_property} holds and so this completes the proof of the lemma.
\end{proof}

\begin{proposition}\label{prop:complete_and_geodesic_metric}
It is a.s.  the case that $\wt{D}$ is a complete and geodesic metric.
\end{proposition}

\begin{proof}
It follows from combining Lemmas~\ref{lem:complete_metric} and ~\ref{lem:geodesic_metric}.
\end{proof}

\subsection{The limit satisfies Weyl scaling}
\label{subsec:metric_satisfies_weyl_scaling}

In this subsection,  we will prove that $\wt{D}$ satisfies Weyl scaling (Axiom~\ref{it:weyl_scaling}).  More precisely,  we will prove the following.

\begin{proposition}\label{prop:weyl_scaling}
The following is true almost surely.  Let $f: \C \to \R$ be a bounded and continuous function.  The we have that 
\begin{align*}
\wt{D}_{h+f}^{\gamma_n} \to e^{\xi(\gamma) f} \cdot \wt{D} \quad \text{as} \quad n \to \infty
\end{align*}
with respect to the local uniform topology on $\C \times \C$.
\end{proposition}

We will use the results from \cite[Section~2]{Pfeffer_2024} in order to prove Proposition~\ref{prop:weyl_scaling}.  In order to describe the results from  \cite{Pfeffer_2024} that we are going to use for the proof of Proposition~\ref{prop:weyl_scaling},  we start with a preliminary definition.

\begin{defn}
We say that a \emph{rational circle} is a circle of the form $O = \partial B_r(u)$ for $r \in \Q \cap [0,\infty)$ and $u \in \Q^2$.  We view a point in $\Q^2$ as a rational circle with radius $0$.  A \emph{rational annulus} is the closure of the region $A$ between two non-intersecting rational circles with positive radii and the same center.  We also say that a rational annulus or a rational circle with positive radius surrounds $z \in \C$ if it does not contain $z$ and it disconnects it from $\infty$.
\end{defn}

The main ingredient of the proof of Proposition~\ref{prop:weyl_scaling} ia the following result  from \cite{Pfeffer_2024}.

\begin{lemma}(\cite[Lemma~2.8]{Pfeffer_2024})
\label{lem:conditions_for_weyl_scaling}
Let $\{D_n\}_{n \in \N}$ be a sequence of metrics in $\C$ that converges to a metric $D_{\infty}$ in the lower semicontinuous topology.  Suppose that the metric space $(\C ,  D_n)$ is a length space for all $n \in \N$,  and that the identity map from $(\C,D_{\infty})$ to $(\C,|\cdot|)$ is continuous.  Let $(f_n)_{n \in \N}$ be a sequence of continuous functions converging uniformly on compact subsets to some continuous function $f_{\infty}$.  Finally suppose that the following two properties hold for some fixed compact sets $K \subseteq K' \subseteq K'' \subseteq  \C$.
\begin{enumerate}
\item \label{it:boundness_of_almost_geodesic_paths}
Let $n \in \N$ and suppose that $P$ is a path in $K$ (resp.  $K'$) between some pair of points $z,w$ such that the $D_n$-length of $P$ is less than $D_n(z,w)+1$,  or the $e^{f_n} \cdot D_n$-length of $P$ is less than $e^{f_n} \cdot D_n(z,w) + 1$.  Then $P$ is contained in $K'$ (resp.  $K''$).
\item \label{it:short_distances_around_small_annuli}
For each nonsingular point $z \in K'$, i.e.,  there exists $w \in \C \setminus \{z\}$ such that $D_{\infty}(z,w) < \infty$,  there exist rational circles $O^m$ surrounding $z$ whose radii tend to zero as $m \to \infty$,  such that the following is true.  If $\wh{O}^m$ denotes the circle with the same center as $O^m$ and twice the radius,  and $A^m$ the annulus bounded by $O^m$ and $\wh{O}^m$,  then
\begin{align*}
\lim_{m \to \infty} \lim_{n \to \infty} (D_n(\text{around} \,\,A^m)) = 0.
\end{align*}
\end{enumerate}

Then the metrics $e^{f_n} \cdot D_n$ restricted to $K$ converge to $e^{f_{\infty}} \cdot D_{\infty}$ restricted to $K$ in the lower semicontinuous topology.
\end{lemma}

The main idea of the proof of Proposition~\ref{prop:weyl_scaling} is the following.  Let $f : \C \to \R$ be a continuous function.  First we will show that the conditions of Lemma~\ref{lem:conditions_for_weyl_scaling} are satisfied with $D_n = \wt{D}_h^{\gamma_n},  f_n = \xi(\gamma_n) f,  f_{\infty} = \xi(\gamma) f$ for all $n \in \N$,  and $D_{\infty} = \wt{D}$.  Since $\wt{D}_{h+f}^{\gamma_n} = e^{\xi(\gamma_n) f} \cdot \wt{D}_h^{\gamma_n}$ for all $n$ by Axiom~\ref{it:weyl_scaling},  Lemma~\ref{lem:conditions_for_weyl_scaling} implies that 
\begin{align*}
\wt{D}_{h+f}^{\gamma_n} \to e^{\xi(\gamma) f} \cdot \wt{D} \quad \text{as} \quad n \to \infty
\end{align*}
with respect to the lower semicontinuous topology a.s.  

Next we will show that the sequence $(\wt{D}_{h+f}^{\gamma_n})_{n \in \N}$ is equicontinuous when restricted to compact sets and hence combining the previous paragraph with the Arzela-Ascoli theorem,  we obtain that
\begin{align*}
\wt{D}_{h+f}^{\gamma_n} \to e^{\xi(\gamma) f} \cdot \wt{D} \quad \text{as} \quad n \to \infty
\end{align*}
with respect to the local uniform topology a.s.

We start by proving that condition~\ref{it:boundness_of_almost_geodesic_paths} of Lemma~\ref{lem:conditions_for_weyl_scaling} is satisfied.  This is the content of the following lemma.

\begin{lemma}\label{lem:condition_I_satisfied}
The following holds a.s.  Let $f : \C \to \R$ be a continuous function and suppose that we have the setup described just after the statement of Lemma~\ref{lem:conditions_for_weyl_scaling} with $D_n = \wt{D}_{h+f}^{\gamma_n} ,  f_n = \xi(\gamma_n) f,  f_{\infty} = \xi(\gamma) f$ and $D_{\infty} = \wt{D}$. Then for every compact set $K \subseteq \C$,  there exist compact sets $K',K'' \subseteq \C$ such that $K \subseteq K' \subseteq K''$ and condition~\ref{it:boundness_of_almost_geodesic_paths} holds.
\end{lemma}

\begin{proof}
Fix any compact set $K \subseteq \C$ and let $M>0$ be such that $K \subseteq B_M(0)$ and set
\begin{align*}
R:=\sup_{z,w \in B_M(0)} \wt{D}(z,w) < \infty.
\end{align*}
Note that \eqref{eqn:unbounded_metric} in the proof of Lemma~\ref{lem:complete_metric} implies that
\begin{align*}
\wt{D}(B_M(0) ,  \C \setminus B_{\wt{M}}(0)) \to \infty \quad \text{as} \quad \wt{M} \to \infty \quad \text{a.s.}
\end{align*}
and so there a.s.  exists $\wt{M} > M$ such that for all $z \in \C$ such that $\wt{D}(z,B_M(0)) < 2R +2$,  we have that $z \in B_{\wt{M}}(0)$.  Also there exists $n_0 \in \N$ such that
\begin{align*}
|\wt{D}_h^{\gamma_n}(z,w) - \wt{D}(z,w)| \leq \frac{R}{2} \quad \text{for all} \quad z,w \in B_{\wt{M}}(0) ,  n \geq n_0,
\end{align*}
which implies that
\begin{align*}
\wt{D}_h^{\gamma_n}(z,w) \leq \frac{3R}{2} \quad \text{for all} \quad z,w \in B_{\wt{M}}(0) ,  n \geq n_0.
\end{align*}

Fix $n \geq n_0$ and $z,w \in K$,  and let $P$ be a path from $z$ to $w$ such that the $\wt{D}_h^{\gamma_n}$-length of $P$ is at most $\wt{D}_h^{\gamma_n}(z,w) + 1$.  We will show that $P \subseteq B_{\wt{M}}(0)$.  Suppose that this doesn't hold.  Then there exists $u \in P \cap \partial B_{\wt{M}}(0)$.  Also we have that $\wt{D}_h^{\gamma_n}(z,u$ is at most the $\wt{D}_h^{\gamma_n}$-length of $P$ and so 
\begin{align*}
\wt{D}_h^{\gamma_n}(z,u) \leq \wt{D}_h^{\gamma_n}(z,w) + 1 \leq \frac{3R}{2} +1 
\end{align*}
which implies that
\begin{align*}
\wt{D}(z,u) \leq \frac{R}{2} + \wt{D}_h^{\gamma_n}(z,u) < 2R + 2.
\end{align*}
It follows that $u \in B_{\wt{M}}(0)$ but that's a contradiction since we have assumed that $u \in \partial B_{\wt{M}}(0)$.  Therefore we obtain that $P \subseteq B_{\wt{M}}(0)$ for any such path $P$.

Next we note that there exists (random) $A \in (0,\infty)$ such that 
\begin{align*}
\sup_{z \in B_{\wt{M}}(0)} \xi(\gamma_n)|f(z)| \leq A \quad \text{for all} \quad n \in \N.
\end{align*}
Again \eqref{eqn:unbounded_metric} in the proof of Lemma~\ref{lem:complete_metric} implies that it is a.s.  the case that there exists $\wh{M} > \wt{M}$ such that if $z \in \C$ is such that
\begin{align}\label{eqn:main_bound}
\wt{D}(z,B_M(0)) < 2e^A(2+3R e^A  / 2),
\end{align}
then $z \in B_{\wh{M}}(0)$.  Also possibly by taking $n_0$ to be larger,  we can assume that 
\begin{align*}
|\wt{D}_h^{\gamma_n}(z,w) - \wt{D}(z,w)| < \frac{R}{2} \quad \text{for all} \quad n \geq n_0 ,  z,w \in B_{\wh{M}}(0).
\end{align*}

Fix $n \geq n_0$ and $z,w \in K$.  Let $P$ be a path from $z$ to $w$ such that the $\wt{D}_{h+f}^{\gamma_n}$-length of $P$ is at most $\wt{D}_{h+f}^{\gamma_n}(z,w) + 1$.  Let also $Q$ be a path from $z$ to $w$ so that its $\wt{D}_h^{\gamma_n}$-length is at most $\wt{D}_h^{\gamma_n}(z,w) + 1$.  Then the choice of $\wt{M}$ implies that $Q \subseteq B_{\wt{M}}(0)$ and so the $\wt{D}_{h+f}^{\gamma_n}$-length of $Q$ is at most $e^A \wt{D}_h^{\gamma_n}(z,w)$ which implies that
\begin{align*}
\wt{D}_{h+f}^{\gamma_n}(z,w) \leq e^A \wt{D}_h^{\gamma_n}(z,w).
\end{align*}
Therefore the $\wt{D}_{h+f}^{\gamma_n}$-length of $P$ is at most $e^A \wt{D}_h^{\gamma_n}(z,w) + 1$.

Finally we will show that $P \subseteq B_{\wh{M}}(0)$.  Suppose that this doesn't hold and we parameterize $P$ by $[0,1]$ so that $P(0) = z$ and $P(1) = w$.  Then we set
\begin{align*}
t = \inf\{s \in [0,1] : P(s) \notin B_{\wh{M}}(0)\} ,\,\,u = P(t)
\end{align*}
and note that $u \in \partial B_{\wh{M}}(0)$.  Then the $\wt{D}_h^{\gamma_n}$-length of $P|_{[0,t]}$ is at most $e^A$ times the $\wt{D}_{h+f}^{\gamma_n}$-length of $P|_{[0,t]}$.  In particular we have that
\begin{align*}
\wt{D}_h^{\gamma_n}(z,u) \leq e^A (e^A \wt{D}_h^{\gamma_n}(z,w) + 1) \leq e^A(1+ 3R e^A / 2)
\end{align*}
and 
\begin{align*}
\wt{D}(z,u) \leq \wt{D}_h^{\gamma_n}(z,u) + \frac{R}{2},
\end{align*}
which implies that
\begin{align*}
\wt{D}(u,B_M(0)) \leq \wt{D}(z,u) < 2e^A (1+ 3R e^A / 2).
\end{align*}
Thus the choice of $\wh{M}$ in \eqref{eqn:main_bound} implies that $u \in B_{\wh{M}}(0)$ but that's a contradiction since the definition of $u$ implies that $u \in \partial B_{\wh{M}}(0)$.  Therefore we have that $P \subseteq B_{\wh{M}}(0)$.  It follows that condition~\ref{it:boundness_of_almost_geodesic_paths} holds with $K' = B_{\wt{M}}(0)$ and $K'' = B_{\wh{M}}(0)$.  This completes the proof of the lemma.
\end{proof}

Next we prove that condition~\ref{it:short_distances_around_small_annuli} holds as well.   This will be achieved in two main steps.  As a first step,  we will show that for every fixed Euclidean annulus $A$,  we have that the sequence of random variables $\{\wt{D}_h^{\gamma_n}(\text{around} \,\,A)\}_{n \in \N}$ is tight.  This will allow us to assume that possibly by considering a subsequence of $(\gamma_n)$,  we have that there exists a coupling of $\{\wt{D}_h^{\gamma_n}\}_{n \in \N}$ and $\wt{D}$ such that $\wt{D}_h^{\gamma_n}(\text{around} \,\,  A)$ converges a.s.  as $n \to \infty$ to some random variable $D(A)$ for every rational annulus $A$.  Also we will show that $\wt{D}(\text{around} \,\,A) < \infty$ for every rational annuls $A$ a.s.  As a second step,  we will show that for every bounded open set $U \subseteq \C$,  there exists a deterministic constant $C>1$ depending only on $(\gamma_n)$ and $U$ such that it is a.s.  the case that the following is true.  There exists $K_0 \in \N$ such that for all $K \geq K_0$ and all $u \in (2^{-2K} \Z^2) \cap U$,  there exists $k \in [K/2 , K]_{\Z}$ such that
\begin{align*}
D(\mathbb{A}_{2^{-k},2^{-k+1}}(u)) \leq C \wt{D}(\partial B_{2^{-k}}(u),\partial B_{2^{-k+1}}(u)).
\end{align*}
Then we will use the above property to construct rational annuli $A^m$ as in condition~\ref{it:short_distances_around_small_annuli}.

We start with proving the first step in the following lemma.

\begin{lemma}\label{lem:tightness_around_annuli}
Fix $z \in \C,  0 < r_1 < r_2 < \infty$.  Then the sequence of random variables 
\begin{align*}
\{\wt{D}_h^{\gamma_n}(\text{around} \,\,  \mathbb{A}_{r_1,r_2}(z))\}_{n \in \N}
\end{align*}
is tight.  Also we have that $\wt{D}(\text{around} \,\,\mathbb{A}_{r_1,r_2}(z)) < \infty$ a.s.
\end{lemma}

\begin{proof}
First we will show the part of the statement of the lemma related to tightness.  Fix $p \in (0,1)$.  Then combining the fact that $\wt{D}$ is a metric which induces the Euclidean topology a.s.  (Proposition~\ref{prop:limit_is_a_metric}) with Proposition~\ref{prop:tightness_proposition_subcritical},  we obtain that there exists a deterministic constant $\delta \in (0,\frac{r_2-r_1}{100})$ such that the following holds.  For all $n \in \N$,  we have the following with probability at least $p$.  For all $x,y \in \partial B_{\frac{r_1+r_2}{2}}(z)$ such that $|x-y| \leq \delta$,  we have that if $P$ is a $\wt{D}_h^{\gamma_n}$-geodesic from $x$ to $y$,  then we have that $P \subseteq B_{\frac{r_2-r_1}{100}}(x)$.  Moreover we have that the $\wt{D}_h^{\gamma_n}$-length of $P$ is at most $1$.

Fix $n \in \N$ and suppose that the event of the previous paragraph occurs.  Let $x_1,\cdots,x_N$ be deterministic points on $\partial B_{\frac{r_1+r_2}{2}}(z)$ ordered in the clockwise direction such that $|x_j-x_{j+1}| \leq \delta$ for all $1 \leq j \leq N$,  where we set $x_{N+1} = x_1$.  For all $1 \leq j \leq N$,  we let $P_{j,n}$ be a $\wt{D}_h^{\gamma_n}$-geodesic from $x_j$ to $x_{j+1}$ and let $P^n$ denote the path given by concatenating the paths $P_{1,n},\cdots,P_{N,n}$.  Then since 
\begin{align*}
P_{j,n} \subseteq B_{\frac{r_1+r_2}{100}}(x_j)\quad  \text{for all} \quad 1 \leq j \leq N,
\end{align*}
we obtain that $P^n \subseteq \mathbb{A}_{r_1,r_2}(z)$ and $P^n$ disconnects $\partial B_{r_1}(z)$ from $\partial B_{r_2}(z)$.  Also since the $\wt{D}_h^{\gamma_n}$-length of $P_{j,n}$ is at most $1$,  we obtain that
\begin{align*}
\len(P^n ;  \wt{D}_h^{\gamma_n}) \leq \sum_{j=1}^N \len(P_{j,n} ; \wt{D}_h^{\gamma_n}) \leq N.
\end{align*}
In particular we have that
\begin{align*}
\wt{D}_h^{\gamma_n}(\text{around} \,\,\mathbb{A}_{r_1,r_2}(z)) \leq N
\end{align*}
with probability at least $p$,  for all $n \in \N$.  This proves the tightness of $\{\wt{D}_h^{\gamma_n}(\text{around} \,\,\mathbb{A}_{r_1,r_2}(z))\}_{n \in \N}$.

Finally using a similar argument as in the above paragraphs and combining with the fact that $\wt{D}_h^{\gamma_n}$ converges to $\wt{D}$ as $n \to \infty$ locally uniformly,  we obtain that 
\begin{align*}
\wt{D}(\text{around}\,\,\mathbb{A}_{r_1,r_2}(z)) < \infty \quad \text{a.s.}
\end{align*}
This completes the proof of the lemma.
\end{proof}

Now we prove that condition~\ref{it:short_distances_around_small_annuli} holds as well.

\begin{lemma}\label{lem:condition_II_satisfied}
Suppose that we have the setup of Lemma~\ref{lem:condition_I_satisfied}.  Then it is a.s.  the case that the following is true.  Fix $K \subseteq \C$ compact.  Then for each $z \in K$,  we can find a sequence of rational annuli $(A^m)_{m \in \N}$ as in condition~\ref{it:short_distances_around_small_annuli} such that
\begin{align*}
\lim_{m \to \infty} \lim_{n \to \infty} (\wt{D}_h^{\gamma_n}(\text{around} \,\,A^m)) = 0.
\end{align*}
\end{lemma}

\begin{proof}
\emph{Step 1.  Setup.}
First we note that Lemma~\ref{lem:tightness_around_annuli} combined with Skorokhod's representation theorem imply that possibly by considering a subsequence of $(\gamma_n)$,  we can assume that we can couple the random variables $\{\wt{D}_h^{\gamma_n}\}_{n \in \N},  \wt{D}$ in a common probability space such that
\begin{align*}
\wt{D}_h^{\gamma_n} \to \wt{D} \quad \text{as} \quad n \to \infty
\end{align*}
with respect to the local uniform topology a.s. ,  and for every rational annulus $A$,  there exists a random variable $D(A)$ defined in the same probability space such that
\begin{align*}
\wt{D}_h^{\gamma_n}(\text{around} \,\,  A) \to D(A) \quad \text{as} \quad n \to \infty \quad \text{a.s.}
\end{align*}

In Step 2,  we will construct the annuli $(A^m)$ as in condition~\ref{it:short_distances_around_small_annuli} while in Step 3 we will complete the proof of the lemma.

\emph{Step 2.  Constructing the annuli $(A^m)$.}
Now we construct the required rational annuli.  Set
\begin{align*}
\gamma_* :=\sup\{\gamma_n : n \in \N\} \in (0,2),
\end{align*}
and for $C>1,u \in \C,k \in \N_0$ and $n \in \N$,  we let
\begin{align*}
E_k^n(u ; C):=\left\{\wt{D}_h^{\gamma_n}(\text{around} \,\,\mathbb{A}_{2^{-k},2^{-k+1}}(u) ) \leq C \wt{D}_h^{\gamma_n}(\text{across}\,\,\mathbb{A}_{2^{-k},2^{-k+1}}(u))\right\}.
\end{align*}
Fix $\epsilon \in (0,1) ,  C>1$ and note that our choice of normalization for $D_h^{\gamma_n}$ (see \eqref{eqn:main_normalization} and \eqref{eqn:normalization_subcritical}) combined with Axioms~\ref{it:weyl_scaling} and ~\ref{it:coordinate_change} imply that
\begin{align}\label{eqn:scaling_distance_around_annulus}
&\p\left[\wt{D}_h^{\gamma_n}(\text{around} \,\,\mathbb{A}_{2^{-k},2^{-k+1}}(u) ) \leq 2^{\xi(\gamma_n) Q(\gamma_n) k } \exp(\xi(\gamma_n) h_{2^{-k}}(u))\right]\notag \\
&= \p\left[\wt{D}_h^{\gamma_n}(\text{around}\,\,\mathbb{A}_{1,2}(0)) \leq 1 \right] \geq \mathfrak{p}(\gamma_*) \quad \text{for all} \quad k \in \N_0,n \in \N,  u \in \C.
\end{align}
Similarly we have that
\begin{align}\label{eqn:scaling_distance_across_annulus}
&\p\left[\wt{D}_h^{\gamma_n}(\text{across}\,\,\mathbb{A}_{2^{-k},2^{-k+1}}(u)) \geq C^{-1} 2^{-\xi(\gamma_n) Q(\gamma_n) k} \exp(\xi(\gamma_n) h_{2^{-k}}(u))\right] \notag\\
&=\p\left[\wt{D}_h^{\gamma_n}(\text{across}\,\,\mathbb{A}_{1,2}(0)) \geq C^{-1}\right] \quad \text{for all} \quad k \in \N_0,n \in \N, u \in \C.
\end{align}

Since 
\begin{align*}
\p\left[\wt{D}_h^{\gamma_n}(\text{across}\,\,\mathbb{A}_{1,2}(0)) \geq C^{-1}\right] \to \p\left[\wt{D}(\text{across}\,\,\mathbb{A}_{1,2}(0)) \geq C^{-1}\right] \quad \text{as} \quad n \to \infty,
\end{align*}
and $\wt{D}(\text{across} \,\,\mathbb{A}_{1,2}(0))>0$ a.s.,  we obtain by combining \eqref{eqn:scaling_distance_around_annulus} with \eqref{eqn:scaling_distance_across_annulus} that there exists $C>1$ such that
\begin{align*}
\p[E_k^n(u;C)] \geq \mathfrak{p}(\gamma_*) - \epsilon \quad \text{for all} \quad k \in \N_0,n \in \N,  u \in \C.
\end{align*}
Therefore by taking $\epsilon \in (0,1)$ sufficiently small and arguing as in the proof of Lemma~\ref{lem:good_event_occurs_almost_everywhere} by applying Lemma~\ref{lem:good_event_at_many_scales} (see also the proof of \cite[Lemma~5.14]{ding2020tightness}),  we obtain that
\begin{align}\label{eqn:event_occurs_at_many_scales}
\p\left[E_k^n(u;C) \,\,\text{occurs for at least one} \,\,k \in [M/2 ,  M]_{\Z} \right] \geq 1- O_M(2^{-3M}) \quad \text{for all} \quad n,M \in \N,  u \in \C,
\end{align}
where the implicit constants are independent of $u,M$,  and $n$.

Fix $U \subseteq \C$ an open bounded set such that $K \subseteq U$.  Then taking $n \to \infty$ in \eqref{eqn:event_occurs_at_many_scales} for fixed $u \in U$ and $M \in \N$,  and then taking a union bound over all $u \in (2^{-2M} \Z^2) \cap U$ and then applying the Borel-Cantelli lemma by summing over all $M \in \N$, we obtain that the following is true a.s.  There exists $M_0 \in \N$ such that for all $M \geq M_0$ and all $u \in (2^{-2M} \Z^2) \cap U$,  there exists $k \in [M/2 ,  M]_{\Z}$ such that
\begin{align}\label{eqn:rational_annuli_inequality}
D(\mathbb{A}_{2^{-k},2^{-k+1}}(u)) \leq C \wt{D}(\text{across} \,\,\mathbb{A}_{2^{-k},2^{-k+1}}(u)).
\end{align}
Therefore we obtain that the following holds a.s.  For all $z \in U$,  there exists a sequence of positive integers $(k_m)_{m \in \N}$ such that $k_m \to \infty$ as $m \to \infty$,  and a sequence of points $(u_m)_{m \in \N}$ such that $u_m \in (2^{-2k_m} \Z^2) \cap U$ and $z \in B_{2^{-k_m}}(u_m)$,  for all $m \in \N$,  and \eqref{eqn:rational_annuli_inequality} holds for some $k \in [k_m / 2,k_m]_{\Z}$.  We then set 
\begin{align*}
A^m: = \mathbb{A}_{2^{-k},2^{-k+1}}(u_m)\quad \text{for all} \quad m \in \N.
\end{align*}

\emph{Step 3.  Conclusion of the proof.}
Fix $z \in U$ and let $(A^m)_{m \in \N}$ be the sequence of rational annuli as in Step 2.  Then possibly by considering a sufficiently sparse subsequence,  we can assume without loss of generality that $A^m$ disconnects $A^{m+1}$ from $\infty$ for all $m \in \N$.  Moreover by taking $m \to \infty$ in \eqref{eqn:rational_annuli_inequality} and since $\wt{D}$ is a finite metric a.s.,  we obtain that condition~\ref{it:short_distances_around_small_annuli} holds for all $z \in U$ a.s.  This completes the proof of the lemma since $K$ was arbitrary.
\end{proof}

Now we are ready to prove Proposition~\ref{prop:weyl_scaling}.

\begin{proof}[Proof of Proposition~\ref{prop:weyl_scaling}]
First we note that combining Lemmas~\ref{lem:conditions_for_weyl_scaling},  ~\ref{lem:condition_I_satisfied} and ~\ref{lem:condition_II_satisfied},  we obtain that the following is true a.s.  For every continuous function $f : \C \to \R$,  we have that
\begin{align}\label{eqn:lower_semicontinuous_convergence}
\wt{D}_{h+f}^{\gamma_n} \to e^{\xi(\gamma) f} \cdot \wt{D} \quad \text{as} \quad n \to \infty
\end{align}
with respect to the lower semicontinuous topology.

Next as explained just after the statement of Lemma~\ref{lem:conditions_for_weyl_scaling},  we show that it is a.s.  the case that for any continuous function $f : \C \to \R$,  we have that the sequence $(\wt{D}_{h+f}^{\gamma_n})_{n \in \N}$ is equicontinuous when restricted to compact sets.  To show the latter claim,  we let $\chi>0$ be a constant as in the statement of Proposition~\ref{prop:tightness_proposition_subcritical} and fix a compact set $K \subseteq \C$.  Let also $K' \subseteq \C$ be a compact set such that $K \subseteq K'$ so that condition~\ref{it:boundness_of_almost_geodesic_paths} in Lemma~\ref{lem:conditions_for_weyl_scaling} is satisfied.  Then applying Proposition~\ref{prop:tightness_proposition_subcritical} and arguing as in the proof of Theorem~\ref{thm:tightness_subcritical},  we obtain that it is a.s.  the case that there exists $C \in (0,\infty)$ such that
\begin{align*}
\wt{D}_h^{\gamma_n}(z,w) \leq C |z-w|^{\chi} \quad \text{for all} \quad z,w \in K',  n \in \N.
\end{align*}

Fix $n \in \N,  z,w \in K$,  and let $P$ be a $\wt{D}_h^{\gamma_n}$-geodesic from $z$ to $w$.  Then condition~\ref{it:boundness_of_almost_geodesic_paths} implies that $P \subseteq K'$ and so the $\wt{D}_{h+f}^{\gamma_n}$-length of $P$ (since we also have that $\wt{D}_{h+f}^{\gamma_n} = e^{\xi(\gamma_n) f} \cdot \wt{D}_h^{\gamma_n}$ by Axiom~\ref{it:weyl_scaling}) is at most $C e^A |z-w|^{\chi}$,  where $A \in (0,\infty)$ is such that
\begin{align*}
\xi(\gamma_n) \sup_{z \in K'} |f(z)| \leq A \quad \text{for all} \quad n \in \N.
\end{align*}
It follows that $(\wt{D}_{h+f}^{\gamma_n})$ restricted to $K$ is equicontinuous.

Finally we can conclude the proof of the proposition as follows.  Fix any subsequence $(\wt{D}_{h+f}^{\gamma_{k_n}})_{n \in \N}$.  Then by local equicontinuity and applying the Arzela-Ascoli theorem,  we obtain that we can find a further subsequence $(\wt{D}_{h+f}^{\gamma_{k_n'}})_{n \in \N}$ and pseudometric $\wh{D}$ such that
\begin{align*}
\wt{D}_{h+f}^{\gamma_{k_n'}}  \to \wh{D} \quad \text{as} \quad n \to \infty
\end{align*}
with respect to the local uniform topology on $\C \times \C$.  Hence combining with \eqref{eqn:lower_semicontinuous_convergence},  we obtain that $\wh{D} = e^{\xi(\gamma) f}  \cdot \wt{D}$ a.s.  Therefore since the subsequence $(\wt{D}_{h+f}^{\gamma_{k_n}})_{n \in \N}$ was arbitrary,  we obtain that
\begin{align*}
\wt{D}_{h+f}^{\gamma_n} \to \wt{D} \quad \text{as} \quad n \to \infty
\end{align*}
with respect to the local uniform topology.  This completes the proof of the proposition.
\end{proof}

\subsection{The limit satisfies Axiom~\ref{it:locality}}
\label{subsec:locality_subcritical}

Now suppose that we have the coupling $(h,\wt{D})$ described in the previous sections.  The main goal of this subsection is to prove the following.

\begin{proposition}\label{prop:locality_property_subcritical}
The metric $\wt{D}$ in the coupling $(h,\wt{D})$ described above satisfies Axiom~\ref{it:locality}.
\end{proposition}

In order to prove Proposition~\ref{prop:locality_property_subcritical},  we will follow the strategy introduced in \cite[Section~2]{Pfeffer_2024}.  Let us start with reviewing some definitions and key results from \cite[Section~2]{Pfeffer_2024} that we are going to use.

First we recall the definitions of local and $\xi$-additive metrics for $h$ which are similar to \cite[Definitions~1.2 and ~1.5]{gwynne2020local}.  Since the method introduced in this subsection will be applied in order to obtain similar results for the supercritical LQG metric (see Subsection~\ref{subsec:locality_general_case}),  we will present these definitions with the assumption that the metrics are continuous removed.

\begin{defn}(\cite[Definition~2.2]{Pfeffer_2024})
\label{def:local_metric}
Let $U \subseteq \C$ be a connected open set,  and let $(h,D)$ be a coupling of $h$ with a random length metric $D$ on $U$.  We say that $D$ is a \emph{local metric} for $h$ if,  for any open set $V \subseteq U$,  the internal metric $D(\cdot,\cdot;V)$ is conditionally independent from the pair 
\begin{align*}
(h,D(\cdot,\cdot ; U \setminus \overline{V}))
\end{align*}
given $h|_{\overline{V}}$.
\end{defn}

\begin{defn}(\cite[Definition~2.3]{Pfeffer_2024})
\label{def:xi_additive_metric}
Let $U \subseteq \C$ be a connected open set,  and let $(h,D)$ be a coupling of $h$ with a random length metric $D$ on $U$ which is local for $h$.   We say that $D$ is \emph{$\xi$-additive} for $h$ if,  for each $z \in U$ and each $r>0$ such that $B_r(z) \subseteq U$,  the metric $e^{-\xi h_r(z)} D$ is local for $h-h_r(z)$.
\end{defn}

As in \cite[Section~2]{Pfeffer_2024},  instead of proving Axiom~\ref{it:locality} directly,  we will prove first a weaker locality property for the subsequential limiting metric $\wt{D}$.  In particular,  we will show the following.

\begin{proposition}\label{prop:limit_is_a_xi_additive_metric_subcritical}
The metric $\wt{D}$ is a $\xi(\gamma)$-additive local metric for $h$ in the coupling $(h,\wt{D})$ described above.
\end{proposition}

Then Proposition~\ref{prop:locality_property_subcritical} will follow from combining Proposition~\ref{prop:limit_is_a_xi_additive_metric_subcritical} with the following result from \cite[Section~2]{Pfeffer_2024}.

\begin{proposition}(\cite[Proposition~2.5]{Pfeffer_2024})
\label{prop:main_prop_from_pfeffer_paper}
Let $U \subseteq \C$ be an open and connected set,  and let $h$ be a whole-plane GFF normalized so that $h_1(0) = 0$.  Let $(h,D)$ be a coupling of $h$ with a random lower semicontinuous metric $D$ on $U$.  Suppose that $D$ is a local $\xi$-additive metric for $h$ that satisfies Axioms~\ref{it:length_space} and 
\begin{enumerate}[(I)]
\item 
\textbf{Translation invariance} 
\label{it:translation_invariance}
For each deterministic $z \in \C$,  a.s.  $D_{h(\cdot +z)} = D_h(\cdot + z,  \cdot + z)$.
\item
\textbf{Tightness across scales} 
\label{it:tightness_across_and-around_annuli}
For each $r>0$,  there exists a deterministic constant $c_r>0$ such that if $A$ is a fixed annulus,  the laws of the distances $c_r^{-1} e^{-\xi h_r(0)} D_h(\text{across}\,\,rA)$ and $c_r^{-1} e^{-\xi h_r(0)} D_h(\text{around}\,\,rA)$ and their inverses for $r>0$ are tight.  Moreover,  there exists $\Lambda > 1$ such that for each $\delta \in (0,1)$,  we have that
\begin{align*}
\Lambda^{-1} \delta^{\Lambda} \leq \frac{c_{\delta r}}{c_r} \leq \Lambda \delta^{-\Lambda} \quad \text{for all} \quad r>0.
\end{align*}
\end{enumerate}

with $c_r = r^{\xi Q + o_r(1)}$ as $r \to 0,  r \in \Q$ for some $Q>0$.  Moreover,  suppose that $D$ is a complete and geodesic metric on the set $U \setminus \{\text{singluar points}\}$.

Then $D$ satisfies Axiom~\ref{it:locality}.  In particular,  $D$ is almost surely determined by $h$.
\end{proposition}

Let us now prove Proposition~\ref{prop:limit_is_a_xi_additive_metric_subcritical}.  We start by stating a version of $\xi(\gamma)$-additivity as in Definitions~\ref{def:local_metric} and ~\ref{def:xi_additive_metric} in the case that $B_r(z) \subseteq V$.  See also \cite[Lemma~2.13]{Pfeffer_2024}.

\begin{lemma}\label{lem:version_of_xi_additivity}
Let $r>0$ and $z \in \C$,  and let $V$ be an open subset of $\C$ that contains $B_r(z)$ and such that $\overline{V} \neq \C$.  Set $\wh{h}:=h - h_r(z)$.  Then the internal metric $\wt{D}(\cdot,\cdot ; V)$ is conditionally independent from the pair $(\wh{h} , e^{-\xi(\gamma) h_r(z)} \wt{D}(\cdot ,  \cdot ; \C \setminus \overline{V}))$ given $\wh{h}|_{\overline{V}}$.
\end{lemma}

\begin{proof}
It follows from the exact same argument used to prove \cite[Lemma~2.13]{Pfeffer_2024}.
\end{proof}

\begin{proof}[Proof of Proposition~\ref{prop:limit_is_a_xi_additive_metric_subcritical}]
First we note that the $\xi(\gamma)$-additivity property is obviously true when $\overline{V} = \C$.  For other choices of $V$,  Lemma~\ref{lem:version_of_xi_additivity} implies that $\wt{D}$ is a $\xi(\gamma)$-additive local metric for $h$ in the case that $B_r(z) \subseteq V$.  Then the general case follows from \cite[Lemma~2.19]{dubedat2020weak}.  This completes the proof of the proposition.
\end{proof}

\begin{proof}[Proof of Proposition~\ref{prop:locality_property_subcritical}]
First we note that Proposition~\ref{prop:main_prop_from_pfeffer_paper} implies that it suffices to check that the conditions in the statement of Proposition~\ref{prop:main_prop_from_pfeffer_paper} are satisfied for the metric $\wt{D}$ with $U = \C$.  

Proposition~\ref{prop:complete_and_geodesic_metric} implies that $\wt{D}$ is a complete and geodesic metric on $\C$.  Note that this implies that $\wt{D}$ satisfies Axiom~\ref{it:length_space}.  Also Proposition~\ref{prop:limit_is_a_xi_additive_metric_subcritical} implies that $\wt{D}$ is a $\xi(\gamma)$-additive local metric for $h$.  Hence since $\wt{D}$ does not have singular points a.s.,  it suffices to show that $\wt{D}$ satisfies conditions~\ref{it:translation_invariance} and ~\ref{it:tightness_across_and-around_annuli} with $c_r = r^{\xi(\gamma) Q(\gamma)}$ in order to complete the proof of the proposition.

To prove condition~\ref{it:tightness_across_and-around_annuli},  we note that the random metrics $\wt{D}_h^{\gamma_n}$ and $r^{-\xi(\gamma_n) Q(\gamma_n)} e^{-\xi(\gamma_n) h_r(0)} \wt{D}_h^{\gamma_n}(r\cdot ,  r \cdot)$ have the same law for all $r>0,n \in \N$.  Moreover the a.s.  convergence of $\wt{D}_h^{\gamma_n}$ to $\wt{D}$ as $n \to \infty$ with respect to the local uniform topology combined with the fact that
\begin{align*}
\xi(\gamma_n) \to \xi(\gamma) \quad \text{and} \quad Q(\gamma_n) \to Q(\gamma) \quad \text{as} \quad n \to \infty
\end{align*}
and the local uniform convergence of $(z,r) \to h_r(z)$ as $n \to \infty$ when viewed as a function in $\C \times (0,\infty)$  in the coupling of $\{(h,\wt{D}_h^{\gamma_n}) ,  (h,\wt{D})\}_{n \in \N}$,  we obtain that
\begin{align*}
r^{-\xi(\gamma_n) Q(\gamma_n)} e^{-\xi(\gamma_n) h_r(0)} \wt{D}_h^{\gamma_n}(r \cdot ,  r \cdot) \to r^{-\xi(\gamma) Q(\gamma)} e^{-\xi(\gamma) h_r(0)} \wt{D}(r \cdot ,  r \cdot) \quad \text{as} \quad n \to \infty
\end{align*}
with respect to the local uniform topology on $\C \times \C$ a.s.,  for all $r >0$.  In particular,  we have that the random metrics $\wt{D}$ and $r^{-\xi(\gamma) Q(\gamma)} e^{-\xi(\gamma) h_r(0)} \wt{D}(r \cdot ,  r \cdot)$ have the same law for all $r>0$.  Hence if $A = \mathbb{A}_{r_1,r_2}(z)$ for $z \in \C ,  0 < r_1 < r_2 < \infty$ is a fixed annulus,  in order to prove that the random variables 
\begin{align*}
r^{-\xi(\gamma) Q(\gamma)} e^{-\xi(\gamma) h_r(0)} \wt{D}(\text{around} \,\,  r A),\quad r^{-\xi(\gamma) Q(\gamma)} e^{-\xi(\gamma) h_r(0)} \wt{D}(\text{across} \,\,  r A)
\end{align*}
and their inverses are tight in $r>0$,  it suffices to prove that 
\begin{align*}
\wt{D}(\text{across}\,\,A) \in (0,\infty)\quad \text{and}
 \quad  \wt{D}(\text{around} \,\, A) \in (0,\infty) \quad \text{a.s.}
\end{align*}

Note that Proposition~\ref{prop:limit_is_a_metric} implies that $\wt{D}(\text{across} \,\,A) \in (0,\infty)$ a.s.  Moreover Lemma~\ref{lem:tightness_around_annuli} implies that $\wt{D}(\text{around} \,\,A) < \infty$ a.s.  Therefore it suffices to prove that $\wt{D}(\text{around} \,\,A) > 0$ a.s.  To prove the latter,  we fix $u,v \in \partial B_{\frac{r_1+r_2}{2}}(z)$ such that $\overline{B_{\frac{r_2-r_1}{2}}(u)} \cap \overline{B_{\frac{r_2-r_1}{2}}(v)} = \emptyset$.   Note that any path in $A$ which disconnects $\partial B_{r_1}(z)$ from $\partial B_{r_2}(z)$ has to intersect both $B_{\frac{r_2-r_1}{2}}(u)$ and $B_{\frac{r_2-r_1}{2}}(v)$.  Thus combining with Proposition~\ref{prop:limit_is_a_metric},  we obtain that
\begin{align*}
\wt{D}(\text{around} \,\,A) \geq \wt{D}\left(B_{\frac{r_2-r_1}{2}}(u) ,  B_{\frac{r_2-r_1}{2}}(v)\right) > 0 \quad \text{a.s.}
\end{align*}
It follows that $\wt{D}$ satisfies condition~\ref{it:tightness_across_and-around_annuli}.

Finally since the metrics $\wt{D}_h^{\gamma_n}(\cdot + z,  \cdot +z)$ and $\wt{D}_{h(\cdot + z)}^{\gamma_n}$ have the same law for all $n \in \N,  z \in \C$,  we obtain by taking $n \to \infty$ and arguing as before that $\wt{D}$ satisfies condition~\ref{it:translation_invariance}.  This completes the proof of the proposition.
\end{proof}

\subsection{The limit is the $\gamma$-LQG metric}
\label{subsec:limit_is_the_gamma_lqg_metric}

Now we are ready to prove Theorem~\ref{thm:main_theorem_subcritical_intro}.  The main input in the proof of Theorem~\ref{thm:main_theorem_subcritical_intro} is the following proposition.

\begin{proposition}\label{prop:main_prop_subcritical}
Suppose that we have the same setup as in the statement of Theorem~\ref{thm:main_theorem_subcritical_intro}.  Then there exists a $\gamma$-LQG metric $h \to \wh{D}_h$ with some fixed normalization so that the following is true.  Let $(\gamma_n)_{n \in \N}$ be a sequence in $(0,2)$ such that $\gamma_n \to \gamma$ as $n \to \infty$.  Then we have that
\begin{align*}
\wh{D}_h^{\gamma_n} \to \wh{D}_h \quad \text{as} \quad n \to \infty
\end{align*}
in probability with respect to the local uniform topology in $\C \times \C$.
\end{proposition}

\begin{proof}
Fix $\gamma \in (0,2)$ and let $(\gamma_n)_{n \in \N}$ be a sequence in $(0,2)$ such that $\gamma_n \to \gamma$ as $n \to \infty$.  Fix also a subsequence $(\gamma_{k_n})_{n \in \N}$ of $(\gamma_n)$.  Then having the setup of Theorem~\ref{thm:tightness_subcritical},  we obtain that Theorem~\ref{thm:tightness_subcritical} combined with Skorohhod's representation theorem imply that we can find a further subsequence $(\gamma_{k_n'})$ and a coupling of $\{(h^n ,  \wt{D}_{h^n}^{\gamma_n})\}_{n \in \N}$,  where $h_n$ is a whole-plane GFF normalized such that its average on $\partial B_1(0)$ is equal to zero,  with a pair $(h,\wt{D})$ such that a.s.  we have that 
\begin{align*}
h^n \to h \quad \text{in} \quad H_{\text{loc}}^{-1}(\C) \quad \text{as} \quad n \to \infty
\end{align*}
and 
\begin{align*}
\wt{D}_{h^n}^{\gamma_{k_n'}} \to \wt{D} \quad \text{as} \quad n \to \infty
\end{align*}
with respect to the local uniform topology in $\C \times \C$.

Propositions~\ref{prop:limit_is_a_metric},  ~\ref{prop:complete_and_geodesic_metric} and ~\ref{prop:locality_property_subcritical} imply that $\wt{D}$ is a metric a.s.  which satisfies Axioms~\ref{it:length_space} and ~\ref{it:locality}.    Moreover Proposition~\ref{prop:weyl_scaling} implies that the following is true a.s.  For every continuous and bounded function $f : \C \to \R$, we have that
\begin{align*}
\wt{D}_{h^n + f}^{\gamma_{k_n'}} \to e^{\xi(\gamma)f} \cdot \wt{D} \quad \text{as} \quad n \to \infty
\end{align*}
with respect to the local uniform topology in $\C \times \C$.

We set $\wt{D}_h:=\wt{D}$ and $\wt{D}_{h+f}:=e^{\xi(\gamma)f} \cdot \wt{D}$.  Then it is easy to see that for every fixed bounded and continuous function $f: \C \to \R$,  the metric $\wt{D}_{h+f}$ satisfies Axioms~\ref{it:length_space} and  ~\ref{it:weyl_scaling} since the same is true for $\wt{D}$.  Moreover, arguing as in the proof of \cite[Lemma~2.6]{Pfeffer_2024},  we obtain that $\wt{D}_{h+f}$ satisfies Axiom~\ref{it:locality} as well.  Furthermore by Axiom~\ref{it:coordinate_change},  we have for fixed $z \in \C,  r>0$ that 
\begin{align*}
\wt{D}_{h^n+f}^{\gamma_{k_n}'}(r\cdot + z ,  r \cdot +z) = \wt{D}_{(h^n+f)(r\cdot +z) + Q \log(r)}^{\gamma_{k_n}'} \quad \text{for all} \quad n \in \N \quad \text{a.s.}
\end{align*}
Since 
\begin{align*}
&\wt{D}_{h^n+f}^{\gamma_{k_n}'}(r \cdot +z,  r\cdot +z) \to \wt{D}_{h+f}(r \cdot + z,  r \cdot +z) \quad \text{as} \quad n \to \infty,\\
&\wt{D}_{(h^n+f)(r \cdot +z) + Q \log(r)}^{\gamma_{k_n}'} \to \wt{D}_{(h+f)(r \cdot + z) + Q \log(r)} \quad \text{as} \quad n \to \infty
\end{align*}
with respect to the local uniform topology on $\C \times \C$ a.s.,  we obtain that $\wt{D}_{h+f}$ satisfies Axiom~\ref{it:coordinate_change}.

To extend the definition of $\wt{D}_{h+f}$ in the case that $f : \C \to \R$ is an unbounded continuous function,  we argue as follows.  For each open and bounded set $V \subseteq \C$,  we set $\wt{D}_{h+f}^V :=\wt{D}_{h+ \phi_V f}(\cdot,\cdot ; V)$,  where $\phi_V$ is a fixed smooth compactly supported function such that $\phi_V \equiv 1$ on $V$.  Then Axiom~\ref{it:locality} implies that $\wt{D}_{h+f}^V$ is a.s.  determined by $(h+\phi_V f)|_V = (h+f)|_V$,  since $\phi_V f$ is a bounded and continuous function.  We defined the metric $\wt{D}_{h+f}$ as the length metric with the property that for any continuous path $P$ in $\C$,  we have that the $\wt{D}_{h+f}$-length of $P$ is equal to its $\wt{D}_{h+f}^V$-length,  for any bounded and open set $V \subseteq \C$ such that $P \subseteq V$.Then it is easy to see that $\wt{D}_{h+f}$ satisfies Axioms~\ref{it:length_space},  \ref{it:locality},\ref{it:weyl_scaling} and \ref{it:coordinate_change}.   It follows that $h \to \wt{D}_h$ defines a $\gamma$-LQG metric with some fixed normalization (see Theorems~\ref{thm:uniqueness_subcritical}).

Finally we choose a normalization for $h \to \wt{D}_h$ which is independent from $(\gamma_n)$.  More precisely,  we let $\wt{\beta}(\gamma) \in (0,\infty)$ be such that
\begin{align*}
\p\left[\wt{D}_h(0,1) \leq \wt{\beta}(\gamma) \right] = \frac{1}{2}
\end{align*}
in the case that $h$ is a whole-plane GFF normalized such that $h_1(0) = 0$.
Note that Lemma~\ref{lem:normalization_well_defined} implies that $\wt{\beta}(\gamma)$ is well-defined.  Since $\wt{D}_{h^n}^{\gamma_{k_n'}} \to \wt{D}_h$ as $n \to \infty$ with respect to the local uniform topology in $\C \times \C$,  and 
\begin{align*}
\p\left[\wt{D}_h(0,1) = \wt{\beta}(\gamma) \right] = 0
\end{align*}
by Lemma~\ref{lem:normalization_well_defined},  we obtain that $\wt{\beta}(\gamma_{k_n'}) \to \wt{\beta}(\gamma)$ as $n \to \infty$,  where $\wt{\beta}(\gamma_{k_n'})$ is such that
\begin{align*}
\p\left[\wt{D}_{h^n}^{\gamma_{k_n'}}(0,1) \leq \wt{\beta}(\gamma_{k_n'})\right] = \frac{1}{2}.
\end{align*}
Again we have by Lemma~\ref{lem:normalization_well_defined} that $\wt{\beta}(\gamma_{k_n'})$ is well-defined for all $n$.
Note that $\wh{D}_{h^n}^{\gamma_{k_n'}} = \wt{\beta}(\gamma_{k_n'})^{-1} \wt{D}_{h^n}^{\gamma_{k_n'}}$ and so $\wh{D}_{h^n}^{\gamma_{k_n'}} \to \wh{D}_h$ as $n \to \infty$ locally uniformly,  where $\wh{D}_h:=\wt{\beta}(\gamma)^{-1} \wt{D}_h$.  Note that
\begin{align*}
\p\left[\wh{D}_h(0,1) \leq 1\right] = \frac{1}{2}.
\end{align*}
Therefore since the subsequence $(\gamma_{k_n})_{n \in \N}$ was arbitrary,  we obtain that
\begin{align*}
\wh{D}_h^{\gamma_n} \to \wh{D}_h \quad \text{in law} \quad \text{as} \quad n \to \infty
\end{align*}
with respect to the local uniform topology in $\C \times \C$.  Hence since the random variables $\wh{D}_h^{\gamma_n}$ and $\wh{D}_h$ are a.s.  determined by $h$ for all $n \in \N$,  by Axiom~\ref{it:locality},  we obtain by combining with Lemma~\ref{lem:convergence_in_prob} that $\wh{D}_h^{\gamma_n} \to \wh{D}_h$ as $n \to \infty$ in probability.  This completes the proof of the theorem.
\end{proof}

\begin{proof}[Proof of Theorem~\ref{thm:main_theorem_subcritical_intro}]
It follows from combining Proposition~\ref{prop:main_prop_subcritical} with Theorem~\ref{thm:uniqueness_subcritical}.
\end{proof}

\section{Tightness: $\gamma \to 0$ case}
\label{sec:gamma_to_zero_case}

In this section,  we will show that the family of metrics $\wt{D}_h^{\gamma}$ in Theorem~\ref{thm:main_theorem_gamma_to_zero_intro} is tight as $\gamma \to 0$ with respect to the local uniform topology on $\C \times \C$.  We will follow the same argument as in Section~\ref{sec:tightness} and obtain analogous statements.  We will re-normalize the metrics $D_h^{\gamma}$ as in \eqref{eqn:main_normalization} but now the constant $\mathfrak{p}$ will not depend on $\gamma$.  Therefore the statements of lemmas and propositions will be different from those in Section~\ref{sec:tightness}.

Let us start with defining our re-normalization of the metrics $D_h^{\gamma}$.  Fix $c \in (0,\frac{1}{2})$ and let $b \in (0,1)$ be a universal constant such that $2 \sqrt{2} \xi_{\text{crit}} < b^{-1}$.  We choose $M \in (1,\infty)$ sufficiently large such that
\begin{align*}
\frac{4 + \xi_{\text{crit}} 2 \sqrt{2} \sqrt{4 + M}}{\sqrt{M}} < b^{-1}.
\end{align*}
Note that the above choices of $b$ and $M$ allows us to apply Proposition~\ref{prop:bound_on_lqg_distances_of_sets} (see \eqref{eqn:choice_of_constant_b} in the proof of Proposition~\ref{prop:bound_on_lqg_distances_of_sets}).  Let $\mathfrak{p}_0 \in (0,1)$ be the constant in the statement  of Proposition~\ref{prop:bound_on_lqg_distances_of_sets} corresponding to the above choice of $M$.  Set also $\zeta = \zeta(\gamma) = \frac{c}{\xi(\gamma)}$ for all $\gamma \in (0,2)$.  Furthermore,  for all $\gamma \in (0,2)$,  we let $\alpha(\gamma) \in (0,\infty)$ be such that
\begin{align}\label{eqn:main_normalization_gamma_to_zero}
\p\left[D_h^{\gamma}(\text{around} \,\,  \mathbb{A}_{1,2}(0)) \leq \alpha(\gamma) \right] = \mathfrak{p}_0
\end{align}
and set 
\begin{align}\label{eqn:normalization_gamma_to_zero}
\wt{D}_h^{\gamma}:=\alpha(\gamma)^{-1} D_h^{\gamma}.
\end{align}
Recall that Lemma~\ref{lem:normalization_well_defined} implies that $\alpha(\gamma)$ is well-defined.

The main goal of the section is to prove the following analogue of Theorem~\ref{thm:tightness_subcritical}.

\begin{theorem}\label{thm:tightness_gamma_to_zero}
There exists $\gamma_0 \in (0,2)$ sufficiently small such that the following is true.  Let $U \subseteq \C$ be an open and connected set. Then the collection of laws of the random metrics $\wt{D}_h^{\gamma}(\cdot ,  \cdot ; U)$ on $U \times U$ for $\gamma \in (0,\gamma_0)$ is tight with respect to the local uniform topology on $U \times U$.
\end{theorem}

The main ingredient in the proof of Theorem~\ref{thm:tightness_gamma_to_zero} is the following analogue of Proposition~\ref{prop:tightness_proposition_subcritical}.

\begin{proposition}\label{prop:tightness_proposition_gamma_to_zero}
Fix $s \in (0,1)$.  Then there exists a constant $\gamma_0 \in (0,2)$ depending only on $s$
such that for every open and connected set $U \subseteq \C$ and every compact and connected set $K \subseteq U$,  we have that
\begin{align*}
\sup_{\gamma \in (0,\gamma_0)} \p\left[ \sup_{z,w \in K} \frac{\wt{D}_h^{\gamma}(z,w ; U)}{|z-w|^s} \geq R \right] \to 0 \quad \text{as} \quad R \to \infty.
\end{align*}
\end{proposition}

\begin{proof}[Proof of Theorem~\ref{thm:tightness_gamma_to_zero} assuming Proposition~\ref{prop:tightness_proposition_gamma_to_zero}]
It follows from the argument in the proof of Theorem~\ref{thm:tightness_subcritical} where we use Proposition~\ref{prop:tightness_proposition_gamma_to_zero} in place of Proposition~\ref{prop:tightness_proposition_subcritical}.
\end{proof}

Therefore it remains to prove Proposition~\ref{prop:tightness_proposition_gamma_to_zero}.  In order to prove Proposition~\ref{prop:tightness_proposition_gamma_to_zero},  we use  Lemmas~\ref{lem:upper_bound_internal_metric_on_points_gamma_to_zero} and \ref{lem:moment_lqg_diameter_bound_for_compact_sets_gamma_to_zero} which are versions of Lemma~\ref{lem:lqg_diameter_upper_bound_internal_metric_on_points} and Proposition~\ref{prop:moment_lqg_diameter_bound_for_compact_sets} respectively.

\begin{lemma}\label{lem:upper_bound_internal_metric_on_points_gamma_to_zero}
Fix $s \in (0,1)$.  Then we can choose $M>1$ sufficiently large and $\gamma_0 \in (0,2)$ sufficiently small (both depending only on $s$) such that the following is true for all $\gamma \in (0,\gamma_0)$.  Fix $\mathfrak{r}>0$ and let $K \subseteq \C$ be a compact set.  Then it holds with polynomially high probability as $\epsilon \to 0$,  at a rate depending only on $M,s$ and $K$,  that
\begin{align*}
\mathfrak{r}^{-\xi(\gamma) Q(\gamma)} e^{-\xi(\gamma) h_{\mathfrak{r}}(0)} \wt{D}_h^{\gamma}(u,v ; B_{2|u-v|}(u)) \leq \left|\frac{u-v}{\mathfrak{r}}\right|^s
\end{align*}
for all $u,v \in K$ with $|u-v| \leq \epsilon \mathfrak{r}$.
\end{lemma}

\begin{lemma}
\label{lem:moment_lqg_diameter_bound_for_compact_sets_gamma_to_zero}
There exist universal constants $\gamma_0 \in (0,2),c_1>0$ such that the following is true.
Fix $p \in (0,c_1 \sqrt{M})$,  let $U \subseteq \C$ be an open and connected set and let $K \subseteq U$ be a compact and connected set with more than a point.  Then for all $\gamma \in (0,\gamma_0)$ and all $\mathfrak{r} > 0$,  we have that
\begin{align*}
\E\left[\left(\mathfrak{r}^{-\xi(\gamma) Q(\gamma)} e^{-\xi(\gamma) h_{\mathfrak{r}}(0)} \sup_{z ,w  \in \mathfrak{r} K} \wt{D}_h^{\gamma}(z,w ;  \mathfrak{r} U) \right)^p \right] \lesssim 1,
\end{align*}
where the implicit constant depends only on $p,M,K$,  and $U$.
\end{lemma}

\begin{proof}[Proof of Proposition~\ref{prop:tightness_proposition_gamma_to_zero} assuming Lemmas~\ref{lem:upper_bound_internal_metric_on_points_gamma_to_zero} and ~\ref{lem:moment_lqg_diameter_bound_for_compact_sets_gamma_to_zero}]
It follows from the argument used to prove Proposition~\ref{prop:tightness_proposition_subcritical} with Lemma~\ref{lem:upper_bound_internal_metric_on_points_gamma_to_zero} in place of Lemma~\ref{lem:lqg_diameter_upper_bound_internal_metric_on_points}
and Lemma~\ref{lem:moment_lqg_diameter_bound_for_compact_sets_gamma_to_zero} in place of Proposition~\ref{prop:moment_lqg_diameter_bound_for_compact_sets}.
\end{proof}

We start by proving Lemma~\ref{lem:moment_lqg_diameter_bound_for_compact_sets_gamma_to_zero}.  The proof of the latter will follow from Lemmas~\ref{lem:circle_average_uniform_bound_gamma_to_zero_case} and ~\ref{lem:upper_bound_on_internal_diamters_of_spuares_gamma_to_zero}.  Lemma~\ref{lem:circle_average_uniform_bound_gamma_to_zero_case} is the analogue of Lemma~\ref{lem:circle_averages_uniform_bound}, while Lemma~\ref{lem:upper_bound_on_internal_diamters_of_spuares_gamma_to_zero} is the analogue of Lemma~\ref{lem:upper_bound_on_internal_diamters_of_spuares}.  We start with stating and proving Lemma~\ref{lem:circle_average_uniform_bound_gamma_to_zero_case}.

\begin{lemma}\label{lem:circle_average_uniform_bound_gamma_to_zero_case}
There exist universal constants $\wt{C}>1,  \gamma_0 \in (0,2)$ such that the following is true for all $\mathfrak{r}>0$ and all $n \in \N$.  For $C>1$,  we let $E_{\mathfrak{r}}^n$ be the event as in the statement of Lemma~\ref{lem:circle_averages_uniform_bound} with $\frac{Q(\gamma)}{2}$ in place of $q$.  Then for all $\gamma \in (0,\gamma_0)$ and all $C>1$,  we have that
\begin{align*}
&\p[(E_{\mathfrak{r}}^n)^c ] \\
&\leq \frac{\wt{C}}{\log(2)\left(1-C^{-\zeta \left(\frac{Q(\gamma)^2}{8(1 + \wt{C} \log(2)^{-1})} - 2\right)}\right)} \exp\left(-\frac{\zeta^2 \log(C)^2}{\log(2)} \left(\frac{Q(\gamma)^2}{8(1+ \wt{C} \log(2)^{-1})} - 2 \right) + \log(\zeta) +  \log(\log(C))\right).
\end{align*}
\end{lemma}

\begin{proof}
Fix $k \in \N$ and suppose that $2^n \in [C^{k \zeta} ,  C^{(k+1) \zeta}]$.  Then Lemma~\ref{lem:upper_bound_on_circle_averages} applied with $\epsilon = 2^{-n},\nu =0 ,R=1$ and $q = \frac{Q(\gamma)}{2} + \frac{1}{(k+1) \zeta}$ implies that there exists a universal constant $\wt{C}>1$ such that
\begin{align*}
&\p\left[(E_{\mathfrak{r}}^n)^c\right] \leq \p\left[\sup\{|h_{2^{-n} \mathfrak{r}}(w) - h_{\mathfrak{r}}(0)| : w \in B_{\mathfrak{r}}(0) \cap (2^{-n-1} \mathfrak{r} \Z^2) \} > \left(\frac{Q(\gamma)}{2} + \frac{1}{(k+1) \zeta}\right) \log(2^n)\right]\\
&\leq \wt{C} \exp\left(-k \zeta \left(\frac{(\frac{Q(\gamma)}{2} + \frac{1}{(k+1) \zeta})^2}{2(1+\wt{C} \log(2)^{-1})} - 2 \right) \log(C) \right) \leq \wt{C} \exp\left(-k \zeta \left(\frac{Q(\gamma)^2}{8(1+\wt{C} \log(2)^{-1})} - 2\right) \log(C) \right) \\
&\text{for all} \quad C>1.
\end{align*}

Hence summing over all $n \in \N$ such that $C^{k\zeta} \leq 2^n \leq C^{(k+1)\zeta}$ implies that 
\begin{align*}
&\p\left[E_{\mathfrak{r}}^n,  \text{for all} \,\, n \in \N \,\,\text{with}\,\,C^{k \zeta} \leq 2^n \leq C^{(k+1) \zeta}\right] \\
&\geq 1 - \frac{\wt{C}}{\log(2)} \exp\left(-k \zeta \left(\frac{Q(\gamma)^2}{8(1+\wt{C} \log(2)^{-1})} - 2\right) \log(C) + \log(\zeta) + \log(\log(C))\right).
\end{align*}
Recall that $Q(\gamma) \to \infty$ as $\gamma \to 0$ and so we can choose $\gamma_0 \in (0,2)$ such that
\begin{align*}
\frac{Q(\gamma)^2}{8 ( 1+ \wt{C} \log(2)^{-1})} - 2 > 0 \quad \text{for all} \quad \gamma \in (0,\gamma_0).
\end{align*}
Therefore summing over all $k \in \N$ and taking a union bound completes the proof of the lemma.
\end{proof}

Next we assume that the same setup just after the proof of Lemma~\ref{lem:circle_averages_uniform_bound}.  In particular,  we let again $\mathcal{R}_{\mathfrak{r}}^n$ be the set of open $2^{-n} \mathfrak{r} \times 2^{-n-1} \mathfrak{r}$ or $2^{-n-1} \mathfrak{r} \times 2^{-n} \mathfrak{r}$ rectangles $R \subseteq \mathfrak{r} \mathbb{S}$ with corners in $2^{-n-1} \mathfrak{r} \Z^2$,  and let $N_C = \lfloor \log_2(C^{\zeta})\rfloor +1$ for $C>1$.  

Similarly,  by applying Proposition~\ref{prop:bound_on_lqg_distances_of_sets} with $A = 2^{\zeta \xi(\gamma) n} = 2^{c n}$ gives that off an event with probability at least
\begin{align*}
1 - \wt{C} \sum_{n \geq N_C} 2^{-\zeta \xi(\gamma) b \sqrt{M} n } \geq 1 - \frac{\wt{C}}{1 - 2^{-c b \sqrt{M}}} C^{-c b \zeta \sqrt{M}},
\end{align*}
we have that for all $n \geq N_C,  R \in \mathcal{R}_{\mathfrak{r}}^n$,  the distance between the two shorter sides of $R$ with respect to $\wt{D}_h^{\gamma}(\cdot ,  \cdot ; R)$ is at most 
\begin{align*}
2^{\zeta \xi(\gamma) n} (2^{-n} \mathfrak{r})^{\xi(\gamma) Q(\gamma)} e^{\xi(\gamma) h_{2^{-n} \mathfrak{r}}(w_R)},
\end{align*}
where $w_R$ denotes the bottom-left corner of $R$.

\begin{lemma}\label{lem:upper_bound_on_internal_diamters_of_spuares_gamma_to_zero}
There exist universal constants $c_1,c_1>0$ and $\gamma_0 \in (0,2)$,  such that for all $\gamma \in (0,\gamma_0) ,  \mathfrak{r}>0$ and $C \geq 2$,  we have that
\begin{align*}
\p\left[\mathfrak{r}^{-\xi(\gamma) Q(\gamma)} e^{-\xi(\gamma) h_{\mathfrak{r}}(0)} \sup_{z ,w  \in \mathfrak{r} \mathbb{S}} \wt{D}_h^{\gamma}(z,w ;  \mathfrak{r} \mathbb{S}) > C\right] \leq c_0 C^{-c_1 \sqrt{M}} + c_0 \exp(-c_1 Q(\gamma)^2 \log(C)^2).
\end{align*}
\end{lemma}

\begin{proof}
Let $\gamma_0 \in (0,2)$ be the constant in Lemma~\ref{lem:circle_average_uniform_bound_gamma_to_zero_case} and fix  $\gamma \in (0,\gamma_0)$.  Suppose that we have the setup described just before the statement of the lemma.  Then combining with Lemma~\ref{lem:circle_average_uniform_bound_gamma_to_zero_case},  we obtain that off an event with probability at most
\begin{align*}
&\frac{\wt{C}}{1 - 2^{-c b \sqrt{M}}} C^{-c b \zeta \sqrt{M}}\\
&+\frac{\wt{C}}{\log(2)\left(1-C^{-\zeta \left(\frac{Q(\gamma)^2}{8(1 + \wt{C} \log(2)^{-1})} - 2\right)}\right)} \exp\left(-\frac{\zeta^2 \log(C)^2}{\log(2)} \left(\frac{Q(\gamma)^2}{8(1+ \wt{C} \log(2)^{-1})} - 2 \right) + \log(\zeta) +  \log(\log(C))\right),
\end{align*}
the following holds.  For all $n \geq N_C$ and all $R \in \mathcal{R}_{\mathfrak{r}}^n$,  there exists a path $P_R$ in $R$ between the two shorter sides of $R$ with $\wt{D}_h^{\gamma}$-length at most 
\begin{align*}
C^{\xi(\gamma)} 2^{-(\frac{Q(\gamma)}{2} - \zeta)\xi(\gamma) n} e^{\xi(\gamma) h_{\mathfrak{r}}(0)}.
\end{align*}
Henceforth we assume that such a path exists.  Recall that \cite[Theorem~1.2]{ding2020fractal} implies that
\begin{align*}
d_{\gamma} \leq 2 + \frac{\gamma^2}{2} + \sqrt{2} \gamma \quad \text{for all} \quad \gamma \in \left(0,\sqrt{\frac{8}{3}}\right].
\end{align*}
Hence since 
\begin{align*}
\zeta = \frac{cd_{\gamma}}{\gamma} \leq \frac{2c}{\gamma} + c \sqrt{2} + \frac{c\sqrt{\gamma}}{2}
\end{align*}
and $c \in (0,\frac{1}{2})$,  we obtain that possibly by taking $\gamma_0 \in (0,2)$ to be smaller,  we have that $\zeta < \frac{Q(\gamma)}{2}$ for all $\gamma \in (0,\gamma_0)$.

Arguing as in the proof of Lemma~\ref{lem:upper_bound_on_internal_diamters_of_spuares},  we obtain that the $\wt{D}_h^{\gamma}(\cdot ,  \cdot ; \mathfrak{r} \mathbb{S})$-diameter of $\mathfrak{r} \mathbb{S}$ is at most 
\begin{align*}
\frac{8 C^{\xi(\gamma) + \zeta}}{1 - 2^{-(\frac{Q(\gamma)}{2} - \zeta) \xi(\gamma)}}  \mathfrak{r}^{\xi(\gamma) Q(\gamma)} e^{\xi(\gamma) h_{\mathfrak{r}}(0)}.
\end{align*}

Note that $Q(\gamma) \to \infty$ and $\zeta \to \infty$ as $\gamma \to 0$.  Thus there exists $\gamma_0 \in (0,2)$ and universal constants $\wt{c} ,  \wt{b} > 0$ such that for all $\gamma \in (0,\gamma_0)$ and all $C \geq 2$,  we have with probability at least 
\begin{align*}
1 - \wt{C} C^{-c b \zeta \sqrt{M}} - \wt{C} \exp(-\zeta^2 Q(\gamma)^2 \wt{b} \log(C)^2)
\end{align*}
that the $\wt{D}_h^{\gamma}(\cdot ,  \cdot ;  \mathfrak{r} \mathbb{S})$-diameter of $\mathfrak{r} \mathbb{S}$ is at most
\begin{align*}
\frac{8 C^{\xi(\gamma) + \zeta}}{1-2^{-(\frac{Q(\gamma)}{2} - \zeta) \xi(\gamma)}} \mathfrak{r}^{\xi(\gamma) Q(\gamma)} e^{\xi(\gamma) h_{\mathfrak{r}}(0)}.
\end{align*}
Therefore replacing $C$ with $C^{\frac{1}{\xi(\gamma) + \zeta}}$ gives that
\begin{align*}
&\p\left[\left(1-2^{-(\frac{Q(\gamma)}{2} - \zeta) \xi(\gamma)}\right)\mathfrak{r}^{\xi(\gamma) Q(\gamma)} e^{-\xi(\gamma) h_{\mathfrak{r}}(0)} \sup_{z ,w  \in \mathfrak{r} \mathbb{S}} \wt{D}_h^{\gamma}(z,w ;  \mathfrak{r} \mathbb{S}) > C\right]\\
&\leq \wt{C} C^{-\frac{c b \zeta \sqrt{M}}{\xi(\gamma) + \zeta}} + \wt{C} \exp\left(-\zeta^2 Q(\gamma)^2 \wt{b} (\xi(\gamma) + \zeta)^{-2} \log(C)^2\right)
\end{align*}
for all $C \geq 2 ,  \mathfrak{r}>0$ and $\gamma \in (0,\gamma_0)$.  

Since $\frac{\zeta}{\xi(\gamma) + \zeta} = \frac{c}{\xi(\gamma)^2 + c}$ and $\xi(\gamma) Q(\gamma) \to 1$ as $\gamma \to 0$ (see \cite[Theorem~1.2]{ding2020fractal}),  we obtain that there exist universal constants $c_0,c_1 >0$ and $\gamma_0 \in (0,2)$ such that for all $\gamma \in (0,\gamma_0),  \mathfrak{r}>0$ and all $C \geq 2$,  we have that
\begin{align*}
\p\left[\mathfrak{r}^{-\xi(\gamma) Q(\gamma)} e^{-\xi(\gamma) h_{\mathfrak{r}}(0)} \sup_{z ,w  \in \mathfrak{r} \mathbb{S}} \wt{D}_h^{\gamma}(z,w ;  \mathfrak{r} \mathbb{S}) > C\right] \leq c_0 C^{-c_1 \sqrt{M}} + c_0 \exp(-c_1 Q(\gamma)^2 \log(C)^2).
\end{align*}
This completes the proof of the lemma.
\end{proof}

\begin{proof}[Proof of Lemma~\ref{lem:moment_lqg_diameter_bound_for_compact_sets_gamma_to_zero}]
It follows from the same argument used to prove Proposition~\ref{prop:moment_lqg_diameter_bound_for_compact_sets} except for using Lemmas~\ref{lem:circle_average_uniform_bound_gamma_to_zero_case} and ~\ref{lem:upper_bound_on_internal_diamters_of_spuares_gamma_to_zero} in place of  Lemmas~\ref{lem:circle_averages_uniform_bound} and ~\ref{lem:upper_bound_on_internal_diamters_of_spuares}.
\end{proof}

Next we prove Lemma~\ref{lem:upper_bound_internal_metric_on_points_gamma_to_zero}.  The latter is a consequence of the following lemma.

\begin{lemma}
\label{lem:diameter_upper_bound_internal_metric_on_balls_gamma_to_zero}
Fix $p \in (0,c_1 \sqrt{M}),  s \in (0,1)$ and let $K \subseteq \C$ be a compact set.  Then there exist constants $\gamma_0 \in (0,2)$ depending only on $p,M$ and $s$ and $C \in (0,\infty)$ depending on $p,M,s$ and $K$ such that
\begin{align*}
\p\left[\sup_{u,v \in B_{\epsilon \mathfrak{r}}(z)} \wt{D}_h^{\gamma}(u,v ; B_{2\epsilon \mathfrak{r}}(z)) \leq \mathfrak{r}^{\xi(\gamma) Q(\gamma)} \epsilon^s e^{\xi(\gamma) h_{\mathfrak{r}}(0)}\right] \leq C \epsilon^{p(1-s)}
\end{align*}

for all $\epsilon \in (0,1),\gamma \in (0,\gamma_0)$ and all $z \in K$.
\end{lemma}

\begin{proof}[Proof of Lemma~\ref{lem:diameter_upper_bound_internal_metric_on_balls_gamma_to_zero}]
It follows from the same argument used to prove Lemma~\ref{lem:lqg_diameter_upper_bound_internal_metric_on_balls} except for using Lemma~\ref{lem:moment_lqg_diameter_bound_for_compact_sets_gamma_to_zero} in place of Proposition~\ref{prop:moment_lqg_diameter_bound_for_compact_sets}.
\end{proof}

\begin{proof}[Proof of Lemma~\ref{lem:upper_bound_internal_metric_on_points_gamma_to_zero}]
It follows from the argument used to prove Lemma~\ref{lem:lqg_diameter_upper_bound_internal_metric_on_points} with Lemma~\ref{lem:diameter_upper_bound_internal_metric_on_balls_gamma_to_zero} in place of Lemma~\ref{lem:lqg_diameter_upper_bound_internal_metric_on_balls}.
\end{proof}

\section{Identifying the limit: $\gamma \to 0$ case}
\label{sec:limit_euclidean_metric}

In this section,  we will complete the proof of Theorem~\ref{thm:main_theorem_gamma_to_zero_intro} by showing that any subsequential limit coming from Theorem~\ref{thm:tightness_gamma_to_zero} is a.s.  equal to the Euclidean metric.  Thus combining with Lemma~\ref{lem:convergence_in_prob} we will complete the proof of Theorem~\ref{thm:main_theorem_gamma_to_zero_intro}.

Let us now describe the setup for the rest of the section.  Let $(\gamma_n)_{n \in \N}$ be a sequence in $(0,2)$ such that $\gamma_n \to 0$ as $n \to \infty$.  Then Theorem~\ref{thm:tightness_gamma_to_zero} combined with Skorokhod's representation theorem imply that we can find a subsequence $(\gamma_{k_n})$ and a coupling of random fields $(h^n)_{n \in \N},  h$ in the same probability space such that $h^n \to h$ as $n \to \infty$ on $H_{\text{loc}}^{-1}(\C)$ a.s.,  and there exists a random function $\wt{D}$ on the same probability space such that 
\begin{align*}
\wt{D}_{h^n}^{\gamma_{k_n}} \to \wt{D} \quad \text{as} \quad n \to \infty
\end{align*}
with respect to the local uniform topology on $\C \times \C$ a.s.  In order to make the notation easier,  we will assume that $\gamma_{k_n} = \gamma_n$ for all $n$ and we will denote the pair $(h^n ,  \wt{D}_{h^n}^{\gamma_n})$ by $(h,\wt{D}_h^{\gamma_n})$.

The main step in completing the proof of Theorem~\ref{thm:main_theorem_gamma_to_zero_intro} is the following proposition which states that $\wt{D}$ is deterministic a.s.

\begin{proposition}\label{prop:metric_deterministic}
$\wt{D}$ is a deterministic metric a.s.
\end{proposition}

Theorem~\ref{thm:main_theorem_gamma_to_zero_intro} will follow from combining Proposition~\ref{prop:metric_deterministic} with the scaling and translation properties of $\wt{D}$.

Let us now focus on proving Proposition~\ref{prop:metric_deterministic}.  Its proof has two main steps.  The first step is to prove that $\wt{D}$ is a.s.  determined by $h$ in the coupling $(h,\wt{D})$.  This is the content of the next proposition.  The second step is to apply a $0-1$ law type argument as explained in Subsection~\ref{subsec:outline}.

\begin{proposition}\label{prop:metric_determined_by_field_gamma_to_zero}
$\wt{D}$ is a.s.  determined by $h$ in the coupling $(h,\wt{D})$.
\end{proposition}

The proof of Proposition~\ref{prop:metric_determined_by_field_gamma_to_zero} essentially follows from the argument in \cite[Section~2.5]{Pfeffer_2024}.  We won't give detailed proofs of the relevant statements since the proofs are the same.  However we choose to include the statements from \cite[Section~2.5]{Pfeffer_2024} that we need in order to make the proof clear.  We start with the following lemma which  states some basic properties of $\wt{D}$ that we are going to use.  Its proof follows from the exact same arguments used in Section~\ref{sec:identifying_limit_subcritical} and so we won't give a detailed proof.

\begin{lemma}\label{lem:metric_properties_gamma_to_zero}
The following conditions hold
\begin{enumerate}
\item \label{it:length_and_complet_gamma_to_zero}
$\wt{D}$ is a length and complete metric a.s.
\item \label{it:weyl_scaling_gamma_to_zero}
It is a.s.  the case that for any bounded and continuous function $f : \C \to \R$,  we have that
\begin{align*}
e^{\xi(\gamma_n)f} \cdot \wt{D}_h^{\gamma_n} \to \wt{D} \quad \text{as} \quad n \to \infty
\end{align*}
locally uniformly.
\item \label{it:translation_and_scaling_gamma_to_zero}
For all $z \in \C ,  r >0$,  we have that the metrics $\wt{D},  r^{-1} \wt{D}(r \cdot ,  r \cdot)$ and $\wt{D}(\cdot + z ,  \cdot + z)$ have the same law.   Also $\wt{D}$ satisfies condition~\ref{it:tightness_across_and-around_annuli} with $c_r = r$ for all $>0$.
\item \label{it:local_metric_gamma_to_zero}
$\wt{D}$ is a local metric for $h$ in the coupling $(h,\wt{D})$ in the sense of Definition~\ref{def:local_metric}.
\end{enumerate}
\end{lemma}

\begin{proof}
\eqref{it:length_and_complet_gamma_to_zero} follows from the exact same argument used to prove Proposition~\ref{prop:complete_and_geodesic_metric} while \eqref{it:weyl_scaling_gamma_to_zero} follows from the same argument used to prove Proposition~\ref{prop:weyl_scaling} combined with the fact that $\xi(\gamma_n) \to 0$ as $n \to \infty$.  Moreover by arguing as in the proof of Proposition~\ref{prop:locality_property_subcritical},  we obtain that \eqref{it:translation_and_scaling_gamma_to_zero} holds and the same argument as in \cite[Lemma~2.13]{Pfeffer_2024} implies that \eqref{it:local_metric_gamma_to_zero} holds as well.  This completes the proof of the lemma.
\end{proof}

Next we state a bi-Lipschitz equivalence result about metrics coupled with $h$ on the same probability space.  See also \cite[Lemma~5.1]{gwynne2020local} and \cite[Lemma~2.28]{Pfeffer_2024}.

\begin{lemma}\label{lem:bi_lipschitz_equivalence_gamma_to_zero}
Let $U \subseteq \C$ be an open and connected set,  and let $(h,D)$ be a coupling of $h$ with a random lower semicontinuous metric $D$ on $U$.  Suppose that $D$  satisfies properties \eqref{it:length_and_complet_gamma_to_zero},  \eqref{it:translation_and_scaling_gamma_to_zero} and \eqref{it:local_metric_gamma_to_zero}.  Then there exists a deterministic constant $C \in (0,\infty)$ such that for any fixed open set $V \subseteq U$ and rational circles $O,O' \subseteq V$,  a.s.
\begin{align*}
C^{-1} \E[D(O,O' ; V) \giv h] \leq D(O,O' ; V) \leq C \E[D(O,O' ; V) \giv h]
\end{align*}
and for any fixed $q>0$,
\begin{align*}
\E[D(O,O' ; V)^q \giv h] < \infty.
\end{align*}
Moreover,  if the pairs $(h,D)$ and $(h,D')$ have the same law,  then a.s.  the metrics $D,D'$ are bi-Lipschitz equivalent with Lipschitz constant $C^2$.
\end{lemma}

\begin{proof}
It follows from the same argument used to prove \cite[Lemma~2.28]{Pfeffer_2024}.
\end{proof}

As in \cite[Section~2.5]{Pfeffer_2024},  we would like to express $\wt{D}$ as a function of a collection of random variables which are conditionally independent given $h$.  This is the content of the following lemma.

\begin{lemma}\label{lem:metric_locally_determined_gamma_to_zero}
Suppose that we have the setup of Lemma~\ref{lem:bi_lipschitz_equivalence_gamma_to_zero}.  Let $\theta$ be sampled uniformly from Lebesgue measure on $[0,1]^2$,  and let $\mathcal{C}_{\theta}^{\epsilon}$ be the set of open $\epsilon \times \epsilon$ squares which intersect $U$ with vertices in $\epsilon(\Z^2 + \theta)$.  Then we have that
\begin{enumerate}
\item For any fixed $\epsilon > 0$ and any path $P$ in $U$ with finite $D$-length chosen in a manner depending only on $h$ and $D$,  a.s.
\begin{align*}
\len(P ; D) = \sum_{S \in \mathcal{C}_{\theta}^{\epsilon}} \len(P \cap S ; D).
\end{align*}
\item The metric $D$ is a.s.  determined by $h,\theta$ and the set of internal metrics 
\begin{align*}
\{D(\cdot ,  \cdot ; S \cap U): S \in \mathcal{C}_{\theta}^{\epsilon}\}.
\end{align*}
\item The internal metrics $\{D(\cdot ,  \cdot ; S \cap U) : S \in \mathcal{C}_{\theta}^{\epsilon}\}$ are conditionally independent given $h$ and $\theta$.
\end{enumerate}
\end{lemma}

\begin{proof}
It follows from the same argument used to prove \cite[Lemma~2.29]{Pfeffer_2024} but with Lemma~\ref{lem:bi_lipschitz_equivalence_gamma_to_zero} in place of \cite[Lemma~2.28]{Pfeffer_2024}.
\end{proof}

As explained in \cite[Section~2.5]{Pfeffer_2024},  Lemmas~\ref{lem:bi_lipschitz_equivalence_gamma_to_zero} and ~\ref{lem:metric_locally_determined_gamma_to_zero} allow us to apply the Efron-Stein inequality and write
\begin{align*}
\text{Var}[D(O,O') \giv h,\theta] \leq \frac{1}{2} \sum_{S \in \mathcal{C}_{\theta}^{\epsilon}} \E[(D^S(O,O') - D(O,O'))^2 \giv h,\theta]
\end{align*}
in the setting of Lemma ~\ref{lem:metric_locally_determined_gamma_to_zero},  where $D^S$ denotes the metric obtained by $D$ by resampling $D(\cdot ,  \cdot ; S \cap U)$ from its conditional law given $(h,\theta)$.

\begin{proof}[Proof of Proposition~\ref{prop:metric_determined_by_field_gamma_to_zero}]
It follows from the argument in the proof of \cite[Proposition~2.5]{Pfeffer_2024} except that we use Lemma~\ref{lem:metric_locally_determined_gamma_to_zero} in place of \cite[Lemma~2.29]{Pfeffer_2024}.
\end{proof}

\begin{proof}[Proof of Proposition~\ref{prop:metric_deterministic}]
Suppose that we have the same setup as in Proposition~\ref{prop:metric_determined_by_field_gamma_to_zero}.  Then combining Proposition~\ref{prop:metric_determined_by_field_gamma_to_zero} with Lemma~\ref{lem:convergence_in_prob},  we obtain that $\wt{D}_h^{\gamma_n} \to \wt{D}$ as $n \to \infty$ in probability,  where the metrics $\wt{D}_h^{\gamma_n}$ are sampled using the same field $h$.  Hence there exists a deterministic subsequence $(\gamma_{k_n})_{n \in \N}$ such that
\begin{align}\label{eqn:almost_sure_convergence_gamma_to_zero}
\wt{D}_h^{\gamma_{k_n}} \to \wt{D} \quad \text{as} \quad n \to \infty \quad \text{a.s.}
\end{align}
with respect to the local uniform topology on $\C \times \C$.

Recall that $h$ can be sampled as follows.  Let $(f_n)_{n \in \N}$ be an orthonormal basis of $H_0(\C)$ with respect to the Dirichlet inner product $(\cdot ,  \cdot)_{\nabla}$ and let $\rho_{0,1}$ denote the uniform probability measure on $\partial \D$.  Then the field $h$ can be sampled as the almost sure limit in $H_{\text{loc}}^{-1}(\C)$ as $n \to \infty$ of 
\begin{align*}
h_n = \sum_{m=1}^n a_m (f_m - (f_m ,  \rho_{0,1})),
\end{align*}
where $(a_m)_{m \in \N}$ is a sequence of i.i.d.  random variables whose law is given by $\mathcal{N}(0,1)$.

Let $\mathcal{F}_n$ denote the $\sigma$-algebra generated by $\{a_m : m \geq n\}$ and set $\mathcal{F} = \cap_{n \geq 1} \mathcal{F}_n$.  Note that Kolmogorov's $0-1$ law implies that $\mathcal{F}$ is trivial.  Fix $m \in \N$.  Then \eqref{eqn:almost_sure_convergence_gamma_to_zero} combined with \eqref{it:weyl_scaling_gamma_to_zero} in Lemma~\ref{lem:metric_properties_gamma_to_zero} imply that
\begin{align*}
\wt{D}_{h-h_m}^{\gamma_{k_n}} \to \wt{D} \quad \text{as} \quad n \to \infty \quad \text{a.s.}
\end{align*}
with respect to the local uniform topology.  Note that $h-h_m$ is $\mathcal{F}_m$-measurable and so $\wt{D}_{h-h_m}^{\gamma_{k_n}}$ is $\mathcal{F}_m$-measurable for all $n \in \N$ by Axiom~\ref{it:locality}.  Therefore we obtain that $\wt{D}$ is $\mathcal{F}_m$-measurable for all $m \in \N$.  In particular,  the metric $\wt{D}$ is $\mathcal{F}$-measurable and so $\wt{D}$ is deterministic a.s.  since $\mathcal{F}$ is trivial.  This completes the proof of the proposition.
\end{proof}

\begin{proof}[Proof of Theorem~\ref{thm:main_theorem_gamma_to_zero_intro}]
First we note that Proposition~\ref{prop:metric_deterministic} implies that there exists a deterministic constant $d \geq 0$ such that $\wt{D}(0,1) = d$ a.s.  Fix $z , w \in \C$ distinct and deterministic points.  Similarly we have that there exists a deterministic constant $\wt{d} \geq 0$ such that $\wt{D}(z,w) = \wt{d}$ a.s.  Moreover condition~\eqref{it:translation_and_scaling_gamma_to_zero} in Lemma~\ref{lem:metric_properties_gamma_to_zero} implies that the random variables $\wt{D}(z,w)$ and $|z-w| \wt{D}(0,1)$ have the same law and hence we have that $\wt{d} = |z-w| d$.  It follows that 
\begin{align*}
\wt{D}(z,w) = d |z-w| \quad \text{for all} \quad z ,w  \in \Q^2 \quad \text{a.s.}
\end{align*}
and so by continuity,  we obtain that 
\begin{align*}
\wt{D}(z,w) = d |z-w| \quad \text{for all} \quad z ,w  \in \C \quad \text{a.s.}
\end{align*}
Therefore it suffices to show that $d = 1$.  But this follows since
\begin{align*}
\p\left[\wt{D}_h^{\gamma_n}(0,1) \leq 1 \right] = \frac{1}{2} \quad \text{for all} \quad n \in \N,
\end{align*}
and 
\begin{align*}
\p\left[|\wt{D}_h^{\gamma_n}(0,1) - d| > \epsilon \right] \to 0 \quad \text{as} \quad n \to \infty, \quad \text{for all} \quad \epsilon>0.
\end{align*}
This completes the proof of the theorem.
\end{proof}

\section{Tightness: General case}
\label{sec:tightness_general_case}

The main goal of this section is to prove that the family of metrics $\wt{D}_h^{\xi}$ in Theorem~\ref{thm:main_theorem_general_case_intro} is tight with respect to the topology of lower semicontinuous functions on $\C \times \C$ whenever $\xi$ lies on compact subsets of $(0,\infty)$.  As explained in Subsection~\ref{subsec:outline},  we will first show a probability estimate on the $\wt{D}_h^{\xi}$-distances between any two fixed compact sets which is uniform in $\xi$ whenever $\xi$ lies on compact subsets of $(0,\infty)$ (see Propositions~\ref{prop:bound_on_lqg_distances_of_sets_general} and ~\ref{prop:lqg_distances_general_case}) and it is analogous to Proposition~\ref{prop:bound_on_lqg_distances_of_sets}.  Next we will combine Propositions~\ref{prop:bound_on_lqg_distances_of_sets_general} and ~\ref{prop:lqg_distances_general_case} with the arguments in \cite[Section~5.2]{ding2020tightness} to construct subsequential limits.

Let us now define the function $\mathfrak{p}_0$ in the statement of Theorem~\ref{thm:main_theorem_general_case_intro}.  For the rest of the section,  we let $h$ be a whole-plane GFF normalized such that $h_1(0) = 0$ and let $D_h^{\xi}$ be the limiting metric in Theorem~\ref{thm:convergence_of_supercritical_lfpp}.  Set $\zeta(\xi) = \frac{Q(\xi)}{4}$ for all $\xi>0$.  We also let $b : (0,\infty) \to (0,\infty)$ be a continuous function and let $M : (0,\infty) \to (0,\infty)$ be a continuous and non-decreasing function such that
\begin{align}\label{eqn:condition_for_b}
b(\xi) < \frac{1}{2\sqrt{2} \xi}
\end{align}
and
\begin{align}\label{eqn:condition_for_M}
\frac{4 - \xi Q(\xi) + \xi 2 \sqrt{2} \sqrt{4 + M(\xi)}}{\sqrt{M(\xi)}} < \frac{1}{b(\xi)},\quad \xi \zeta(\xi) b(\xi) \sqrt{M(\xi)} > 1,\quad \text{and} \quad b(\xi) \sqrt{M(\xi)} > 1
\end{align}
for all $\xi > 0$.

For all $\xi>0$,  we let $\mathfrak{p}_0(\xi) \in (0,1)$ be the minimum number in $(0,1)$ such that the statement of Lemma~\ref{lem:good_event_occurs_almost_everywhere} holds with $M = M(\xi)$.  We then let $\alpha(\xi) \in (0,\infty)$ be such that
\begin{align}\label{eqn:main_normalization_general_case}
\p\left[D_h^{\xi}(\text{around} \,\,  \mathbb{A}_{1,2}(0)) \leq \alpha(\xi) \right] = \mathfrak{p}_0(\xi).
\end{align}
Note that Lemma~\ref{lem:normalization_well_defined} implies that $\alpha(\xi)$ is well-defined for all $\xi>0$.  We then set 
\begin{align}\label{eqn:normalization_general_case}
\wt{D}_h^{\xi}:=\alpha(\xi)^{-1} D_h^{\xi}.
\end{align}

We start with the following analogue of Proposition~\ref{prop:bound_on_lqg_distances_of_sets}.

\begin{proposition}\label{prop:bound_on_lqg_distances_of_sets_general}
Let $U \subseteq \C$ be an open and connected set and let $K_1,K_2 \subseteq U$ be connected,  disjoint compact sets which are not singletons.  Then for all $\mathfrak{r}>0$,  it holds with probability at least $1-O_A(A^{-b(\xi) \sqrt{M(\xi)}})$ as $A \to \infty$,  at a rate which is uniform in the choice of $\mathfrak{r}$ and $\xi$ and depends only on $K_1,K_2$ and $U$,  that
\begin{align*}
\wt{D}_h^{\xi}(\mathfrak{r} K_1 ,  \mathfrak{r} K_2 ; \mathfrak{r} U) \leq A \mathfrak{r}^{\xi Q(\xi)} e^{\xi h_{\mathfrak{r}}(0)}.
\end{align*}
\end{proposition}

\begin{proof}
It follows from the exact same argument used to prove Proposition~\ref{prop:bound_on_lqg_distances_of_sets}.
\end{proof}

Next we prove a version of Proposition~\ref{prop:bound_on_lqg_distances_of_sets_general} which does not require the sets $K_1,K_2$ not to be singletons.

\begin{proposition}\label{prop:lqg_distances_general_case}
Fix $\xi^* \in (0,\infty)$ and let $U \subseteq \C$ be an open set,  and let $K_1,K_2 \subseteq U$ be connected,  disjoint compact sets that are allowed to be singletons.  Then for all $\epsilon \in (0,1)$,  there exists a constant $C \in (0,\infty)$ depending only on $\epsilon,\xi^* ,  K_1 ,  K_2$ and $U$,  such that
\begin{align*}
\p\left[\wt{D}_h^{\xi}(\mathfrak{r} K_1 ,  \mathfrak{r} K_2 ; \mathfrak{r} U) \leq C \mathfrak{r}^{\xi Q(\xi)} e^{\xi h_{\mathfrak{r}}(0)}\right] \geq 1 - \epsilon \quad \text{for all} \quad \mathfrak{r}>0,  \xi \in (0,\xi^*].
\end{align*}
\end{proposition}

\begin{proof}
Fix $\xi^* > 0,  \mathfrak{r} > 0$ and $\xi \in (0,\xi^*]$.   Then combining Proposition~\ref{prop:bound_on_lqg_distances_of_sets_general} with Axioms~\ref{it:weyl_scaling} and ~\ref{it:translation_and_scale_invariance},  we obtain that for all $z \in K_1$,  we have off an event with probability at most $O_n(2^{-\xi \zeta(\xi) b(\xi) \sqrt{M(\xi)}n})$ as $n \to \infty$ (at a universal rate) that
\begin{align}\label{eqn:non_centered_upper_bounds_on_annuli_distances}
&(2^{-n}\mathfrak{r})^{-\xi Q(\xi)} e^{-\xi h_{2^{-n}\mathfrak{r}}(\mathfrak{r}z)} \wt{D}_h^{\xi}(\text{across} \,\,  \mathbb{A}_{2^{-n} \mathfrak{r} ,  2^{-n+2} \mathfrak{r}}( \mathfrak{r}z)) \leq 2^{\xi \zeta(\xi) n},\notag\\
&(2^{-n}\mathfrak{r})^{-\xi Q(\xi)} e^{-\xi h_{2^{-n}\mathfrak{r}}(\mathfrak{r}z)} \wt{D}_h^{\xi}(\text{around} \,\,  \mathbb{A}_{2^{-n} \mathfrak{r} ,  2^{-n+1} \mathfrak{r}}(\mathfrak{r}z)) \leq 2^{\xi \zeta(\xi) n}.
\end{align}
Moreover Lemma~\ref{lem:upper_bound_on_circle_averages} implies that
\begin{align}\label{eqn:recentering_prob_bound}
\p\left[|h_{2^{-n} \mathfrak{r}}(\mathfrak{r} z) - h_{\mathfrak{r}}(0)| \leq \zeta(\xi) n \log(2)\right] \geq 1 - O_n(2^{-\zeta(\xi)^2 n / 8}) \quad \text{as} \quad n \to \infty
\end{align}
at a rate which depends only on $K_1$ and $\xi^*$.  Hence combining \eqref{eqn:non_centered_upper_bounds_on_annuli_distances} with \eqref{eqn:recentering_prob_bound},  we obtain that there exist constants $\wt{C} \in (0,\infty),  n_0 \in \N$ depending only on $K_1$ and $\xi^*$ such that for all $z \in K_1$ and all $n \geq n_0$,  we have with probability at least 
\begin{align*}
1 - \wt{C} 2^{-\xi \zeta(\xi) b(\xi) \sqrt{M(\xi)} n} - \wt{C} 2^{-\zeta(\xi)^2 n /8}
\end{align*}
that
\begin{align}\label{eqn:centered_prob_bound_fixed_scale}
&\mathfrak{r}^{-\xi Q(\xi)} e^{-\xi h_{\mathfrak{r}}(0)} \wt{D}_h^{\xi}(\text{across} \,\,  \mathbb{A}_{2^{-n} \mathfrak{r} ,  2^{-n+2} \mathfrak{r}}( \mathfrak{r}z)) \leq 2^{-\xi (Q(\xi) - 2 \zeta(\xi))n},\notag\\
&\mathfrak{r}^{-\xi Q(\xi)} e^{-\xi h_{\mathfrak{r}}(0)} \wt{D}_h^{\xi}(\text{around} \,\,  \mathbb{A}_{2^{-n} \mathfrak{r} ,  2^{-n+1} \mathfrak{r}}(\mathfrak{r}z)) \leq 2^{-\xi (Q(\xi) - 2\zeta(\xi)) n}.
\end{align}
Note that for all $N \in \N$,  we have that
\begin{align*}
1 - \wt{C} \sum_{n \geq N} 2^{-\zeta(\xi)^2 n / 8} - \wt{C} \sum_{n \geq N} 2^{-\xi \zeta(\xi) b(\xi) \sqrt{M(\xi)}n} =1 - \frac{\wt{C} 2^{-\zeta(\xi)^2 N / 8}}{1-2^{-\zeta(\xi)^2 / 8}} -\frac{\wt{C} 2^{-\xi b(\xi) \zeta(\xi) \sqrt{M(\xi)}N}}{1 - 2^{-\xi \zeta(\xi) b(\xi) \sqrt{M(\xi)}}}.
\end{align*}
Recall that $Q(\xi) \to \infty$ as $\xi \to 0$.  Hence since $\xi \zeta(\xi) b(\xi) \sqrt{M(\xi)} > 1$ by \eqref{eqn:condition_for_M},  we obtain by combining with \eqref{eqn:centered_prob_bound_fixed_scale} and summing over all $n \in \N ,  n \geq N$ that there exist constants $a,A>0$ depending only on $\xi^*$ and $K_1$ such that the following holds for all $\xi \in (0,\xi^*],  \mathfrak{r}>0,N \geq n_0$ and all $z \in K_1$.  Off an event with probability at most $A 2^{-aN}$,  we have that
\begin{align}\label{eqn:centered_prob_bound_fixed_scale_uniform}
&\mathfrak{r}^{-\xi Q(\xi)} e^{-\xi h_{\mathfrak{r}}(0)} \wt{D}_h^{\xi}(\text{across} \,\,  \mathbb{A}_{2^{-n} \mathfrak{r} ,  2^{-n+2} \mathfrak{r}}( \mathfrak{r}z)) \leq 2^{-\xi (Q(\xi) - 2 \zeta(\xi))n},\notag\\
&\mathfrak{r}^{-\xi Q(\xi)} e^{-\xi h_{\mathfrak{r}}(0)} \wt{D}_h^{\xi}(\text{around} \,\,  \mathbb{A}_{2^{-n} \mathfrak{r} ,  2^{-n+1} \mathfrak{r}}(\mathfrak{r}z)) \leq 2^{-\xi (Q(\xi) - 2\zeta(\xi)) n}.
\end{align}
for all $n \geq N$.  Henceforth we fix $z \in K_1$ and we assume that the above event occurs.

Note that every path across the annulus $\mathbb{A}_{2^{-n} \mathfrak{r} ,  2^{-n+2} \mathfrak{r}}(\mathfrak{r} z)$ intersects every path around $\mathbb{A}_{2^{-n+1} \mathfrak{r} ,  2^{-n+2} \mathfrak{r}}(\mathfrak{r} z)$ and every path around $\mathbb{A}_{2^{-n} \mathfrak{r} ,  2^{-n+1} \mathfrak{r}}(\mathfrak{r} z)$.  Thus by concatenating paths across the annuli 
$\mathbb{A}_{2^{-n} \mathfrak{r} ,  2^{-n+2} \mathfrak{r}}(\mathfrak{r} z)$ and paths around the annuli $\mathbb{A}_{2^{-n} \mathfrak{r} ,  2^{-n+1} \mathfrak{r}}(\mathfrak{r} z)$,  we obtain a path from $\partial B_{2^{-N} \mathfrak{r}}(\mathfrak{r} z)$ to any arbitrarily small Euclidean neighborhood of $\mathfrak{r} z$.  Also \cite[Lemma~3.2]{Pfeffer_2024} implies that it is a.s.  the case that
\begin{align*}
\text{diam}(X ; \wt{D}_h^{\xi}) = \text{diam}(\overline{X} ; \wt{D}_h^{\xi}) \quad \text{for every Borel set} \quad X \subseteq \C.
\end{align*}
Therefore we obtain that
\begin{align*}
&\mathfrak{r}^{-\xi Q(\xi)} e^{-\xi h_{\mathfrak{r}}(0)} \wt{D}_h^{\xi}(\mathfrak{r} z ,  \partial B_{2^{-N} \mathfrak{r}}(\mathfrak{r} z)) \\
&\leq \mathfrak{r}^{-\xi Q(\xi)} e^{-\xi h_{\mathfrak{r}}(0)} \left(\sum_{n \geq N} \left(\wt{D}_h^{\xi}(\text{around} \,\,  \mathbb{A}_{2^{-n} \mathfrak{r} ,  2^{-n+1} \mathfrak{r}}(\mathfrak{r} z)) + \wt{D}_h^{\xi}(\text{across} \,\,  \mathbb{A}_{2^{-n} \mathfrak{r} , 2^{-n+2} \mathfrak{r}}(\mathfrak{r} z)) \right)\right)\\
&\leq 2 \sum_{n \geq N} 2^{-\xi (Q(\xi) - 2 \zeta(\xi))n} =2 \frac{2^{- \xi Q(\xi) N / 2}}{1 - 2^{-\xi Q(\xi) / 2}}.
\end{align*}
Note that \cite[Theorem~1.2]{ding2020fractal} implies that $d_{\gamma} =2 + o_{\gamma}(1)$ as $\gamma \to 0$ and so $\xi Q(\xi) \to 1$ as $\xi \to 0$ since $Q(\xi(\gamma)) = \frac{2}{\gamma} + \frac{\gamma}{2}$ for all $\gamma \in (0,2)$.  Therefore combining with\eqref{eqn:centered_prob_bound_fixed_scale_uniform},  we obtain that possibly by taking $a>0$ to be smaller and $A$ to be larger (depending only on $\xi^*$ and $K_1$),  we have off an event with probability at most $A 2^{-a N}$ that
\begin{align}\label{eqn:pointwise_distance_in_K_1}
\mathfrak{r}^{-\xi Q(\xi)} e^{-\xi h_{\mathfrak{r}}(0)} \wt{D}_h^{\xi}(\mathfrak{r} z ,  \partial B_{2^{-N} \mathfrak{r}}(\mathfrak{r} z)) \leq A 2^{-aN}.
\end{align}

Next we fix $w \in K_2$.  Then arguing as in \eqref{eqn:pointwise_distance_in_K_1},  we obtain that possibly by taking $A$ to be larger and $0<a<1$ to be smaller (depending only on $\xi^*,  K_1$ and $K_2$),  we have that
\begin{align}\label{eqn:pointwise_distance_in_K_2}
\mathfrak{r}^{-\xi Q(\xi)} e^{-\xi h_{\mathfrak{r}}(0)} \wt{D}_h^{\xi}(\mathfrak{r} w ,  \partial B_{2^{-N} \mathfrak{r}}(\mathfrak{r} w)) \leq A 2^{-a N}
\end{align}
off an event with probability at most $A 2^{-aN}$.  Moreover since $b(\xi) \sqrt{M(\xi)} > 1$,  Proposition~\ref{prop:bound_on_lqg_distances_of_sets_general} implies that there exists $B \in (0,\infty)$ sufficiently large (depending only on $z,w,K_1,K_2,\xi^*,U$ and $N$) such that off an event with probability at most $A 2^{-aN}$,  we have that
\begin{align}\label{eqn:lqg_distance_between_circles}
\mathfrak{r}^{-\xi Q(\xi)} e^{-\xi h_{\mathfrak{r}}(0)} \wt{D}_h^{\xi}(\partial B_{2^{-N} \mathfrak{r}}(\mathfrak{r} z) ,  \partial B_{2^{-N} \mathfrak{r}}(\mathfrak{r} w) ; \mathfrak{r} U) \leq B.
\end{align}

Therefore combining \eqref{eqn:pointwise_distance_in_K_1},  \eqref{eqn:pointwise_distance_in_K_2} and \eqref{eqn:lqg_distance_between_circles},  we obtain that off an event with probability at most $3A 2^{-aN}$,  we have that
\begin{align*}
\mathfrak{r}^{-\xi Q(\xi)} e^{-\xi h_{\mathfrak{r}}(0)} \wt{D}_h^{\xi}(\mathfrak{r} K_1 ,  \mathfrak{r} K_2 ;  \mathfrak{r} U) \leq B + 2A 2^{-aN}.
\end{align*}
Hence the proof of the proposition is complete by taking $N \in \N$ sufficiently large (depending only on $\xi^*,  K_1,K_2$ and $U$) such that $1-3A2^{-aN} > 1-\epsilon$. 
\end{proof}

Now we will use Propositions~\ref{prop:bound_on_lqg_distances_of_sets_general} and ~\ref{prop:lqg_distances_general_case} to construct subsequential limits.
Fix $\xi>0$ and let $(\xi_n)_{n \in \N}$ be a sequence in $(0,\infty)$ such that $\xi_n \to \xi$ as $n \to \infty$.  We will use the same construction as in \cite[Section~5.2]{ding2020tightness} in order to construct a subsequential limit in law of the sequence $(\wt{D}_h^{\xi_n})_{n \in \N}$ with respect to the lower semicontinuous topology of functions on $\C \times \C$.  First we note that by countably many applications of Proposition~\ref{prop:lqg_distances_general_case},  we can find a subsequence $(\xi_{k_n})$ of $(\xi_n)$ and a coupling of the GFF $h$ with random variables $D(O_1,O_2)$ for rational circles $O_1,O_2$,  random variables $D(\text{around} \,\,  A)$ for rational annuli $A$ for which the following is true.  Let $h^n$ be a random variable with the same law as $h$.  Then we have the following joint convergence in law as $n \to \infty$:
\begin{align*}
&\wt{D}_{h^n}^{\xi_{k_n}}(O_1,O_2) \to D(O_1,O_2) \quad \text{for all} \quad \text{rational circles} \quad O_1,O_2,\\
&\wt{D}_h^{\xi_{k_n}}(\text{around} \,\,A) \to D(\text{around} \,\,A) \quad \text{for all rational annuli} \quad A,
\end{align*}
and the functions $(z,r) \to h^n_r(z)$ defined on $\C \times (0,\infty)$ converge to $(z,r) \to h_r(z)$ with respect to the local uniform topology.  By Skorokhod's representation theorem,  we can find a coupling of $\{(h_n ,  \wt{D}_{h^n}^{\xi_{k_n}})\}_{n \in \N}$ with the field $h$ and the random variables $D(O_1,O_2)$ and $D(\text{around} \,\,A)$ such that the above convergences occur a.s.

Now we define the function $\wt{D} : \C \times \C \to \R_+ \cup \{\pm \infty\}$ by 
\begin{align}\label{eqn:definition_of_supercritical_metric}
\lim_{O_z \downarrow z ,  O_w \downarrow w} D(O_z,O_w)
\end{align}
if $z \neq w$ and $0$ if $z = w$,  where the limit is over any sequence of rational circles with positive radii $O_z$ surrounding $z$ and $O_w$ surrounding $w$ whose radii shrink to zero.

\begin{proposition}\label{prop:lower_semicontinuity}
In the above coupling,  we have that
\begin{align*}
\wt{D}_{h^n}^{\xi_{k_n}} \to \wt{D} \quad \text{as} \quad n \to \infty
\end{align*}
with respect to the topology on lower semicontinuous functions.  
\end{proposition}

\begin{proof}
It follows from the exact same argument used to prove \cite[Proposition~5.6]{ding2020tightness}.
\end{proof}

Finally we mention the following consequence of Propositions~\ref{prop:lqg_distances_general_case} and ~\ref{prop:lower_semicontinuity}.

\begin{lemma}\label{lem:metric_finite_at_fixed_points}
For all $z ,  w \in \C$ fixed,  we have that $\wt{D}(z,w) < \infty$ a.s.
\end{lemma}

\begin{proof}
It follows from combining Propositions~\ref{prop:lqg_distances_general_case} and ~\ref{prop:lower_semicontinuity}.  
\end{proof}

\begin{rem}\label{rem:tightness_of_internal_metrics_general_case}
We note that Proposition~\ref{prop:lqg_distances_general_case} implies that for every open and connected set $U \subseteq \C$,  rational circles $O_1,O_2 \subseteq U$ and rational annuli $A \subseteq U$,  we have that the collections of random variables $\{\wt{D}_{h^n}^{\xi_n}(O_1,O_2 ; U)\}_{n \in \N}$ and $\{\wt{D}_{h^n}^{\xi_n}(\text{around} \,\,  A)\}_{n \in \N}$ are tight.  Therefore by Skorokhod's representation theorem we can find a subsequence $(\xi_{k_n})_{n \in \N}$ and a coupling on the same probability space of $\{(h^n ,  \wt{D}_{h^n}^{\xi_{k_n}}(\cdot ,  \cdot ; U))\}_{n \in \N}$ with a random filed $h$ and random variables $D(O_1 ,  O_2 ; U)$ and $D(\text{around} \,\,  A)$,  where $O_1,O_2$ and $A$ are as above,  such that a.s.  as $n \to \infty$ \begin{align*}
&\wt{D}_{h^n}^{\xi_{k_n}}(O_1,O_2 ; U) \to D(O_1,O_2 ; U) \quad \text{for all} \quad \text{rational circles} \quad O_1,O_2 \subseteq U,\\
&\wt{D}_{h^n}^{\xi_{k_n}}(\text{around} \,\,A) \to D(\text{around} \,\,A) \quad \text{for all rational annuli} \quad A \subseteq U,
\end{align*}
and the functions $(z,r) \to h^n_r(z)$ converge to $(z,r) \to h_r(z)$ with respect to the local uniform topology.  Hence we can define a limiting lower semicontinuous function $\wt{D}(\cdot ,  \cdot ; U)$ as in \eqref{eqn:definition_of_supercritical_metric} with $D(O_z ,  O_w ; U)$ in place of $D(O_z , O_w)$.  Then the argument in the proof of Proposition~\ref{prop:lower_semicontinuity} implies that
\begin{align*}
\wt{D}_{h^n}^{\xi_{k_n}}(\cdot ,  \cdot ;U) \to \wt{D}(\cdot ,  \cdot ; U) \quad \text{as} \quad n \to \infty
\end{align*}
with respect to the lower semicontinuous topology of functions on $U$.  In particular the collection of random variables $\{\wt{D}_{h^n}^{\xi_n}(\cdot ,  \cdot ;U)\}_{n \in \N}$ is tight with respect to the lower semicontinuous topology of functions on $U$. 
\end{rem}

\section{Identifying the limit: General case}
\label{lem:identifying_limit_general_case}

In this section,  we are going to complete the proof of Theorem~\ref{thm:main_theorem_general_case_intro}.  In particular we will show that the limit $\wt{D}$ in Proposition~\ref{prop:lower_semicontinuity} is non-trivial a.s.  Moreover we will show that if we make the further assumption that $\wt{D}$ satisfies the triangle inequality a.s.,  then $\wt{D}$ is a metric a.s.  This will be done in Subsection~\ref{subsec:limit_is_a_metric_general_case}.  Next (assuming that $\wt{D}$ satisfies the triangle inequality a.s.) we will show in Subsections~\ref{subsec:limit_complete_and_geodesic_metric_general_case} - \ref{subsec:locality_general_case} that $\wt{D}$ satisfies the axioms of an  LQG metric with parameter $\xi$ (Definition~\ref{def:strong_lqg_metric_general}).  Finally we will complete the proof of Theorem~\ref{thm:main_theorem_general_case_intro} in Subsection~\ref{subsec:limit_xi_lqg_metric_general_case} by combining with Theorem~\ref{thm:uniqueness_supercritical}.

For the rest of the section,  we fix $\xi>0$ and suppose that we have the same setup as in Section~\ref{sec:tightness_general_case} and let $\wt{D}$ be the limiting metric in Proposition~\ref{prop:lower_semicontinuity}. 

\subsection{The limit is a metric: General case}
\label{subsec:limit_is_a_metric_general_case}

The main purpose of this subsection is to prove the following. 

\begin{proposition}\label{prop:limit_is_a_metric_general_case}
Let $\wt{D}$ be the limit in the statement of Proposition~\ref{prop:lower_semicontinuity}.  Then for every distinct and deterministic points $z,w \in \C$,  we have that $\wt{D}(z,w) > 0$ a.s.  Moreover if in addition we assume that $\wt{D}$ satisfies the triangle inequality a.s.,  then we have that $\wt{D}$ is a metric a.s.
\end{proposition}

The proof of Proposition~\ref{prop:limit_is_a_metric_general_case} essential follows from the same argument as in the proof of Proposition~\ref{prop:limit_is_a_metric} but there are some subtleties because the metrics converge with respect to a weaker topology,  i.e.,  the topology on lower semicontinuous functions.  We are going to highlight the main differences and give the relevant explanations.

As in Subsection~\ref{subsec:limit_is_a_metric_subcritical},  the main component in the proof of Proposition~\ref{prop:limit_is_a_metric_general_case} is the following proposition.

\begin{proposition}\label{prop:annuli_crossing_always_positive_general_case}
Suppose that we have the same setup as in Proposition~\ref{prop:limit_is_a_metric_general_case} and assume that $\wt{D}$ satisfies the triangle inequality a.s.  Then we have that
\begin{align*}
\lim_{\epsilon \to 0} \liminf_{n \to \infty} \left(\p\left[\wt{D}_{h^n}^{\xi_{k_n}}(\text{across} \,\ \mathbb{A}_{r_1,r_2}(0)) > \epsilon \right] \right) = 1 \quad \text{for all} \quad 0<r_1 < r_2 < \infty.
\end{align*}
\end{proposition}

Proposition~\ref{prop:annuli_crossing_always_positive_general_case} will follow from Lemmas~\ref{lem:uniform_lower_bound_across_annuli_general_case} and ~\ref{lem:crossing_positive_general_case} which are the analogues of Lemmas~\ref{lem:uniform_lower_bound_across_annuli} and ~\ref{lem:crossing_positive} respectively.  In order to prove Lemma~\ref{lem:uniform_lower_bound_across_annuli_general_case},  we need to know that the $\wt{D}_h^{\xi_{k_n}}$-distance between any two fixed points is bounded away from zero with positive probability (uniformly in $n \in \N$).  This is the content of the following lemma and it is the analogue of Lemma~\ref{lem:uniform_lower_bound_between_points}.

\begin{lemma}\label{lem:uniform_lower_bound_between_points_general_case}
Suppose that we have the same setup as in the statement of Proposition~\ref{prop:lower_semicontinuity}.  Then there exists some universal constant $r_0 \in (0,1)$ such that for all $r \in (0,r_0)$,  there exists a constant $c>0$ depending only on $r$ and the sequence $(\xi_{k_n})$ such that
\begin{align*}
\liminf_{n \to \infty} \inf_{x \in \partial B_r(0),  y \in \partial B_{2r}(0)}\left( \p\left[\wt{D}_{h^n}^{\xi_{k_n}}(x,y ; B_{3r}(0)) > c\right]\right)>0.
\end{align*}
\end{lemma}

\begin{proof}
Let $0<\xi_* < \xi^* < \infty$ be such that $\xi_n \in [\xi_* ,  \xi^*]$ for all $n \in \N$.  Since the function $M : (0,\infty) \to (0,\infty)$ is a continuous function,  we obtain that there exists $\wt{\xi} \in [\xi_* ,  \xi^*]$ such that $M(\xi) \leq M(\wt{\xi})$ for all $\xi \in [\xi_* ,  \xi^*]$.  Thus the definition of $\mathfrak{p}_0(\xi)$ (see \eqref{eqn:main_normalization_general_case}) implies that $\mathfrak{p}_0(\xi) \leq \mathfrak{p}_0(\wt{\xi})$ for all $\xi \in [\xi_* ,  \xi^*]$.  Then we have that
\begin{align*}
\p\left[\wt{D}_{h^n}^{\xi_{k_n}}(\text{around} \,\,  \mathbb{A}_{1,2}(0)) > 1\right]\geq  1 - \mathfrak{p}_0(\wt{\xi})>0 \quad \text{for all} \quad n \in \N.
\end{align*}
Therefore,  since $\wt{D}_{h^n}^{\xi_{k_n}}$ satisfies the axioms of a strong LQG metric for all $n \in \N$,  we obtain the statement of the lemma by arguing as in the proof of Lemma~\ref{lem:uniform_lower_bound_between_points}.
\end{proof}

\begin{lemma}\label{lem:uniform_lower_bound_across_annuli_general_case}
Let $r_0 \in (0,1)$ be the constant of Lemma~\ref{lem:uniform_lower_bound_between_points_general_case}.  Then for all $r \in (0,r_0)$,  there exists a constant $c>0$ depending only on $r$ and the sequence $(\xi_n)$ such that
\begin{align*}
\liminf_{n \to \infty} \p\left[\wt{D}_{h^n}^{\xi_{k_n}}(\partial B_r(0) ,  \partial B_{2r}(0)) > c \right] > 0.
\end{align*}
\end{lemma}

\begin{proof}
Fix $r \in (0,r_0)$ and let $c'>0$ be the constant in Lemma~\ref{lem:uniform_lower_bound_between_points_general_case} corresponding to $r$.  Set 
\begin{align*}
A:=\inf_{n \in \N} \inf_{x \in \partial B_r(0),  y \in \partial B_{2r}(0)}\left( \p\left[\wt{D}_{h^n}^{\xi_{k_n}}(x,y ; B_{3r}(0)) > c' \right]\right)>0.
\end{align*}

Note that the proof of Proposition~\ref{prop:lqg_distances_general_case} and in particular \eqref{eqn:centered_prob_bound_fixed_scale_uniform} imply that there exist constants $a,A > 0$ depending only on $(\xi_{k_n})_{n \in \N}$ and $r$ such that for all $n \in \N,  z \in \partial B_r(0) \cup \partial B_{2r}(0)$ and all $N \in \N$,  it holds off an event with probability at most $A 2^{-aN}$ that
\begin{align*}
&\wt{D}_{h^n}^{\xi_{k_n}}(\text{across} \,\,  \mathbb{A}_{2^{-m},2^{-m+2}}(z)) \leq 2^{-\xi_{k_n} Q(\xi_{k_n}) m /2}\\
&\wt{D}_{h^n}^{\xi_{k_n}}(\text{around} \,\,  \mathbb{A}_{2^{-m},2^{-m+2}}(z)) \leq 2^{-\xi_{k_n} Q(\xi_{k_n}) m /2}
\end{align*}
for all $m \geq N$.  Since every path across $\mathbb{A}_{2^{-m} ,  2^{-m+2}}(z)$ intersects every path around $\mathbb{A}_{2^{-m+1},2^{-m+2}}(z)$ and every path around $\mathbb{A}_{2^{-m} ,  2^{-m+1}}(z)$,  combining with Lemma~\ref{lem:uniform_lower_bound_between_points_general_case},  we obtain that there exists $m_0 \in \N$ such that for all $m \in \N,  m \geq m_0$,  we have that
\begin{align*}
\p\left[\wt{D}_{h^n}^{\xi_{k_n}}(B_{r 2^{-m}}(x) ,  B_{r 2^{-m}}(y) ; B_{3r}(0)) > \frac{c'}{2} \right] > \frac{A}{2}
\end{align*} 
for all $n \in \N$ and all $x \in \partial B_r(0),  y \in \partial B_{2r}(0)$. Moreover setting $X:=\partial B_r(0) \cap (r 2^{-m}) \Z^2,  Y:=\partial B_{2r}(0) \cap (r 2^{-m}) \Z^2$ and arguing as in the proof of Lemma~\ref{lem:uniform_lower_bound_across_annuli},  we obtain that 
\begin{align*}
\p\left[\wt{D}_{h^n}^{\xi_{k_n}}(\partial B_r(0) ,  \partial B_{2r}(0)) > \frac{c'}{2}\right] \geq \p\left[\bigcap_{x \in X ,  y \in Y} \left\{\wt{D}_{h^n}^{\xi_{k_n}}(B_{r 2^{-m}}(x) ,  B_{r 2^{-m}}(y) ; B_{3r}(0)) > \frac{c'}{2} \right\}\right].
\end{align*}
Hence the proof of the lemma is complete by arguing as in the proof of Lemma~\ref{lem:uniform_lower_bound_across_annuli}.
\end{proof}

Before we state and prove Lemma~\ref{lem:crossing_positive_general_case},  we state and prove Lemmas~\ref{lem:metric_unbounded_general_case} and ~\ref{lem:connectivity_lemma_general_case} which we are going to use in the proof of Lemma~\ref{lem:crossing_positive_general_case} and they are the analogues of Lemmas~\ref{lem:metric_unbounded} and ~\ref{lem:connectivity_lemma} respectively.

\begin{lemma}\label{lem:metric_unbounded_general_case}
For all $r,T>0$,  it holds that
\begin{align*}
\lim_{R \to \infty} \liminf_{n \to \infty} \left(\p\left[ \wt{D}_{h^n
}^{\xi_{k_n}}(\text{across} \,\,\mathbb{A}_{r,R}(0)) > T \right] \right) = 1.
\end{align*}
\end{lemma}

\begin{proof}
It follows from the same argument used to prove Lemma~\ref{lem:metric_unbounded} except that we use Lemma~\ref{lem:uniform_lower_bound_across_annuli_general_case} in place of Lemma~\ref{lem:uniform_lower_bound_across_annuli}.
\end{proof}

\begin{lemma}\label{lem:connectivity_lemma_general_case}
Fix $r>0$.  Then the set $\{x \in \C : \wt{D}(x,\partial B_r(0)) = 0\}$ is a.s.  a closed,  bounded and connected set.
\end{lemma}

\begin{proof}
The proof is essentially the same as the proof of Lemma~\ref{lem:connectivity_lemma} where we use Lemma~\ref{lem:metric_unbounded_general_case} in place of Lemma~\ref{lem:metric_unbounded}.
\end{proof}

\begin{lemma}\label{lem:crossing_positive_general_case}
Suppose in addition that $\wt{D}$ satisfies the triangle inequality a.s.  Fix $r>0$ and set
\begin{align*}
\mathcal{Z} = \{\wt{D}(\text{across} \,\,  \mathbb{A}_{r,2r}(0)) > 0\}.
\end{align*}
Then we have that $\p[\mathcal{Z}] \in \{0,1\}$.
\end{lemma}

\begin{proof}
The claim of the lemma follows from the same argument used to prove Lemma~\ref{lem:crossing_positive}.  More precisely,  as in the proof of Lemma~\ref{lem:crossing_positive},  we set
\begin{align*}
&E_1 = \{x \in \C : \wt{D}(x ,  \partial B_r(0)) = 0\},\\
&E_2 = \{x \in \C : \wt{D}(x ,  \partial B_{2r}(0) = 0\}
\end{align*}
and note that
\begin{align*}
\mathcal{Z} = \{E_1 \cap \partial B_{2r}(0) = \emptyset\} = \{E_2 \cap \partial B_r(0) = \emptyset\}.
\end{align*}
However,  there is one subtle difference in our case.  In particular,  in the subcritical regime in the proof of Lemma~\ref{lem:crossing_positive},  we know that the limiting metric satisfies the triangle inequality a.s.  since the limits are considered with respect to the local uniform topology on $\C \times \C$.  This implied that $\mathcal{Z} = \{E_1 \cap E_2 = \emptyset\}$.  In our case,  we don't know a priori whether $\wt{D}$ satisfies the triangle inequality since this property is not always preserved under limits with respect to the lower semicontinuous topology of functions on $\C \times \C$.  Therefore the assumption in the statement of the lemma guarantees that $\mathcal{Z} = \{E_1 \cap E_2 = \emptyset\}$ a.s.  and so the rest of the argument is identical with the argument given in the proof of Lemma~\ref{lem:crossing_positive}. 
\end{proof}

\begin{proof}[Proof of Proposition~\ref{prop:annuli_crossing_always_positive_general_case}]
It follows from the argument used to prove Proposition~\ref{prop:annuli_crossing_always_positive} except that we use Lemmas~\ref{lem:uniform_lower_bound_across_annuli_general_case} and ~\ref{lem:crossing_positive_general_case} in place of Lemmas~\ref{lem:uniform_lower_bound_across_annuli} and ~\ref{lem:crossing_positive} respectively.
\end{proof}

Next we show that $\wt{D}$ is non-trivial a.s.  without making the assumption that it satisfies the triangle inequality a.s.

\begin{lemma}\label{lem:metric_non_trivial}
Fix $z,w \in \C$ distinct and deterministic points.  Then it is a.s.  the case that $\wt{D}(z,w) > 0$.
\end{lemma}

\begin{proof}
Fix $r \in (0,1)$ sufficiently small such that $z$ (resp.  $w$) lies in the bounded (resp.  unbounded) connected component of $\C \setminus \mathbb{A}_{\frac{r}{2}{r}}(0)$.  Lemma~\ref{lem:uniform_lower_bound_across_annuli_general_case} implies that there exist $0<s_1 < s_2 < 1$ sufficiently small and $c,p \in (0,1)$ such that
\begin{align}\label{eqn:lower_bound_on_crossing_prob}
\p\left[\wt{D}_{h^n}^{\xi_{k_n}}(\text{across} \,\,  \mathbb{A}_{s_1,s_2}(0)) > c \right] \geq p \quad \text{for all} \quad n \in \N.
\end{align}

Fix $s \in (0,s_1)$ and set $r_k:=r s^k$ for all $k \in \N_0$.  We consider the event
\begin{align*}
E_{r_k}^n:=\left\{\wt{D}_{h^n}^{\xi_{k_n}}(\text{across} \,\,  \mathbb{A}_{s_1 r_k ,  s_2 r_k}(z)) > c r_k^{\xi_{k_n} Q(\xi_{k_n})} e^{\xi_{k_n} h^n_{r_k}(z)}\right\}
\end{align*}
Then Axiom~\ref{it:translation_and_scaling_gamma_to_zero} combined with \eqref{eqn:lower_bound_on_crossing_prob} imply that
\begin{align*}
\p\left[E_{r_k}^n\right] \geq p \quad \text{for all} \quad n \in \N,  k \in \N_0.
\end{align*}
Moreover the event $E_{r_k}^n$ is a.s.  determined by $(h^n - h^n_{r_k})|_{\mathbb{A}_{s_1 r_k ,  s_2 r_k}(z)}$.  Let also $\mathcal{N}^k(K)$ denote the number of $k \in [1,K]_{\Z}$ for which $E_{r_k}^n$ occurs.  Then \cite[Lemma~3.1]{gwynne2020local} implies that there exist constants $a>0,  b \in (0,1)$ and $\wt{c}>0$ depending only on $p,s_1$ and $s_2$ such that
\begin{align*}
\p\left[\mathcal{N}^n(K) < b K\right] \leq \wt{c} e^{-a K} \quad \text{for all} \quad n ,  K \in \N.
\end{align*}
In particular we have that
\begin{align}\label{eqn:limsup_bound}
\p\left[\limsup_{n \to \infty} \{\mathcal{N}^n(K) \geq b K\}\right] \geq 1 - \wt{c} e^{-aK} \quad \text{for all} \quad K \in \N.
\end{align}

We will show that
\begin{align*}
\p\left[\wt{D}(z,w) > 0\right] \geq 1 - \wt{c} e^{-aK} \quad \text{for all}\quad K \in \N,
\end{align*}
and hence letting $K \to \infty$ will complete the proof of the lemma.  Indeed fix $K \in \N$ and suppose that the event $\{\limsup_{n \to \infty} \{\mathcal{N}^n(K) \geq b K\}$ occurs.  Then there exists a subsequence $(\xi_{\lambda_n})$ of $(\xi_{k_n})$ such that $\{\mathcal{N}^{\lambda_n}(K) \geq b K \}$ occurs for all $K \in \N$.  Let $(z_n),(w_n)$ be sequences such that $z_n \to z ,  w_n \to w$ and $\wt{D}_{h^{\lambda_n}}^{\xi_{\lambda_n}}(z_n , w_n) \to \wt{D}(z,w)$ as $n \to \infty$.  

Let $n_0 \in \N$ be such that $z_n \in B_{r_{K+1}}(z)$ and $w_n \notin B_r(z)$ for all $n \geq n_0$.  Fix $n \geq n_0$ and suppose that $E_{r_k}^{\lambda_n}$ occurs for some $k \in [1,K]_{\Z}$.  Then it holds that
\begin{align*}
\wt{D}_{h^{\lambda_n}}^{\xi_{\lambda_n}}(z_n , w_n) > c r_k^{\xi_{\lambda_n} Q(\xi_{\lambda_n})} e^{\xi_{\lambda_n} h^{\lambda_n}_{r_k}(z)} \geq c r_k^{\xi_{\lambda_n} Q(\xi_{\lambda_n})} \exp\left\{\xi_{\lambda_n} \min_{k \in [1,K]_{\Z}} h^{\lambda_n}_{r_k}(z) \right\}.
\end{align*}
Note that the right hand side of the above inequality converges to 
\begin{align*}
c r_k^{\xi Q(\xi)} \exp\left\{\xi \min_{k \in [1,K]_{\Z}} h_{r_k}(z) \right\}
\end{align*}
as $n \to \infty$.  This implies that
\begin{align*}
\wt{D}(z,w) \geq c r_k^{\xi Q(\xi)} \exp\left\{\xi \min_{k \in [1,K]_{\Z}} h_{r_k}(z) \right\}.
\end{align*}
Hence we complete the proof of the lemma by combining with \eqref{eqn:limsup_bound}.
\end{proof}

\begin{proof}[Proof of Proposition~\ref{prop:limit_is_a_metric_general_case}]
The first claim in the statement of the proposition follows from Lemma~\ref{lem:metric_non_trivial}.  The second claim follows by arguing as in the proof of Proposition~\ref{prop:limit_is_a_metric} and using Proposition~\ref{prop:annuli_crossing_always_positive_general_case} in place of Proposition~\ref{prop:annuli_crossing_always_positive}.
\end{proof}

\subsection{The limit is a complete and geodesic metric: General case}
\label{subsec:limit_complete_and_geodesic_metric_general_case}
For this subsection and Subsections~\ref{subsec:weyl_scaling_general_case} and ~\ref{subsec:locality_general_case},  we assume that $\wt{D}$ satisfies the triangle inequality a.s.  The main result of this subsection is the following.

\begin{proposition}\label{prop:complete_and_geodesic_metric_general_xi}
It is a.s.  the case that the restriction of $\wt{D}$ to $\C \setminus \{\text{singular points}\}$ is a finite,  complete and geodesic metric.
\end{proposition}

Before proving Proposition~\ref{prop:complete_and_geodesic_metric_general_xi},  we state and prove the following useful lemma about the singular points of $\wt{D}$.

\begin{lemma}\label{lem:singular_points}
The following is true a.s.  Suppose that $z,z'w,w' \in \C$ are such that $z \neq z',  w \neq w',  \wt{D}(z,z') < \infty$ and $\wt{D}(w,w') < \infty$.  Then $\wt{D}(z,w) < \infty$.  In particular,  if $z,w \in \C$ are such that $\wt{D}(z,w) = \infty$,  then either $z$ or $w$ is a singular point for $\wt{D}$.
\end{lemma}

\begin{proof}
It follows from the argument in the proof of \cite[Lemma~5.23]{ding2020tightness}.
\end{proof}

Note that Lemma~\ref{lem:singular_points} implies that $\wt{D}$ is a finite metric on $\C \setminus \{\text{singular points}\}$.

Next we show that $\wt{D}$ is complete when restricted to $\C \setminus \{\text{singular points}\}$ a.s.

\begin{proposition}\label{prop:complete_metric_general_case}
Almost surely,  every $\wt{D}$-Cauchy sequence is convergent.  In particular,  the restriction of $\wt{D}$ to $\C \setminus \{\text{singular points}\}$ is complete.
\end{proposition}

\begin{proof}
Let $(z_n)_{n \in \N}$ be a Cauchy sequence with respect to $\wt{D}$.  Then arguing as in the proof of Lemma~\ref{lem:complete_metric} and using Lemma~\ref{lem:metric_unbounded_general_case} in place of Lemma~\ref{lem:metric_unbounded},  we obtain that a.s.  we have that
\begin{align}\label{eqn:distances_to_infty}
\inf_{z \in \C \setminus B_R(0),  w \in B_r(0)} \wt{D}(z,w) \to \infty \quad \text{as} \quad R \to \infty
\end{align}
for all $r>0$.  Hence we obtain that it is a.s.  the case that there exists $R>0$ such that $|z_n| \leq R$ for all $n \in \N$.  Also,  there exists a subsequence $(z_{k_n})_{n \in \N}$ and $z \in \overline{B_R(0)}$ such that $z_{k_n} \to z$ as $n \to \infty$.

We will show that $\wt{D}(z_n ,  z) \to 0$ as $n \to \infty$.  Indeed fix $\epsilon>0$.  Then there exists $n_0 \in \N$ such that $\wt{D}(z_n,z_m) \leq \epsilon$ for all $n,m \geq n_0$.  Thus it holds that
\begin{align*}
\wt{D}(z,z_n) \leq \liminf_{m \to \infty} \wt{D}(z_{k_m} ,  z_n) \leq \epsilon \quad \text{for all} \quad n \geq n_0,
\end{align*}
which implies that $\wt{D}(z,z_n) \to 0$ as $n \to \infty$.  This completes the proof of the Proposition.
\end{proof}

Now we prove that $\wt{D}$ restricted to $\C \setminus \{\text{singular points}\}$ is a geodesic metric a.s.

\begin{proposition}\label{prop:geodesic_metric_general_case}
Almost surely,  the restriction of $\wt{D}$ to $\C \setminus \{\text{singular points}\}$ is a geodesic metric.
\end{proposition}

\begin{proof}
The statement of the proposition follows from combining Lemma~\ref{lem:singular_points} with the same argument used to prove Lemma~\ref{lem:geodesic_metric} with Lemma~\ref{lem:metric_unbounded_general_case} in place of Lemma~\ref{lem:metric_unbounded}.
\end{proof}

\begin{proof}[Proof of Proposition~\ref{prop:complete_and_geodesic_metric_general_xi}]
It follows from combining Lemma~\ref{lem:singular_points},  Proposition~\ref{prop:complete_metric_general_case} and Proposition~\ref{prop:geodesic_metric_general_case}.
\end{proof}

\subsection{The limit satisfies Weyl scaling: General case}
\label{subsec:weyl_scaling_general_case}

The main goal of this subsection is to prove the following analogue of Proposition~\ref{prop:weyl_scaling}.

\begin{proposition}\label{prop:weyl_scaling_general_case}
The following is true a.s.  Let $f : \C \to \R$ be a bounded and continuous function and let $K \subseteq \C$ be a compact set.  Then the metrics $e^{\xi_{k_n} f} \cdot \wt{D}_{h^n}^{\xi_{k_n}}$ restricted to $K$ converge to $e^{\xi f} \cdot \wt{D}$ restricted to $K$ as $n \to \infty$ with respect to the topology of lower semicontinuous functions.
\end{proposition}

We will follow the same argument used to prove Proposition~\ref{prop:weyl_scaling} in order to prove Proposition~\ref{prop:weyl_scaling_general_case}.  In particular,  we will prove that the conditions in the statement of Lemma~\ref{lem:conditions_for_weyl_scaling} hold a.s.  for the metrics $\wt{D}_{h^n}^{\xi_{k_n}}$.  We start by proving that condition~\ref{it:boundness_of_almost_geodesic_paths} holds a.s.

\begin{lemma}\label{lem:condition_I_satisfied_general_case}
The following holds a.s.  Let $f : \C \to \R$ be a bounded and continuous function.  Then for every compact set $K \subseteq \C$,  there exist compact sets $K',K'' \subseteq \C$ such that $K \subseteq K' \subseteq K''$ and condition~\ref{it:boundness_of_almost_geodesic_paths} in the statement of Lemma~\ref{lem:conditions_for_weyl_scaling} holds with $D_n = \wt{D}_{h^n}^{\xi_{k_n}} ,  f_n = \xi_{k_n} f,  f_{\infty} = \xi f$ and $D_{\infty} = \wt{D}$. 
\end{lemma}

\begin{proof}
We will only prove the claim of the lemma for the metrics $\wt{D}_{h^n}^{\xi_{k_n}}$ since the same argument works for the metrics $e^{\xi_{k_n} f} \cdot \wt{D}_{h^n}^{\xi_{k_n}}$.

Recall that we have shown in the proof of Proposition~\ref{prop:complete_metric_general_case} (see \eqref{eqn:distances_to_infty}) that it is a.s.  the case that
\begin{align}\label{eqn:crossing_large}
\wt{D}(\text{across} \,\,  \mathbb{A}_{r,R}(0)) \to \infty \quad \text{as} \quad R \to \infty \quad \text{for all} \quad r>0 \quad \text{a.s.}
\end{align}
Let $K \subseteq \C$ be a compact set and fix $r>0$ such that $K \subseteq B_r(0)$.  We will show that
\begin{align}\label{eqn:crossing_uniformly_large}
\inf_{n \in \N} \left\{\wt{D}_{h^n}^{\xi_{k_n}}(\text{across} \,\,  \mathbb{A}_{r,R}(0)) \right\}\to \infty \quad \text{as} \quad R \to \infty \quad \text{a.s.}
\end{align}
Indeed fix $T>0$ and let $R > 4r$ be such that 
\begin{align}\label{eqn:crossing_distance_big}
\wt{D}(\text{across} \,\,  \mathbb{A}_{2r ,  	R/2}(0)) > T. 
\end{align}
Note that \eqref{eqn:crossing_large} implies that we can always choose such $R$.  Suppose for contradiction that we have that
\begin{align*}
\wt{D}_{h^n}^{\xi_{k_n}}(\text{across} \,\,  \mathbb{A}_{r,R}(0)) \leq T
\end{align*}
for infinitely many values of $n \in \N$.  Then there exists a subsequence $(\xi_{k_n'})$ of $(\xi_{k_n})$ and points $u_n \in \partial B_r(0),  v_n \in \partial B_R(0)$ such that
\begin{align*}
\wt{D}_{h^{k_n'}}^{\xi_{k_n'}}(u_n,v_n) \leq T \quad \text{for all} \quad n \in \N.
\end{align*}
Since any path connecting $u_n$ to $v_n$ must cross the annulus $\mathbb{A}_{2r ,  R/2}(0)$,  we obtain that 
\begin{align*}
\wt{D}_{h^{k_n'}}^{\xi_{k_n'}}(\text{across} \,\,  \mathbb{A}_{2r,R/2}(0)) \leq T \quad \text{for all} \quad n \in \N
\end{align*}
which implies that there exist points $x_n \in \partial B_{2r}(0),  y_n \in \partial B_{R/2}(0)$ such that
\begin{align*}
\wt{D}_{h^{k_n'}}^{\xi_{k_n'}}(x_n,y_n) \leq T \quad \text{for all} \quad n \in \N.
\end{align*}
Possibly by passing into a subsequence,  we can assume that there exist $x \in \partial B_{2r}(0),  y \in \partial B_{R/2}(0)$ such that $x_n \to x$ and $y_n \to y$ as $n \to \infty$.  It follows that
\begin{align*}
\wt{D}(\text{across} \,\,  \mathbb{A}_{2r , R/2}(0)) \leq \wt{D}(x,y) \leq \liminf_{n \to \infty} \wt{D}_{h^{k_n'}}^{\xi_{k_n'}}(x_n ,  y_n) \leq T
\end{align*}
and that contradicts \eqref{eqn:crossing_distance_big}.  Therefore there exists $n_0 \in \N$ such that
\begin{align}\label{eqn:crossing_eventually_large}
\wt{D}_{h^n}^{\xi_{k_n}}(\text{across} \,\,  \mathbb{A}_{r,R}(0)) > T \quad \text{for all} \quad n \geq n_0.
\end{align}
Also we have that
\begin{align}\label{eqn:crossing_large_fixed_n}
\wt{D}_{h^n}^{\xi_{k_n}}(\text{across} \,\,  \mathbb{A}_{r,R}(0)) \to \infty \quad \text{as} \quad R \to \infty
\end{align}
for every fixed $n \in \N$.  Therefore combining \eqref{eqn:crossing_eventually_large} with \eqref{eqn:crossing_large_fixed_n} we obtain \eqref{eqn:crossing_uniformly_large}.

Next we let $r_1,r_2 \in \Q$ be such that $r<r_1 < r_2 < 2r$.  Then we have that
\begin{align*}
\wt{D}_{h^n}^{\xi_{k_n}}(\text{around} \,\,  \mathbb{A}_{r,2r}(0)) \leq \wt{D}_{h^n}^{\xi_{k_n}}(\text{around} \,\,  \mathbb{A}_{r_1,r_2}(0)) \quad \text{for all} \quad n \in \N
\end{align*}
and 
\begin{align*}
\wt{D}_{h^n}^{\xi_{k_n}}(\text{around} \,\,  \mathbb{A}_{r_1,r_2}(0)) \to D(\text{around} \,\,  \mathbb{A}_{r_1,r_2}(0)) < \infty.
\end{align*}
Thus combining with \eqref{eqn:crossing_uniformly_large},  we obtain that if we choose $T>0$ such that
\begin{align*}
\wt{D}_{h^n}^{\xi_{k_n}}(\text{around} \,\,  \mathbb{A}_{r,2r}(0)) + 1 < T \quad \text{for all} \quad n \in \N,
\end{align*}
then we can choose $R>2r$ such that
\begin{align}\label{eqn:main_ineq}
\wt{D}_{h^n}^{\xi_{k_n}}(\text{around} \,\,  \mathbb{A}_{r,2r}(0))+1 < \wt{D}_{h^n}^{\xi_{k_n}}(\text{across} \,\,  \mathbb{A}_{2r,R}(0)) \quad \text{for all} \quad n \in \N.
\end{align}

Fix $z,w \in K,  n \in \N$ and let $P$ be a path between $z$ and $w$ such that the $\wt{D}_{h^n}^{\xi_{k_n}}$-length of $P$ is less than $\wt{D}_{h^n}^{\xi_{k_n}}(z,w) + 1$.  Suppose that $P \not \subseteq B_R(0)$.  Then we have that the $\wt{D}_{h^n}^{\xi_{k_n}}$-length of $P$ is at least 
\begin{align*}
\wt{D}_{h^n}^{\xi_{k_n}}(z,\partial B_{2r}(0)) + \wt{D}_{h^n}^{\xi_{k_n}}(w,\partial B_{2r}(0)) + \wt{D}_{h^n}^{\xi_{k_n}}(\text{across} \,\,  \mathbb{A}_{2r,R}(0))
\end{align*}
and so combining with \eqref{eqn:main_ineq} we obtain that
\begin{align*}
\wt{D}_{h^n}^{\xi_{k_n}}(z,w) > \wt{D}_{h^n}^{\xi_{k_n}}(z,\partial B_{2r}(0)) + \wt{D}_{h^n}^{\xi_{k_n}}(w,\partial B_{2r}(0)) + \wt{D}_{h^n}^{\xi_{k_n}}(\text{around} \,\,  \mathbb{A}_{r,2r}(0)).
\end{align*}
But this is a contradiction since it is easy to see that
\begin{align*}
\wt{D}_{h^n}^{\xi_{k_n}}(z,w) \leq \wt{D}_{h^n}^{\xi_{k_n}}(z,\partial B_{2r}(0)) + \wt{D}_{h^n}^{\xi_{k_n}}(w,\partial B_{2r}(0)) + \wt{D}_{h^n}^{\xi_{k_n}}(\text{around} \,\,  \mathbb{A}_{r,2r}(0)).
\end{align*}
It follows that $P \subseteq \overline{B_R(0)}$.  The same argument works with the metrics $e^{\xi_{k_n} f } \cdot \wt{D}_{h^n}^{\xi_{k_n}}$ in place of $\wt{D}_{h^n}^{\xi_{k_n}}$.  This completes the proof of the lemma.
\end{proof}

Next we show that condition~\ref{it:short_distances_around_small_annuli} holds a.s.  as well.

\begin{lemma}\label{lem:condition_II_satisfied_general_case}
It is a.s.  the case that the following is true.  Fix $K \subseteq \C$ compact.  Then for each $z \in K$,  we can find a sequence of rational annuli $(A^m)_{m \in \N}$ as in condition~\ref{it:short_distances_around_small_annuli} such that
\begin{align*}
\lim_{m \to \infty} \lim_{n \to \infty} (\wt{D}_{h^n}^{\xi_{k_n}}(\text{around} \,\,A^m)) = 0.
\end{align*}
\end{lemma}

\begin{proof}
It follows from the same argument used to prove Lemma~\ref{lem:condition_II_satisfied}.
\end{proof}

In order to prove that all of the assumptions of Lemma~\ref{lem:conditions_for_weyl_scaling} are satisfied,  it remains to prove that the identity map from $(\C ,  \wt{D})$ to $\C$ endowed with the Euclidean metric is continuous a.s.  This is the content of the following lemma.

\begin{lemma}\label{lem:identity_map_continuous_general_case}
The identity map from $(\C ,  \wt{D})$ to $(\C ,  |\cdot|)$ is continuous a.s.,  where $|\cdot|$ denotes the Euclidean metric.
\end{lemma}

\begin{proof}
It follows from the argument in the proof of \cite[Theorem~1.2]{ding2023tightness} where we use Proposition~\ref{prop:annuli_crossing_always_positive_general_case} in place of \cite[Proposition~6.4]{ding2023tightness}.
\end{proof}

\begin{proof}[Proof of Proposition~\ref{prop:weyl_scaling_general_case}]
Lemmas~\ref{lem:condition_I_satisfied_general_case},  ~\ref{lem:condition_II_satisfied_general_case} and ~\ref{lem:identity_map_continuous_general_case} imply that the assumptions in the statement of Lemma~\ref{lem:conditions_for_weyl_scaling} hold a.s.  Hence the claim in the statement of the proposition follows from Lemma~\ref{lem:conditions_for_weyl_scaling}.
\end{proof}

\subsection{The limit satisfies Axiom~\ref{it:locality}: General case}
\label{subsec:locality_general_case}

Now we prove that $\wt{D}$ satisfies Axiom~\ref{it:locality}.  We will follow the same argument used to prove Proposition~\ref{prop:locality_property_subcritical}.  In particular,  we will show that the assumptions of Proposition~\ref{prop:main_prop_from_pfeffer_paper} are satisfied for $\wt{D}$.

We start by proving that $\wt{D}$ is a local $\xi$-additive metric for $h$.

\begin{proposition}\label{prop:limit_is_a_xi_additive_metric_general_case}
The metric $\wt{D}$ is a local $\xi$-additive metric for $h$ in the coupling $(h,\wt{D})$.
\end{proposition}

\begin{proof}
It follows from the argument used to prove Proposition~\ref{prop:limit_is_a_xi_additive_metric_subcritical}.
\end{proof}

\begin{proposition}\label{prop:locality_geeneral_case}
The metric $\wt{D}$ satisfies Axiom~\ref{it:locality} in the coupling $(h,\wt{D})$.
\end{proposition}

\begin{proof}
It suffices to show that the assumptions of Proposition~\ref{prop:main_prop_from_pfeffer_paper} are satisfied for $\wt{D}$.  First we note that Proposition~\ref{prop:complete_and_geodesic_metric_general_xi} implies that $\wt{D}$ restricted to $\C \setminus \{\text{singular points}\}$ is a.s.  a finite,   complete and geodesic metric.  Moreover for every $z \in \C$ fixed,  we have that $\wt{D}_{h^n(\cdot + z)}^{\xi_{k_n}}$ converges to $\wt{D}(\cdot + z ,  \cdot + z)$ as $n \to \infty$ with respect to the topology on lower semicontinuous functions a.s.  In particular we have that the metrics $\wt{D}$ and $e^{-\xi h_1(z)} \wt{D}(\cdot + z ,  \cdot +z)$ have the same law.  Hence condition~\ref{it:translation_invariance} is satisfied.  Moreover Proposition~\ref{prop:limit_is_a_xi_additive_metric_general_case} implies that $\wt{D}$ is a $\xi$-additive local metric for $h$.  Therefore it remains to prove that condition~\ref{it:tightness_across_and-around_annuli} is satisfied.

Fix any Euclidean annulus $A$.  Then for every fixed $r>0$,  since $\wt{D}_{h^n}^{\xi_{k_n}}$ is a strong LQG metric with parameter $\xi_{k_n}$ for all $n \in \N$ and 
\begin{align*}
r^{\xi_{k_n} Q(\xi_{k_n})} e^{-\xi_{k_n} h^n_r(0)} \wt{D}_{h^n}^{\xi_{k_n}}(r \cdot ,  r \cdot) \to r^{-\xi Q(\xi)} e^{-\xi h_r(0)} \wt{D}(r \cdot ,  r \cdot) \quad \text{as} \quad n \to \infty
\end{align*}
with respect to the lower semicontinuous topology,  we obtain that the metrics $\wt{D}$ and $r^{-\xi Q(\xi)} e^{-\xi h_r(0)} \wt{D}(r \cdot ,  r \cdot)$ have the same law.  Therefore in order to prove condition~\ref{it:tightness_across_and-around_annuli},  it suffices to prove that $\wt{D}(\text{across} \,\,A) \in (0,\infty)$ and $\wt{D}(\text{around} \,\,A) \in (0,\infty)$ a.s.  

Since $\wt{D}$ is a lower semicontinuous metric,  we obtain that $\wt{D}(\text{across} \,\,A) > 0$ a.s.  Moreover we have that 
\begin{align*}
\wt{D}(\text{across} \,\,A) \leq \liminf_{n \to \infty} \wt{D}_{h^n}^{\xi_{k_n}}(\text{across} \,\,A).
\end{align*}
Let $O_1,O_2$ be two rational circles such that $O_1$ (resp.  $O_2$) is contained in the bounded (resp.  unbounded) connected component of $\C \setminus A$.  Then we have that
\begin{align*}
\wt{D}_{h^n}^{\xi_{k_n}}(\text{across} \,\,A) \leq \wt{D}_{h^n}^{\xi_{k_n}}(O_1 ,  O_2) \quad \text{for all} \quad n \in \N.
\end{align*}
It follows that
\begin{align*}
\wt{D}(\text{across} \,\,A) \leq D(O_1,O_2) < \infty \quad \text{a.s.}
\end{align*}

Now we prove that $\wt{D}(\text{around} \,\,  A) \in (0,\infty)$ a.s.  We start by proving that $\wt{D}(\text{around} \,\,  A) < \infty$ a.s.  We fix deterministic disjoint and connected compact sets $X_j ,  Y_j$ and open sets $U_j \subseteq A$ for $j=1,2$ such that $X_j \cup Y_j \subseteq U_j$ for all $j=1,2$,  and such that the union of a path from $X_1$ to $Y_1$ in $U_1$ and a path from $X_2$ to $Y_2$ in $U_2$ necessarily contains a path around $A$.  Since $\wt{D}(X_j ,  Y_j ; U_j) < \infty$ for all $j=1,2$ a.s.  by Proposition~\ref{prop:bound_on_lqg_distances_of_sets_general} and since $\wt{D}$ is a length metric,  we obtain that it is a.s.  the case that there exist paths $P_1,P_2$ in $U_1$ and $U_2$ respectively such that $P_j$ connects $X_j$ to $Y_j$ and $\len(P_j ; \wt{D}) < \infty$ for all $j=1,2$ a.s.  Let $P$ be the path obtained by concatenating $P_1$ and $P_2$.  Then we have that $P \subseteq A$ and $P$ disconnects the inner from the outer boundary of $A$,  and 
\begin{align*}
\len(P ; \wt{D}) \leq \len(P_1 ; \wt{D}) + \len(P_2 ; \wt{D}) < \infty.
\end{align*}
It follows that $\wt{D}(\text{around} \,\,  A) < \infty$ a.s.

Finally we show that $\wt{D}(\text{around} \,\,  A) > 0$ a.s.  and hence completing the proof of the proposition.  We will first show that it is a.s.  the case that $\wt{D}(K_1,K_2) > 0$ for any disjoint compact sets $K_1,K_2 \subseteq \C$.  Indeed,  fix such sets $K_1,K_2$ and let $(z_n),(w_n)$ be sequences of points in $K_1$ and $K_2$ respectively such that
\begin{align*}
\wt{D}(z_n,w_n) \to \wt{D}(K_1,K_2) \quad \text{as} \quad n \to \infty.
\end{align*}
Possibly by passing into a subsequence,  we can assume that there exist $z \in K_1 ,  w \in K_2$ such that $z_n \to z$ and $w_n \to w$ as $n \to \infty$.  Then since $\wt{D}$ is lower semicontinuous,  we have that
\begin{align*}
0< \wt{D}(z,w) \leq \liminf_{n \to \infty} \wt{D}(z_n ,  w_n) = \wt{D}(K_1 ,  K_2)
\end{align*}
since $\wt{D}$ is a metric.

Suppose that $A = \mathbb{A}_{r_1,r_2}(z)$ for some $0<r_1 < r_2 < \infty$ and $z \in \C$.  Set
\begin{align*}
&K_1 =\{w \in A : w = z + r e^{i\frac{\pi}{2}} ,  r \in [r_1 ,  r_2]\},  \\
&K_2 = \{w \in A : w = z + r e^{i \frac{3\pi}{2}} ,  r \in [r_1 ,  r_2]\}.
\end{align*}
Note that any path in $A$ disconnecting the inner from the outer boundary of $A$ must connect $K_1$ to $K_2$ and so its $\wt{D}$-length has to be at least $\wt{D}(K_1,K_2)$.  Therefore combining with the previous paragraph,  we obtain that $\wt{D}(\text{around} \,\,A) > 0$ a.s.  This completes the proof of the proposition.
\end{proof}

\subsection{The limit is the $\xi$-LQG metric: General case}
\label{subsec:limit_xi_lqg_metric_general_case}

Finally we are ready to prove Theorem~\ref{thm:main_theorem_general_case_intro}.  The main ingredient of its proof is the following theorem.

\begin{theorem}\label{thm:main_theorem_general_case}
Let $h$ be a whole-plane GFF normalized such that $h_1(0) = 0$.  For all $\xi>0$,  we let $\wt{D}_h^{\xi}$ be the metric $D_h^{\xi}$ normalized as in \eqref{eqn:normalization_for_general_xi}.  Fix $\xi>0$ and let $(\xi_n)_{n \in \N}$ be a sequence in $(0,\infty)$ such that $\xi_n \to \xi$ as $n \to \infty$.  Then the sequence of metrics $(\wt{D}_h^{\xi_n})_{n \in \N}$ is tight with respect to the lower semicontinuous topology of functions.  Moreover if $\wt{D}$ is any subsequential limit in law,  we have that $\wt{D}$ is non-trivial in the sense that $\wt{D}(z,w) \in (0,\infty)$ a.s.  for any fixed distinct and  deterministic points $z,w \in \C$.  If in addition we assume that $\wt{D}$ satisfies the triangle inequality a.s.,  then the following is true.  Let $(\xi_{k_n})_{n \in \N}$ be a subsequence of $(\xi_n)_{n \in \N}$ such that 
\begin{align*}
\wt{D}_h^{\xi_{k_n}} \to \wt{D} \quad \text{in law} \quad \text{as} \quad n \to \infty.
\end{align*}
Then there exists a $\xi$-LQG metric $h \to \wt{D}_h$ such that
\begin{align*}
\wt{D}_h^{\xi_{k_n}} \to \wt{D}_h \quad \text{in probability} \quad \text{as} \quad n \to \infty.
\end{align*}
\end{theorem}

\begin{proof}
Fix $\xi > 0$ and let $(\xi_n)_{n \in \N}$ be a sequence in $(0,\infty)$ such that $\xi_n \to \xi$ as $n \to \infty$.  Proposition~\ref{prop:lower_semicontinuity} implies that $(\wt{D}_h^{\xi_n})_{n \in \N}$ is tight with respect to the lower semicontinuous topology.  Let $(\xi_{k_n})_{n \in \N}$ be a subsequence of $(\xi_n)_{n \in \N}$ and suppose that the random fields $(h^n)_{n \in \N}$ are coupled in the same probability space with some function $\wt{D}$ as in Proposition~\ref{prop:lower_semicontinuity} such that
\begin{align*}
\wt{D}_{h^n}^{\xi_{k_n}} \to \wt{D} \quad \text{as} \quad n \to \infty
\end{align*}
with respect to the topology of lower semicontinuous functions.  Then Lemmas~\ref{lem:metric_finite_at_fixed_points} and ~\ref{lem:metric_non_trivial} imply that $\wt{D}(z,w) \in (0,\infty)$ a.s.  for every fixed and distinct points $z,w \in \C$.

Suppose in addition that $\wt{D}$ satisfies the triangle inequality a.s.  Let $h$ be the limit of $h^n$ as $n \to \infty$ and recall that $h$ has the law of a whole-plane GFF normalized so that $h_1(0) = 0$.  Then Propositions~\ref{prop:complete_and_geodesic_metric_general_xi} and ~\ref{prop:locality_geeneral_case} imply that $\wt{D}$ satisfies Axioms~\ref{it:length_space} and ~\ref{it:locality} in the coupling $(h,\wt{D})$. 
Recall also that for fixed $z \in \C ,  r>0$,  Axiom~\ref{it:coordinate_change} implies that 
\begin{align*}
\wt{D}_{h^n}^{\xi_{k_n}}(r\cdot + z,  r \cdot +z) = \wt{D}_{h^n(r \cdot + z) + Q \log(r)}^{\xi_{k_n}} \quad \text{for all} \quad n \in \N \quad \text{a.s.}
\end{align*}
Moreover we have that
\begin{align*}
&\wt{D}_{h^n}^{\xi_{k_n}}(r\cdot + z,  r \cdot +z) \to \wt{D}_h(r \cdot + z,  r \cdot +z) \quad \text{as} \quad n \to \infty,  \\
& \wt{D}_{h^n(r \cdot + z) + Q \log(r)}^{\xi_{k_n}} \to \wt{D}_{h(r \cdot +z) + Q \log(r)} \quad \text{as} \quad n \to \infty
\end{align*}
with respect to the lower semicontinuous topology of functions.  Thus we obtain that $\wt{D}$ satisfies Axiom~\ref{it:coordinate_change} and Proposition~\ref{prop:lqg_distances_general_case} implies that $\wt{D}$ satisfies Axiom~\ref{it:finiteness}.

Next we will construct a $\xi$-LQG metric $h \to \wt{D}_h$ in the same probability space such that $\wt{D} = \wt{D}_h$ a.s.  Recall that Proposition~\ref{prop:weyl_scaling_general_case} implies that it is a.s.  the case that if $f : \C \to \R$ is any bounded and continuous function,  we have that $\wt{D}_{h^n + f}^{\xi_{k_n}}$ converges to $e^{\xi f} \cdot \wt{D}$ as $n \to \infty$ with respect to the lower semicontinuous topology of functions.  Then we set $\wt{D}_{h+f}:=e^{\xi f} \cdot \wt{D}$ whenever $f : \C \to \R$ is bounded and continuous.  Furthermore we can extend the definition of $\wt{D}_{h+f}$ for general continuous functions $f : \C \to \R$ in the same way as in the proof of Proposition~\ref{prop:main_prop_subcritical}.  Then arguing as in the proof of Theorem~\ref{thm:main_theorem_subcritical_intro},  we obtain that $h \to \wt{D}_h$ is indeed an LQG metric with parameter $\xi$.  Finally since $\wt{D}_h$ is a.s.  determined by $h$,  Lemma~\ref{lem:convergence_in_prob} implies that 
\begin{align*}
\wt{D}_h^{\xi_{k_n}} \to \wt{D}_h \quad \text{in probability} \quad \text{as} \quad n \to \infty.
\end{align*}
This completes the proof of the theorem.
\end{proof}

\begin{proof}[Proof of Theorem~\ref{thm:main_theorem_general_case_intro}]
Fix $\xi>0$ and let $(\xi_n)_{n \in \N}$ be a sequence in $(0,\infty)$ such that $\xi_n \to \xi$ as $n \to \infty$.  Then Theorem~\ref{thm:main_theorem_general_case} implies that there exists a subsequence $(\xi_{k_n})_{n \in \N}$ and a coupling of the random fields $(h^n)_{n \in \N}$ in the same probability space with some random function $\wt{D}$ as in Proposition~\ref{prop:lower_semicontinuity} such that a.s.
\begin{align*}
\wt{D}_{h^n}^{\xi_{k_n}} \to \wt{D} \quad \text{as} \quad n \to \infty
\end{align*}
with respect to the topology of lower semicontinuous functions.  For all $n \in \N$,  we let $\wt{\beta}(\xi_n) \in (0,\infty)$ be such that
\begin{align*}
\p\left[\wt{D}_{h^n}^{\xi_{k_n}}(0,1) \leq \wt{\beta}(\xi_n)\right] = \frac{1}{2} 
\end{align*}
and note that $\wh{D}_{h^n}^{\xi_{k_n}} = \wt{\beta}(\xi_n)^{-1} \wt{D}_{h^n}^{\xi_n}$ for all $n \in \N$.  

We claim that there exists $M \in (1,\infty)$ such that $M^{-1} \leq \wt{\beta}(\xi_{k_n}) \leq M$ for all $n \in \N$.  Indeed by Theorem~\ref{thm:main_theorem_general_case} we have that $\wt{D}(0,1)>0$ a.s.  and so there exists $M>1$ such that
\begin{align*}
\p\left[\wt{D}(0,1) > M^{-1}\right] > \frac{1}{2}.
\end{align*}
It follows that
\begin{align*}
\frac{1}{2} <  \p\left[\wt{D}(0,1) > M^{-1} \right] \leq \liminf_{n \to \infty} \p\left[\wt{D}_{h^n}^{\xi_{k_n}}(0,1) \geq M^{-1} \right]
\end{align*}
and hence $\wt{\beta}(\xi_{k_n}) \geq M^{-1}$ for all $n \in \N$ sufficiently large.  Moreover Proposition~\ref{prop:lqg_distances_general_case} implies that $(\wt{D}_{h^n}^{\xi_{k_n}}(0,1))_{n \in \N}$ is tight and so by possibly taking $M$ to be larger,  we have that $\wt{\beta}(\xi_{k_n}) \leq M$ for all $n \in \N$.  It follows that there exists $M>1$ such that
\begin{align*}
M^{-1} \leq \wt{\beta}(\xi_{k_n}) \leq M \quad \text{for all} \quad n \in \N.
\end{align*}
Hence by possibly passing into a subsequence,  we can assume that there exists $\wt{\beta} \in (0,\infty)$ such that $\wt{\beta}(\xi_{k_n}) \to \wt{\beta}$ as $n \to \infty$.  Then we have that $\wh{D}_{h^n}^{\xi_{k_n}} \to \wh{D}$ as $n \to \infty$ a.s.  with respect to the lower semicontinuous topology,  where $\wh{D} = \wt{\beta}^{-1} \wt{D}$.  Therefore Theorem~\ref{thm:main_theorem_general_case} implies that $\wh{D}(z,w) \in (0,\infty)$ a.s.  for every fixed and distinct points $z,w \in \C$.

Suppose further that $\wh{D}$ satisfies the triangle inequality a.s.  Then the same is true for $\wt{D}$ and so Theorem~\ref{thm:main_theorem_general_case} implies that there exists an LQG metric with parameter $\xi>0$ $h \to \wt{D}_h$ such that
\begin{align*}
\wt{D}_h^{\xi_{k_n}} \to \wt{D}_h \quad \text{in probability} \quad \text{as} \quad n \to \infty.
\end{align*}
Note that $h \to \wh{D}_h:=\wt{\beta}^{-1} \wt{D}_h$ is also an LQG metric with parameter $\xi$.  Therefore combining with Theorem~\ref{thm:uniqueness_supercritical},  we obtain that it suffices to show that
\begin{align}\label{eqn:normalizing_constant}
\p\left[\wh{D}_h(0,1) \leq 1 \right] = \frac{1}{2}.
\end{align}

To show \eqref{eqn:normalizing_constant},  we first note that by the convergence with respect to the lower semicontinuous topology,  we have that 
\begin{align*}
\p\left[\wh{D}_h(0,1) > 1\right] \leq \liminf_{n \to \infty} \p\left[\wh{D}_h^{\xi_{k_n}}(0,1) > 1\right] = \frac{1}{2}
\end{align*}
which implies that
\begin{align}\label{eqn:normalization_lower_bound}
\p\left[\wh{D}_h(0,1) \leq 1 \right] \geq \frac{1}{2}.
\end{align}
Recall that the proof of Proposition~\ref{prop:lqg_distances_general_case}  (see in particular \eqref{eqn:centered_prob_bound_fixed_scale} and \eqref{eqn:centered_prob_bound_fixed_scale_uniform}) implies that
\begin{align*}
\sup_{n \in \N} \p\left[|\wh{D}_h^{\xi_{k_n}}(0,1) - \wh{D}_h^{\xi_{k_n}}(\partial B_r(0) ,  \partial B_r(1))| > \epsilon \right] \to 0 \quad \text{as} \quad r \to 0
\end{align*}
for all $\epsilon > 0$.  Fix $\epsilon > 0$ and let $r \in (0,1)$ be such that
\begin{align*}
\p\left[|\wh{D}_h^{\xi_{k_n}}(0,1) - \wh{D}_h^{\xi_{k_n}}(\partial B_r(0) ,  \partial B_r(1))| > \epsilon \right] \leq \epsilon \quad \text{for all} \quad n \in \N.
\end{align*}
Then we have that
\begin{align*}
\p\left[\wh{D}_h^{\xi_{k_n}}(\partial B_r(0) ,  \partial B_r(1)) > 1 -\epsilon \right] \geq \frac{1}{2} - \epsilon \quad \text{for all} \quad n \in \N.
\end{align*}
Also we have that
\begin{align*}
\p\left[\wh{D}_h(0,1) \geq 1 - \epsilon \right] \geq \limsup_{n \to \infty} \p\left[\wh{D}_h^{\xi_{k_n}}(\partial B_r(0) ,  \partial B_r(1)) \geq 1 - \epsilon \right]
\end{align*}
and so 
\begin{align*}
\p\left[\wh{D}_h(0,1) \geq 1 - \epsilon \right] \geq \frac{1}{2} - \epsilon \quad \text{for all} \quad \epsilon > 0.
\end{align*}
Letting $\epsilon \to 0$ implies that $\p\left[\wh{D}_h(0,1) \geq 1 \right] \geq \frac{1}{2}$ and so combining with Lemma~\ref{lem:normalization_well_defined},  we obtain that
\begin{align}\label{eqn:normalization_upper_bound}
\p\left[\wh{D}_h(0,1) > 1\right] \geq \frac{1}{2}.
\end{align}
Therefore \eqref{eqn:normalizing_constant} follows from combining \eqref{eqn:normalization_lower_bound} with \eqref{eqn:normalization_upper_bound}.  This completes the proof of the theorem.
\end{proof}

\section{Tightness: $\xi \to \infty$ case}
\label{sec:xi_to_infty}

The main goal of this section is to prove Theorem~\ref{thm:main_thm_xi_to_infty_intro}.  We are going to follow the same strategy as in Section~\ref{sec:tightness_general_case}.  However since we re-normalize the metrics differently,  we will obtain different estimates.  As in Section~\ref{sec:tightness_general_case},  in order to show that the family of metrics $\wt{D}_h^{\xi}$ in Theorem~\ref{thm:main_thm_xi_to_infty_intro} is tight as $\xi \to \infty$,  we will prove a probability estimate on the $\wt{D}_h^{\xi}$-distances between any two fixed compact and connected sets which are not singletons,  which is uniform in $\xi$ as $\xi \to \infty$ (see Proposition~\ref{prop:bound_on_lqg_distances_of_sets_xi_to_infty}).  Once we prove these estimates,  we will use the argument in Section~\ref{sec:tightness_general_case} to construct subsequential limits of $\wt{D}_h^{\xi}$ as $\xi \to \infty$.  Moreover we will show that any subsequential limit is non-trivial a.s.  (see Lemmas~\ref{lem:limit_not_trivial_almost_surely} and ~\ref{lem:limit_finite_xi_to_infty}).  Finally we will show that if in addition we make the assumption that a subsequential limit satisfies the triangle inequality a.s.,  then it has to be a metric a.s.  (see Proposition~\ref{prop:limit_is_a_metric_xi_to_infty}) and hence completing the proof of Theorem~\ref{thm:main_thm_xi_to_infty_intro}.

Now we define the re-normalization of the metrics $D_h^{\xi}$ that we are going to use.  We fix $b \in (0,\frac{1}{2\sqrt{2}})$ and let $M>1$ be sufficiently large such that
\begin{align}\label{eqn:condition_for_M_xi_to_infty}
b \sqrt{M} > 1 \quad \text{and} \quad \frac{b 2 \sqrt{2} \sqrt{4+M}}{\sqrt{M}} < 1.
\end{align}
Let also $\mathfrak{p} \in (0,1)$ be the constant in the statement of Lemma~\ref{lem:good_event_occurs_almost_everywhere} corresponding to $M$.  Moreover for all $\xi \in (1,\infty)$,  we let $\alpha(\xi) \in (0,\infty)$ be such that
\begin{align}\label{eqn:main_normalization_xi_to_infty}
\p\left[D_h^{\xi}(\text{around} \,\,  \mathbb{A}_{1,2}(0)) \leq \alpha(\xi) \right] = \mathfrak{p}
\end{align}
and set 
\begin{align}\label{eqn:normalization_xi_to_infty}
\wt{D}_h^{\xi}:=(\alpha(\xi)^{-1} D_h^{\xi})^{1/\xi}. 
\end{align}
Note that $\wt{D}_h^{\xi}$ is a metric for all $\xi \geq 1$.

Let $\xi_0 \in (1,\infty)$ be such that
\begin{align}\label{eqn:condition_for_xi}
\frac{b}{\sqrt{M}}\left(2 \sqrt{2} \sqrt{4+M} - Q(\xi) + \frac{4}{\xi} \right) < 1 \quad \text{for all} \quad \xi \geq \xi_0,
\end{align}
and note that $\xi_0$ is well-defined since \cite[Proposition~2.5]{ding2020tightness} implies that $\lim_{\xi \to \infty} Q(\xi) = 0$.  Then we have the following analogue of Proposition~\ref{prop:bound_on_lqg_distances_of_sets_general}.

\begin{proposition}\label{prop:bound_on_lqg_distances_of_sets_xi_to_infty}
Suppose that we have the setup described above and fix $\xi \in [\xi_0 ,  \infty)$.  Let $U \subseteq \C$ be an open and connected set and let $K_1,K_2 \subseteq U$ be two connected,  disjoint compact sets which are not singletons.  Then for all $\mathfrak{r}>0$,  it holds with probability at least $1 - O_A(A^{-b \sqrt{M}})$ as $A \to \infty$ at a rate which depends only on $K_1,K_2$ and $U$ and it is uniform in the choice of $\mathfrak{r}$,  that 
\begin{align*}
\wt{D}_h^{\xi}(\mathfrak{r} K_1 ,  \mathfrak{r} K_2 ; \mathfrak{r} U) \leq A \mathfrak{r}^{Q(\xi)} e^{h_{\mathfrak{r}}(0)}.
\end{align*}
\end{proposition}

\begin{proof}
The argument is very similar to the argument in the proof of Proposition~\ref{prop:bound_on_lqg_distances_of_sets}.  As in the proof of Proposition~\ref{prop:bound_on_lqg_distances_of_sets},  we can assume that $U$ is bounded and let $R>1$ be such that $U \subseteq B_R(0)$.  Also for $z \in \C,  r>0$,  we let $E_r^{\xi}(z)$ be the event that
\begin{align*}
D_h^{\xi}(\text{around} \,\,  \mathbb{A}_{r,2r}(z)) \leq \alpha(\xi) r^{\xi Q(\xi)} e^{\xi h_r(z)}.
\end{align*}
Note that the choice of $\alpha(\xi)$ combined with the Axioms of a strong LQG metric with parameter $\xi$ (Definition~\ref{def:strong_lqg_metric_general}) imply that
\begin{align*}
\p[E_r^{\xi}(z)] = \mathfrak{p} \quad \text{for all} \quad z \in \C,  r >0.
\end{align*}

For all $\epsilon \in (0,1),  \mathfrak{r}>0$,  we let $F_{\mathfrak{r}}^{\epsilon}$ be the event that the following is true.  For all $z \in B_{\mathfrak{r} \epsilon^{-M}}(0)$,  there exist $r \in [\epsilon^2 \mathfrak{r} ,  \epsilon \mathfrak{r}] \cap \{2^{-k} \mathfrak{r}\}_{k \in \N},   w \in B_{\mathfrak{r} \epsilon^{-M}}(0) \cap \left(\frac{\epsilon^2 \mathfrak{r}}{4} \Z^2\right)$ such that $E_r^{\xi}(w)$ occurs and $z \in B_{\frac{\mathfrak{r} \epsilon^2}{2}}(w)$.  Then arguing as in Lemma~\ref{lem:good_event_occurs_almost_everywhere},  we obtain that
\begin{align*}
\p[F_{\mathfrak{r}}^{\epsilon}] \geq 1 - \epsilon^M \quad \text{as} \quad \epsilon \to 0,
\end{align*}
at a universal rate.  Moreover if $F_{\mathfrak{r}}^{\epsilon}$ occurs,  we let $P_{w,r}$ be a path with $D_h^{\xi}$-length at most $\alpha(\xi) r^{\xi Q(\xi)} e^{\xi h_r(0)}$ disconnecting $\partial B_r(w)$ from $\partial B_{2r}(w)$ and chosen according to some arbitrary and fixed measurable way,  for all $r \in [\epsilon^2 \mathfrak{r} ,  \epsilon \mathfrak{r}] \cap \{2^{-k} \mathfrak{r}\}_{k \in \N}$ such that $E_r^{\xi}(w)$ occurs.  Note that the argument in Lemma~\ref{lem:connecting_compact_sets_with_good_annuli} implies that if $F_{\mathfrak{r}}^{\epsilon}$ occurs,  we have that the union of paths $P_{w,r}$ for $w \in B_{\epsilon \mathfrak{r}}(\mathfrak{r} U) \cap \left(\frac{\epsilon^2 \mathfrak{r}}{4} \Z^2\right)$ and $r \in [\epsilon^2 \mathfrak{r} ,  \epsilon \mathfrak{r} ] \cap \{2^{-k} \mathfrak{r}\}_{k \in \N}$ such that $E_r^{\epsilon}(w)$ occurs contains a path from $\mathfrak{r} K_1$ to $\mathfrak{r} K_2$ which is contained in $\mathfrak{r} U$.  Henceforth we assume that this event occurs.

Note that the number of paths $P_{w,r}$ as above is at most $\epsilon^{-4 - o_{\epsilon}(1)}$ as $\epsilon \to 0$ (at a rate which depends only on $K_1,K_2$ and $U$).  Recall also that Lemma~\ref{lem:upper_bound_on_circle_averages} implies that there exists a constant $C>1$ depending only on $R$ such that for all $\epsilon \in (0,1)$,  off an event with probability at most $C \epsilon^{\frac{4+M}{1 + C \log(\epsilon^{-1})^{-1}}}$,  it holds that
\begin{align*}
\sup\left\{|h_r(w) - h_{\mathfrak{r}}(0)| : w \in B_{R \mathfrak{r}}(0) \cap \left(\frac{\epsilon^2 \mathfrak{r}}{4} \Z^2 \right),  r \in [\epsilon^2 \mathfrak{r} ,  \epsilon \mathfrak{r}]\right\} \leq 2 \sqrt{2} \sqrt{4+M} \log(\epsilon^{-1}).
\end{align*}
Combining we obtain that off an event with probability at most $\epsilon^M$ as $\epsilon \to 0$ (at a rate which depends only on $K_1,K_2$ and $U$),  we have that
\begin{align*}
D_h^{\xi}(\mathfrak{r} K_1 ,  \mathfrak{r} K_2 ; \mathfrak{r} U) &\leq \epsilon^{-4 - o_{\epsilon}(1)} \sup\left\{\len(P_{w,r} ;  D_h^{\xi}) : w \in B_{\epsilon \mathfrak{r}}(\mathfrak{r} U) \cap \left(\frac{\epsilon^2 \mathfrak{r}}{4} \Z^2\right) ,  r \in [\epsilon^2 \mathfrak{r} ,  \epsilon \mathfrak{r}] \cap \{2^{-k} \mathfrak{r}\}_{k \in \N} \right\}\\
&\leq \alpha(\xi) \epsilon^{-4 - o_{\epsilon}(1)} \sup\left\{r^{\xi Q(\xi)} e^{\xi h_r(w)} : w \in B_{\epsilon \mathfrak{r}}(\mathfrak{r} U) \cap \left(\frac{\epsilon^2 \mathfrak{r}}{4} \Z^2\right),  r \in [\epsilon^2 \mathfrak{r} , \epsilon \mathfrak{r}] \cap \{2^{-k} \mathfrak{r}\}_{k \in \N}\right\}\\
&\leq \alpha(\xi) \epsilon^{-4 +\xi Q(\xi)- \xi 2 \sqrt{2} \sqrt{4+M} - o_{\epsilon}(1)} \mathfrak{r}^{\xi Q(\xi)} e^{\xi h_{\mathfrak{r}}(0)}
\end{align*}
which implies that
\begin{align*}
\wt{D}_h^{\xi}(\mathfrak{r} K_1 ,  \mathfrak{r} K_2 ; \mathfrak{r} U) \leq \epsilon^{-\frac{4}{\xi} +Q(\xi)-2\sqrt{2} \sqrt{4+M} - o_{\epsilon}(1)} \mathfrak{r}^{Q(\xi)} e^{h_{\mathfrak{r}}(0)}.
\end{align*}

As in the proof of Proposition~\ref{prop:bound_on_lqg_distances_of_sets},  given $A>1$,  we set $\epsilon = A^{-\frac{b}{\sqrt{M}}}$.  Recall that \eqref{eqn:condition_for_xi} implies that
\begin{align*}
\epsilon^{-\frac{4}{\xi} +Q(\xi)-2\sqrt{2} \sqrt{4+M}} = A^{\frac{b}{\sqrt{M}}(2 \sqrt{2} \sqrt{4+M} -Q(\xi)+ \frac{4}{\xi})} < A
\end{align*}
and so 
\begin{align*}
\wt{D}_h^{\xi}(\mathfrak{r} K_1 ,  \mathfrak{r} K_2 ; \mathfrak{r} U) \leq A \mathfrak{r}^{Q(\xi)} e^{h_{\mathfrak{r}}(0)}
\end{align*}
if $A>1$ is sufficiently large (depending only on $K_1,K_2$,  and $U$).  This completes the proof of the proposition.
\end{proof}

Let us now use Proposition~\ref{prop:bound_on_lqg_distances_of_sets_xi_to_infty} to construct subsequential limits of $\wt{D}_h^{\xi}$.  Let $(\xi_n)_{n \in \N}$ be a sequence in $(1,\infty)$ such that $\xi_n \to \infty$ as $n \to \infty$.  By arguing as in Section~\ref{sec:tightness_general_case} and applying Proposition~\ref{prop:bound_on_lqg_distances_of_sets_xi_to_infty} countably many times,  we obtain that we can find a subsequence $(\xi_{k_n})_{n \in \N}$ of $(\xi_n)_{n \in \N}$ and a coupling of the GFFs $(h^n)_{n \in \N}$ with a GFF $h$ and random variables $D(O_1,O_2)$ for rational circles $O_1,O_2$ and random variables $D(\text{around} \,\,  A)$ for rational annuli $A$,  such that both $h^n$ and $h$ have the law of a whole-plane GFF normalized such that its average on $\partial B_1(0)$ is equal to zero and such that a.s.
\begin{align*}
&\wt{D}_{h^n}^{\xi_{k_n}}(O_1,O_2) \to D(O_1,O_2) \quad \text{for all rational circles} \quad O_1,O_2,\\
&\wt{D}_{h^n}^{\xi_{k_n}}(\text{around} \,\,A) \to D(\text{around} \,\,A) \quad \text{for all rational annuli} \quad A,
\end{align*}
and the functions $(z,r) \to h^n_r(z)$ defined on $\C \times (0,\infty)$ converge to $(z,r) \to h_r(z)$ with respect to the local uniform topology.

Then arguing as in Section~\ref{sec:tightness_general_case} ,  we can construct a lower semicontinuous function $\wt{D} : \C \times \C \to \R \cup \{\pm \infty\}$ coupled with $h$ as in \eqref{eqn:definition_of_supercritical_metric}.  

\begin{rem}\label{rem:metric_well_defined}
We note that as in  \cite[Section~5.2]{ding2020tightness},  the key observation behind the definition of $\wt{D}$ (see \eqref{eqn:definition_of_supercritical_metric}) is the following.  If $O_z,O_z'$ are rational circles surrounding $z$ and $O_w,O_w'$ are rational circles surrounding $w$ such that $O_z \cap O_w  =\emptyset$,  $O_z$ surrounds $O_z'$,  and $O_w$ surrounds $O_w'$,  then we have that
\begin{align}\label{eqn:monotonicity_property_limit}
D(O_z ,  O_w) \leq D(O_z' ,  O_w').
\end{align}
Indeed this follows since $(\wt{D}_{h^n}^{\xi_{k_n}})^{\xi_{k_n}}$ is a length metric for all $n$ and so 
\begin{align*}
(\wt{D}_{h^n}^{\xi_{k_n}}(O_z,O_w))^{\xi_{k_n}} \leq (\wt{D}_{h^n}^{\xi_{k_n}}(O_z' ,  O_w'))^{\xi_{k_n}} \quad \text{for all} \quad n \in \N
\end{align*}
which implies that
\begin{align}\label{eqn:monotonicity_property}
\wt{D}_{h^n}^{\xi_{k_n}}(O_z,O_w) \leq \wt{D}_{h^n}^{\xi_{k_n}}(O_z' ,  O_w') \quad \text{for all} \quad n \in \N.
\end{align}
Thus \eqref{eqn:monotonicity_property_limit} follows from \eqref{eqn:monotonicity_property} by taking $n \to \infty$.  Note also that $\wt{D}_{h^n}^{\xi_{k_n}}$ is not necessarily a length metric but \eqref{eqn:monotonicity_property_limit} still holds and so this allows us to define the limiting metric $\wt{D}$ as in \eqref{eqn:definition_of_supercritical_metric}.
\end{rem}

\,\,\,\,\,We have the following analogue of Proposition~\ref{prop:lower_semicontinuity}.

\begin{proposition}\label{prop:lower_semicontinuity_xi_to_infty}
In the above coupling,  we have that
\begin{align*}
\wt{D}_{h^n}^{\xi_{k_n}} \to \wt{D} \quad \text{as} \quad n \to \infty
\end{align*}
with respect to the topology on lower semicontinuous functions.  
\end{proposition}

\begin{proof}
It follows from the same argument as in the proof of Proposition~\ref{prop:lower_semicontinuity}.
\end{proof}

\begin{rem}
We note that arguing as in Remark~\ref{rem:tightness_of_internal_metrics_general_case} and the proof of Proposition~\ref{prop:lower_semicontinuity_xi_to_infty},  we obtain that for every open and connected set $U \subseteq \C$,  the collection of random variables $\{\wt{D}_{h^n}^{\xi_n}(\cdot ,  \cdot ; U)\}_{n \in \N}$ is tight with respect to the lower semicontinuous topology of functions on $U$.
\end{rem}

\,\,\,\,\,\,\,Next we mention some of the properties of the limiting metric $\wt{D}$ in the coupling $(h,\wt{D})$.  First we note that the Axioms of strong LQG metrics imply that for fixed $z \in \C ,  r>0$,
\begin{align*}
\wt{D}_{h^n(\cdot + z) - h^n_1(z)}^{\xi_{k_n}} = e^{-h^n_1(z)} \wt{D}_{h^n}^{\xi_{k_n}}(\cdot + z ,  \cdot + z) \quad \text{a.s.}
\end{align*}
Also we have that the random variables $\wt{D}_{h^n}^{\xi_{k_n}}$ and $\wt{D}_{h^n(\cdot + z) - h^n_1(z)}^{\xi_{k_n}}$ have the same law and that
\begin{align*}
e^{-h^n_1(z)} \wt{D}_{h^n}^{\xi_{k_n}}(\cdot + z ,  \cdot + z) \to e^{-h_1(z)} \wt{D}(\cdot + z ,  \cdot +z) \quad \text{as} \quad n \to \infty
\end{align*}
with respect to the lower semicontinuous topology.  It follows that
\begin{align}\label{eqn:translation_invariance_of_the_limit}
\wt{D} = e^{-h_1(z)} \wt{D}(\cdot + z ,  \cdot +z) \quad \text{in law}.
\end{align}

Moreover we have that
\begin{align*}
\wt{D}_{h^n(r \cdot) - h^n_r(0)}^{\xi_{k_n}} = r^{-Q(\xi_{k_n})} e^{-h^n_r(0)} \wt{D}_{h^n}^{\xi_{k_n}}(r \cdot ,  r \cdot)
\end{align*}
and note that the random metrics $\wt{D}_{h^n(r \cdot) - h^n_r(0)}^{\xi_{k_n}}$ and $\wt{D}_{h^n}^{\xi_{k_n}}$ have the same law.  Since 
\begin{align*}
r^{-Q(\xi_{k_n})} e^{-h^n_r(0)} \wt{D}_{h^n}^{\xi_{k_n}}(r \cdot ,  r \cdot)
\to e^{-h_r(0)} \wt{D}(r \cdot ,  r \cdot) \quad \text{as} \quad n \to \infty
\end{align*}
with respect to the lower semicontinuous topology,  we obtain that
\begin{align}\label{eqn:scaling_invariance_of_the_limit}
\wt{D} = e^{-h_r(0)} \wt{D}(r \cdot ,  r \cdot) \quad \text{in law}.
\end{align}

Furthermore we fix $\theta \in [0,2\pi)$.  Then \cite[Proposition~1.9]{ding2023uniqueness} implies that
\begin{align*}
\wt{D}_{h^n}^{\xi_{k_n}}(e^{i\theta} \cdot ,  e^{i\theta} \cdot) = \wt{D}_{h^n(e^{i\theta} \cdot)}^{\xi_{k_n}} \quad \text{a.s.}
\end{align*}
Since the random fields $h^n(e^{i\theta} \cdot)$ and $h^n$ have the same law and 
\begin{align*}
\wt{D}_{h^n}^{\xi_{k_n}}(e^{i\theta} \cdot ,  e^{i\theta} \cdot) \to \wt{D}(e^{i\theta} \cdot ,  e^{i\theta} \cdot) \quad \text{as} \quad n \to \infty
\end{align*}
with respect to the lower semicontinuous topology a.s.,  it follows that
\begin{align}\label{eqn:rotational_invariance_of_the_limit}
\wt{D} = \wt{D}(e^{i\theta} \cdot ,  e^{i\theta} \cdot) \quad \text{in law}.
\end{align}

Now we focus on proving the following.

\begin{proposition}\label{prop:limit_is_a_metric_xi_to_infty}
Let $\wt{D}$ be the limit in the statement of Proposition~\ref{prop:lower_semicontinuity_xi_to_infty}.  Then if we assume that $\wt{D}$ satisfies the triangle inequality a.s.,  we have that $\wt{D}$ is a metric a.s.
\end{proposition}

The proof of Proposition~\ref{prop:limit_is_a_metric_xi_to_infty} is essentially the same as the proof of Proposition~\ref{prop:limit_is_a_metric_general_case} by obtaining statements which are analogous to those of Proposition~\ref{prop:annuli_crossing_always_positive_general_case} and Lemmas~\ref{lem:uniform_lower_bound_between_points_general_case}-\ref{lem:crossing_positive_general_case}.  However there is one subtle difference in obtaining the statement which is analogous to that of Lemma~\ref{lem:metric_unbounded_general_case}.  More precisely,  by following the argument in the proof of Lemma~\ref{lem:metric_unbounded_general_case} (see also the proof of Lemma~\ref{lem:metric_unbounded}),  we obtain that there exists a constant $c>0$ such that with high probability (uniformly in $n \in \N$),  we have that
\begin{align*}
\wt{D}_{h^n}^{\xi_{k_n}}(\partial B_{2^k}(0) ,  \partial B_{2^{k+1}}(0)) > c \exp\left(h^n_{2^k}(0) + Q(\xi_{k_n}) \log(2) k\right)
\end{align*}
for many values of $k \in \N$.  Recall that 
\begin{align*}
h^n_{2^k}(0) = B^n_{\log(2) k} \quad \text{for all} \quad n, k \in \N,
\end{align*}
where $B^n$ is a one-dimensional standard Brownian motion.  Thus if we had that $\inf_{n \in \N} Q(\xi_{k_n}) > 0$,  then the proof would work in the same way as the proofs of Lemmas~\ref{lem:metric_unbounded} and ~\ref{lem:metric_unbounded_general_case}.  However we have that $Q(\xi_{k_n}) \to 0$ as $n \to \infty$ (see \cite[Section~2.5]{ding2020tightness}) and so we will follow a slightly different approach.

We start by proving the analogue of Lemma~\ref{lem:uniform_lower_bound_between_points_general_case} (see also Lemma~\ref{lem:uniform_lower_bound_between_points}).

\begin{lemma}\label{lem:uniform_lower_bound_between_points_xi_to_infty}
For all $r>0$ there exist constants $c>0 ,  \epsilon \in (0,\frac{r}{2})$ depending only on $r$ and the sequence $(\xi_n)$ such that
\begin{align*}
\liminf_{n \to \infty}\left( \inf_{x \in \partial B_r(0) ,  y \in \partial B_{2r}(0)}\left(\p\left[\wt{D}_{h^n}^{\xi_{k_n}}(\partial B_{\epsilon}(x) ,  \partial B_{\epsilon}(y) ; B_{3r}(0)) > c \right] \right) \right) > 0.
\end{align*}
\end{lemma}

\begin{proof}
\emph{Step 1.  Outline and setup.}
We will follow an argument which is similar to the argument presented in Lemmas~\ref{lem:uniform_lower_bound_between_points} and ~\ref{lem:uniform_lower_bound_between_points_general_case}.  Recall that
\begin{align}\label{eqn:lower_bound_on_around_annulus}
\p\left[(\wt{D}_{h^n}^{\xi_{k_n}})^{\xi_{k_n}}(\text{around} \,\,  \mathbb{A}_{1,2}(0)) > 1 \right] = 1 - \mathfrak{p} \quad \text{for all} \quad n \in \N
\end{align}
by the choice of our re-normalization (see \eqref{eqn:main_normalization_xi_to_infty} and \eqref{eqn:normalization_xi_to_infty}).  As in the proof of Lemma~\ref{lem:uniform_lower_bound_between_points},  we will choose points $x_1,\cdots,x_N$ on $\partial B_{\frac{3}{2}}(0)$ such that $\max_{1 \leq j \leq N} |x_j - x_{j+1}|$ is sufficiently small.  Then in Step 2,  we will combined \eqref{eqn:lower_bound_on_around_annulus} with the argument in the proof of Lemma~\ref{lem:good_event_occurs_almost_everywhere} to prove that with positive probability (uniformly in $n \in \N$),  the event in \eqref{eqn:lower_bound_on_around_annulus} holds and for all $1 \leq j \leq N$,  we can find $r_{j,n} > 0$ and paths $Q_{j,n}$ such that 
\begin{align*}
Q_{j,n} \subseteq \mathbb{A}_{r_{j,n} ,  2 r_{j,n}}(x_j) \subseteq B_{2 |x_j - x_{j+1}|}(x_j),
\end{align*}
$h_{r_{j,n}}^n(x_j)$ is small,  $Q_{j,n}$ disconnects $\partial B_{r_{j,n}}(x_j)$ from $\partial B_{2 r_{j,n}}(x_j)$ and 
\begin{align*}
\len(Q_{j,n} ; D_{h^n}^{\xi_{k_n}}) \leq \alpha(\xi_{k_n}) r_{j,n}^{\xi_{k_n} Q(\xi_{k_n})} e^{\xi_{k_n} h_{r_{j,n}}^n(x_j)}.
\end{align*}

In Step 3,  for all $1 \leq j \leq n$,  we will let $P_{j,n}$ be a path in $B_{2|x_j-x_{j+1}|}(x_j)$ such that
\begin{align*}
\len(P_{j,n} ; (\wt{D}_{h^n}^{\xi_{k_n}})^{\xi_{k_n}}) \leq  (\wt{D}_{h^n}^{\xi_{k_n}})^{\xi_{k_n}}(Q_{j,n} ,  Q_{j+1,n} ; B_{4 |x_j - x_{j+1}|}(x_j)) + \frac{1}{N^2}.
\end{align*}
Then the concatenation $P_n$ of the paths $P_{1,n},\cdots,P_{N,n},Q_{1,n},\cdots,Q_{N,n}$ disconnects $\partial B_1(0)$ from $\partial B_2(0)$.  Thus since $\len(P_n ; (\wt{D}_{h^n}^{\xi_{k_n}})^{\xi_{k_n}}) > 1$,  we have that there exists $1 \leq j \leq n$ such that
\begin{align}\label{eqn:main_local_ineq}
\wt{D}_{h^n}^{\xi_{k_n}}(Q_{j,n} ,  Q_{j+1,n} ; B_{4 |x_j - x_{j+1}|}(x_j)) \geq 
\left((\wt{D}_{h^n}^{\xi_{k_n}})^{\xi_{k_n}}(Q_{j,n} , Q_{j+1,n} ; B_{4 |x_j - x_{j+1}|}(x_j))\right)^{1/ \xi_{k_n}} > \frac{1}{2N}.
\end{align}

Finally in Step 4,  we will combine \eqref{eqn:main_local_ineq} with Axiom~\ref{it:coordinate_change} and the rotation-invariance property of $\wt{D}_{h^n}^{\xi_{k_n}}$ (see \cite[Proposition~1.9]{ding2023uniqueness}) to complete the proof of the lemma.

\emph{Step 2.  Construction of the paths $Q_{j,n}$.}
Fix $\delta \in (0,1)$ sufficiently small (to be chosen in a universal way) and let $x_1,\cdots,x_N,x_{N+1}$ be distinct points on $\partial B_{\frac{3}{2}}(0)$ ordered in the clockwise way such that $x_{N+1} = x_1$ and 
\begin{align*}
\frac{\delta}{6} < |x_j - x_{j+1}| < \frac{\delta}{4} \quad \text{for all} \quad 1 \leq j \leq N.
\end{align*}
Fix also $q \in (\mathfrak{p} ,  1)$ and let $K \in \N$ be sufficiently large (to be chosen).  Recall that for all $n \in \N$ and all $1 \leq j \leq N$,  the process $\left(h_{e^{-r}}^n(x_j) - h_{\frac{1}{2}}^n(x_j)\right)_{r \geq \log(2)}$ has the law of a standard Brownian motion.

For all $1 \leq j \leq N$,  we set
\begin{align*}
\tau_{j,K}^n:=\inf\{r \geq -\log(\delta / 6) : h_{e^{-r}}^n(x_j) - h_{\frac{1}{2}}^n(x_j) \leq -K\}
\end{align*}
and note that $\tau_{j,K}^n < \infty$ a.s.  We choose $A>0$ sufficiently large (depending only on $q$ and $N$) and $\wt{K} \in \N$ sufficiently large (depending only on $q,N$ and $K$) such that off an event with probability at most $1-q$,  we have that
\begin{align}\label{eqn:event_conditions}
\sup_{\log(2) \leq r \leq \log(2) K} \left |h_{e^{-r - \tau_{j,K}^n}}^n(x_j) - h_{e^{-\tau_{j,K}^n}}^n(x_j) \right | \leq A K^{1/2},\,\,\, \tau_{j,K}^n \leq \wt{K} \log(2),\,\,\, \left |h_{\frac{1}{2}}^n(x_j)\right | \leq A \,\,\, \text{for all} \quad 1 \leq j \leq N.
\end{align}
Recall the definition of the event
\begin{align*}
E_r^{\xi}(z) = \left\{D_h^{\xi}(\text{around} \,\,  \mathbb{A}_{r,2r}(z)) \leq \alpha(\xi) r^{\xi Q(\xi)} e^{\xi h_r(z)}\right\}
\end{align*}
in the proof of Proposition~\ref{prop:bound_on_lqg_distances_of_sets_xi_to_infty} for $z \in \C ,  r>0$.  For all $n \in \N,  \epsilon,\mathfrak{r}>0,  z \in \C$,  we  let also $F_{\mathfrak{r}}^{n,\epsilon}(z)$ denote the event that $E_r^{\xi_{k_n}}(z)$ occurs for some $r \in [\epsilon^2 \mathfrak{r} ,  \epsilon \mathfrak{r}] \cap \{2^{-k} \mathfrak{r}\}_{k \in \N}$.  Note that Axiom~\ref{it:coordinate_change} combined with the proof of Lemma~\ref{lem:good_event_occurs_almost_everywhere} imply that
\begin{align}\label{eqn:high_prob_event}
\p[F_{\mathfrak{r}}^{n,\epsilon}(z)] \geq 1 - O(\epsilon^M) \quad \text{as} \quad \epsilon \to 0 \quad \text{for all} \quad \mathfrak{r}>0,z \in \C,
\end{align}
at a universal rate.  Moreover if we set $\epsilon_K = e^{-\log(2) K / 2}$,  Axiom~\ref{it:locality} implies that $F_{e^{-\tau_{j,K}^n}}^{n,\epsilon_K}(x_j)$ is a.s.  determined by 
\begin{align*}
\left(h^n - h^n_{e^{-\tau_{j,K}^n}}(x_j)\right)|_{B(x_j,e^{-\tau_{j,K}^n})}
\end{align*}
and so combining with \eqref{eqn:high_prob_event} we obtain that
\begin{align}\label{eqn:high_conditional_prob}
\p\left[ F_{e^{-\tau_{j,K}^n}}^{n,\epsilon_K}(x_j) \giv h^n|_{\C \setminus B(x_j ,  e^{-\tau_{j,K}^n})}\right] \geq 1 - O(\epsilon_K^{M}) \quad \text{as} \quad K \to \infty \quad \text{a.s.}
\end{align}

Combining \eqref{eqn:lower_bound_on_around_annulus},  \eqref{eqn:event_conditions} and \eqref{eqn:high_conditional_prob},  we obtain that off an event with probability at most $1 - q + \mathfrak{p} + O(N \epsilon_K^M)$,  we have that $D_{h^n}^{\xi_{k_n}}(\text{around} \,\,  \mathbb{A}_{1,2}(0)) > \alpha(\xi_{k_n})$,  \eqref{eqn:event_conditions} holds and $F_{e^{-\tau_{j,K}^n}}^{n,\epsilon_K}(x_j)$ holds for all $1 \leq j \leq N$.  Suppose that we are working on this event and we choose $K \in \N$ sufficiently large (depending only on $\mathfrak{p},q,N$ and $M$) such that
\begin{align*}
1 - q + \mathfrak{p} + O(N \epsilon_K^M) < 1 - \frac{q-\mathfrak{p}}{2} < 1.
\end{align*}
Then for all $1 \leq j \leq N$,  there exists $r_{j,n} \in (2^{-K-\wt{K}} ,  \frac{\delta}{6})$ such that $E_{r_{j,n}}^{\xi_{k_n}}(x_j)$ occurs.  In particular,  we have that 
\begin{align*}
&D_{h^n}^{\xi_{k_n}}(\text{around} \,\,  \mathbb{A}_{r_{j,n},2r_{j,n}}(x_j)) \leq \alpha(\xi_{k_n}) r_{j,n}^{\xi_{k_n} Q(\xi_{k_n})} e^{\xi_{k_n} h_{r_{j,n}}^n(x_j)},\\
&h_{r_{j,n}}^n(x_j) \leq A K^{1/2} - K + A < -\frac{K}{2}
\end{align*}
for $K$ sufficiently large.  It follows that
\begin{align*}
(\wt{D}_{h^n}^{\xi_{k_n}})^{\xi_{k_n}}(\text{around} \,\,  \mathbb{A}_{r_{j,n},2r_{j,n}}(x_j)) \leq e^{-\xi_{k_n} K / 2}
\end{align*}
and let $Q_{j,n}$ be a path in $\mathbb{A}_{r_{j,n},2r_{j,n}}(x_j)$ such that
\begin{align*}
\len(Q_{j,n} ;  (\wt{D}_{h^n}^{\xi_{k_n}})^{\xi_{k_n}}) \leq e^{-\xi_{k_n} K /2}.
\end{align*}

\emph{Step 3. Constructing a path around $\mathbb{A}_{1,2}(0)$.}
Next we let $P_{j,n}$ be a path in $B_{4|x_j - x_{j+1}|}(x_j)$ such that
\begin{align*}
\len(P_{j,n} ; (\wt{D}_{h^n}^{\xi_{k_n}})^{\xi_{k_n}}) \leq  (\wt{D}_{h^n}^{\xi_{k_n}})^{\xi_{k_n}}(Q_{j,n} ,  Q_{j+1,n} ; B_{4|x_j - x_{j+1}|}(x_j)) + \frac{1}{N^2},
\end{align*}
and let $P_n$ denote the concatenation of the paths $P_{1,n},\cdots,P_{N,n},Q_{1,n},\cdots,Q_{N,n}$.  Then we can choose $\delta \in (0,1)$ sufficiently small (in a universal way) such that $P_n$ disconnects $\partial B_1(0)$ from $\partial B_2(0)$.  Also we have that
\begin{align*}
1 < \len(P_n ; (\wt{D}_{h^n}^{\xi_{k_n}})^{\xi_{k_n}}) \leq &\sum_{j=1}^N \len(P_{j,n} ; (\wt{D}_{h^n}^{\xi_{k_n}})^{\xi_{k_n}}) + \sum_{j=1}^N \len(Q_{j,n} ; (\wt{D}_{h^n}^{\xi_{k_n}})^{\xi_{k_n}})\\
&\leq N e^{-K / 2} + \sum_{j=1}^N \len(P_{j,n} ; (\wt{D}_{h^n}^{\xi_{k_n}})^{\xi_{k_n}})
\end{align*}
and hence we have that 
\begin{align*}
\sum_{j=1}^N \len(P_{j,n} ; (\wt{D}_{h^n}^{\xi_{k_n}})^{\xi_{k_n}}) > 1 - N e^{-K/2}
\end{align*}
if we take $K$ sufficiently large such that $N e^{-K/2} < \frac{1}{2}$.  It follows that there exists $1 \leq j \leq N$ such that
\begin{align*}
\len(P_{j,n} ; (\wt{D}_{h^n}^{\xi_{k_n}})^{\xi_{k_n}}) > \frac{1}{N} - e^{-K/2}.
\end{align*}
Note that
\begin{align*}
&(\wt{D}_{h^n}^{\xi_{k_n}})^{\xi_{k_n}}\left(\partial B_{2^{-K-\wt{K}}}(x_j) ,  
\partial B_{2^{-K-\wt{K}}}(x_{j+1}) ; B_{4|x_j - x_{j+1}|}(x_j)\right)\\
&\geq (\wt{D}_{h^n}^{\xi_{k_n}})^{\xi_{k_n}}(Q_{j,n} ,  Q_{j+1,n} ; B_{4 |x_j - x_{j+1}|}(x_j)) \geq \len(P_{j,n} ; (\wt{D}_{h^n}^{\xi_{k_n}})^{\xi_{k_n}}) > \frac{1}{N} - e^{-K/2} - \frac{1}{N^2}
\end{align*}
and hence
\begin{align*}
&\wt{D}_{h^n}^{\xi_{k_n}}\left(\partial B_{2^{-K-\wt{K}}}(x_j) ,  \partial B_{2^{-K-\wt{K}}}(x_{j+1}) ; B_{4|x_j - x_{j+1}|}(x_j)\right)\\
& \geq \left((\wt{D}_{h^n}^{\xi_{k_n}})^{\xi_{k_n}}\left(\partial B_{2^{-K-\wt{K}}}(x_j) ,  \partial B_{2^{-K-\wt{K}}}(x_{j+1}) ; B_{4|x_j - x_{j+1}|}(x_j)\right)\right)^{1 / \xi_{k_n}} > \frac{1}{2N}.
\end{align*}
It follows that for all $n \in \N$,  there exist deterministic points $x_{j_n},x_{j_n+1}$ such that
\begin{align}\label{eqn:main_prob_lower_bound}
\p\left[\wt{D}_{h^n}^{\xi_{k_n}}\left(\partial B_{2^{-K-\wt{K}}}(x_{j_n}) ,  \partial B_{2^{-K-\wt{K}}}(x_{j_n + 1}) ; B_{4 |x_{j_n} - x_{j_n + 1}|}(x_{j_n})\right) > \frac{1}{2N} \right] > \frac{q-\mathfrak{p}}{2N}.
\end{align}

\emph{Step 4.  Completion of the proof.}
Finally we will complete the proof of the lemma.
Set $d_n:=|x_{j_n}-x_{j_n+1}| \in (\frac{\delta}{6},\frac{\delta}{4})$.  Then \eqref{eqn:main_prob_lower_bound} combined with Axiom~\ref{it:coordinate_change} and the rotation-invariance of $\wt{D}_{h^n}^{\xi_{k_n}}$ (see \cite[Proposition~1.9]{ding2023uniqueness}) imply that
\begin{align*}
&\p\left[\wt{D}_{h^n}^{\xi_{k_n}}\left(\partial B_{2^{-K-\wt{K}}}(x_{j_n}) ,  \partial B_{2^{-K-\wt{K}}}(x_{j_n + 1}) ; B_{4 d_n}(x_{j_n})\right) > \frac{1}{2N}\right]\\
& = \p\left[\wt{D}_{h^n}^{\xi_{k_n}}\left(\partial B_{\frac{2^{-K-\wt{K}}}{d_n}}(0) ,  \partial B_{\frac{2^{-K-\wt{K}}}{d_n}}(1) ; B_4(0) \right) > \frac{1}{2N}e^{-h_{d_n}^n(0)} d_n^{-Q(\xi_{k_n})}\right] > \frac{q - \mathfrak{p}}{2N}.
\end{align*}
Thus we obtain that there exist constants $s,c \in (0,\frac{1}{2})$ such that
\begin{align}\label{eqn:lower_bound_on_distances_macroscopic_scale}
\p\left[\wt{D}_{h^n}^{\xi_{k_n}}(\partial B_s(0) ,  \partial B_s(1) ; B_4(0))>c \right] \geq \frac{q-\mathfrak{p}}{2N} \quad \text{for all} \quad n \in \N.
\end{align}

Now fix $r>0$ and $x \in \partial B_r(0),  y \in \partial B_{2r}(0)$.  Then combining Axiom~\ref{it:coordinate_change} with \cite[Proposition~1.9]{ding2023uniqueness},  and since $B_{3r}(-x) \subseteq B_{4r}(0)$ and $|y-x| \geq r$,  we obtain that if $\theta \in [0,2\pi)$ is such that $y-x = e^{i\theta} |y-x|$,  then for all $\epsilon \in (0,\frac{r}{2})$,  we have that
\begin{align}\label{eqn:lower_bound_on_distances_microscopic_scales}
&\p\left[\wt{D}_{h^n}^{\xi_{k_n}}(\partial B_{\epsilon}(x) ,  \partial B_{\epsilon}(y) ; B_{3r}(0)) > c\right] \geq \p\left[\wt{D}_{h^n(\cdot + x)}^{\xi_{k_n}}(\partial B_{\epsilon}(0) , \partial B_{\epsilon}(y-x) ; B_{4r}(0)) > c\right]\notag \\
&=\p\left[\wt{D}_{h^n(|y-x|\cdot +x)}^{\xi_{k_n}}\left(\partial B_{\frac{\epsilon}{|y-x|}}(0) ,  \partial B_{\frac{\epsilon}{|y-x|}}\left(\frac{y-x}{|y-x|}\right) ; B_{\frac{4r}{|y-x|}}(0) \right) > c \right] \notag \\
&=\p\left[\wt{D}_{h^n(|y-x| e^{i\theta} \cdot + x)}^{\xi_{k_n}}\left(\partial B_{\frac{\epsilon}{|y-x|}}(0) ,  \partial B_{\frac{\epsilon}{|y-x|}}(1) ; B_{\frac{4r}{|y-x|}}(0)\right) > c\right] \notag \\
&\geq \p\left[\wt{D}_{h^n(|y-x| e^{i\theta} \cdot + x)}^{\xi_{k_n}}\left(\partial B_{\frac{\epsilon}{|y-x|}}(0) ,  \partial B_{\frac{\epsilon}{|y-x|}}(1) ; B_{4}(0)\right) > c\right] \notag \\
&\geq \p\left[\wt{D}_{h^n(|y-x| e^{i\theta} \cdot + x)}^{\xi_{k_n}}\left(\partial B_{\frac{\epsilon}{r}}(0) ,  \partial B_{\frac{\epsilon}{r}}(1) ; B_{4}(0)\right) > c\right] \notag \\
&= \p\left[\wt{D}_{h^n(|y-x| e^{i\theta} \cdot + x) - h_{|y-x|}^n(x)}^{\xi_{k_n}}\left(\partial B_{\frac{\epsilon}{r}}(0) ,  \partial B_{\frac{\epsilon}{r}}(1) ; B_{4}(0)\right) > c e^{-h_{|y-x|}^n(x)}\right].
\end{align}
Note that the random fields $h^n(|y-x| e^{i\theta} \cdot + x) - h_{|y-x|}^n(x)$ and $h^n$ have the same law for all $n \in \N$.  Thus the proof of the lemma is complete by combining \eqref{eqn:lower_bound_on_distances_macroscopic_scale} and \eqref{eqn:lower_bound_on_distances_microscopic_scales} with $\epsilon = s r$ and the fact that $(z,t) \to h_t^n(z)$ converges to $(z,t) \to h_t(z)$ locally uniformly on $\C \times (0,\infty)$ as $n \to \infty$ a.s.  by the choice of our coupling.
\end{proof}

Next we state and prove the analogue of Lemma~\ref{lem:uniform_lower_bound_across_annuli_general_case} (see also Lemma~\ref{lem:uniform_lower_bound_across_annuli}).

\begin{lemma}\label{lem:uniform_lower_bound_across_annuli_xi_to_infty}
For all $r>0$,  there exists $c'>0$ depending only on $r$ and the sequence $(\xi_n)$ such that
\begin{align*}
\liminf_{n \to \infty} \left(\p\left[\wt{D}_{h^n}^{\xi_{k_n}}(\partial B_r(0) ,  \partial B_{2r}(0)) > c' \right] \right) > 0.
\end{align*}
\end{lemma}

\begin{proof}
Fix $r>0$ and let $\epsilon,c$ be the constants in Lemma~\ref{lem:uniform_lower_bound_between_points_xi_to_infty} corresponding to $r$.   Set 
\begin{align*}
A:=\inf_{n \in \N} \left(\inf_{x \in \partial B_r(0) ,  y \in \partial B_{2r}(0)} \left(\p\left[\wt{D}_{h^n}^{\xi_{k_n}}(\partial B_{\epsilon}(x) ,  \partial B_{\epsilon}(y) ; B_{3r}(0)) > c \right] \right) \right).
\end{align*}
As in the proof of Lemma~\ref{lem:uniform_lower_bound_across_annuli},  we fix $m \in \N$ and set 
\begin{align*}
X:=\partial B_r(0) \cap (r 2^{-m} \Z^2),\quad Y:=\partial B_{2r}(0) \cap (r 2^{-m} \Z^2).
\end{align*}
We choose $m \in \N$ sufficiently large such that $r 2^{-m} < \epsilon$ and 
\begin{align*}
\partial B_r(0) \subseteq \bigcup_{x \in X} B_{r 2^{-m}}(x),\quad \partial B_{2r}(0) \subseteq \bigcup_{y \in Y} B_{r 2^{-m}}(y).
\end{align*}
Then we have that
\begin{align}\label{eqn:main_inclusion_ineq}
\p\left[\wt{D}_{h^n}^{\xi_{k_n}}(\partial B_r(0) ,  \partial B_{2r}(0)) > c \right] \geq \p\left[\bigcap_{x \in X ,  y \in Y} \left\{\wt{D}_{h^n}^{\xi_{k_n}}(\partial B_{r 2^{-m}}(x) ,  \partial B_{r 2^{-m}}(y) ; B_{3r}(0)) > c \right \} \right] \quad \text{for all} \quad n \in \N.
\end{align}

Note that Lemma~\ref{lem:uniform_lower_bound_between_points_xi_to_infty} implies that
\begin{align}\label{eqn:main_lower_bound}
\p\left[\wt{D}_{h^n}^{\xi_{k_n}}(\partial B_{r 2^{-m}}(x) ,  \partial B_{r 2^{-m}}(y) ; B_{3r}(0)) > c \right] \geq A > 0 \quad \text{for all} \quad (x,y) \in X \times Y.
\end{align}
Therefore the proof of the lemma is complete by using \eqref{eqn:main_inclusion_ineq},  \eqref{eqn:main_lower_bound} and arguing as in Step 3 in the proof of Lemma~\ref{lem:uniform_lower_bound_between_points}.
\end{proof}

Before stating and proving the analogue of Lemma~\ref{lem:metric_unbounded_general_case},  we state and prove the following elementary lemma about the one-dimensional standard Brownian motion that we are going to use.

\begin{lemma}\label{lem:brownian_motion_lemma}
Let $B=(B_t)_{t \in [0,1]}$ be a standard one-dimensional Brownian motion starting from zero.  Fix $p \in (0,1)$.  Then there exist $\epsilon ,  \delta \in (0,1)$ and $K_0 \in \N$ depending only on $p$ such that the following holds for all $K \in \N$ with $K \geq K_0$.  With probability at least $p$,  we have that there are at least $\epsilon K$ number of $j \in [1,K-1]_{\Z}$ such that $B_t > 
\delta$ for all $t \in [\frac{j}{K} ,  \frac{j+1}{K}]$.
\end{lemma}

\begin{proof}
First we note that $\text{Leb}(\{t \in [0,1] : B_t > 0\}) > 0$ a.s.  and that
\begin{align*}
\text{Leb}(\{t \in [0,1] : B_t > \delta\}) \to \text{Leb}(\{t \in [0,1] : B_t > 0\}) \quad \text{as} \quad \delta \to 0 \quad \text{a.s.}
\end{align*}
Thus there exist $\epsilon ,  \delta \in (0,1)$ depending only on $p$ such hat with probability at least $1 - \frac{1-p}{2}$,  we have that
\begin{align}\label{eqn:lower_bound_on_lebesgue_measure}
\text{Leb}(\{t \in [0,1] : B_t > \delta \}) > \epsilon.
\end{align}

Fix $K \in \N$ and suppose that \eqref{eqn:lower_bound_on_lebesgue_measure} holds.  We claim that there are at least $\epsilon K$ number of $j \in [1,K-1]_{\Z}$ with the property that there exists $t \in [\frac{j}{K} ,  \frac{j+1}{K}]$ such that $B_t > \delta$.  Indeed suppose that this is not the case.  Then the set $\{t \in [0,1] : B_t > \delta\}$ can be covered by at most $\epsilon K$ number of intervals of the form $[\frac{j}{K} ,  \frac{j+1}{K}]$ for $j \in [1,K-1]_{\Z}$.  Thus we have that
\begin{align*}
\text{Leb}(\{t \in [0,1] : B_t > \delta\}) \leq \epsilon
\end{align*}
and that contradicts \eqref{eqn:lower_bound_on_lebesgue_measure}.  This proves the claim.

By the almost sure absolute continuity of $B$,  we obtain that there exists $K_0 \in \N$ sufficiently large (depending only on $p$ and $\delta$) such that for all $K \in \N$ with $K \geq K_0$,  we have with probability at least $1 - \frac{1-p}{2}$ that
\begin{align}\label{eqn:absolute_continuity}
|B_t - B_s| < \frac{\delta}{2} \quad \text{for all} \quad t,s \in [0,1] \quad \text{with} \quad |t-s| < \frac{1}{K}.
\end{align}
Hence for all $K \in \N$ with $K \geq K_0$,  we have with probability at least $p$ that \eqref{eqn:lower_bound_on_lebesgue_measure} and \eqref{eqn:absolute_continuity} hold.  Suppose that \eqref{eqn:lower_bound_on_lebesgue_measure} and \eqref{eqn:absolute_continuity} hold and let $j \in [1,K-1]_{\Z}$ be such that $B_t > \delta$ for some $t \in [\frac{j}{K} ,  \frac{j+1}{K}]$.  Recall that the previous paragraph implies that the number of $j \in [1,K-1]_{\Z}$ with this property is at least $\epsilon K$.  Then \eqref{eqn:absolute_continuity} implies that $B_s > \frac{\delta}{2}$ for all $s \in [\frac{j}{K} ,  \frac{j+1}{K}]$.  Therefore we have completed the proof of the lemma.
\end{proof}

Next we state and prove the analogue of Lemma~\ref{lem:metric_unbounded_general_case}.

\begin{lemma}\label{lem:metric_unbounded_xi_to_infty_case}
For all $r,T>0$,  we have that
\begin{align*}
\lim_{R \to \infty} \left(\liminf_{n \to \infty} \left(\p\left[\wt{D}_{h^n}^{\xi_{k_n}}(\partial B_r(0) ,  \partial B_R(0)) > T\right] \right) \right)= 1.
\end{align*}
\end{lemma}

\begin{proof}
We will prove the claim of the lemma only in the case that $r = 1$ since the exact same argument works for general $r>0$.  The argument will be similar with the arguments in the proofs of Lemmas~\ref{lem:metric_unbounded} and ~\ref{lem:metric_unbounded_general_case}.  

Fix $q \in (0,1)$ and let $\delta ,  \epsilon \in (0,1)$ be the constants in Lemma~\ref{lem:brownian_motion_lemma} corresponding to $q$.  Note that 
Lemma~\ref{lem:uniform_lower_bound_across_annuli_xi_to_infty} implies that there exist $1<s_1<s_2,  c ,  p \in (0,1)$ such that
\begin{align}\label{eq:bounded_crossing_prob}
\p\left[\wt{D}_{h^n}^{\xi_{k_n}}(\partial B_{s_1}(0), \partial B_{s_2}(0)) > c \right] \geq p \quad \text{for all} \quad n \in \N.
\end{align}
Fix $t \in (0,s_2^{-1})$ and set $r_k:=t^k$ for all $k \in \N$.  Then \eqref{eq:bounded_crossing_prob} combined with Axioms~\ref{it:weyl_scaling} and ~\ref{it:translation_and_scale_invariance} imply that 
\begin{align}\label{eq:bounded_scaled_crossing_prob}
\p\left[\wt{D}_{h^n}^{\xi_{k_n}}\left(\partial B_{\frac{s_1}{r_k}}(0) ,  \partial B_{\frac{s_2}{r_k}}(0)\right) > c r_k^{-Q(\xi_{k_n})} \exp\left(h^n_{r_k^{-1}}(0)\right) \right] \geq p \quad \text{for all} \quad n, k  \in \N.
\end{align}
We consider the field $\wt{h}^n := h^n \circ \phi$ where $\phi(z) = -\frac{1}{z}$ for all $z \in \C \cup \{\infty\}$,  and note that $\wt{h}^n$ has the law of a whole-plane GFF normalized so that its average on $\partial B_1(0)$ is equal to zero.  Also we consider the event 
\begin{align*}
E_{r_k}^n:=\left\{\wt{D}_{h^n}^{\xi_{k_n}}\left(\partial B_{\frac{s_1}{r_k}}(0) ,  \partial B_{\frac{s_2}{r_k}}(0)\right) > c r_k^{-Q(\xi_{k_n})} \exp\left(h_{r_k^{-1}}^n(0)\right)\right\} \quad \text{for all} \quad n,k \in \N.
\end{align*}
Note that Axiom~\ref{it:locality} implies that $E_{r_k}^n$ is a.s.  determined by $\left(\wt{h}^n - \wt{h}_{r_k}^n(0)\right)|_{\mathbb{A}_{\frac{r_k}{s_2} ,  \frac{r_k}{s_1}}(0)}$ for all $n,k \in \N$.  

We will follow an argument which is similar to that in the proof of \cite[Lemma~3.1]{gwynne2020local}.  Let us introduce some notation first.  For $r>0$,  we define the $\sigma$-algebra
\begin{align*}
\mathcal{F}_r:=\sigma\left(\left(\wt{h}^n - \wt{h}_r^n(0)\right)|_{\C \setminus B_r(0)}\right)
\end{align*}
and note that \cite[Lemma~3.2]{gwynne2020local} implies that $\mathcal{F}_{r'}^n \subseteq \mathcal{F}_r^n$ for all $r' > r > 0$.  Moreover we note that for all $r>0$,  we can write 
\begin{align*}
(\wt{h}^n - \wt{h}_r^n(0))|_{\C \setminus B_r(0)} = \wt{h}^{0,n,r} + \wt{\Fh}^{n,r},
\end{align*}
where $\wt{h}^{0,n,r}$ is a zero-boundary GFF on $B_r(0)$ and $\wt{\Fh}^{n,r}$ is harmonic on $B_r(0)$.  Furthermore,  for all $0<r<R<\infty$,  we set 
\begin{align*}
\mathcal{M}_r^{n,R}:=\sup_{z \in B_r(0)} |\wt{\Fh}^{n,R}(z) - \wt{\Fh}^{n,R}(0)|.
\end{align*}
We also set $s':=\frac{1+s_1^{-1}}{2} \in (0,1)$ and for $r>0$,  we let $H_r^n$ denote the Radon-Nikodym derivative of the conditional law of $(\wt{h}^n - \wt{h}_r^n(0))|_{B_{s_1^{-1} r}(0)}$ given $\mathcal{F}_r^n$ with respect to the marginal law.  We also fix a parameter $M \geq 1$ which we will choose eventually sufficiently close to $1$,  depending only on $s_1,s_2,\epsilon,p,$ and $q$.

Note that \cite[Lemma~3.3]{gwynne2020local} implies that for all $\alpha>0$,  there exists $C \in (0,\infty)$ depending only on $\alpha,s_1,s'$ and $M$ such that a.s.  on the event that $\{\mathcal{M}_{s' r}^{n,r} \leq M\}$,  we have that
\begin{align}\label{eq:R_N_derivative_moment_bounds}
\max\left\{\E\left[(H_r^n)^{\alpha} \giv \mathcal{F}_r^n \right] ,  \E\left[(H_r^n)^{-\alpha} \giv \mathcal{F}_r^n \right] \right\} \leq C.
\end{align}
Thus combining \eqref{eq:R_N_derivative_moment_bounds} with H\"older inequality imply that there exists $p_M = p_M(s_1,s_2,p,M) \in (0,1)$ such that
\begin{align}\label{eq:conditional_prob_lower_bound}
\p\left[E_{r_k}^n \giv \mathcal{F}_{r_k}^n \right] \geq p_M \quad \text{a.s.} \quad \text{on} \quad \{\mathcal{M}_{s' r_k}^{r_k} \leq M\}.
\end{align}

For $j \in \N$,  we let $n_j$ be the $j$th smallest $k \in \N$ for which 
\begin{align*}
\mathcal{M}_{s' r_k}^{r_k} \leq M \quad \text{and} \quad \wt{h}_{r_k}^n(0) = B_{-\log(t) k} \geq \delta k^{1/2},
\end{align*}
where $B = (B_t)_{t \geq 0}$ is an one-dimensional standard Brownian motion.  Let also $\mathcal{N}_M(K)$ denote the number of $k \in [0,K-1]_{\Z}$ such that $\mathcal{M}_{s' r_k}^{r_k} \leq M$.  Then \cite[Lemma~3.4]{gwynne2020local} implies that there exists a constant $c_0 \in (0,\infty)$ depending only on $\epsilon$ and $s'$ such that if we choose $M>1$ sufficiently large (depending only on $\epsilon$ and $s'$),  we have that
\begin{align}\label{eq:many_occurences}
\p\left[\mathcal{N}_M(K) \leq \left(1-\frac{\epsilon}{2}\right) K \right] \leq c_0 e^{-K} \quad \text{for all} \quad K \in \N.
\end{align}
Moreover Lemma~\ref{lem:brownian_motion_lemma} combined with Brownian scaling imply that there exists $K_0 \in \N$ sufficiently large (depending only on $q$) such that for all $K \geq K_0$,  we have with probability at least $q$ that $B_{-\log(t) k} > \delta K^{1/2}$ for at least $\epsilon K$ number of $k \in [1,K-1]_{\Z}$.  Thus combining with \eqref{eq:many_occurences},  we obtain that there exists $j_0 \in \N$ depending only on $\epsilon$ and $q$ such that
\begin{align}\label{eq:number_of_upcrossings}
\p\left[n_j \leq \frac{2j}{\epsilon}\right ] \geq q - c_0 e^{-\frac{2j}{\epsilon}} \quad \text{for all} \quad j \geq j_0.
\end{align}

Finally for $K \in \N$,  we let $\mathcal{N}^n(K)$ be the number of $k \in [0,K-1]_{\Z}$ for which $E_{r_k}^n$ occurs.  Note also that \eqref{eq:conditional_prob_lower_bound} implies that $\mathcal{N}^n(n_j)$ stochastically dominates a binomial distribution with $j$ trials and success probability $p_M$.  Then \cite[Lemma~3.5]{gwynne2020local} implies that there exists a constant $A \in (0,\infty)$ depending only on $p_M$ such that
\begin{align}\label{eq:conditionally_many_crossings}
\p\left[\mathcal{N}^n(n_j) > \frac{j p_M}{2} \giv n_j \leq \frac{2j}{\epsilon} \right] \geq 1 -e^{-A j } \quad \text{for all} \quad j \geq j_0.
\end{align}
Therefore combining \eqref{eq:number_of_upcrossings} with \eqref{eq:conditionally_many_crossings},  we obtain that 
\begin{align}\label{eq:final_inequality}
\p\left[\mathcal{N}^n(n_j) > \frac{j p_M}{2} ,  n_j \leq \frac{2j}{\epsilon} \right] \geq \left(q - c_0 e^{-\frac{2j}{\epsilon}}\right) \left(1 - e^{-A j}\right) \quad \text{for all} \quad j \geq j_0.
\end{align}
Note that if the event in \eqref{eq:final_inequality} occurs,  then combining with \eqref{eqn:monotonicity_property} in Remark~\ref{rem:metric_well_defined} we obtain that
\begin{align*}
\wt{D}_{h^n}^{\xi_{k_n}}\left(\partial B_1(0) ,   \partial B_{t^{-\frac{2 j}{\epsilon}}}(0)\right) \geq c e^{\delta j^{1/2}}
\end{align*}
which implies that
\begin{align*}
\liminf_{j \to \infty} \left(\liminf_{n \to \infty} \left(\p\left[\wt{D}_{h^n}^{\xi_{k_n}}\left(\partial B_1(0) ,  \partial B_{\frac{2j}{\epsilon}}(0)\right) \geq c e^{\delta j^{1/2}} \right]\right)\right) \geq q.
\end{align*}
Therefore using \eqref{eqn:monotonicity_property} in Remark~\ref{rem:metric_well_defined} again,  we obtain that
\begin{align*}
\lim_{R \to \infty} \left(\liminf_{n \to \infty} \left(\p\left[\wt{D}_{h^n}^{\xi_{k_n}}(\partial B_1(0) ,  \partial B_R(0)) > T\right] \right) \right)\geq q \quad \text{for all} \quad T>0.
\end{align*}
This completes the proof of the lemma since $q \in (0,1)$ was arbitrary.
\end{proof}

\begin{proof}[Proof of Proposition~\ref{prop:limit_is_a_metric_xi_to_infty}.]
The proof is essentially the same as the proof of Proposition~\ref{prop:limit_is_a_metric_general_case}.  More precisely,  we use Lemma~\ref{lem:metric_unbounded_xi_to_infty_case} in order to obtain analogous statements of Lemmas~\ref{lem:connectivity_lemma_general_case} and ~\ref{lem:crossing_positive_general_case}.  In particular we can show by arguing as in the proof of Lemma~\ref{lem:crossing_positive_general_case} that
\begin{align}\label{eqn:crossing_positive_xi_to_infty}
\p\left[\wt{D}(\partial B_r(0) ,  \partial B_{2r}(0)) > 0 \right] \in \{0,1\} \quad \text{for all} \quad r>0.
\end{align}
Hence using Lemma~\ref{lem:uniform_lower_bound_across_annuli_xi_to_infty} and arguing as in the proof of Proposition~\ref{prop:annuli_crossing_always_positive_general_case},  we obtain that
\begin{align}\label{eqn:crossing_assymptotically_positive}
\lim_{\epsilon \to 0} \liminf_{n \to \infty} \left(\p\left[\wt{D}_{h^n}^{\xi_{k_n}}(\partial B_{r_1}(0) ,  \partial B_{r_2}(0)) > \epsilon \right]\right) = 1 \quad \text{for all} \quad 0 < r_1 < r_2 < \infty.
\end{align}
Therefore the proof of the proposition is complete by using \eqref{eqn:crossing_assymptotically_positive} and arguing as in the proof of Proposition~\ref{prop:limit_is_a_metric_general_case}.
\end{proof}

Now we focus on proving that $\wt{D}$ is non-trivial a.s.  First we will show that $\wt{D} \not \equiv 0$ a.s.  This is the content of the following lemma.

\begin{lemma}\label{lem:limit_not_trivial_almost_surely}
Fix $z,w  \in \C$ distinct and deterministic points.  Then we have that $\wt{D}(z,w) > 0$ a.s.
\end{lemma}

\begin{proof}
First we note that combining \eqref{eqn:translation_invariance_of_the_limit},  \eqref{eqn:scaling_invariance_of_the_limit} and \eqref{eqn:rotational_invariance_of_the_limit},  we obtain that it suffices to show that $\wt{D}(0,1) > 0$ a.s.  From now on,  we will focus on proving the latter claim.

Note that Lemma~\ref{lem:uniform_lower_bound_across_annuli_xi_to_infty} implies that there exist $0 < s_1 < s_2 < 1,  c,p \in (0,1)$ such that
\begin{align}\label{eqn:uniform_crossing_distance}
\p\left[\wt{D}_{h^n}^{\xi_{k_n}}(\partial B_{s_1}(0) ,  \partial B_{s_2}(0)) > c \right] \geq p \quad \text{for all} \quad n \in \N.
\end{align}
Fix $t \in (0,s_1)$ and set $r_k:=t^k$ for all $k \in \N_0$.  We consider the event
\begin{align*}
E_{r_k}^n:=\left\{\wt{D}_{h^n}^{\xi_{k_n}}(\partial B_{s_1 r^k}(0) ,  \partial B_{s_2 r_k}(0)) > c r_k^{Q(\xi_{k_n})} e^{h^n_{r_k}(0)}\right\}.
\end{align*}
Then combining \eqref{eqn:uniform_crossing_distance} with the Axioms of strong LQG metrics,  we obtain that
\begin{align*}
\p\left[E_{r_k}^n\right] \geq p \quad \text{for all} \quad n,k \in \N.
\end{align*}
Let $\mathcal{N}^n(K)$ denote the number of $k \in [1,K]_{\Z}$ for which $E_{r_k}^n$ occurs.  Then arguing as in the proof of Lemma~\ref{lem:metric_non_trivial} and using \cite[Lemma~3.1]{gwynne2020local},  we obtain that there exist constants $a>0,b \in (0,1),\wt{c}>0$ such that
\begin{align}\label{eqn:limit_non_trivial_with_very_high_prob}
\p\left[\wt{D}(0,1)\right] \geq \p\left[\limsup_{n \to \infty} \{\mathcal{N}^n(K) \geq b K\}\right] \geq 1 - \wt{c} e^{-aK} \quad \text{for all} \quad K \in \N.
\end{align}
Therefore the proof of the lemma is complete by taking $K \to \infty$ in \eqref{eqn:limit_non_trivial_with_very_high_prob}.
\end{proof}

Next we show that $\wt{D} \not \equiv \infty$ a.s.

\begin{lemma}\label{lem:limit_finite_xi_to_infty}
It holds that $\wt{D} \not \equiv \infty$ a.s.
\end{lemma}

\begin{proof}
Fix $K_1,K_2 \subseteq \C$ disjoint compact and connected sets which are both not singletons.  Since $\wt{D}_{h^n}^{\xi_{k_n}}$ converges to $\wt{D}$ as $n \to \infty$ with respect to the lower semicontinuous topology of functions,  we obtain that
\begin{align*}
\wt{D}(K_1,K_2) \leq \liminf_{n \to \infty} \wt{D}_{h^n}^{\xi_{k_n}}(K_1,K_2).
\end{align*}

Fix $\epsilon \in (0,1)$ and note that Proposition~\ref{prop:bound_on_lqg_distances_of_sets_xi_to_infty} implies that there exists $A>0$ sufficiently large such that
\begin{align*}
\p\left[\wt{D}_{h^n}^{\xi_{k_n}}(K_1,K_2) \geq A \right] \leq \epsilon \quad \text{for all} \quad n \in \N.
\end{align*}
Thus we obtain that
\begin{align*}
\p\left[\wt{D}(K_1,K_2) \geq 2A \right] \leq \p\left[\liminf_{n \to \infty} \left\{\wt{D}_{h^n}^{\xi_{k_n}}(K_1,K_2) \geq A \right \}\right] \leq \liminf_{n \to \infty} \p\left[\wt{D}_{h^n}^{\xi_{k_n}}(K_1,K_2) \geq A \right] \leq \epsilon.
\end{align*}
Since $\epsilon \in (0,1)$ was arbitrary,  we obtain that $\wt{D}(K_1,K_2) < \infty$ a.s.  In particular,  we have that $\wt{D} \not \equiv \infty$ a.s.  This completes the proof of the lemma. 
\end{proof}

Finally we are ready to prove Theorem~\ref{thm:main_thm_xi_to_infty_intro}.

\begin{proof}[Proof of Theorem~\ref{thm:main_thm_xi_to_infty_intro}]
Fix a sequence $(\xi_n)_{n \in \N}$ in $[1,\infty)$ such that $\xi_n \to \infty$ as $n \to \infty$.  Then Proposition~\ref{prop:lower_semicontinuity_xi_to_infty} implies that $(\wt{D}_h^{\xi_n})_{n \in \N}$ is a tight sequence of metrics with respect to the lower semicontinuous topology of functions.  Let $\wt{D}$ be any subsequential limit in law of $(\wt{D}_h^{\xi_n})_{n \in \N}$.  Then combining Lemmas~\ref{lem:limit_not_trivial_almost_surely} and ~\ref{lem:limit_finite_xi_to_infty},  we obtain that $\wt{D}(z,w) > 0$ a.s.  for any fixed distinct and deterministic points $z,w \in \C$,  and $\wt{D} \not \equiv \infty$ a.s.  Finally Proposition~\ref{prop:limit_is_a_metric_xi_to_infty} implies that if in addition we assume that $\wt{D}$ satisfies the triangle inequality a.s.,  then we have that $\wt{D}$ is a metric a.s.   This completes the proof of the theorem.
\end{proof}

\section{Appendix}
\label{sec:apendix}

In this section,  we record two facts about LQG metrics and convergence of random variables taking values on complete and separable metric spaces.  Lemma~\ref{lem:normalization_well_defined} shows that the way that we have re-normalized the metrics $D_h^{\xi}$ for $\xi>0$ in Theorems~\ref{thm:main_theorem_subcritical_intro}-\ref{thm:main_thm_xi_to_infty_intro} is well-defined.  Lemma~\ref{lem:convergence_in_prob} allows us to deduce convergence in probability of random variables under certain assumptions.

\begin{lemma}\label{lem:normalization_well_defined}
Let $h$ be a whole-plane GFF normalized such that $h_1(0) = 0$.  Fix $\xi$ and let $D_h^{\xi}$ denote the limit in probability as $\epsilon \to 0$ of the random metrics $a_{\epsilon}^{-1} D_h^{\epsilon}$ as in Theorem~\ref{thm:convergence_of_supercritical_lfpp}.  Then for all $x>0$,  we have that
\begin{align*}
 \p\left[D_h^{\xi}(\text{around} \,\,  \mathbb{A}_{1,2}(0)) = x \right] = \p\left[D_h^{\xi}(0,1) = x\right] = 0.
\end{align*}
\end{lemma}

\begin{proof}
Fix $x>0$.  We will only show that 
\begin{align*}
\p\left[D_h^{\xi}(\text{around} \,\,  \mathbb{A}_{1,2}(0)) = x \right] = 0,
\end{align*}
since the exact same argument works for $D_h^{\xi}(0,1)$.  We will argue by contradiction and suppose that
\begin{align}\label{eqn:contradiction_hypothesis}
\p\left[D_h^{\xi}(\text{around} \,\,  \mathbb{A}_{1,2}(0)) = x \right] > 0.
\end{align}

Fix any $y>0$ and let $z \in \R$ be such that $x e^{-\xi(\gamma) z} = y$.  Let also $\phi$ be a smooth and compactly supported function such that $\phi|_{\mathbb{A}_{1,2}(0)} \equiv z$.  Then since the random variables $D_h^{\xi}(\text{around}\,\,\mathbb{A}_{1,2}(0))$ and $D_{h+\phi}^{\xi}(\text{around} \,\,\mathbb{A}_{1,2}(0))$ are a.s.  determined by $h|_{\mathbb{A}_{1,2}(0)}$ and $(h+\phi)|_{\mathbb{A}_{1,2}(0)}$ respectively and $h+\phi = h+ z$ on $\mathbb{A}_{1,2}(0)$,  we obtain by combining with Weyl's scaling that
\begin{align}\label{eqn:change_of_fields}
\p\left[D_{h+\phi}^{\xi}(\text{around} \,\,\mathbb{A}_{1,2}(0)) = x \right] = \p\left[D_h^{\xi}(\text{around} \,\,\mathbb{A}_{1,2}(0)) =y\right].
\end{align}

Note that \cite[Proposition~3.4]{ms2016imag1} implies that the laws of $h|_{\mathbb{A}_{1,2}(0)}$ and $(h+\phi)|_{\mathbb{A}_{1,2}(0)}$ are mtutually absolutely continuous and so combining with \eqref{eqn:contradiction_hypothesis} and \eqref{eqn:change_of_fields},  we obtain that
\begin{align*}
\p\left[D_h^{\xi}(\text{around} \,\,\mathbb{A}_{1,2}(0)) = y \right] > 0.
\end{align*}
In particular,  we have that the function
\begin{align*}
F(y):=\p\left[D_h^{\xi}(\text{around} \,\,\mathbb{A}_{1,2}(0)) \leq y\right]
\end{align*}
is discontinuous for all $y>0$.  But that's a contradiction since $F$ is non-decreasing and the set of discontinuity points of a non-decreasing function is at most countable.  It follows that
\begin{align*}
\p\left[D_h^{\xi}(\text{around}\,\,\mathbb{A}_{1,2}(0)) = x \right] = 0 \quad \text{for all} \quad x >0.
\end{align*}
This completes the proof of the lemma.
\end{proof}

\begin{lemma}(\cite[Lemma~1.3]{dubedat2020weak})
\label{lem:convergence_in_prob}
Let $(\Omega_1,d_1)$ and $(\Omega_2,d_2)$ be complete separable metric spaces.  Let $X$ be a random variable taking values in $\Omega_1$ and let $\{Y^n\}_{n \in \N}$ and $Y$ be random variables taking values in $\Omega_2$,  all defined on the same probability space,  such that $(X,Y^n) \to (X,Y)$ as $n \to \infty$ in law.  If $Y$ is a.s.  determined by $X$,  then $Y^n \to Y$ as $n \to \infty$ in probability.
\end{lemma}

\bibliographystyle{abbrv}
\bibliography{biblio}

\end{document}